\documentclass[10pt]{report}

\usepackage{thesistemplate}

\begin{document}

\addtocontents{toc}{~\hfill Page\par}


\begin{titlepage}
\thispagestyle{empty}\enlargethispage{\the\footskip}%
\begin{center}
	{\setstretch{1.66} {Complexity and Avoidance}\par }%
	\vskip.4in
	By
	\vskip .3in
	{Hayden Robert Jananthan}
	\vskip .3in
	
	Dissertation\\
		Submitted to the Faculty of the \\
		Graduate School of Vanderbilt University \\
		in partial fulfillment of the requirements \\
		for the degree of \vspace{.2in}
	
DOCTOR OF PHILOSOPHY \vspace{.1in}
	
	in \vspace{.1in}
	
	{Mathematics} \vspace{.2in}
	
	June 30th, 2021 \vspace{.2in}
	
	Nashville, Tennessee
\end{center}

\vfill


\begin{center}
Approved,

Douglas H. Fisher, Ph.D. \vspace{.2in}

Alexander Y. Olshanskii, Ph.D. \vspace{.2in}

Denis V. Osin, Ph.D. \vspace{.2in}

Stephen G. Simpson, Ph.D. \vspace{.2in}

Constantine Tsinakis, Ph.D. \vspace{.2in}

\end{center}

\end{titlepage}

\newpage


\pagenumbering{roman} 
\setcounter{page}{2}

~\vfill

\begin{center}
Copyright \copyright\ 2021 by Hayden Robert Jananthan \\ All Rights Reserved
\end{center}

\vfill~

\newpage


\addcontentsline{toc}{chapter}{DEDICATION}

~\vfill

\begin{center}
Dedicated to Lekha
\end{center}

\vfill~

\newpage


\chapter*{ACKNOWLEDGEMENTS}
\addcontentsline{toc}{chapter}{ACKNOWLEDGEMENTS}

\begin{singlespace}
My utmost thanks is to my advisor, Steve Simpson, who has guided me through this nearly five-year journey, starting all the way back to the our first independent studies course in my second semester. Steve's support, knowledge, and direction has been invaluable in my growth as a mathematician, and I know I've been lucky to have him as my advisor.

I would also like to thank my fellow graduate students and my instructors during my time here at Vanderbilt. Overall, the Vanderbilt Mathematics Department has proven to be a welcoming and tight-knit community. Between breath-taking hikes in the mountains, costume parties and soirees, and the simple satisfaction of drinks at KayBob's or food at McDougal's (among many other things), I know that I've made wonderful friends and wonderful memories here. Special thanks should be given to my roommate, Dumindu De Silva, who has never complained about me regularly mulling about our apartment at absurd hours of the night and morning.

Thanks is due to some of my colleagues, mentors, and instructors at MIT. My first forays into research, publishing, and supervising are thanks to Dr.\ Jeremy Kepner and Dr.\ Vijay Gadepally, and I continue to learn countless things from them and the Lincoln Laboratory Supercomputing Center as a whole. I must also thank Prof.\ Henry Cohn, who's brilliant instruction of 18.510 secured my interest in logic early in my academic journey, and Prof.\ Michael Sipser, who gave me my first taste of computability theory in 18.404. 

I would like to thank my family, all of whom have been supportive throughout the process and have put up with many a mathematical jargon-filled phone call. Finally, I want to give my deepest thanks to my wife Lekha, a beautiful person inside and out who has been a rock that has anchored me over these years. I'm not sure where I would be without her support, and I look forward to a bright future. 

\end{singlespace}

\newpage


\setcounter{tocdepth}{2}
\tableofcontents

\newpage

%
%


\listoffigures

\newpage


\addtocontents{toc}{Chapter \hfill~\par}

\normalsize
\doublespacing
\pagenumbering{arabic}
\setcounter{page}{1}

\chapter{Introduction}
\label{introduction chapter}

A subset $P$ of $\baire$ may be considered a `problem' whose `solutions' are its elements, as in the problems ``Find a completion of $\mathsf{PA}$'' or ``Find a $1$-random infinite binary sequence'', corresponding to the subsets $\cpa$ and $\mlr$, respectively. In this context, we call $P$ a \emph{mass problem}. To compare the `degree of unsolvability' of two mass problems $P$ and $Q$, one approach is to use weak reducibility, where $P \weakleq Q$ if and only if every member of $Q$ computes a member of $P$. 

Two well-studied hierarchies of mass problems are the complexity and diagonally non-recursive hierarchies. The former consists of the sets
\begin{equation*}
\complex(f) \coloneq \{ X \in \cantor \mid \text{$\pfc(X \restrict n) \geq f(n) - O(1)$ for all $n$}\}
\end{equation*}
where $f \colon \mathbb{N} \to [0,\infty)$ is an unbounded, nondecreasing, computable function (an \emph{order function}) and $\pfc$ is prefix-free Kolmogorov complexity. In other words, $\complex(f)$ consists of all infinite binary sequences whose first $n$ bits cannot be described with less than $f(n)$ bits of information, up to addition of a constant. The latter hierarchy consists of the sets
\begin{equation*}
\dnr(p) \coloneq \{ X \in \baire \mid \text{$X(n) \nsimeq \varphi_n(n)$ and $X(n) < p(n)$ for all $n$}\}
\end{equation*}
where $p \colon \mathbb{N} \to (1,\infty)$ is a nondecreasing, computable function and $\varphi_n$ is the $n$-th $1$-place partial recursive function. In other words, $\dnr(p)$ consists of all $p$-bounded infinite sequences which avoid the diagonal of a fixed enumeration of the $1$-place partial recursive functions. 

Although the two hierarchies are quite different in presentation, the connections between them have been widely studied. Among them is a result of Kjos-Hanssen, Merkle, \& Stephan \cite[Theorem 2.3]{kjoshanssen2006kolmogorov} which shows that the complexity and diagonally non-recursive hierarchies are tightly coupled when going downward:

\begin{thm*}
\textnormal{\cite[Theorem 2.3]{kjoshanssen2006kolmogorov}}
Suppose $X \in \cantor$. Then the following are equivalent.
\begin{enumerate}[(i)]
\item $X \in \complex(f)$ for some order function $f\colon \mathbb{N} \to \co{0,\infty}$.
\item There exists an order function $p\colon \mathbb{N} \to (1,\infty)$ and a $Y \in \dnr(p)$ such that $Y$ is computable from $X$.
\end{enumerate}
\end{thm*}

Two other connections were proven by Greenberg \& Miller \cite{greenberg2011diagonally}, relating the diagonally non-recursive hierarchy to the upper levels of the complexity hierarchy. The former says that regardless of how slow-growing an order function $p$ is, there is an $X \in \dnr(p)$ which cannot compute a maximally complex infinite binary sequence (an element of $\complex(\id_\mathbb{N})$), while the latter gives an upshot that if $p$ is sufficiently slow-growing then any $X \in \dnr(p)$ computes highly complex infinite binary sequences.

\begin{thm*}
\textnormal{\cite[Theorem 5.11]{greenberg2011diagonally}}
If $p\colon \mathbb{N} \to (1,\infty)$ is an order function, then there exists $X \in \dnr(p)$ such that $X$ computes no member of $\complex(\lambda n.n)$.
\end{thm*}

\begin{thm*}
\textnormal{\cite[Theorem 4.9]{greenberg2011diagonally}}
For all sufficiently slow-growing order functions $p\colon \mathbb{N} \to (1,\infty)$, every $X \in \dnr(p)$ computes a member of $\bigcap_{0 \leq \delta < 1}{\complex(\lambda n. \delta n)}$.
\end{thm*}

In \cite{simpson2017turing}, Simpson introduced a variation of $\dnr$, $\ldnr$ (\textbf{L}inearly \textbf{U}niversal \textbf{A}voidance)\footnote{The notation used by Simpson in \cite{simpson2017turing} was $\mathrm{LDNR}$, standing for \textbf{L}inearly \textbf{D}iagonally \textbf{N}on-\textbf{R}ecursive.}, to remove its dependence on any specific choice of enumeration of the partial recursive functions as well as to more closely tie the growth rate of $p$ to the degree of unsolvability of the class $\ldnr(p)$. For any order function $p$ there are order functions $p^+$ and $p^-$ such that any $X \in \dnr(p^+)$ computes a member of $\ldnr(p)$ and any $Y \in \ldnr(p)$ computes a member of $\dnr(p^-)$, so all of the aforementioned results linking the complexity and diagonally non-recursive hierarchies translate to the $\ldnr$ hierarchy. 

An observation made by Bienvenu \& Porter \cite{bienvenu2016deep}, Greenberg, Miller \cite{miller2020assorted}, and Slaman is that the behavior of $\dnr(p)$ (for specific types of enumerations of the partial recursive functions) changes significantly depending on whether the series $\sum_{n=0}^\infty{p(n)^{-1}}$ converges (in which case $p$ is called \emph{fast-growing}) or diverges (in which case $p$ is called \emph{slow-growing}), and this observation applies to $\ldnr$ as well \cite[Theorem 5.4]{simpson2017turing}. Thus, we may consider the $\ldnr$ hierarchy as being made up of two sub-hierarchies, the \emph{fast-growing $\ldnr$ hierarchy} (consisting of $\ldnr(p)$ for fast-growing $p$) and the \emph{slow-growing $\ldnr$ hierarchy} (consisting of $\ldnr(p)$ for slow-growing $p$).

Within $\mathcal{E}_\weak$ (where our degrees of interest lie) there is a subregion in its upper reaches consisting of so-called `deep degrees'. The slow-growing $\ldnr$ hierarchy lies in this subregion, while both the fast-growing $\ldnr$ and complexity hierarchies lies in its complement. The notion of `shift complexity' provides a randomness notion lying in that deep region, providing another way to study the connections between the slow-growing $\ldnr$ hierarchy and randomness/complexity notions. 

Our goal is to explore the relationships between the complexity, fast-growing $\ldnr$, shift complexity, and slow-growing $\ldnr$ hierarchies, expanding existing relationships and providing explicit bounds on the growth rates of the corresponding order functions.

\section{Summary of Chapters}

Each chapter is summarized below. Additionally, Figure \ref{summary of general reductions figure} summarizes the main general reductions proven, Figure \ref{summary of specific example reductions figure} summarizes specific examples of reductions, and Figures \ref{summary of result references figure} and \ref{summary of question references figure} collect references to the results, sections, and questions pertaining to each of the explored relationships between the hierarchies of interest. Figure \ref{tikz map of Ew} gives a visual representation of $\mathcal{E}_\weak$ and how the hierarchies of interest sit within it.

\subsubsection*{\cref{introduction chapter}} The remainder of this chapter covers notation, conventions, and terminology. \cref{basic conventions and notation section} covers basic notions, such as notation and terminology for number systems, set theoretic functions \& relations, strings over a set, the Cantor \& Baire spaces, and various encoding functions. \cref{computability definitions notation and conventions section} gives a brief overview of the relevant notation and terminology from computability theory. Finally, \cref{reducibility notions section} briefly reviews the Turing, weak, \& strong reducibility notions and the classes of mass problems we will be principally interested in. 

\subsubsection*{\cref{background chapter}} This chapter serves to introduce many of the main notions discussed within the remainder of the document. We start by giving a brief overview of partial randomness, reviewing the notation, terminology, and some basic results. Following that, we discuss $\dnr$ and its dependence on a choice of an enumeration of the partial recursive functions, using its definition to motivate the definition of the class $\avoid^\psi(p)$ for a recursive $p \colon \mathbb{N} \to (1,\infty)$ and a partial recursive $\psi \colonsub \mathbb{N} \to \mathbb{N}$. After defining the family of linearly universal partial recursive functions, we define $\ldnr(p)$, covering some of the basic reducibility results between those classes. The fast-growing, slow-growing dichotomy is examined, where we state and prove several technical results used later. Finally, we define depth and discuss the weak degrees of deep $\Pi^0_1$ classes, the basic structure of the region of deep degrees in $\mathcal{E}_\weak$ and its relation to the fast-growing $\ldnr$ and slow-growing $\ldnr$ hierarchies.

\subsubsection*{\cref{complexity and avoidance downward relationships chapter}} This chapter is centered around the relationships between the complexity and fast-growing $\ldnr$ hierarchies. One way in which we do this is by strengthening \cite[Theorem 2.3]{kjoshanssen2006kolmogorov}, addressing the problems ``given $f$, find $q$ such that $\ldnr(q) \weakleq \complex(f)$'' and ``given $p$, find $g$ such that $\complex(g) \weakleq \ldnr(p)$'' and giving explicit bounds for each. In particular, one of our main theorems is the following.

\begin{repthm}{main downward theorem}
To each sub-identical order function $f\colon \mathbb{N} \to \co{0,\infty}$ there is a fast-growing order function $q\colon \mathbb{N} \to (1,\infty)$ such that $\ldnr(q) \strongleq \complex(f)$, and to each fast-growing order function $p\colon \mathbb{N} \to (1,\infty)$ there is a sub-identical order function $g\colon \mathbb{N} \to \co{0,\infty}$ such that $\complex(g) \strongleq \ldnr(p)$.
\end{repthm}

We also address the `upward' problem ``given $p$ fast-growing, find sub-identical $g$ such that $\ldnr(p) \weakleq \complex(g)$'', giving a partial answer.

\begin{repthm}{complex hierarchy outpaces fast nice ldnr hierarchy}
If $p\colon \mathbb{N} \to (1,\infty)$ is a fast-growing order function such that $\sum_{n=0}^\infty{p(n)^{-1}}$ is a recursive real, then there exists a convex sub-identical order function $g$ such that $\ldnr(p) \strongleq \complex(g) \neq \mlr$.
\end{repthm}

%

\subsubsection*{\cref{complexity and avoidance upward relationships chapter}} In this chapter we address the problem ``given $f$, find $q$ such that $\complex(f) \weakleq \ldnr(q)$'', giving a partial answer and providing explicit bounds for those cases. Our main results are the following.

%
%

\begin{repthm}{sqrt complex from ldnr}
Given an order function $\Delta \colon \mathbb{N} \to [0,\infty)$ such that $\lim_{n \to \infty}{\Delta(n)/\sqrt{n}} = 0$ and any rational $\epsilon \in (0,1)$, 
\begin{equation*}
\complex\bigl(\lambda n. n-\sqrt{n}\cdot \Delta(n)\bigr) \weakleq \ldnr\bigl(\lambda n. \exp_2\bigl((1-\epsilon)\Delta(\log_2\log_2 n)\bigr)\bigr).
\end{equation*}
More generally, $\complex(\lambda n. n - \sqrt{n}\cdot \Delta(n)) \weakleq \ldnr(q)$ for any order function $q$ satisfying
\begin{equation*}
q\left( \exp_2((1-\epsilon)^{-1} \cdot [(n+1)^2 - (n+1) \cdot \Delta((n+1)^2)] \cdot \ell(n)) \right) \leq \ell(n)
\end{equation*}
for almost all $n \in \mathbb{N}$, where $\ell(n) = \exp_2\left((1-\epsilon)[(n+1) \cdot \Delta((n+1)^2) - n \cdot \Delta(n^2)]\right)$. 
\end{repthm}

%
%

\subsubsection*{\cref{generalized shift complexity chapter}} In this chapter we examine classes of `shift complex' sequences -- in which the prefix-free complexity of all segments of a sequence $X$ are quantified rather than only the initial segments -- with respect to (non)negligibility and depth, as well as relationships with the complexity and $\ldnr$ hierarchies. Our main results are the following.

\begin{repthm}{quantified Khan shift complex theorem explicit example bound}
\textnormal{(Corollary of \cref{quantified Khan shift complex theorem})} 
Fix a rational $\epsilon > 0$. For all rational $\delta \in (0,1)$ we have 
\begin{equation*}
\shiftcomplex(\delta) \weakleq \ldnr\bigl(\lambda n. (\log_2 n)^{1-\epsilon}\bigr).
\end{equation*}
\end{repthm}

\begin{repthm}{sub-identical complex strongly computes shift complex}
Suppose $f$ is a sub-identical order function such that $\sum_{m=0}^\infty{f(2^m)/2^m}$ converges to a recursive real. Then there is an order function $g$ such that $\shiftcomplex(f) \strongleq \complex(g)$ and for which $\lim_{n \to \infty}{g(n)/n}=0$.
\end{repthm}

\subsubsection*{\cref{bushy tree chapter}} This chapter focuses on the relationships between the fast and slow-growing $\ldnr$ subhierarchies and their structures, with applications to the depth properties of the boundaries of the slow-growing $\ldnr$ hierarchy and relationships with the shift complexity hierarchy. Our main results are the following.

\begin{repthm}{ldnr incomparable}
For all order functions $p_1\colon \mathbb{N} \to (1,\infty)$ and $p_2\colon \mathbb{N} \to (1,\infty)$, there exists a slow-growing order function $q\colon \mathbb{N} \to (1,\infty)$ such that $\ldnr(p_1) \nweakleq \ldnr(q) \nweakleq \ldnr(p_2)$. \\ In particular, for any order function $p\colon \mathbb{N} \to (1,\infty)$, there exists a slow-growing order function $q\colon \mathbb{N} \to (1,\infty)$ such that $\ldnr(p)$ and $\ldnr(q)$ are weakly incomparable.
\end{repthm}

\begin{repthm}{ldnr-slow not deep}
$\ldnr_\slow$ is not of deep degree.
\end{repthm}

\begin{repthm}{slow-growing hierarchy has no minimum}
There is no order function $q\colon \mathbb{N} \to (1,\infty)$ such that $\ldnr_\slow \weakeq \ldnr(q)$.
\end{repthm}

\begin{repthm}{SC not weakly below ldnr_slow}
$\shiftcomplex \weaknleq \ldnr_\slow$.
\end{repthm}

\subsubsection*{\cref{structure of filter of deep degrees}} Using results of the previous chapters, we further explore the structure of the region of deep degrees in $\mathcal{E}_\weak$ and a larger region consisting of `pseudo-deep' degrees in $\mathcal{E}_\weak$. Our main result is the following.

\begin{repthm}{proper nesting of deep-related filters}
Define
\begin{align*}
\mathscr{F}_\mathrm{deep} & \coloneq \{ \mathbf{p} \in \mathcal{E}_\weak \mid \text{$\mathbf{p}$ a deep degree}\}. \\
\mathscr{F}_\mathrm{pseudo} & \coloneq \{ \mathbf{p} \in \mathcal{E}_\weak \mid \text{$\mathbf{p} = \inf \mathcal{C}$ for some $\mathcal{C} \subseteq \mathscr{F}_\mathrm{deep}$}\}. \\
\mathscr{F}_\mathrm{diff} & \coloneq \{ \mathbf{p} \in \mathcal{E}_\weak \mid \forall P \in \mathbf{p} \forall X \in \mlr \qspace ( \exists Y \in P \qspace (Y \turingleq X) \to (0' \turingleq X))\}.
\end{align*}
Then $\mathscr{F}_\mathrm{pseudo}$ is a principal filter while $\mathscr{F}_\mathrm{deep}$ and $\mathscr{F}_\mathrm{diff}$ are nonprincipal filters. Consequently, $\mathscr{F}_\mathrm{deep} \subsetneq \mathscr{F}_\mathrm{pseudo} \subsetneq \mathscr{F}_\mathrm{diff}$.
\end{repthm}

%



\begin{figure}[p]
\caption{Summary of general reductions between the hierarchies of interest. Within each row, the conditions listed in the right-most column are assumed for the given reduction. In addition, all functions are assumed to be order functions, $f$ sub-identical, and $p$ fast-growing. $\mathbb{R}_\rec$ is the set of recursive reals.}
\label{summary of general reductions figure}

\begin{center}
\renewcommand{\arraystretch}{3}
\begin{tabular}{l C | c}

\ref{complex reals compute ldnr functions} \vspace{0cm}
& 
$\displaystyle \ldnr\bigl(\lambda n. \exp_2((f^\inverse \circ h)(n)+1)\bigr) \strongleq \complex(f)$
& 
$\scalebox{.8}{$\displaystyle\sum_{n=0}^\infty$} {\frac{1}{2^{h(n)}}} \in \mathbb{R}_\rec $
\\ \hline


\ref{ldnr functions computes complex reals} \vspace{0cm}
&
$\displaystyle \complex\Bigl( \bigl(\lambda n. \raisebox{.1em}{\scalebox{.9}{$\displaystyle\sum$}}_{i < r(n)}{\lfloor \log_2 p(i) \rfloor} \bigr)^\inverse\Bigr) \strongleq \ldnr(p)$
&
$\parbox{2.5cm}{\centering $\displaystyle \lim_{n \to \infty} \textstyle {\frac{r(n)}{2^n}} = \infty$}$
\\ \hline

\ref{complex hierarchy outpaces fast nice ldnr hierarchy quantified} \vspace{0cm}
&
$\displaystyle \ldnr(p) \strongleq \complex\bigl(\lambda n. \log_2 \tilde{p}\bigl(p^\inverse(2^{n+1})-1\bigr)\bigr)$
&
$\parbox{2.5cm}{\centering $\scalebox{0.8}{$\displaystyle \sum_{n=0}^\infty$} {\frac{1}{\tilde{p}(n)}} \in \mathbb{R}_\rec$, \\ $\displaystyle \lim_{n \to \infty} \textstyle {\frac{p(n)}{\tilde{p}(n+3)}} = \infty$}$
\\ \hline

\ref{sqrt complex from ldnr} \vspace{0cm}
&
$\displaystyle \complex\bigl(\lambda n. n-\sqrt{n}\cdot \Delta(n)\bigr) \weakleq \ldnr\bigl(\lambda n. \exp_2\bigl((1-\epsilon)\Delta(\log_2\log_2 n)\bigr)\bigr)$
& 
$\parbox{2.5cm}{\centering $0 < \epsilon < 1$, \\ $\displaystyle \lim_{n \to \infty} \textstyle {\frac{\Delta(n)}{\sqrt{n}}} = 0$}$
\\ \hline

\ref{quantified Khan shift complex theorem explicit example bound} \vspace{0cm}
&
$\displaystyle \shiftcomplex(f) \weakleq \ldnr\bigl(\lambda n. (\log_2 n)^{1-\epsilon}\bigr)$
&
$\parbox{2.5cm}{\centering $\!\!\!\limsup_n {\frac{f(n)}{n}} < 1$, \\ $0 < \epsilon < 1$}$
\\ \hline

\ref{partial randomness and shift complexity} \vspace{0cm}
&
$\displaystyle \shiftcomplex(f) \strongleq \complex(\delta)$
&
$\parbox{2.5cm}{\centering $\scalebox{0.8}{$\displaystyle \sum_{m=0}^\infty$} {\frac{f(2^m)}{2^m}} < \infty$, \\ $0 < \delta \leq 1$}$

\end{tabular}
\renewcommand{\arraystretch}{1}
\end{center}
\end{figure}

\begin{figure}[p]
\caption{Summary of specific examples of reductions between the hierarchies of interest. Within each row, the conditions listed in the right-most column are assumed for the given reduction.}
\label{summary of specific example reductions figure}

\begin{center}
\renewcommand{\arraystretch}{3}
\begin{tabular}{l C | c}

\ref{reduction from delta-complex sequences}~ \vspace{0cm}
&
$\displaystyle \ldnr\Bigl(\lambda n. 4\sqrt[\delta]{n \cdot \log_2 n \mdots \log_2^{k-1} n \cdot (\log_2^k n)^{(1+\epsilon)}}\Bigr) \strongleq \complex(\delta)$
&
$\parbox{2.5cm}{\centering $0 < \delta \leq 1$, \\ $k \in \mathbb{N}$, $0 < \epsilon$}$
\\ \hline

\ref{reduction from power-complex sequences} \vspace{0cm}
&
$\displaystyle \ldnr\bigl(\lambda n. 4 \exp_2\bigl(\sqrt[\alpha]{(1+\epsilon)\log_2 n}\bigr)\bigr) \strongleq \complex(\lambda n. n^\alpha)$
&
$\parbox{2.5cm}{\centering $0 < \alpha \leq 1$, \\ $0 < \epsilon$}$
\\ \hline

\ref{logarithmic complex example} \vspace{0cm}
&
$\ldnr\bigl(\lambda n. 4 \exp_2(n^{(1+\epsilon)/\beta})\bigr) \strongleq \complex(\lambda n. \beta \log_2 n)$
&
$\parbox{2.5cm}{\centering $0 < \beta$, \\ $0 < \epsilon$}$
\\ \hline

\ref{exponential ldnr example} \vspace{0cm}
&
$ \complex(\lambda n. (1/2 - \epsilon)\log_2 n) \strongleq \ldnr(\lambda n. 2^n)$
&
$\parbox{2.5cm}{\centering $0 < \epsilon < 1/2$}$
\\ \hline

\ref{greenberg miller theorem 4.9 improved} \vspace{0cm}
&
$\displaystyle \complex(\lambda n. n-(1+\epsilon)\sqrt{n} \log_2 n) \weakleq \ldnr\bigl(\lambda n. (\log_2 n)^{1/2 - \epsilon}\bigl)$
&
$\parbox{2.5cm}{\centering $0 < \epsilon < 1/2$}$
\\ \hline

\ref{quantified Khan shift complex theorem explicit example bound} \vspace{0cm}
&
$\displaystyle \shiftcomplex(\delta) \weakleq \ldnr(\lambda n. \bigl(\log_2 n)^{1-\epsilon}\bigr)$
&
$\parbox{2.5cm}{\centering $0 < \delta < 1$, \\ $0 < \epsilon < 1$}$
\\ \hline

\ref{complex and shift-complex example 1} \vspace{0cm}
&
$\displaystyle \shiftcomplex(\lambda n. n^\alpha) \strongleq \complex(\lambda n. n^{\alpha+\epsilon})$
&
$\parbox{2.5cm}{\centering $0 < \alpha < 1$, \\ $0 < \epsilon \leq 1-\alpha$}$
\\ \hline

\ref{complex and shift-complex example 2} \vspace{0cm}
&
$\displaystyle \shiftcomplex\bigl(\lambda n. n/(\log_2 n)^{\alpha+1+\epsilon}\bigr) \strongleq \complex\bigl(\lambda n. n/(\log_2 n)^\alpha\bigr)$
&
$\parbox{2.5cm}{\centering $0 < \alpha$, \\ $0 < \epsilon$}$

\end{tabular}
\renewcommand{\arraystretch}{1}
\end{center}
\end{figure}

\begin{figure}[p]
\caption{Collected result and section references related to reductions of the form $P \weakleq Q$ or $P \nweakleq Q$, where $P$ is a member of the hierarchy corresponding to the row and $Q$ is a member of the hierarchy corresponding to the column. The $\complex$ row also contains references related to (non)negligibility and depth.}
\label{summary of result references figure}
\begin{center}
\renewcommand{\arraystretch}{3}
\begin{tabular}{c || P | P | P | P}
~ & $\complex$ & $\ldnr_\fast$ & $\ldnr_\slow$ & $\shiftcomplex$ \\ \hline\hline
$\complex$ & & \ref{ldnr and mlr}(a), \ref{complex weakly equivalent to ldnrrec}, \ref{main downward theorem}, \S \ref{complex below ldnr section} & \ref{ldnr and mlr}(b), \ref{ldnr p deep degree if p slow-growing}, \ref{improved greenberg and miller's randomness conclusion}, \S \ref{quantifying greenberg miller proof}, \S \ref{quantifying greenberg miller proof general}, \ref{ldnr-slow not deep} & \S\ref{shift complexity and depth subsection}, \S\ref{strong shift complexity and depth subsection} \\ \hline
$\ldnr_\fast$ & \ref{complex weakly equivalent to ldnrrec}, \ref{main downward theorem}, \S \ref{ldnr below complex section}, \S \ref{complexity to avoidance upwards} & \ref{ldnr basic facts} & \ref{ldnr basic facts}, \S\ref{fast and slow-growing ldnr hierarchies subsection}, \ref{ldnr incomparable} & \\ \hline
$\ldnr_\slow$ & & \ref{separation of fast-growing and slow-growing ldnr hierarchies} & \ref{ldnr basic facts}, \ref{separation of fast-growing and slow-growing ldnr hierarchies}, \ref{ldnr incomparable}, \ref{slow-growing hierarchy has no minimum} & \\ \hline
$\shiftcomplex$ & \ref{relating generalized shift complexity and complexity subsection}, \ref{extracting generalized shift complexity from sublinear complexity subsection} & & \ref{SC not weakly below ldnr_slow}, \ref{SC not weakly below ldnr_slow corollary} & ~\newline 
\end{tabular}
\renewcommand{\arraystretch}{1}
\end{center}
\end{figure}

\begin{figure}[p]
\caption{Collected open question references related to reductions of the form $P \weakleq Q$ or $P \nweakleq Q$, where $P$ is a member of the hierarchy corresponding to the row and $Q$ is a member of the hierarchy corresponding to the column. The $\complex$ row also contains references related to (non)negligibility and depth.}
\label{summary of question references figure}
\begin{center}
\renewcommand{\arraystretch}{3}
\begin{tabular}{c || P | P | P | P}
~ & $\complex$ & $\ldnr_\fast$ & $\ldnr_\slow$ & $\shiftcomplex$ \\ \hline\hline
$\complex$ & & \ref{quantify upper bound on q fast-growing}, \ref{fast-growing ldnr above half-random} \newline & \ref{what does complex below ldnr apply to}, \ref{quantify upper bound on q} & \ref{are unions of delta-shift complex classes deep}, \ref{depth of generalized shift complex classes} \\ \hline
$\ldnr_\fast$ & \ref{recursive sum necessary for complex above ldnr question} & ~\newline & \ref{better understanding incomparable order functions} & \\ \hline
$\ldnr_\slow$ & & & \ref{better understanding incomparable order functions}, \ref{minimal slow-growing ldnr question}, \ref{slow-growing ldnr downwards-directed question} & \ref{comparing delta-shift complex to slow-growing ldnr} \\ \hline
$\shiftcomplex$ & \ref{comparing generalized shift complex to complex} & & \ref{shift complexity and avoidance section}, \ref{quantify which slow-growing order functions do not compute shift complexity} & \ref{separating delta-shift complex degrees}, \ref{separating strongly delta-shift complex degrees}, \ref{separating shift complexity from strong shift complexity}, \ref{relationship between shift complex and strongly shift complex}, \ref{separating generalized shift complexity}
\end{tabular}
\renewcommand{\arraystretch}{1}
\end{center}
\end{figure}



\begin{figure}[p]
\caption{Visual representation of $\mathcal{E}_\weak$ and the relationships between the hierarchies of interest within. $p$ denotes a slow-growing order function, $q$ denotes a slow-growing order function, $f$ denotes a sub-identical order function, $g$ denotes an order function satisfying $\sum_{m=0}^\infty{g(2^m)/2^m} < \infty$, and $\delta$ denotes a rational number in $(0,1)$.}
\label{tikz map of Ew}

\makebox[\textwidth][c]{
\begin{tikzpicture}[
dot/.style = {circle, fill, minimum size=5pt,
              inner sep=0pt, outer sep=0pt},
decoration={snake,amplitude=1pt,segment length=2mm,pre=lineto,pre length=3pt,post=lineto,post length=3pt},
>=stealth',
]

\coordinate (0) at (0,0);
\coordinate (upperleft) at (-7,9);
\coordinate (topleft) at (-7.07,10);
\coordinate (upperright) at (7,9);
\coordinate (topright) at (7.07,10);
\coordinate (DeepBottom) at (2,4.25);

\fill[gray!10] (0,4) circle (4);
\draw[dotted] (0,4) circle (4);

\begin{scope}
	\clip (0,4) circle (4);
	\draw[fill=gray!20,thick,dotted] (1,4.5) parabola (3,10) -- (1,10) -- cycle;
	\draw[fill=gray!20,thick,dotted] (1,4.5) parabola (-2,10) -- (1,10) -- cycle;
	\draw[very thick,draw=gray!20] (1,4.5) -- (1,10);
\end{scope}	

\draw[thick, dashed] (0) to[out=180, in=275] (upperleft);
\draw[thick, dotted] (upperleft) to[out=95, in=275] (topleft);
\draw[thick, dashed] (0) to[out=0, in=265] (upperright);
\draw[thick, dotted] (upperright) to[out=85, in=265] (topright);
                    
\node (Dw) at (-6,8.75) {\huge$\mathcal{D}_\weak$};
\node (Ew) at (-3.75,7) {\huge$\mathcal{E}_\weak$};
                    
\node[dot, label={[xshift=.65cm]$\mathbf{1} \mbox{=} \sup\mathcal{E}_\weak$}] (cpa) at (0,8) {};
\node[dot, label=right:$\mathbf{d}_q$] (luaq) at (.7,6.7) {};
\node[dot, label=right:$\mathbf{s}_\delta$] (shiftcomplexdelta) at (-.4,5.7) {};
\node[dot, label=right:$\mathbf{d}_\slow$] (luaslow) at (1.6,5) {};
\node[dot, label=below right:$\weakdeg(L)$] (pdeep) at (1,4.5) {};
\node[dot, label=above:$\mathbf{r}_1$] (mlr) at (-1.5,4.5) {};
\node[dot, label=left:$\mathbf{d}_p$] (luap) at (-2.5,3) {};
\node[dot, label=above right:$\mathbf{r}_f$] (complexf) at (-1.5,3) {};
\node[dot, label=right:$\mathbf{s}_g$] (shiftcomplexg) at (-.5,3) {};
\node[dot, label=below:$\mathbf{d}_\rec$] (drec) at (-1.5,2) {};
\node[dot, label={[xshift=.67cm]$\mathbf{0}=\inf\mathcal{D}_\weak$}] (rec) at (0,0) {};

\node (dslowpdeep) at (1.27,4.8) {$\rotatebox{40}{=}$};
\node (dslowpdeepq) at (1.2,5.1) {?};

\draw[<->,decorate] ([xshift=.05cm]luap.east) to ([xshift=-.05cm]complexf.west);
\draw[<->,decorate] ([xshift=.05cm]complexf.east) to ([xshift=-.05cm]shiftcomplexg.west);
\draw[->,decorate,bend right] ([xshift=-.05cm,yshift=-.1cm]luaq.south west) to ([xshift=.05cm,yshift=.05cm]shiftcomplexdelta.north);
\draw[->,decorate,bend right=70,looseness=1.25] ([xshift=-.1cm]luaq.north west) to ([xshift=-.1cm,yshift=.1cm]luap.north);

\draw[->,bend right] ([yshift=-.05cm,xshift=-.1cm]mlr.west) to ([yshift=.1cm]luap.north);
\draw[->] ([yshift=-.125cm]mlr.south) to ([yshift=.1cm]complexf.north);
\draw[->,bend left] ([yshift=-.05cm,xshift=.1cm]mlr.east) to ([yshift=.1cm]shiftcomplexg.north);
\draw[->,bend right] ([yshift=-.1cm]luap.south) to ([yshift=.05cm,xshift=-.1cm]drec.west);
\draw[->] ([yshift=-.1cm]complexf.south) to ([yshift=.1cm]drec.north);
\draw[->,bend left] ([yshift=-.1cm]shiftcomplexg.south) to ([yshift=0.05cm,xshift=.1cm]drec.east);
\draw[->] ([yshift=-.1cm,xshift=.05cm]luaq.south) to ([yshift=0.05cm,xshift=-.05cm]luaslow.north);
\draw[->] ([yshift=-.1cm,xshift=.1cm]cpa.south) to ([yshift=0.1cm,xshift=-.1cm]luaq.north);
\draw[->,bend left=50] ([yshift=-.08cm]pdeep.south) to ([yshift=-.1cm,xshift=.1cm]drec.east);

\fill[draw=black!0, fill=gray!10] (-2.06,4.225) circle (4pt);
\draw[->, decorate, bend right=60, looseness=1.45] ([xshift=-.1cm]luaq.south west) to ([xshift=-.1cm,yshift=.1cm]complexf.north);

\draw[bend right, looseness=.5] ([xshift=.1cm,yshift=.01cm]mlr.east) to node[midway,description,fill=gray!10]{\textsf{x}} ([xshift=-.05cm,yshift=-.1cm]pdeep.west);
\draw ([xshift=-.1cm,yshift=-.1cm]shiftcomplexdelta.south west) to node[midway,description,fill=gray!10]{\textsf{x}} ([xshift=.1cm,yshift=.1cm]mlr.north east);
\draw[bend left=40] ([yshift=-.1cm,xshift=-.05cm]luaq.south) to node[pos=.7,description,fill=gray!10]{\textsf{x}} ([xshift=.1cm,yshift=.05cm]mlr.east);

\matrix [draw,fill=white] at (6.75,7) {
  \node [label=right:$\mathcal{E}_\weak$] {\testclr{gray!10}~:}; \\
  \node [label=right:deep region of $\mathcal{E}_\weak$] {\testclr{gray!20}~:}; \\
  \node [label=right:$\weakdeg\bigl(\ldnr(r)\bigr)$] {$\mathbf{d}_r$:}; \\
  \node [label=right:$\weakdeg\bigl(\ldnr_\rec\bigr){=}\weakdeg(\complex)$] {\hspace{-7pt}$\mathbf{d}_\rec$:}; \\
  \node [label=right:\hspace{-7pt}$\weakdeg\bigl(\ldnr_\slow\bigr)$] {\hspace{-5pt}$\mathbf{d}_\slow$:~~~}; \\
  \node [label=right:$\weakdeg\bigl(\shiftcomplex(h)\bigr)$] {$\mathbf{s}_h$:}; \\
  \node [label=right:$\weakdeg\bigl(\complex(h)\bigr)$] {$\mathbf{r}_h$:}; \\
  \node [label=right:$\bigcup\{\text{deep $\Pi^0_1$ $P\subseteq \cantor$}\}$] {~$L$:}; \\
  \node [label=right:${-} \weakge {-}$] {$\to$\,:}; \\
  \node [label=right:$\exists$ $\weakgeq$-relations] {{\hspace{-4.5pt}\xrsquigarrow{~~~}}:}; \\
  \node [label=right:$\weakleq$-incomparable] {{\hspace{-4.5pt}\xxarrow{~~~}}:}; \\
};

\end{tikzpicture}
}
\end{figure}
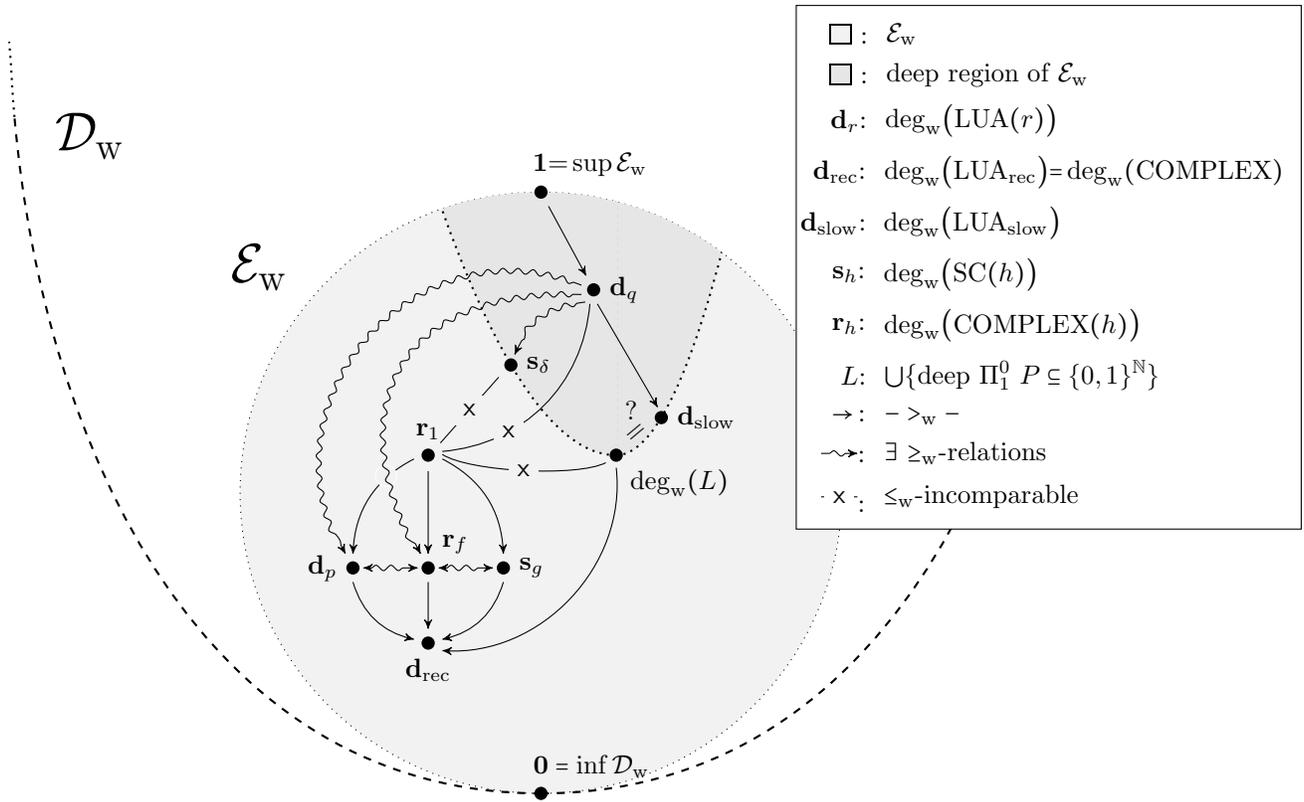

\section{Basic Conventions and Notation}
\label{basic conventions and notation section}

$\mathbb{N}$ is the set of natural numbers (including $0$). 
$\mathbb{Z}$ is the set of integers. 
$\mathbb{Q}$ is the set of rational numbers. 
$\mathbb{R}$ is the set of real numbers. 
Open, closed, and half-open intervals in $\mathbb{R}$ are written $(a,b)$, $[a,b]$, and $\co{a,b},\oc{a,b}$, respectively, for any $-\infty \leq a \leq b \leq \infty$. 
If $\mathbb{F}$ is any one of $\mathbb{N}$, $\mathbb{Z}$, $\mathbb{Q}$, or $\mathbb{R}$ and $a \in \mathbb{F}$, then we write $\mathbb{F}_{\geq a} \coloneq \{ x \in \mathbb{F} \mid a \leq x\}$ and $\mathbb{F}_{> a} \coloneq \{ x \in \mathbb{F} \mid a < x\}$. 

$n \bmod m$ is the remainder after dividing $n$ by $m$. ${}^mn$ denotes the $m$-th tetration of $n$. $\log_2^k$ denotes the composition of $k$-many base-$2$ logarithms. $\lceil {-} \rceil$ and $\lfloor {-} \rfloor$ denote the ceiling and floor functions, respectively. When it would increase readibility, we write $\exp_a(n) = a^n$ for $a > 0$. 

Set containment is denoted by $\subseteq$ and proper containment is denoted by $\subset$ or $\subsetneq$. $\mathcal{P}(S)$ denotes the power set of $S$, while $\mathcal{P}_\fin(S)$ denotes the set of finite subsets of $S$. $B^A$ denotes the set of all functions with domain $A$ and codomain $B$. 

Ordered $n$-tuples (or just `$n$-tuples') are denoted using angled brackets, as in $\langle x_0,x_1,x_2,\ldots,x_{n-1} \rangle$. We assume that if $\langle x_0,x_1,x_2,\ldots,x_{n-1}\rangle = \langle y_0,y_1,y_2,\ldots,y_{m-1}\rangle$, then $n=m$. We sometimes identify an $n$-tuple $\langle x_0,x_1,x_2,\ldots,x_{n-1} \rangle$ with the set $\{ \langle 0,x_0\rangle, \langle 1,x_1\rangle, \langle 2,x_2\rangle, \ldots, \langle n-1,x_{n-1}\rangle\}$ and with the function $f \colon \{0,1,2,\ldots,n-1\} \to \{x_0,x_1,x_2,\ldots,x_{n-1}\}$ defined by $f(i) \coloneq x_i$ for each $i \in \{0,1,2,\ldots,n-1\}$. More generally, we identify a function $f \colon A \to B$ with its graph $\{\langle a,b \rangle \in A \times B \mid b = f(a)\}$. 

Given a function $f \colon A \to B$, we write $\dom f = A$ and $\cod f = B$. Given subsets $A_0 \subseteq A$ and $B_0 \subseteq B$, we write $f[A_0] \coloneq \{ f(a) \mid a \in A_0\}$ and $f^{-1}[B_0] \coloneq \{ a \in A \mid f(a) \in B_0\}$ for the preimage of $B_0$. $\im f$ is defined to be equal to $f[A]$. $f \restrict A_0$ denotes the restriction of $f$ to $A_0$. If $\dom f = \mathbb{N}$, then we write $f \restrict n$ for the restriction $f \restrict \{0,1,2,\ldots,n-1\}$.

If $f \colon \mathbb{N} \to \mathbb{R}$ is nondecreasing, we define $f^\inverse \colon \mathbb{R} \to \mathbb{N}$ by $f^\inverse(x) \coloneq \text{least $m$ such that $f(m) \geq x$}$. 

Given functions $f,g \colon \mathbb{N} \to \mathbb{R}$, we write $f \domleq g$, read ``$g$ \textdef{dominates} $f$'' or ``$f$ is \textdef{dominated by} $g$'' to mean that $f(n) \leq g(n)$ for almost all $n \in \mathbb{N}$.

A \textdef{partial function} $f \colonsub A \to B$ is a function $f \colon A_0 \to B$ for some $A_0 \subseteq A$. Given $a \in A$, then $f(a)$ is said to be \textdef{defined} or to \textdef{converge}, written $f(a) \converge$, if $a \in \dom f$, otherwise $f(a)$ is said to be \textdef{undefined} or \textdef{diverge}, written $f(a) \diverge$. If $f \colon a \mapsto b$, we sometimes write $f(a) \converge = b$. If $f$ and $g$ are partial functions ${\subseteq} A \to B$ and $a \in A$, then we write $f(a) \simeq g(a)$ to mean that either $a \in \dom f \cap \dom g$ and $f(a) = g(a)$ (i.e., $f(a)$ and $g(a)$ both converge and are equal) or that $a \notin \dom f \cup \dom g$ (i.e., $f(a)$ and $g(a)$ both diverge). $f$ is \textdef{total} if $\dom f = A$. 

Given a set $S$, a \textdef{string over $S$} -- or simply a \textdef{string} if $S$ is understood -- is any element of $S^n$ for some $n \in \mathbb{N}$. $S^\ast$ is the set of all strings over $S$, i.e., $S^\ast = \bigcup_{n \in \mathbb{N}}{S^n}$. 
Given a string $\sigma \in S^\ast$, its \textdef{length} $|\sigma|$ is the unique $n \in \mathbb{N}$ for which $\sigma \in S^n$. 
Given $k < |\sigma|$, $\sigma(k)$ is the $k$-th coordinate of $\sigma$, so that $\langle \sigma(0), \sigma(1),\ldots,\sigma(|\sigma|-1)\rangle = \sigma$. 
If $\sigma = \langle s_0,s_1,\ldots,s_n\rangle$ and $\tau = \langle t_0,t_1,\ldots,t_m\rangle$ are strings over $S$, their \textdef{concatenation} $\sigma \concat \tau$ is given by
\begin{equation*}
\sigma \concat \tau = \langle s_0,s_1,\ldots,s_n,t_0,t_1,\ldots,t_m\rangle.
\end{equation*}
Given $\sigma \in S^\ast$ and $n \leq |\sigma|$, $\sigma \restrict n$ denotes the string $\langle \sigma(0), \sigma(1), \ldots, \sigma(n-1)\rangle$. 
Given $\sigma,\tau \in S^\ast$, then $\sigma$ is an \textdef{initial segment} of $\tau$ (equivalently, $\tau$ is an \textdef{extension} of $\sigma$) if $\sigma = \tau \restrict |\sigma|$, written $\sigma \subseteq \tau$. $\sigma$ is a \textdef{proper} initial segment of $\tau$ (equivalently, $\tau$ is a \textdef{proper} extension of $\sigma$) if $\sigma \subseteq \tau$ and $\sigma \neq \tau$, written $\sigma \subset \tau$. $\sigma$ and $\tau$ are \textdef{compatible} if either $\sigma \subseteq \tau$ and $\tau \subseteq \sigma$, otherwise they are \textdef{incompatible}. 
A set of strings $A \subseteq S^\ast$ is \textdef{prefix-free} if $\sigma \nsubseteq \tau$ for all distinct elements $\sigma,\tau$ in $A$.
If $\leq$ is a partial order on $S$, then the \textdef{lexicographical ordering} $\leq_\lex$ on $S^\ast$ is defined by setting $\sigma \leq_\lex \tau$ if $\sigma \subseteq \tau$ or $\sigma(k) < \tau(k)$ for the least index $k$ at which $\sigma(k) \neq \tau(k)$. The \textdef{shortlex ordering} $\leq_\shortlex$ on $S^\ast$ is defined by setting $\sigma \leq_\shortlex \tau$ if $|\sigma| < |\tau|$ or if $|\sigma| = |\tau|$ and $\sigma \leq_\lex \tau$, i.e., we order by length first, then lexicographically. 

Suppose $\sigma \in S^\ast$ and $f \colon \mathbb{N} \to S$ are given. $\sigma$ is an \textdef{initial segment} of $f$ (equivalently, $f$ is an \textdef{extension} of $\sigma$) if $f \restrict |\sigma| = \sigma$. $\sigma$ and $f$ are \textdef{incompatible} if $\sigma \nsubset f$. Finally, we define $\sigma \concat f \colon \mathbb{N} \to S$ by
\begin{equation*}
(\sigma \concat f)(n) \coloneq \begin{cases} \sigma(n) & \text{if $n < |\sigma|$,} \\ f(n-|\sigma|) & \text{otherwise.} \end{cases} 
\end{equation*}

The \textdef{Baire space} $\baire$ is endowed with the topology with basic open sets 
\begin{equation*}
\bbracket{\sigma} \coloneq \{ f \in \baire \mid \sigma \subset f \}
\end{equation*}
for $\sigma \in \mathbb{N}^\ast$. $\baire$, with this topology, is a non-empty zero-dimensional perfect polish space whose compact subsets have empty interior (with these properties characterizing $\baire$ up to homeomorphism). 

The \textdef{Cantor space} $\cantor$ is endowed with the subspace topology coming from $\baire$. Alternatively, it has the topology with basic open sets
\begin{equation*}
\bbracket{\sigma}_2 \coloneq \{ X \in \cantor \mid \sigma \subset X\}
\end{equation*}
for $\sigma \in \{0,1\}^\ast$. $\cantor$, with this topology, is a non-empty zero-dimensional compact perfect Polish space (with these properties characterizing $\cantor$ up to homeomorphism). We make the usual identification between elements of $\cantor$ and subsets of $\mathbb{N}$. The \textdef{fair coin measure} $\lambda$ is the outer measure on $\cantor$ induced by the assignments $\lambda(\bbracket{A}_2) \coloneq \sum_{i=1}^n{2^{-|\sigma_i|}}$ where $A = \{ \sigma_1,\sigma_2,\ldots,\sigma_n\} \subseteq \{0,1\}^\ast$ is prefix-free.

Given $f_0,f_1,\ldots,f_{n-1} \in \baire$, we define $f_0 \oplus f_1 \oplus \cdots \oplus f_{n-1} \in \baire$ by 
\begin{equation*}
(f_0 \oplus f_1 \oplus \cdots \oplus f_{n-1})(x) \coloneq f_{x \bmod n}(\lfloor x/n \rfloor).
\end{equation*}
E.g., $(f_0 \oplus f_1)(2x) = f_0(x)$ and $(f_0 \oplus f_1)(2x+1) = f_1(x)$. The assignment $\langle f_0,f_1,\ldots,f_{n-1}\rangle \mapsto f_0 \oplus f_1 \oplus \cdots \oplus f_{n-1}$ defines a homeomorphism $(\baire)^n \to \baire$, and restricting to $(\cantor)^n$ also yields a homeomorphism $(\cantor)^n \to \cantor$. 

For $k \geq 2$, the functions $\pi^{(k)} \colon \mathbb{N}^k \to \mathbb{N}$ denote the bijections defined recursively by
\begin{align*}
\pi^{(2)}(x,y) & \coloneq 2^x(2y+1)-1, \\
\pi^{(k+1)}(x_1,x_2,\ldots,x_k,x_{k+1}) & \coloneq \pi^{(2)}(\pi^{(k)}(x_1,x_2,\ldots,x_k),x_{k+1}).
\end{align*}
We additionally define $\pi^{(1)} \colon \mathbb{N} \to \mathbb{N}$ and $\pi^{(0)} \colon \{\langle\rangle\} \to \mathbb{N}$ by setting $\pi^{(1)} \coloneq \id_\mathbb{N}$ and $\pi^{(0)}(\langle\rangle) \coloneq 0$.

We define a bijection $\str \colon \mathbb{N} \to \{0,1\}^\ast$ by 
\begin{equation*}
\str(n) = \sigma \iff n+1 = \sum_{i=0}^{k-1}{\sigma(i) \cdot 2^i}+2^k.
\end{equation*}
Note that $n \leq m$ if and only if $\str(n) \leq_\shortlex \str(m)$.

We define a bijection $\#_\infty \colon \mathbb{N}^\ast \to \mathbb{N}$ by setting
\begin{equation*}
\#_\infty(\sigma) \coloneq \sum_{i=0}^{|\sigma|-1}{\exp_2(\sigma(0) + \sigma(1) + \cdots + \sigma(i) + i)}.
\end{equation*}
for each $\sigma \in \mathbb{N}^\ast$.
Note that if $\sigma \subseteq \tau$, then $\#_\infty(\sigma) \leq \#_\infty(\tau)$.

If $\mathbb{D}$ is an understood domain of discourse and $S \subseteq \mathbb{D}$, then the \textdef{characteristic function for $S$} is the function $\chi_S \colon \mathbb{D} \to \{0,1\}$ defined by
\begin{equation*}
\chi_S(x) \coloneq \begin{cases} 1 & \text{if $x \in S$,} \\ 0 & \text{otherwise.} \end{cases}
\end{equation*}

\section{Computability - Definitions, Notation, and Conventions}
\label{computability definitions notation and conventions section}

Here we briefly review the definitions of recursiveness/computability for various objects. With the possible exception of notation given in \cref{basic conventions and notation section} and through the remaining chapters, we have attempted to adhere to standard notation and terminology whenever possible, so the reader is encouraged to consult any of the standard references (e.g., \cite{rogers1967theory}, \cite{soare2016turing}, \cite{simpson2009degrees}, etc.) for additional background. 

\subsection{(Partial) Recursive Functions and Sets}

We define the collections of elementary, primitive, or partial recursive functions as the smallest collections of partial functions ${\subseteq} \mathbb{N}^k \to \mathbb{N}$ closed under particular operations. 

\begin{definition}
\qspace
\begin{itemize}
\item The \textdef{initial functions} consist of the zero function $Z \colon \mathbb{N} \to \mathbb{N}$ ($\forall x \qspace (Z(x) \coloneq 0)$), the successor function $S \colon \mathbb{N} \to \mathbb{N}$ ($\forall x \qspace (S(x) \coloneq x+1)$), and for each $k \in \mathbb{N}_{>0}$ and $j \in \{0,1,\ldots,k-1\}$ the projection $\pi_j^k \colon \mathbb{N}^k \to \mathbb{N}$ ($\forall x_0,x_1,\ldots,x_{k-1} \qspace (\pi_j^k(x_0,x_1,\ldots,x_{k-1}) \coloneq x_j)$).

\item Given $f$ $k$-ary and $g_1,g_2,\ldots,g_k$ each $n$-ary, their \textdef{generalized composition} is the $n$-ary function $h$ where for $\mathbf{x} \in \mathbb{N}^n$ we have
\begin{equation*}
h(\mathbf{x}) \coloneq f(g_1(\mathbf{x}),g_2(\mathbf{x}),\ldots,g_k(\mathbf{x})).
\end{equation*}

\item Given $f$ $(k+1)$-ary, its \textdef{bounded sum} and \textdef{bounded product} are the $k$-ary functions $g$ and $h$, respectively, where for $n \in \mathbb{N}$ and $\mathbf{x} \in \mathbb{N}^k$ we have
\begin{equation*}
g(n,\mathbf{x}) \coloneq \sum_{i=0}^n{f(i,\mathbf{x})} \quad \text{and} \quad h(n,\mathbf{x}) \coloneq \prod_{i=0}^n{f(i,\mathbf{x})}.
\end{equation*}

\item Given $f$ $(k+2)$-ary and $g$ $k$-ary, the result of \textdef{primitive recursion} applied to $f$ and $g$ is the $(k+1)$-ary function $h$ defined recursively for $n \in \mathbb{N}$ and $\mathbf{x} \in \mathbb{N}^k$ by
\begin{align*}
h(0,\mathbf{x}) & \coloneq g(\mathbf{x}) \\
h(n+1,\mathbf{x}) & \coloneq f(n,h(n,\mathbf{x}),\mathbf{x}).
\end{align*}

\item Given $f$ $(k+1)$-ary, its \textdef{minimization} is the $k$-ary function $g$ where for $\mathbf{x} \in \mathbb{N}^k$ we have
\begin{equation*}
g(\mathbf{x}) \coloneq \text{least $y$ such that $f(y,\mathbf{x}) = 1$}.
\end{equation*}

\end{itemize}
\end{definition}

\begin{definition}
The collection of \ldots
\begin{description}
\item[$\ldots$] \textdef{elementary recursive functions} is the smallest collection $\mathcal{C}$ of total functions of the form $\mathbb{N}^k \to \mathbb{N}$ containing the initial functions and the function $\dotminus \colon \mathbb{N}^2 \to \mathbb{N}$ defined by $x \dotminus y \coloneq \max\{x - y, 0\}$ and closed under generalized composition and taking bounded sums and products.

\item[$\ldots$] \textdef{primitive recursive functions} is the smallest collection $\mathcal{C}$ of total functions of the form $\mathbb{N}^k \to \mathbb{N}$ containing the initial functions and closed under generalized composition and primitive recursion.

\item[$\ldots$] \textdef{partial recursive functions} is the smallest collection $\mathcal{C}$ of partial functions of the form ${\subseteq} \mathbb{N}^k \to \mathbb{N}$ containing the initial functions and closed under generalized composition, primitive recursion, and minimization of its total members.


\item[$\ldots$] \textdef{total recursive functions} is the collection of total partial recursive functions.
\end{description}
\end{definition}

\begin{remark}
There are many other characterizations of the above classes. Regarding the partial recursive functions (which are exactly the partial functions computed by Turing machine programs or by register machine programs) Church's Thesis claims that any reasonable characterization of the `effectively computable' partial functions is equivalent to being partial recursive.
\end{remark}

The notion of recursiveness is extended to subsets of $\mathbb{N}^k$:

\begin{definition}[recursive predicate]
A predicate $S \subseteq \mathbb{N}^k$ is \textdef{recursive} if its characteristic function $\chi_S \colon \mathbb{N}^k \to \{0,1\}$ is recursive.
\end{definition}

To extend the notion of partial recursiveness to partial functions whose domains or codomains are not $\mathbb{N}^k$ for some $k$, we make use of \Godel\ numbers. 

\begin{definition}
Suppose $S$ and $T$ are countable sets $S$ and $T$ and then fix injections $\#_S \colon S \to \mathbb{N}$ and $\#_T \colon T \to \mathbb{N}$ for which $\im \#_S$ and $\im \#_T$ are both recursive subsets of $\mathbb{N}$, which we informally call \textdef{\Godel\ numberings} of $S$ and $T$, respectively. Then a partial function $f \colonsub S \to T$ is \textdef{partial recursive} (with respect to $\#_S$ and $\#_T$) if the partial function $\#_T \circ f \circ \#_S^{-1} \colonsub \mathbb{N} \to \mathbb{N}$ is partial recursive.
\end{definition}

\begin{convention} \label{godel numbering convention}
Unless stated otherwise, we assume the following \Godel\ numberings:
\begin{itemize}
\item $\mathbb{N}^k$ is \Godel\ numbered by $\pi^{(k)}$. 
\item $\{0,1\}^\ast$ is \Godel\ numbered by $\#_2 \coloneq \str^{-1}$.
\item $\mathbb{N}^\ast$ is \Godel\ numbered by $\#_\infty$.
\item If $\# \colon S \to \mathbb{N}$ is a \Godel\ numbering of $S$, then $S^\ast$ is \Godel\ numbered by setting $\#(\langle s_0,s_1,\ldots,s_{k-1}\rangle) \coloneq \#_\infty(\langle \#(s_0), \#(s_1),\ldots, \#(s_{k-1})\rangle)$.
\item If $\# \colon S \to \mathbb{N}$ is a \Godel\ numbering of $S$, then $\mathcal{P}_\fin(S)$ is \Godel\ numbered by setting $\#(T) \coloneq \\ \#(\langle s_0,s_1,\ldots,s_{k-1}\rangle)$, where $T = \{s_0,s_1,\ldots,s_{k-1}\}$ and $\#(s_0) < \#(s_1) < \cdots < \#(s_{k-1})$.
\item $\mathbb{Z}$ is \Godel\ numbered by setting $\#_\mathbb{Z}(n) \coloneq 2n$ if $n \geq 0$ and $\#_\mathbb{Z}(n) \coloneq 2n+1$ if $n < 0$.
\item $\mathbb{Q}$ is \Godel\ numbered by setting $\#_\mathbb{Q}(r) \coloneq \pi^{(2)}(n,m)$, where if $r = p/q$ with $\gcd(p,q) = 1$ and $q \geq 1$ then $n = \#_\mathbb{Z}(p)$ and $q$ is the $m$-th positive integer coprime with $p$.
\end{itemize}
\end{convention}

\begin{remark}
Under any reasonable \Godel\ numberings, such as in \cref{godel numbering convention}, relevant operations and relations on the \Godel\ numbered objects yield recursive operations and relations on their \Godel\ numbers. E.g., for strings (either in $\mathbb{N}^\ast$ or $\{0,1\}^\ast$), this includes for operations the length and concatenation functions and for relations the the lexicographical ordering, the shortlex ordering, and $\subseteq$; for integers and rationals, this includes virtually all number theoretic operations, the standard total orderings, and the divisibility relation.

In some cases our choices of \godel\ numberings are purposeful, as in the cases of $\pi^{(k)}$ and $\#_2$. Unless otherwise stated, we have chosen particular \godel\ numberings for the remaining cases solely for exactness, and in principle the reader may substitute them with their preferred encodings. 
\end{remark}

\subsection{Enumerations of the Partial Recursive Functions}

Thanks to their connection to effective algorithms, a partial recursive function $\theta \colonsub \mathbb{N}^k \to \mathbb{N}$ can be described by a single natural number $e$ which we may think of as encoding an algorithm computing $\theta$. 

\begin{definition}[effective enumeration]
An \textdef{enumeration} $\varphi_0,\varphi_1,\varphi_2,\ldots$ of the $k$-place partial recursive functions is any surjection from $\mathbb{N}$ onto the set of all $k$-place partial recursive functions $\theta \colonsub \mathbb{N}^k \to \mathbb{N}$. Such an enumeration is \textdef{effective} if the partial function $\Phi(e,x_1,x_2,\ldots,x_k) \simeq \varphi_e(x_1,x_2,\ldots,x_k)$ is partial recursive. 
\end{definition}

\begin{convention}
Unless otherwise specified, $\varphi_0,\varphi_1,\varphi_2,\ldots$ will be an enumeration of the $1$-place partial recursive functions. 
For concision, $\varphi_\bullet$ will denote an enumeration $\varphi_0,\varphi_1,\varphi_2,\ldots$.

Given an enumeration $\varphi_\bullet$ and $k \in \mathbb{N}_{\geq 1}$, we denote by $\varphi_\bullet^{(k)}$ the enumeration of the $k$-place partial recursive functions defined by $\varphi_e^{(k)}(x_1,\ldots,x_k) \simeq \varphi_e(\pi^{(k)}(x_1,\ldots,x_k))$ for all $e,x_1,\ldots,x_k \in \mathbb{N}$.
\end{convention}

For many purposes, an enumeration $\varphi_\bullet$ satisfying stronger properties than simply being effective is required. An example of such a property is that the $S^m_n$ Theorem holds:

\begin{property}[$S^m_n$ Theorem] 
For all $m,n \in \mathbb{N}$ there exists a primitive\footnote{It is essential that $\pi^{(k)} \colon \mathbb{N}^k \to \mathbb{N}$ be primitive recursive for each $k$ so that the fulfillment of the $S^m_n$ Theorem (particularly, the primitive recursiveness of $S^m_n$) does not depend on the choice of $\pi^{(k)}$ for $k \in \mathbb{N}$.}
 recursive function $S^m_n\colon \mathbb{N}^{m+1} \to \mathbb{N}$ such that $\varphi_{S^m_n(e,x_1,\ldots,x_m)}^{(n)}(y_1,\ldots,y_n) \simeq \varphi_e^{(m+n)}(x_1,\ldots,x_m,y_1,\ldots,y_n)$ for all $e,x_1,\ldots,x_m,y_1,\ldots,y_n \in \mathbb{N}$.
\end{property}

A weaker version that often suffices for application (and does for our uses) is the following.

\begin{property}[Parametrization Theorem] \label{parametrization theorem}
For any partial recursive function $\theta \colonsub \mathbb{N}^2 \to \mathbb{N}$, there exists a total recursive function $f \colon \mathbb{N} \to \mathbb{N}$ such that $\varphi_{f(e)}(x) \simeq \theta(e,x)$ for all $e,x \in \mathbb{N}$.
\end{property}

The Parametrization Theorem has further implications.

\begin{prop} \label{implications of parametrization theorem}
\textnormal{(well-known)}
Suppose $\varphi_\bullet$ is an effective enumeration for which the Parametrization Theorem holds.
\begin{enumerate}[(a)]
\item For all $m,n\in\mathbb{N}$ there exists a total recursive function $S^m_n \colon \mathbb{N}^{m+1} \to \mathbb{N}$ such that
\begin{equation*}
\varphi_{S^m_n(e,x_1,\ldots,x_m)}^{(n)}(y_1,\ldots,y_n) \simeq \varphi_e^{(m+n)}(x_1,\ldots,x_m,y_1,\ldots,y_n)
\end{equation*}
for all $e, x_1, \ldots, x_m, y_1, \ldots, y_n \in \mathbb{N}$.

\item \textnormal{Recursion Theorem:} For any partial recursive function $\theta \colonsub \mathbb{N}^{n+1} \to \mathbb{N}$ there exists an $e \in \mathbb{N}$ such that $\varphi_e^{(n)}(x_1,\ldots,x_n) \simeq \theta(e,x_1,\ldots,x_n)$ for all $x_1,\ldots,x_n \in \mathbb{N}$.

\end{enumerate}
\end{prop}

The Parametrization \emph{Theorem} is stated as a property of a enumeration rather than a feature of all effective enumerations due to the following observation.

\begin{definition}[admissible enumeration]
An enumeration $\varphi_\bullet$ is \textdef{admissible} if it is effective and for which the Parametrization Theorem holds.
\end{definition}

\begin{prop} \label{effective but not admissible}
\textnormal{\cite[following Definition 3]{rogers1958godel}}
There exists an effective enumeration $\varphi_\bullet$ which is not admissible.
\end{prop}

In \cite{rogers1958godel}, Rogers examines the relationships between different enumerations of the partial recursive functions\footnote{Within \cite{rogers1958godel}, Rogers uses the term `numbering' where we use `enumeration' and `semi-effective' where we use `effective'.}. There, an enumeration is defined as a surjection $\rho$ from a recursive subset $D_\rho$ of $\mathbb{N}$ onto the set of all $1$-place partial recursive functions. `Effectiveness' is defined as above, and the only modification necessary for `admissibility' is that the total recursive function $f$ in the statement of the Parametrization Theorem take values in $D_\rho$.

\begin{definition}
Suppose $\rho$ and $\tau$ are enumerations. We write $\rho \preceq \tau$ if there is a recursive function $g \colon D_\rho \to D_\tau$ such that $\rho = \tau \circ g$. 
\end{definition}

\begin{lem}
$\preceq$ is a preorder on the set of all effective enumerations.
\end{lem}

\begin{prop} \label{admissible equals fully-effective}
\textnormal{\cite[Exercise 5.10]{soare1987recursively}}
Suppose $\tau$ is an admissible enumeration. Then $\tau$ is admissible if and only if $\rho \preceq \tau$ for any effective enumeration $\rho$.
\end{prop}
\begin{proof} 
Suppose $\rho$ is an effective enumeration. It suffices to show that if $\tau$ is admissible, then \begin{enumerate*}[(i)] \item $\rho \preceq \tau$, and \item if $\tau \preceq \rho$, then $\rho$ is admissible. \end{enumerate*}

Suppose $\tau$ is admissible. $\rho$ is an effective enumeration, so the partial function $\Phi(e,x) \simeq \rho(e)(x)$ is recursive. Because $\tau$ is admissible, there exists a total recursive function $f \colon \mathbb{N} \to D_\tau$ such that $\tau(f(e))(x) \simeq \Phi(e,x)$ for all $e,x \in \mathbb{N}$, or equivalently that $\tau \circ f = \rho$. Thus, $\rho \preceq \tau$.

Now suppose $\tau \preceq \rho$, so that there is a total recursive function $g \colon D_\tau \to D_\rho$ such that $\tau = \rho \circ g$. Suppose $\theta \colonsub \mathbb{N}^2 \to \mathbb{N}$ is any partial recursive function. Because $\tau$ is admissible, there exists a total recursive function $f \colon \mathbb{N} \to \mathbb{N}$ such that $\tau(f(e))(x) \simeq \theta(e,x)$ for all $e,x \in \mathbb{N}$. But then $\rho((g\circ f)(e))(x) \simeq \theta(e,x)$ for all $e,x \in \mathbb{N}$, showing $\rho$ is admissible.
\end{proof}

\begin{cor}
Let $E$ denote the set of effective enumerations, let $\cong$ be the equivalence relation induced by $\preceq$ (i.e., $\rho \cong \tau$ if and only if $\rho \preceq \tau$ and $\tau \preceq \rho$), and let $\leq$ be the partial order on $E/{\cong}$ induced by $\preceq$ (i.e., $\rho/{\cong} \leq \tau/{\cong}$ if and only if $\rho \preceq \tau$). Then the poset $(E/{\cong},\leq)$ has a maximum, and that maximum is equal to the set of all admissible enumerations.
\end{cor}


\begin{remark}
All standard enumerations of the $1$-place partial recursive functions are admissible.
\end{remark}

\begin{convention}
Unless otherwise specified, any effective enumeration $\varphi_\bullet$ used is assumed to be admissible. Given a partial recursive $\theta$, an \textdef{index for $\theta$} is any $e \in \mathbb{N}$ for which $\theta = \varphi_e$.
\end{convention}

Many admissible enumerations have a natural way to interpret the statement, ``$\varphi_e(x)$ converges to $y$ within $s$ steps.''

\begin{notation}
Given an admissible enumeration $\varphi_\bullet$ of the $k$-place partial recursive functions, the notation $\varphi_{e,s}(x_1,x_2,\ldots,x_k)$ is used to denote the output of a partial recursive function $\langle e, x_1, x_2, \ldots, x_k, s\rangle \mapsto \varphi_{e,s}(x_1,x_2,\ldots,x_k)$ satisfying the following properties:
\begin{enumerate}[(a)]
\item $\varphi_e(x_1,x_2,\ldots,x_k) \converge$ if and only if $\varphi_{e,s}(x_1,x_2,\ldots,x_k) \converge$ for some $s \in \mathbb{N}$, in which case $\varphi_e(x_1,x_2,\ldots,x_k) = \varphi_{e,s}(x_1,x_2,\ldots,x_k)$. 
\item If $s < t$ and $\varphi_{e,s}(x_1,x_2,\ldots,x_k) \converge$, then $\varphi_{e,t}(x_1,x_2,\ldots,x_k) \converge = \varphi_{e,s}(x_1,x_2,\ldots,x_k)$.
\item The set $\{ \langle e, x_1, x_2, \ldots, x_k, s\rangle \mid \varphi_{e,s}(x_1,x_2,\ldots,x_k) \converge \}$ is recursive.\footnote{In contrast, $\{ \langle e,x_1,x_2,\ldots,x_k\rangle \mid \varphi_e(x_1,x_2,\ldots,x_k) \converge \}$ is nonrecursive and in fact $\{ \pi^{(k+1)}(e,x_1,x_2,\ldots,x_k) \mid \varphi_e(x_1,x_2,\ldots,x_k) \converge \}$ is many-one equivalent to the Halting Problem.}
\end{enumerate}
\end{notation}

\subsection{Partial Recursive Functionals}

Although $\cantor$ and $\baire$ cannot be \godel\ numbered, reasoning about initial segments of sequences in $\cantor$ and $\baire$ allows us to talk about partial recursiveness for partial functions involving these spaces.

\begin{prop} \label{characterizations of partial recursive functional}
\textnormal{(well-known)}
Suppose $\Psi \colonsub \baire \to \baire$ is given. The following are equivalent.
\begin{enumerate}[(i)]
\item There exists a partial recursive function $\Gamma_1 \colonsub \mathbb{N}^\ast \times \mathbb{N} \to \mathbb{N}$ such that 
\begin{equation*}
\forall \sigma \forall \sigma' \forall x \forall y \qspace \bigl( (\langle\sigma,x,y\rangle \in \Gamma_1 \wedge \sigma \subseteq \sigma') \implies \langle\sigma',x,y\rangle \in \Gamma_1 \bigr)
\end{equation*}
for which $\Psi(X) \simeq Y$ if and only if $\forall x \exists n \qspace \langle X \restrict n, x, Y(x) \rangle \in \Gamma_1$.

\item There exists a partial recursive function $\Gamma_2 \colonsub \mathbb{N}^\ast \times \mathbb{N} \to \mathbb{N}$ such that 
\begin{equation*}
\forall \sigma \forall \sigma' \forall x \qspace \bigl( (\langle\sigma,x\rangle \in \dom \Gamma_2 \wedge \langle\sigma',x\rangle \in \dom \Gamma_2 \wedge \sigma \subseteq \sigma') \implies \sigma = \sigma' \bigr)
\end{equation*}
for which $\Psi(X) \simeq Y$ if and only if $\forall x \exists x \qspace \langle X \restrict n, x, Y(x) \rangle \in \Gamma_2$.

\item There exists a partial recursive function $\Gamma_3 \colonsub \mathbb{N}^\ast \to \mathbb{N}^\ast$ such that 
\begin{equation*}
\forall \sigma \forall \sigma' \forall \tau \forall \tau' \qspace \bigl( (\langle\sigma,\tau\rangle \in \Gamma_3 \wedge \langle\sigma',\tau'\rangle \in \Gamma_3 \wedge \sigma \subseteq \sigma') \implies (\tau \subseteq \tau' \vee \tau' \subseteq \tau) \bigr)
\end{equation*}
for which $\Psi(X) \simeq Y$ if and only if $Y \simeq \bigcup\{ \tau \mid \exists \sigma \subset X \qspace (\langle \sigma,\tau\rangle \in \Gamma_3)\}$. 
\end{enumerate}
\end{prop}

\begin{definition}[partial recursive functional] \label{partial recursive functional definition}
A partial functional $\Psi \colonsub \baire \to \baire$ is \textdef{partial recursive} if any (equivalently, all) of the conditions in \cref{characterizations of partial recursive functional} hold.
\end{definition}

\begin{notation}
We will implicitly identify a partial recursive functional $\Psi$ with a partial recursive function $\Gamma_\Psi \colonsub \mathbb{N}^\ast \to \mathbb{N}^\ast$ as in \cref{characterizations of partial recursive functional}(iii). We define:
\begin{align*}
\Psi^X & \coloneq \bigcup\{ \tau \in \mathbb{N}^\ast \mid \exists \sigma \subset X \qspace (\langle \sigma,\tau\rangle \in \Gamma_\Psi)\} & \text{for $X \in \baire$,} \\
\Psi^{-1}(\sigma) & \coloneq \{ X \in \baire \mid \Psi^X \supseteq \sigma\} = \bbracket{\{ \tau \in \mathbb{N}^\ast 
\mid \exists \sigma' \supseteq \sigma \qspace (\langle \tau,\sigma'\rangle \in \Gamma_\Psi)\}} & \text{for $\sigma \in \mathbb{N}^\ast$, and} \\
\Psi^{-1}(S) & \coloneq \bigcup_{\sigma \in S}{\Psi^{-1}(\sigma)} & \text{for $S \subseteq \mathbb{N}^\ast$.}
\end{align*}
Note that $\Psi^X$ is a member of $\mathbb{N}^\ast \cup \baire$, and that $\Psi^X \in \baire$ if and only if $X \in \dom \Psi$, in which case $\Psi(X) = \Psi^X$.
\end{notation}

For partial functions of the form $\Psi \colonsub (\baire)^k \times \mathbb{N}^\ell \to (\baire)^m \times \mathbb{N}^n$ ($k,m \geq 1$) to be partial recursive, we may either make appropriate adjustments to the above definitions, or reduce to the case ${\subseteq} \baire \to \baire$ by associating a tuple $\langle X_0, X_1, \ldots, X_{k-1}, x_0,x_1,\ldots,x_{\ell-1}\rangle$ with $\langle x_0,x_1,\ldots,x_{\ell-1}\rangle \concat (X_0 \oplus X_1 \oplus \cdots \oplus X_{k-1})$.

\subsection{The Arithmetical Hierarchy}


The arithmetical hierarchy provides one stratification of the complexity of subsets of $(\baire)^k \times \mathbb{N}^\ell$, and although we will only be interested in very low levels of that hierarchy, it is more convenient to define the general notion than those individual levels separately.

\begin{definition}[arithmetical hierarchy]
A subset $S \subseteq (\baire)^k \times \mathbb{N}^\ell$ is $\Sigma^0_0$ and $\Pi^0_0$ if $S$ is recursive. Given $\Sigma^0_n$ and $\Pi^0_n$ have been defined, we say that $S$ is $\Sigma^0_{n+1}$ if there exists a $\Pi^0_n$ subset $R \subseteq (\baire)^k \times \mathbb{N}^{\ell+1}$ such that 
\begin{equation*}
S(X_1,\ldots,X_k,y_1,\ldots,y_\ell) \equiv \exists m \qspace R(X_1,\ldots,X_k,y_1,\ldots,y_\ell,m).
\end{equation*}
for all $X_1,\ldots,X_k \in \baire$ and $y_1,\ldots,y_\ell \in \mathbb{N}$. Likewise, $S$ is $\Pi^0_{n+1}$ if there exists a $\Sigma^0_n$ subset $R \subseteq (\baire)^k \times \mathbb{N}^{\ell+1}$ such that
\begin{equation*}
S(X_1,\ldots,X_k,y_1,\ldots,y_\ell) \equiv \forall m \qspace R(X_1,\ldots,X_k,y_1,\ldots,y_\ell,m)
\end{equation*}
for all $X_1,\ldots,X_k \in \baire$ and $y_1,\ldots,y_\ell \in \mathbb{N}$.
\end{definition}

Other names are given to the lowest nontrivial level of the arithmetical hierarchy for subsets of $\mathbb{N}^k$.

\begin{definition}[(co-)recursively enumerable]
A subset $X \subseteq \mathbb{N}^k$ is \textdef{recursively enumerable}, or \textdef{r.e.}, if $X$ is $\Sigma^0_1$, and \textdef{co-recursively enumerable}, or \textdef{co-r.e.}, if $X$ is $\Pi^0_1$.
\end{definition}

\begin{remark}
More generally, we can extend the use of `r.e.' and `co-r.e.' to sets of objects which are \godel\ numbered. In particular, it makes sense to speak of a set of strings $S \subseteq \mathbb{N}^\ast$ being r.e.\ or co-r.e.
\end{remark}

%
%

Often we are interested in sequences of sets each of which is at the same level of the arithmetical hierarchy in a uniform way. 

\begin{definition}[uniformly $\Sigma^0_n$/$\Pi^0_n$]
A sequence $\langle S_i \rangle_{i \in I}$ of subsets $S_i \subseteq (\baire)^k \times \mathbb{N}^\ell$ is \textdef{uniformly $\Sigma^0_n$} (respectively, \textdef{uniformly $\Pi^0_n$}) if it is $\Sigma^0_n$ (resp., $\Pi^0_n$) when considered as a subset of $(\baire)^k \times \mathbb{N}^{\ell+1}$.
\end{definition}

\subsection{Recursive Reals and Real-Valued Functions}

Like $\cantor$ and $\baire$, although $\mathbb{R}$ cannot be \godel\ numbered, we can reason about real-valued functions by approximating real outputs by recursive sequences of rational numbers.

\begin{definition}[recursive/computable real]
$\alpha \in \mathbb{R}$ is \textdef{left recursively enumerable}, or \textdef{left r.e.}, if there is a recursive sequence $\langle p_n \rangle_{n \in \mathbb{N}}$ of rational numbers converging monotonically to $\alpha$ from below. 

$\alpha$ is \textdef{right recursively enumerable}, or \textdef{right r.e.}, if there is a recursive sequence $\langle q_n \rangle_{n \in \mathbb{N}}$ of rational numbers converging monotonically to $\alpha$ from above. 

If $\alpha$ is both left r.e.\ and right r.e., then $\alpha$ is \textdef{recursive} or \textdef{computable}.
\end{definition}

\begin{lem} \label{equivalent characterizations of recursive reals}
\textnormal{(well-known)}
Suppose $\alpha$ is a real number. 
\begin{enumerate}[(a)]
\item $\alpha$ is recursive if and only if there exists a monotone recursive sequence $\langle \alpha_k \rangle_{k \in \mathbb{N}}$ of rational numbers such that $|\alpha - \alpha_k| \leq 2^{-k}$ for each $k \in \mathbb{N}$. 
\item $\alpha$ is left r.e.\ if and only if there exists a sequence $\langle \alpha_k \rangle_{k \in \mathbb{N}}$ of uniformly recursive reals $\alpha_k \leq \alpha$ converging to $r$.
\item $\alpha$ is right r.e.\ if and only if there exists a sequence $\langle \alpha_k \rangle_{k \in \mathbb{N}}$ of uniformly recursive reals $\alpha_k \geq \alpha$ converging to $\alpha$.
\end{enumerate}
\end{lem}
\begin{proof}
Straight-forward.
\end{proof}

\begin{definition}[recursive/computable real-valued function of a discrete variable]
$f \colon \mathbb{N} \to \mathbb{R}$ is \textdef{left recursively enumerable}, or \textdef{left r.e.}, if there is a recursive function $p_{-,-} \colon \mathbb{N}^2 \to \mathbb{Q}$ such that for each $x \in \mathbb{N}$, the sequence $\langle p_{n,x} \rangle_{n \in \mathbb{N}}$ converges monotonically to $f(x)$ from below. 

$f$ is \textdef{right recursively enumerable}, or \textdef{right r.e.}, if there is a recursive funciton $q_{-,-} \colon \mathbb{N}^2 \to \mathbb{Q}$ such that for each $x \in \mathbb{N}$, the sequence $\langle q_{n,x} \rangle_{n \in \mathbb{N}}$ converges monotonically to $f(x)$ from above. 

If $f$ is both left r.e.\ and right r.e., then $f$ is \textdef{recursive} or \textdef{computable}.
\end{definition}

\begin{remark}
In the definitions of left r.e., right r.e., and recursive functions $f \colon \mathbb{N} \to \mathbb{R}$, we may replace $\mathbb{N}$ with any set $S$ which can be \godel\ numbered, associating a function $g \colon S \to \mathbb{R}$ with the function $g \circ \#^{-1} \colon \mathbb{N} \to \mathbb{R}$, where $\# \colon S \to \mathbb{N}$ is a \godel\ numbering.
\end{remark}

One of the principal ways in which recursive real-valued functions appear is to quantify growth rate. 

\begin{definition}[order function]
An \textdef{order function} is a function $f \colon \mathbb{N} \to \mathbb{R}$ which is nondecreasing, unbounded, and computable.
\end{definition}

Now we describe what it means for a real-valued function of a real variable to be computable. 

\begin{definition}[recursive/computable real-valued function of a real variable]
A function $\chi\colon \mathbb{N} \to \mathbb{Q}$ \textdef{represents} $x \in \mathbb{R}$ if $|\chi(n) - x| \leq 2^{-n}$ for all $n \in \mathbb{N}$.

Suppose $I \subseteq \mathbb{R}$ is an interval. A function $f \colon I \to \mathbb{R}$ is \textdef{recursive} or \textdef{computable} if there exists a partial recursive functional $\Psi \colonsub \mathbb{Q}^\mathbb{N} \to \mathbb{Q}^\mathbb{N}$ such that whenever $\chi\colon \mathbb{N} \to \mathbb{Q}$ represents $x \in I$ and $\im \chi \subseteq I$, $\Psi(\chi) \converge$ and $\Psi(\chi)$ represents $f(x)$.
\end{definition}

\begin{remark}
By `interval' we include both the bounded intervals $\oo{a,b}$, $\co{a,b}$, $\oc{a,b}$, $[a,b]$, and the unbounded intervals $(-\infty,\infty)$, $\co{a,\infty}$, $\oc{-\infty,b}$, $(a,\infty)$, $(-\infty,b)$.
\end{remark}

We will assume the following basic facts concerning computable real-valued functions of a real variable.

\begin{prop}
\textnormal{(well-known)}
\begin{enumerate}[(a)]
\item The functions $x \mapsto \alpha$, $x \mapsto x^\beta$, $x \mapsto \log_2 x$, $x \mapsto 2^x$, and $x \mapsto \lfloor x \rfloor$ are computable, where each function has its natural domain and $\alpha \in \mathbb{R}$ and $\beta \in \mathbb{R}_{\geq 0}$ are computable.

\item If $f \colon \mathbb{N} \to \mathbb{R}$ is computable, then the piecewise-linear extension $\overline{f} \colon \co{0,\infty} \to \mathbb{R}$ defined by $\overline{f}(x) \coloneq (f(\lfloor x \rfloor+1) - f(\lfloor x \rfloor))(x-\lfloor x \rfloor) + f(\lfloor x \rfloor)$ for each $x \in \co{0,\infty}$ is computable.

\item If $f \colon \co{0,\infty} \to \mathbb{R}$ is computable, then $f \restrict \mathbb{N}$ is computable.

\item If $f \colon I \to \mathbb{R}$ is computable and $J$ is an interval whose endpoints are either infinite or finite, computable reals, then $f \restrict (I \cap J)$ is computable.

\item If $f,g \colon I \to \mathbb{R}$ are computable, then $f+g$ and $f\cdot g$ are computable.

\item If $f \colon I \to \mathbb{R}$ and $g\colon J \to \mathbb{R}$ are computable and $\im f \subseteq J$, then their composition $g \circ f$ is computable.

\end{enumerate}
\end{prop}
\begin{proof}
Cumbersome but routine.
\end{proof}

%
%
%
%
%

\section{Reducibility Notions}
\label{reducibility notions section}

\subsection{Turing Reducibility}

The principal way to measure the `degree of unsolvability' of an infinite sequence $X \in \baire$ is Turing reducibility.

\begin{thm} \label{equivalents of turing reducibility}
\textnormal{(well-known)}
Suppose $X,Y \in \baire$. The following are equivalent.
\begin{enumerate}[(i)]
\item $X$ is a member of the smallest collection of partial functions ${\subseteq} \mathbb{N}^k \to \mathbb{N}$ containing the initial functions, containing $Y$, and closed under generalized composition, primitive recursion, and minimization.
\item There exists an oracle Turing machine which computes $X$ given an oracle for $Y$.
\item There exists an oracle register machine program which computes $X$ given an oracle for $Y$.
\item There exists a recursive functional $\Psi \colonsub \baire \to \baire$ such that $\Psi(Y)\converge = X$.
\end{enumerate}
\end{thm}

\begin{definition}[Turing reducibility]
Given $X,Y \in \baire$, we say that $X$ is \textdef{Turing reducible} to $Y$, written $X \turingleq Y$, if any of the equivalent conditions in \cref{equivalents of turing reducibility} hold. We may also say that $X$ is \textdef{$Y$-computable}, \textdef{$Y$-recursive}, or that \textdef{$Y$ computes $X$}. 

If $X \turingleq Y$ and $Y \turingleq X$, then we say that $X$ and $Y$ are \textdef{Turing equivalent} and write $X \turingeq Y$. 
\end{definition}

$\turingleq$ is a preorder, and hence $\turingeq$ is an equivalence relation.

\begin{definition}[Turing degree]
The \textdef{Turing degree} of an infinite sequence $X \in \baire$ is the $\turingeq$-equivalence class containing $X$, written $\turingdeg(X)$. 

The collection of all Turing degrees is written $\mathcal{D}_\turing$. $\turingleq$ induces a partial order $\leq$ on $\mathcal{D}_\turing$ defined by setting $\turingdeg(X) \leq \turingdeg(Y)$ if and only if $X \turingleq Y$.
\end{definition}

\begin{example}
$\mathbf{0}$ is the Turing degree of any recursive $X \in \baire$. It is the minimum of $\mathcal{D}_\turing$.
\end{example}

\begin{example}
Consider the following subsets of $\mathbb{N}$, which we tacitly identify with their characteristic functions:
\begin{align*}
H_1 & \coloneq \{ e \in \mathbb{N} \mid \varphi_e(0) \converge\}, \\
H_2 & \coloneq \{ e \in \mathbb{N} \mid \varphi_e(e) \converge\}, \\
H_3 & \coloneq \{ e \in \mathbb{N} \mid \text{$\varphi_i(n) \converge$, where $\pi^{(2)}(i,n) = e$}\}.
\end{align*}
All three Turing equivalent to one another. Firstly, $H_1 \turingleq H_3$ because $e \in H_1$ if and only if $\pi^{(2)}(e,0) \in H_3$; likewise, $H_2 \turingleq H_3$ because $e \in H_1$ if and only if $\pi^{(2)}(e,e) \in H_3$. For the opposite direction, consider the partial recursive function $\theta(\pi^{(2)}(i,n),m) \simeq \varphi_i(n)$; the Parametrization Theorem yields a total recursive function $f \in \baire$ such that $\varphi_{f(e)}(m) \simeq \theta(e,m)$ for all $e,m\in\mathbb{N}$. Then
\begin{equation*}
e \in H_3 \iff f(e) \in H_1 \iff f(e) \in H_2,
\end{equation*}
showing $H_3 \turingleq H_1$ and $H_3 \turingleq H_2$. $H_3$ is the Halting problem, though thanks to the above equivalences we may also refer to $H_1$ or $H_2$ as the Halting problem. The Turing degree of $H_1 \turingeq H_2 \turingeq H_3$ is called $\mathbf{0}'$ and $X \in \baire$ is said to be \textdef{complete} if $\mathbf{0}' \leq \turingdeg(X)$.
\end{example}

Some simple facts about Turing degrees include the following:

\begin{lem}
Suppose $X,Y \in \baire$ are given.
\begin{enumerate}[(a)]
\item $\turingdeg(X \oplus Y) = \sup\{ \turingdeg(X), \turingdeg(Y)\}$. Consequently, $(\mathcal{D}_\turing,\leq)$ is a join semi-lattice. 
\item There exists $Z \in \cantor$ such that $X \turingeq Z$. 
\end{enumerate}
\end{lem}
\begin{proof}
Straight-forward.
\end{proof}

In analogy with admissible enumerations of the $k$-place partial recursive functions we may also consider admissible enumerations of the $k$-place partial recursive functionals ${\subseteq} \baire \times \mathbb{N}^k \to \mathbb{N}$. 

\begin{definition}[admissible enumeration of partial recursive functionals]
An \textdef{admissible enumeration} of the $k$-place partial recursive functionals ${\subseteq} \baire \times \mathbb{N}^k \to \mathbb{N}$ is a sequence $\Phi_0,\Phi_1, \Phi_2,\ldots$ of such functionals such that:
\begin{enumerate}[(i)]
\item For every partial recursive functional $\Psi \colonsub \baire \times \mathbb{N}^k \to \mathbb{N}$ there is $e \in \mathbb{N}$ such that $\Psi = \Phi_e$.
\item The partial functional $\Phi(e,X,x_1,x_2,\ldots,x_k) \simeq \Phi_e(X,x_1,x_2,\ldots,x_k)$ is partial recursive.
\item For any partial recursive functional $\Theta \colonsub \baire \times \mathbb{N}^{k+1} \to \mathbb{N}$ there exists a total recursive $g \colon \mathbb{N} \to \mathbb{N}$ such that, for all $f \in \baire$ and $e,x_1,x_2,\ldots,x_k \in \mathbb{N}$,
\begin{equation*}
\Phi_{g(e)}(f,x_1,x_2,\ldots,x_k) \simeq \Theta(f,e,x_1,x_2,\ldots,x_k).
\end{equation*}
\end{enumerate}
\end{definition}

\begin{notation}
The notation $\langle X,e,x_1,x_2,\ldots,x_k \rangle \mapsto \varphi_e^{X}(x_1,x_2,\ldots,x_k)$ will often be used to denote an (admissible) enumeration of the $k$-place partial recursive functionals ${\subseteq} \baire \times \mathbb{N}^k \to \mathbb{N}$. 

The notation $\varphi_{e,s}^{\tau}(x_1,x_2,\ldots,x_k)$ is used to denote the output of a partial recursive function \\ $\langle e, x_1, x_2, \ldots, x_k, s, \tau\rangle \mapsto \varphi_{e,s}^{\tau}(x_1,x_2,\ldots,x_k)$ satisfying the following properties:
\begin{enumerate}[(i)]
\item For every $X \in \baire$, $\varphi_e^{X}(x_1,x_2,\ldots,x_k) \converge$ if and only if $\varphi_{e,s}^{X\restrict s}(x_1,x_2,\ldots,x_k) \converge$ for some $s \in \mathbb{N}$, in which case $\varphi_e^{X}(x_1,x_2,\ldots,x_k) = \varphi_{e,s}^{X\restrict s}(x_1,x_2,\ldots,x_k)$.
\item If $\sigma \subseteq \tau$, $s \leq t$, and $\varphi_{e,s}^{\sigma}(x_1,x_2,\ldots,x_k) \converge$, then $\varphi_{e,t}^{\tau}(x_1,x_2,\ldots,x_k) \converge = \varphi_{e,s}^{\sigma}(x_1,x_2,\ldots,x_k)$.
\item The set $\{ \langle e, x_1, x_2, \ldots, x_k, s, \tau\rangle \mid \varphi_{e,s}^{\tau}(x_1,x_2,\ldots,x_k) \converge\}$ is recursive.
\end{enumerate}
\end{notation}

\subsection{Mass Problems and Weak, Strong Reducibility}

When measuring the `degree of unsolvability' of a \emph{subset} of $\baire$ there are two notions of reducibility we consider.

\begin{definition}[weak reducibility]
Given $P,Q \subseteq \baire$, $P$ is \textdef{weakly reducible}
to $Q$, written $P \weakleq Q$, if for every $Y \in Q$ there exists $X \in Q$ such that $X \turingleq Y$.

If $P \weakleq Q$ and $Q \weakleq P$, then $P$ and $Q$ are \textdef{weakly equivalent} 
, written $P \weakeq Q$.
\end{definition}

\begin{definition}[strong reducibility]
Given $P,Q \subseteq \baire$, $P$ is \textdef{strongly reducible} 
 to $Q$, written $P \strongleq Q$, if there exists a recursive functional $\Psi \colon Q \to P$.

If $P \strongleq Q$ and $Q \weakleq P$, then $P$ and $Q$ are \textdef{strongly equivalent} 
, written $P \strongeq Q$.
\end{definition}

\begin{remark}
Weak reducibility is also called \emph{Muchnik reducibility}, and strong reducibility is also called \emph{Medvedev reducibility}.
\end{remark}

An interpretation of a subset $P \subseteq \baire$ is as a \emph{problem} (and within this context subsets of $\baire$ are sometimes called \textdef{mass problems}) whose elements are its \emph{solutions}. If $P \weakleq Q$, then any solution to $Q$ computes a solution to $P$, though this is not necessarily a uniform procedure. If $P \strongleq Q$, then there is a uniform procedure to turn a solution to $Q$ into a solution to $P$. See \cite{simpson2007extension} and \cite{simpson2005mass} for additional information and motivation. 

As with Turing reducibility, both weak and strong reducibility are preorders on $\mathcal{P}(\baire)$ and $\weakeq$ and $\strongeq$ define equivalence relations on $\mathcal{P}(\baire)$.

\begin{definition}[weak and strong degrees]
Suppose $P \subseteq \baire$. The \textdef{weak degree} of $P$, $\weakdeg(P)$, is the $\weakeq$-equivalence class containing $P$, and the \textdef{strong degree} of $P$, $\strongdeg(P)$, is the $\strongeq$-equivalence class containing $P$.

The set of all weak degrees is denoted by $\mathcal{D}_\weak$, while the set of all strong degrees is denoted by $\mathcal{D}_\strong$.
\end{definition}

Here we collect some simple facts about $\mathcal{D}_\weak$ and $\mathcal{D}_\strong$ that we use repeatedly and often implicitly.

\begin{prop} \label{properties of weak and strong reducibility}
\textnormal{(well-known)}
Let $P$ and $Q$ be mass problems and let $\mathcal{Q}$ a family of mass problems.
\begin{enumerate}[(a)]
\item $P \strongleq Q$ implies $P \weakleq Q$.
\item $Q \subseteq P$ implies $P \strongleq Q$.
\item $P \weakeq P^{\turingleq}$, where $P^{\turingleq} \coloneq \{ Y \in \baire \mid \exists X \in P \qspace (X \turingleq Y)\}$ is the \textdef{Turing upward closure} of $P$.
\item $\inf\{\weakdeg(P) \mid P \in \mathcal{Q}\} = \weakdeg\left(\bigcup\mathcal{Q}\right)$.
\item $\sup\{\weakdeg(P) \mid P \in \mathcal{Q}\} = \weakdeg\left(\bigcap\{ P^{\turingleq} \mid P \in \mathcal{Q}\}\right)$.
\item $\langle \mathcal{D}_\weak,\leq\rangle$ is a completely distributive lattice.
\item $\inf\{\strongdeg(P),\strongdeg(Q)\} = \strongdeg(P \times Q)$, where $P \times Q \coloneq \{ X \oplus Y \mid X \in P \wedge Y \in Q\}$.
\item $\sup\{\strongdeg(P),\strongdeg(Q)\} = \strongdeg(\{\langle 0 \rangle \concat X \mid X \in P\} \cup \{\langle 1 \rangle \concat Y \mid Y \in Q\})$.
\item $\REC \strongleq P \strongleq \emptyset$, where $\REC = \{ X \in \baire \mid \text{$X$ is recursive}\}$.
\item $\langle \mathcal{D}_\strong,\leq \rangle$ is a bounded distributive lattice.
\item For all $X,Y \in \baire$, $X \turingleq Y$ if and only if $\{X\} \weakleq \{Y\}$, and if and only if $\{X\} \strongleq \{Y\}$. Consequently, $\turingdeg(X) \mapsto \weakdeg(\{X\})$ and $\turingdeg(X) \mapsto \strongdeg(\{X\})$ define embeddings of $\langle \mathcal{D}_\turing,\leq \rangle$ into $\langle \mathcal{D}_\weak,\leq\rangle$ and $\langle \mathcal{D}_\strong,\leq \rangle$, respectively.
\item $P \weakeq \REC$ if and only if $P \cap \REC \neq \emptyset$.
\end{enumerate}
\end{prop}

Examples of mass problems, weak reductions, and strong reductions will be encountered throughout the remainder of this thesis, though we give one example presently which is neither empty nor weakly or strongly equivalent to $\{f\}$ for some $f \in \baire$. 

\begin{example}
Suppose $P \cap \REC = \emptyset$. Then $P \subseteq \REC^\mathsf{c}$, so $\REC^\mathsf{c} \strongleq P$. It follows that $\weakdeg(\REC^\mathsf{c})$ and $\strongdeg(\REC^\mathsf{c})$ are immediate successors of the minimum of $\mathcal{D}_\weak$ and $\mathcal{D}_\strong$, respectively.

That $\REC^\mathsf{c} \weakneq \{X\}$ (and consequently $\REC^\mathsf{c} \strongneq \{X\}$) for any $X \in \baire$ follows from the fact that there exist incomparable minimal nonrecursive Turing degrees.
\end{example}

\subsection{\texorpdfstring{$\mathcal{E}_\weak$}{E_w} and the Embedding Lemma}

Particular attention is given to the weak degrees of nonempty $\Pi^0_1$ subsets of $\cantor$, which we call \textdef{$\Pi^0_1$ classes}, serving a similar role in $\mathcal{D}_\weak$ as the collection $\mathcal{E}_\turing$ of r.e.\ Turing degrees (i.e., Turing degrees of r.e.\ sequences) in $\mathcal{D}_\turing$. 

\begin{definition}
$\mathcal{E}_\weak \coloneq \{ \weakdeg(P) \mid \text{$P \subseteq \cantor$ is $\Pi^0_1$}\}$.
\end{definition}

One motivation for $\mathcal{E}_\weak$ in comparison to $\mathcal{E}_\turing$ is that specific, natural examples of weak degrees in $\mathcal{E}_\weak$ can be given while there are no known specific, natural r.e.\ degrees aside from $\mathbf{0}$ and $\mathbf{0}'$. See \cite{simpson2007extension} and \cite{simpson2005mass} for additional details.

\begin{prop}
Suppose $P \subseteq \cantor$ is given. The following are equivalent.
\begin{enumerate}[(a)]
\item $P$ is $\Pi^0_1$.
\item There exists a recursive tree $T \subseteq \{0,1\}^\ast$ such that $P$ is the set of paths through $T$.
\item There exists $e \in \mathbb{N}$ such that $P = \{ X \in \cantor \mid \varphi_e^X(0) \diverge\}$.
\end{enumerate}
\end{prop}

Although our interests lie chiefly within $\mathcal{E}_\weak$, the weak degrees we consider are often most naturally represented by mass problems which are not $\Pi^0_1$ subsets of $\cantor$. The following result shows that this issue can be side-stepped as long as the mass problem is sufficiently low in the arithmetical hierarchy.

\begin{prop}[Embedding Lemma] \label{embedding lemma}
\textnormal{\cite[Lemma 17.1]{simpson2009degrees}}
Suppose $P$ is a nonempty $\Pi^0_1$ subset of $\cantor$ and $S$ is a $\Sigma^0_3$ subset of $\baire$. Then there exists a nonempty $\Pi^0_1$ subset $Q$ of $\cantor$ such that $Q \weakeq P \cup S$.
\end{prop}

One particular case which is especially well-behaved is with \emph{recursively bounded $\Pi^0_1$ classes}. 

\begin{definition}
Suppose $h \colon \mathbb{N} \to (0,\infty)$ is a computable function. We write
\begin{align*}
h^n & \coloneq \{ \sigma \in \mathbb{N}^n \mid \forall i < n \qspace ( \sigma(i) < h(i) )\}, \\
h^\ast & \coloneq \{ \sigma \in \mathbb{N}^\ast \mid \forall i < |\sigma| \qspace ( \sigma(i) < h(i) )\} = \bigcup_{n\in\mathbb{N}}{h^n}, \\
h^\mathbb{N} & \coloneq \{ X \in \baire \mid \forall i \qspace ( X(i) < h(i) )\}.
\end{align*}
\end{definition}

In other words, $h^n$ is the set of $h$-bounded strings of length $n$, $h^\ast$ is the set of all $h$-bounded strings, and $h^\mathbb{N}$ is the set of $h$-bounded infinite sequences. 

\begin{lem}
The subspace topology on $h^\mathbb{N} \subseteq \baire$ has a basis $\{ \bbracket{\sigma}_h \mid \sigma \in h^\ast\}$, where for $\sigma \in h^\ast$ we define
\begin{equation*}
\bbracket{\sigma}_h \coloneq \{ X \in h^\mathbb{N} \mid \sigma \subset X\}.
\end{equation*}
\end{lem}
\begin{proof}
For all $\sigma \in h^\ast$, $h^\mathbb{N} \cap \bbracket{\sigma} = \bbracket{\sigma}_h$. If $\sigma \in \mathbb{N}^\ast \setminus h^\ast$, then $h^\mathbb{N} \cap \bbracket{\sigma} = \emptyset$.
\end{proof}

\begin{prop} \label{hN is recursively homeomorphic to cantor space}
\textnormal{(well-known)}
$h^\mathbb{N}$ is recursively homeomorphic to $\cantor$.
\end{prop}
\begin{proof}
Define $\psi \colon h^\ast \to \{0,1\}^\ast$ recursively as follows: $\psi(\langle\rangle) = \langle\rangle$ and given $\psi(\sigma)$ has been defined, let $\psi(\sigma \concat \langle i \rangle) \coloneq \psi(\sigma) \concat \langle 1\rangle^i \concat \langle 0 \rangle$ for each $i < h(|\sigma|)-1$ and $\psi(\sigma \concat \langle h(|\sigma|)-1\rangle) \coloneq \psi(\sigma) \concat \langle 1 \rangle^{h(|\sigma|)-1}$. We make the following observations: \begin{enumerate*}[(i)] \item for all $\sigma, \sigma' \in h^\ast$, $\sigma \subseteq \sigma'$ if and only if $\psi(\sigma) \subseteq \psi(\sigma')$, and \item for all $\sigma \in h^\ast$, $\bbracket{\psi(\sigma)}_2 = \bigcup_{i < h(|\sigma|)}{\bbracket{\psi(\sigma\concat\langle i\rangle)}_2}$. \end{enumerate*}

Now define $\Psi \colon h^\mathbb{N} \to \cantor$ by $\Psi(X) \coloneq \bigcup_{n \in \mathbb{N}}{\psi(X \restrict n)}$. Observation (i) above implies $\Psi$ is well-defined and injective, while observation (ii) implies $\Psi$ is surjective. Given $\tau \in \{0,1\}^\ast$, $\Psi^{-1}[\bbracket{\tau}_2] = \bigcup\{ \bbracket{\sigma}_h \mid \psi(\sigma) \supseteq \tau\}$, showing $\Psi$ is continuous. Conversely, given $\sigma \in h^\ast$, $\Psi[\bbracket{\sigma}_h] = \bbracket{\psi(\sigma)}_2$ -- that $\Psi[\bbracket{\sigma}_h] \subseteq \bbracket{\psi(\sigma)}_2$ is immediate, while the reverse inclusion follows from observation (ii) above -- and hence $\Psi$ is an open map. Since $\Psi$ is clearly recursive, it is a recursive homeomorphism.
\end{proof}

\begin{definition}[recursively bounded]
$P \subseteq \baire$ is \textdef{recursively bounded}, or \textdef{r.b.}, if there exists an recursive function $h \colon \mathbb{N} \to (1,\infty)$ such that $P \subseteq h^\mathbb{N}$. 

In particular, a \textdef{recursively bounded $\Pi^0_1$ class}, or a \textdef{r.b.\ $\Pi^0_1$ class}, is a recursively bounded and $\Pi^0_1$ subset of $\baire$.
\end{definition}

\begin{prop} \label{extending recursive functionals on recursively bounded pi01 classes}
\textnormal{\cite[Theorem4.7]{simpson2005mass}}
Suppose $P$ is a r.b.\ $\Pi^0_1$ class and that $\Psi \colon P \to \baire$ is a recursive functional. 
\begin{enumerate}[(a)]
\item The image $\Psi[P]$ is recursively bounded and $\Pi^0_1$.
\item $\Psi$ extends to a total recursive functional $\tilde{\Psi} \colon \baire \to \baire$.
\end{enumerate}
\end{prop}

\begin{cor} \label{recursively bounded pi01 class recursively homeomorphic to pi01 class in cantor space}
Suppose $P$ is a r.b.\ $\Pi^0_1$ class. Then there exists a $\Pi^0_1$ subset $Q$ of $\cantor$ which is recursively homeomorphic to $P$.
\end{cor}
\begin{proof}
Let $\Psi$ be the recursive homeomorphism defined in the proof of \cref{hN is recursively homeomorphic to cantor space}. Then \cref{extending recursive functionals on recursively bounded pi01 classes}(a) shows that the image $Q$ of $P$ under $\Psi$ is another $\Pi^0_1$ class.
\end{proof}

\clearpage
\chapter{Complexity, Avoidance, and Depth}
\label{background chapter}

In \cref{introduction chapter} we gave brief definitions of $\complex(f)$ and $\dnr(p)$ for $f$ an order function and $p$ a nondecreasing computable function, as well as alluded to a variation on $\dnr$ which we termed $\ldnr$. 

In \cref{algorithmic randomness and complexity section}, we define Martin-\Lof\ randomness through three paradigms as a precursor to their generalizations which give rise to $f$-randomness and strong $f$-randomness for any computable function $f \colon \{0,1\}^\ast \to \mathbb{R}$, with Martin-\Lof\ randomness corresponding to $(\lambda \sigma.|\sigma|)$-randomness. The classes $\complex(f)$ are defined and shown to lie in $\mathcal{E}_\weak$. Finally, we list some of the properties of prefix-free and conditional prefix-free complexity that we will use later.

In \cref{dnr and avoidance section}, we show that the effect of the growth rate of $p$ on $\weakdeg\dnr(p)$ depends explicitly on the choice of admissible enumeration used, and motivate the definition of the class $\avoid^\psi(p)$ for any computable $p \colon \mathbb{N} \to (1,\infty)$ and partial recursive $\psi$. Linearly universal partial recursive functions are defined, followed by defining $\ldnr(p)$ as the union of the classes $\avoid^\psi(p)$ as $\psi$ ranges over those linearly universal partial recursive functions. Basic and technical results are covered for the linearly universal partial recursive functions, $\ldnr(p)$, and $\avoid^\psi(p)$ more generally.

In \cref{fast and slow-growing order functions section}, we formally define the notion of being fast-growing and slow-growing for an order function and address the problem ``Given a recursive sequence of fast-growing (resp., slow-growing) order functions $\langle p_k\rangle_{k \in \mathbb{N}}$, find fast-growing (resp., slow-growing) order functions $q^+$ and $q^-$ such that $p_k \domleq q^+$ and $q^- \domleq p_k$ for all $k \in \mathbb{N}$.'' Towards that end, for the slow-growing case we prove that such a $q^-$ always exists and that a $q^+$ exists with additional hypotheses on $\langle p_k \rangle_{k \in \mathbb{N}}$ (\cref{lower bound of sequence of divergent series}), but that a $q^+$ need not exist in general (\cref{max of slow-growing example}). On the other hand, for the fast-growing case we can prove that $q^+$ always exists and that $q^-$ exists with additional hypotheses on the $p_k$'s and the sequence $\langle p_k \rangle_{k \in \mathbb{N}}$ (\cref{lower and upper bounds of sequence of fast-growing functions}).

%
%

To better understand the extra hypothesis of requiring $\sum_{n=0}^\infty{p(n)^{-1}}$ not only be finite but also recursive, we prove the following equivalence:

\begin{repprop}{infinite sum recursive iff improper integral recursive}
Suppose $p \colon \mathbb{N} \to (0,\infty)$ is a fast-growing order function and let $\overline{p} \colon [0,\infty) \to (0,\infty)$ be any continuous nonincreasing extension of $p$. Then $\sum_{n=0}^\infty{p(n)^{-1}}$ is a recursive real if and only if $\int_0^\infty{\overline{p}(x)^{-1}\dd x}$ is a recursive real.
\end{repprop}

\cref{depth section} introduces the notion of depth for r.b.\ $\Pi^0_1$ classes, of which our interest is based on strong general properties of deep r.b.\ $\Pi^0_1$ classes and the fact that the classes $\ldnr_p$ are deep exactly when $p$ is slow-growing. Depth is shown to be well-behaved with respect to $\strongleq$ while slightly less well-behaved for $\weakeq$. We end the section with a discussion of the applicability of `depth' to subsets of $\baire$ which are not r.b.\ $\Pi^0_1$ classes.

\section{Algorithmic Randomness and Complexity}
\label{algorithmic randomness and complexity section}

Downey \& Hirschfeldt identify three paradigms through which one can attempt to make precise the idea of `algorithmic randomness' or `algorithmic complexity'. \cite[Chapter 6]{downey2010algorithmic}
\begin{description}
\item[The measure-theoretic paradigm:] If $X \in \cantor$ is `random', then it should pass all `statistical tests' (e.g., $X$ should obey the Law of Large Numbers, the Law of the Iterated Logarithm, etc.). Any `statistical test' should be such that the set of sequences failing that statistical test should be `effectively null' (so $X$ should not fall into any effectively null subset of $\cantor$).

\item[The computational paradigm:] If $X \in \cantor$ is `random', then the initial segments of $X$ should be `maximally difficult' to describe, in the sense that we should need to know roughly $n$ bits of information in order to describe $X \restrict n$. 

\item[The unpredictability paradigm:] If $X \in \cantor$ is `random' and we imagine that each bit of $X$ represents the result of a coin flip whose outcome we are betting on, then there shouldn't be any strategy by which we make arbitrarily high earnings.

\end{description}

There are several ways to make precise the notion of randomness from any of these three paradigms (some of them inequivalent), and our interest will not solely be on `randomness' but on notions of `partial randomness'. Quantifying `how random' a partially random sequence is yields the complexity hierarchy.

\subsection{Martin-\texorpdfstring{\Lof}{Lof}\ Randomness}

Martin-\Lof\ randomness is among the most standard ways to capture the notion of a sequence being algorithmically random, and our definitions of partial randomness will be generalizations of those for Martin-\Lof\ randomness. 

The first definition we give comes from the within the measure-theoretic paradigm.

\begin{definition}[Martin-\Lof\ randomness]
A \textdef{Martin-\Lof\ test}, or \textdef{ML test}, is a sequence $\langle U_i \rangle_{i \in \mathbb{N}}$ of uniformly $\Sigma^0_1$ subsets of $\cantor$ such that $\lambda(U_i) \leq 2^{-i}$ for each $i \in \mathbb{N}$. Such an ML test \textdef{covers} $X \in \cantor$ if $X \in \bigcap_{i \in \mathbb{N}}{U_i}$. $X \in \cantor$ is \textdef{Martin-\Lof\ random} (or \textdef{$1$-random}) if no ML test covers $X$. The set of all Martin-\Lof\ random sequences is denoted by $\mlr$.
\end{definition}

The computational paradigm involves measuring the `complexity' of initial segments of an $X \in \cantor$. There are two relevant notions of complexity, that of \emph{prefix-free complexity} and \emph{a priori complexity}.

First, prefix-free complexity:

\begin{definition}[prefix-free machine]
A \textdef{machine} is a partial recursive function $M \colonsub \{0,1\}^\ast \to \{0,1\}^\ast$. 
A machine $M$ is \textdef{prefix-free} if $\dom M$ is prefix-free. 
A prefix-free machine $U$ is said to be \textdef{universal} if whenever $M$ is another prefix-free machine there is a $\rho \in \{0,1\}^\ast$ such that $U(\rho \concat \tau) \simeq M(\tau)$ for all $\tau \in \{0,1\}^\ast$.
\end{definition}

\begin{lem} 
\textnormal{\cite[Proposition 3.5.1.(ii)]{downey2010algorithmic}}
There exists a universal prefix-free machine.
\end{lem}

\begin{definition}[prefix-free complexity]
Fix a universal prefix-free machine $U$. Given $\sigma \in \{0,1\}^\ast$, its \textdef{prefix-free complexity} (with respect to $U$) is defined by
\begin{equation*}
\pfc(\sigma) = \pfc_U(\sigma) \coloneq \min\{ |\tau| \mid U(\tau) \converge = \sigma\}.
\end{equation*}
\end{definition}

The second notion of complexity we use is that of a priori complexity, which has a more pronounced measure-theoretic leaning than that of prefix-free complexity:

\begin{definition}[continuous semimeasure]
A \textdef{continuous semimeasure} is a function $\nu \colon \{0,1\}^\ast \to [0,1]$ such that $\nu(\langle\rangle)=1$ and $\nu(\sigma) \geq \nu(\sigma \concat \langle 0 \rangle) + \nu(\sigma \concat \langle 1 \rangle)$ for all $\sigma \in \{0,1\}^\ast$. 
A continuous semimeasure $\nu$ is \textdef{left recursively enumerable}, or \textdef{left r.e.}, if it is left r.e.\ in the usual sense. 
A left r.e.\ continuous semimeasure $\mathbf{M}$ is \textdef{universal} if whenever $\nu$ is another left r.e.\ continuous semimeasure there is a $c \in \mathbb{N}$ such that $\nu(\sigma) \leq c \cdot \mathbf{M}(\sigma)$ for all $\sigma \in \{0,1\}^\ast$.
\end{definition}

\begin{lem} 
\textnormal{\cite[Theorem 3.16.2]{downey2010algorithmic}}
There exists a universal left r.e.\ continuous semimeasure.
\end{lem}

\begin{definition}[a priori complexity]
Fix a universal left r.e.\ continuous semimeasure $\mathbf{M}$. Given $\sigma \in \{0,1\}^\ast$, its \textdef{a priori complexity} (with respect to $\mathbf{M}$) is defined by 
\begin{equation*}
\apc(\sigma) = \apc_\mathbf{M}(\sigma) \coloneq - \log_2 \mathbf{M}(\sigma).
\end{equation*}
\end{definition}

The complexity of an $X \in \cantor$ can be quantified by the growth rate of the complexities of its initial segments.

\begin{definition}[(strong) $1$-complexity]
$X \in \cantor$ is \textdef{$1$-complex} if there exists $c \in \mathbb{N}$ such that $\pfc(X \restrict n) \geq n - c$ for all $n \in \mathbb{N}$, and is \textdef{strongly $1$-complex} if there exists $c \in \mathbb{N}$ such that $\apc(X \restrict n) \geq n - c$ for all $n \in \mathbb{N}$.
\end{definition}

The predictability paradigm includes supermartingales as one way to capture a notion of betting. 

\begin{definition}[supermartingale and success]
A \textdef{supermartingale} is a function $d \colon \{0,1\}^\ast \to [0,\infty)$ such that $2d(\sigma) \geq d(\sigma \concat \langle 0 \rangle) + d(\sigma \concat \langle 1 \rangle)$ for all $\sigma \in \{0,1\}^\ast$.
A supermartingale $d$ is \textdef{left recursively enumerable}, or \textdef{left r.e.}, if it left r.e.\ in the usual sense.
A left r.e.\ supermartingale $d$ \textdef{succeeds on $X \in \cantor$} if $\limsup_n{d(X \restrict n)} = \infty$.
\end{definition}

Each of these approaches (among several others) ultimately give the same notion of an $X \in \cantor$ being `algorithmically random'.

\begin{prop} 
\textnormal{\cite[Chapter 6]{downey2010algorithmic}}
Suppose $X \in \cantor$. The following are equivalent.
\begin{enumerate}[(i)]
\item $X$ is Martin-\Lof\ random.
\item $X$ is $1$-complex.
\item $X$ is strongly $1$-complex.
\item No left r.e.\ supermartingale succeeds on $X$.
\end{enumerate}
\end{prop}

\subsection{Partial Randomness}

Even if $X \notin \mlr$, $X$ may still exhibit some degree of randomness. \emph{Partial} randomness can be approached from the measure-theoretic, computational, and predictability paradigms just as in the case of Martin-\Lof\ randomness, though the resulting definitions need not be equivalent in general. Notions of `partial $f$-randomness' can be motivated by interpreting Martin-\Lof\ randomness as corresponding to the choice $f(\sigma) \coloneq |\sigma|$ for $\sigma \in \{0,1\}^\ast$.

\begin{notation}
Unless otherwise specified, $f$ denotes a computable function $f \colon \{0,1\}^\ast \to \mathbb{R}$. 
\end{notation}

Concerning the measure-theoretic paradigm, the map $\sigma \mapsto 2^{-f(\sigma)}$ no longer induces a pre-measure on the algebra of basic open subsets of $\cantor$; given a $\Sigma^0_1$ subset $U \subseteq \cantor$, it matters how $U$ is expressed as a union of basic open sets. For this reason, our emphasis is on r.e.\ subsets of $\{0,1\}^\ast$ instead of $\Sigma^0_1$ subsets of $\cantor$. Moreover, we consider two distinct ways to capture the idea of the `$f$-weight' of a subset of $\{0,1\}^\ast$.

\begin{definition}[direct and prefix-free $f$-weight]
The \textdef{direct $f$-weight} of a set of strings $S \subseteq \{0,1\}^\ast$ is defined by
\begin{equation*}
\dwt_f(S) \coloneq \sum_{\sigma \in S}{2^{-f(\sigma)}}.
\end{equation*}
Its \textdef{prefix-free $f$-weight} is defined by
\begin{equation*}
\pwt_f(S) \coloneq \sup\{ \dwt_f(A) \mid \text{prefix-free $A \subseteq S$}\}.
\end{equation*}
\end{definition}

\begin{definition}[(strong) $f$-randomness]
Suppose $\langle S_i \rangle_{i \in \mathbb{N}}$ is a sequence of uniformly r.e.\ subsets of $\{0,1\}^\ast$. $\langle S_i \rangle_{i \in \mathbb{N}}$ is a \textdef{$f$-ML test} if $\dwt_f(S_i) \leq 2^{-i}$ for each $i \in \mathbb{N}$ and a \textdef{weak $f$-ML test} if $\pwt_f(S_i) \leq 2^{-i}$ for each $i \in \mathbb{N}$. $\langle S_i \rangle_{i \in \mathbb{N}}$ \textdef{covers} $X \in \cantor$ if $X \in \bigcap_{i \in \mathbb{N}}{\bbracket{S_i}}$.
$X \in \cantor$ is \textdef{$f$-random} if no $f$-ML test covers $X$ and \textdef{strongly $f$-random} if no weak $f$-ML test covers $X$.
\end{definition}

Replacing the map $\sigma \mapsto |\sigma|$ with $f$ quickly generalizes the notion of $1$-complexity and strong $1$-complexity:

\begin{definition}[(strong) $f$-complexity]
$X \in \cantor$ is \textdef{$f$-complex} if there exists $c \in \mathbb{N}$ such that $\pfc(X \restrict n) \geq f(X \restrict n) - c$ for all $n \in \mathbb{N}$, and is \textdef{strongly $f$-complex} if there exists $c \in \mathbb{N}$ such that $\apc(X \restrict n) \geq f(X \restrict n) - c$ for all $n \in \mathbb{N}$.
\end{definition}

For supermartingales and success, we generalize the notion of success.

\begin{definition}[$f$-success]
A left r.e.\ supermartingale $d$ \textdef{$f$-succeeds on $X \in \cantor$} if 
\begin{equation*}
\limsup_n{\left(d(X \restrict n) \cdot 2^{n-f(X \restrict n)}\right)} = \infty.
\end{equation*}
\end{definition}

Unlike when $f(\sigma) \coloneq |\sigma|$, these notions are no longer all necessarily equivalent, instead forming two groups. In summary:

\begin{prop} 
\textnormal{\cite[Theorem 2.6, Theorem 2.8]{higuchi2014propagation} \cite[Theorem 4.1.6, Theorem 4.1.8, Theorem 4.2.3]{hudelson2013partial}}
Suppose $X \in \cantor$ and $f \colon \{0,1\}^\ast \to \mathbb{R}$ is computable. 
\begin{enumerate}[(a)]
\item $X$ is $f$-random if and only if it is $f$-complex.
\item $X$ is strongly $f$-random if and only if it is strongly $f$-complex, and if and only if no left r.e.\ supermartingale $f$-succeeds on $X$.
\end{enumerate}
\end{prop}

\begin{remark}
In \cite{hudelson2013partial}, Hudelson generalizes the supermartingale approach to partial randomness by modifying the definition of a supermartingale, defining a \emph{left r.e.\ $f$-supermartingale} to be a left r.e.\ function $d \colon \{0,1\}^\ast \to [0,\infty)$ such that 
\begin{equation*}
2^{-f(\sigma)}\cdot d(\sigma) \geq 2^{-f(\sigma\concat \langle 0 \rangle)}\cdot d(\sigma \concat \langle 0 \rangle) + 2^{-f(\sigma\concat \langle 1 \rangle)}\cdot d(\sigma \concat \langle 1 \rangle)
\end{equation*}
for all $\sigma \in \{0,1\}^\ast$, with $d$ \emph{succeeding on $X \in \cantor$} if $\limsup_n{d(X \restrict n)} = \infty$. For every $X \in \cantor$, there exists an $f$-supermartingale $d$ succeeding on $X$ if and only if there exists a supermartingale $\tilde{d}$ $f$-succeeding on $X$.  

Our choice to use ordinary supermartingales and $f$-success follows the approaches used for $f$ of the form $f(\sigma) \coloneq \delta \cdot |\sigma|$ for $\delta \in (0,1]$ in \cite{calude2006partial} and \cite{greenberg2011diagonally}.
\end{remark}

Although $f$-randomness does not in general imply strong $f$-randomness (see, e.g., \cite[Theorem 4.3.2]{hudelson2013partial}), if $g$ grows sufficiently faster than $f$, then $g$-randomness will imply strong $f$-randomness:

\begin{prop} 
\textnormal{\cite[Theorem 3.5]{higuchi2014propagation}}
Suppose $f,g \colon \{0,1\}^\ast \to \mathbb{R}$ are computable functions and $X \in \cantor$ is $g$-random. If there exists a nondecreasing $h \colon \mathbb{R} \to \mathbb{R}$ such that $\sum_{n=1}^\infty{2^{-h(n)}} < \infty$ and for which $g(\sigma) \geq f(\sigma) + h(f(\sigma))$ for all $\sigma \in \{0,1\}^\ast$, then $X$ is strongly $f$-random.
\end{prop}

\begin{cor} 
\textnormal{\cite[Theorem 3.6]{higuchi2014propagation}}
Suppose $k>0$ and $\epsilon>0$, and let $g = f+\log_2 f + \log_2 \log_2 f + \cdots + \log_2^{k-1} f + (1+\epsilon) \log_2^k f$. Then any $g$-random $X \in \cantor$ is strongly $f$-random.
\end{cor}

\subsection{Randomness and Complexity as Mass Problems}

For each computable $f \colon \{0,1\}^\ast \to \mathbb{R}$ there is an associated mass problem consisting of all $X \in \cantor$ which are $f$-complex (equivalently, $f$-random). We use the following notation:

\begin{definition}
Suppose $f \colon \{0,1\}^\ast \to \mathbb{R}$ is computable and $c \in \mathbb{N}$. Then
\begin{align*}
\complex(f,c) & \coloneq \{ X \in \cantor \mid \forall n \qspace (\pfc(X \restrict n) \geq f(X \restrict n) - c)\}, \\
\complex(f) & \coloneq \{ X \in \cantor \mid \text{$X$ is $f$-complex}\} = \bigcup_{c=0}^\infty{\complex(f,c)}.
\end{align*}
\end{definition}

\begin{notation}
Given $\delta \in (0,1]$, define $f \colon \{0,1\}^\ast \to \mathbb{R}$ by $f(\sigma) \coloneq \delta |\sigma|$. We write $\complex(\delta,c)$ for $\complex(f,c)$ and $\complex(\delta)$ for $\complex(f)$.
\end{notation}

\begin{prop}
The sets $\complex(f,c)$ are $\Pi^0_1$, uniform in $c$. Thus, $\complex_f$ is $\Sigma^0_2$ and consequently $\weakdeg(\complex(f)) \in \mathcal{E}_\weak$ whenever $\complex(f) \neq \emptyset$.
\end{prop}
\begin{proof}
Suppose $X \in \cantor$ and $c \in \mathbb{N}$. Fix a universal prefix-free machine $U$ and let $e$ be such that $U(\tau) \simeq \sigma$ if and only if $\varphi_e(\str^{-1}(\tau)) \simeq \str^{-1}(\sigma)$. Then
\begin{equation*}
X \in \complex(f,c) \iff \forall n \forall s \forall \tau \qspace (\varphi_{e,s}(\str^{-1}(\tau)) \converge = \str^{-1}(X \restrict n) \to |\tau| \geq f(X \restrict n) - c).
\end{equation*}
The uniformity of the above predicate shows that $\complex(f)$ is $\Sigma^0_2$. 

If $\complex(f) \neq \emptyset$, then $\complex(f,c) \neq \emptyset$ for some $c \in \mathbb{N}$. In particular, $\complex(f)$ contains a nonempty $\Pi^0_1$ subset of $\cantor$. Thus, the \nameref{embedding lemma} implies $\weakdeg(\complex(f)) \in \mathcal{E}_\weak$.
\end{proof}

Often $f(\sigma)$ depends only on $|\sigma|$; such $f$ are said to be \textdef{length-invariant} and correspond exactly with computable functions of the form $\mathbb{N} \to \mathbb{R}$. We extend our notation for $\complex(f)$ to such functions.

\begin{notation}
If $f \colon \mathbb{N} \to \mathbb{R}$ is a computable function and $c \in \mathbb{N}$, then $\complex(f,c)$ stands for $\complex(\tilde{f},c)$, where $\tilde{f} \colon \{0,1\}^\ast \to \mathbb{R}$ is defined by $\tilde{f}(\sigma) \coloneq f(|\sigma|)$ for $\sigma \in \{0,1\}^\ast$.
\end{notation}

Because we are often interested in \emph{partial} randomness and in light of \cref{ldnr and mlr}, we often only consider $f$ satisfying the following condition:

\begin{definition}[sub-identical]
A function $f \colon \{0,1\}^\ast \to [0,\infty)$ is \textdef{sub-identical} if $\lim_{n \to \infty}{(n - f(X \restrict n))} = \infty$ for every $X \in \cantor$. 

Likewise, a function $f \colon \mathbb{N} \to [0,\infty)$ is \textdef{sub-identical} if $\lim_{n \to \infty}{(n-f(n))} = \infty$.
\end{definition}

Thanks to the presence of $c$ in the definition $\complex(f) = \bigcup_{c \in \mathbb{N}}{\complex(f,c)}$, if $f(n) = g(n)$ for almost all $n$, then $\complex_f = \complex_g$. 

\begin{convention}
Given $f \colonsub \mathbb{N} \to \mathbb{R}$, suppose $f \restrict \mathbb{N}_{\geq a}$ is a total, computable, nondecreasing, unbounded function for some $a \in \mathbb{N}$. Define $\tilde{f} \colon \mathbb{N} \to \mathbb{R}$ by $\tilde{f}(n) \coloneq f(n)$ for $n \geq a$ and $\tilde{f}(n) \coloneq f(a)$ otherwise. Then $\tilde{f}$ is an order function, and we let $\complex(f)$ denote the class $\complex(\tilde{f})$. This allows us to make sense of something like, e.g., $\complex(\log_2)$.
\end{convention}

\subsection{Properties of Prefix-Free Complexity}

Here we collect some of the results concerning prefix-free complexity we make use of later. Several of these results involve \emph{conditional} prefix-free complexity.

\begin{definition}[oracle prefix-free machine]
An \textdef{oracle prefix-free machine} is a partial recursive function $M \colonsub \{0,1\}^\ast \times \{0,1\}^\ast \to \{0,1\}^\ast$ such that
\begin{enumerate*}[(i)] 
\item if $M^\tau(\sigma) \converge$ and $\tau \subseteq \tau'$, then $M^{\tau'}(\sigma) \converge = M^\tau(\sigma)$ and 
\item $\{ \sigma \mid \langle \sigma,\tau \rangle \in \dom M\}$ is prefix-free for every $\tau \in \{0,1\}^\ast$.
\end{enumerate*}
An oracle prefix-free machine $U$ is \textdef{universal} if for every oracle prefix-free machine $M$ there exists $\rho \in \{0,1\}^\ast$ such that $U^\tau(\rho\concat\sigma) \simeq M^\tau(\sigma)$ for all $\sigma \in \{0,1\}^\ast$.
\end{definition}

\begin{lem} 
\textnormal{\cite[Section 3.2]{downey2010algorithmic}}
There exists a universal oracle prefix-free machine $U$.
\end{lem}

\begin{definition}[conditional prefix-free complexity] \cite[Section 3.2]{downey2010algorithmic}
Fix a universal oracle prefix-free machine $U$. Given $\sigma,\tau \in \{0,1\}^\ast$, the \textdef{conditional prefix-free complexity of $\sigma$ given $\tau$} (with respect to $U$) is defined by
\begin{equation*}
\pfc(\sigma \mid \tau) = \pfc_U(\sigma \mid \tau) \coloneq \min\{ |\rho| \mid U^{\overline{\tau}}(\rho) \converge = \sigma\},
\end{equation*}
where $\overline{\tau} = \langle \tau(0),\tau(0),\tau(1),\tau(1),\ldots,\tau(|\tau|-1),\tau(|\tau|-1),0,1\rangle$. 
\end{definition}

\begin{notation}

Temporarily write $\vec{\sigma} \coloneq \str\bigl( \pi^{(k)}(\str^{-1}(\sigma_1),\str^{-1}(\sigma_2),\ldots,\str^{-1}(\sigma_k)) \bigr)$ for $\sigma_1,\sigma_2,\ldots,\sigma_k \in \{0,1\}^\ast$. Then given $\sigma_1,\sigma_2,\ldots,\sigma_k, \tau_1,\tau_2,\ldots,\tau_m \in \{0,1\}^\ast$, we define
\begin{align*}
\pfc(\sigma_1,\sigma_2,\ldots,\sigma_k) & \coloneq \pfc(\vec{\sigma}), \\
\pfc(\sigma_1,\sigma_2,\ldots,\sigma_k \mid \tau_1,\tau_2,\ldots,\tau_m) & \coloneq \pfc(\vec{\sigma} \mid \vec{\tau}).
\end{align*}
\end{notation}

\begin{notation}
Given functions $f,g \colon S \to \mathbb{R}$, we write $f \leq^+ g$ to mean that there exists $c \in \mathbb{N}$ such that $f(x) \leq g(x) + c$ for all $x \in S$. We write $f =^+ g$ if $f \leq^+ g$ and $g \leq^+ f$. By an abuse of notation, we may write $f(x) \leq^+ g(x)$ or $f(x) =^+ g(x)$ where $x$ is an indeterminate to indicate that $f \leq^+ g$ or $f =^+ g$, respectively.
\end{notation}

\begin{prop} \label{prefix-free complexity facts}
\mbox{}
\begin{enumerate}[(a)]
\item \textnormal{\cite[Theorem 3.6.1 \& Corollary 3.6.2]{downey2010algorithmic}}
Kraft-Chaitin: Suppose $(d_i,\tau_i)_{i\in\mathbb{N}}$ is a recursive sequence of pairs $(d_i,\tau_i) \in \mathbb{N} \times \{0,1\}^\ast$ such that $\sum_{i=0}^\infty{2^{-d_i}} \leq 1$. Then there exists a prefix-free machine $M$ and strings $\sigma_i$ such that $|\sigma_i| = d_i$ and $M(\sigma_i) = \tau_i$ for all $i \in \mathbb{N}$ and $\dom M = \{ \sigma_i \mid i \in \mathbb{N}\}$. Consequently, $\pfc(\tau_i) \leq^+ d_i$.

\item \textnormal{\cite[Proposition 3.5.4]{downey2010algorithmic}} 
If $f\colon \{0,1\}^\ast \to \{0,1\}^\ast$ is computable, then $\pfc(f(\sigma)) \leq^+ \pfc(\sigma)$ for all $\sigma \in \{0,1\}^\ast$.

\item \textnormal{\cite[Corollary 3.7.5]{downey2010algorithmic}} 
For any $k$ and $\epsilon>0$, $\pfc(\sigma) \leq^+ |\sigma| + \log|\sigma| + \log\log|\sigma| + \cdots + (1+\epsilon)\log^k |\sigma|$.

\item \textnormal{\cite[Proposition 3.7.13]{downey2010algorithmic}} 
For all $\sigma,\tau \in \{0,1\}^\ast$, $\pfc(\sigma\concat \tau) \leq \pfc(\sigma,\tau) \leq^+ \pfc(\sigma) + \pfc(\tau)$.

\item \textnormal{\cite[Theorem 3.10.2]{downey2010algorithmic}} 
$\pfc(\sigma,\tau) =^+ \pfc(\sigma) + \pfc(\tau \mid \sigma, \pfc(\sigma)) =^+ \pfc(\sigma) + \pfc(\tau \mid \sigma^\ast)$, where $\sigma^\ast$ is the lexicographically least $\tau$ such that $U(\tau)\converge = \sigma$ in $s$ stages for the least possible $s$.

\end{enumerate}
\end{prop}

\begin{cor}
Suppose $f\colon (\{0,1\}^\ast)^k \to \{0,1\}^\ast$ is computable. Then for all $\sigma_1,\sigma_2,\ldots,\sigma_k$,
\begin{equation*}
\pfc(f(\sigma_1,\sigma_2,\ldots,\sigma_k)) \leq^+ \pfc(\sigma_1) + \pfc(\sigma_2) + \cdots + \pfc(\sigma_k).
\end{equation*}
\end{cor}

\section{\texorpdfstring{$\dnr$}{DNR} and Avoidance}
\label{dnr and avoidance section}

The diagonally non-recursive (or $\dnr$) hierarchy consists of the sets $\dnr(p)$ for $p \colon \mathbb{N} \to (1,\infty)$ computable, where `diagonally non-recursive' is with respect to a fixed admissible enumeration $\varphi_\bullet$.

\begin{definition}[diagonally non-recursive]
The set of \textdef{diagonally non-recursive} sequences, $\dnr \subseteq \baire$, is defined by $\dnr \coloneq \{ f \in \baire \mid \forall n \qspace (f(n) \nsimeq \varphi_n(n))\}$. Additionally, given a recursive $p \colon \mathbb{N} \to (1,\infty)$, we define:
\begin{equation*}
\dnr(p) \coloneq \{ X \in \baire \mid X \in \dnr \wedge \forall n \qspace (X(n) < p(n))\}.
\end{equation*}
\end{definition}

The weak degree of $\dnr(p)$ has a dependence on the growth rate of $p$. At one extreme, $\mathbf{1} \coloneq \weakdeg(\dnr(2))$ is the maximum weak degree in $\mathcal{E}_\weak$ \cite[Theorem 14.6]{simpson2009degrees}. As the growth rate of $p$ rises the weak degree of $\dnr(p)$ falls; \cite[Theorem 2.3]{kjoshanssen2006kolmogorov} shows that the class $\dnr_\rec \coloneq \bigcup\{ \dnr(p) \mid \text{$p$ recursive}\}$ is weakly equivalent to the class $\complex \coloneq \bigcup\{\complex(f) \mid \text{$f$ recursive}\}$. Several results (e.g., \cite{greenberg2011diagonally}, \cite{khan2013shift}) show that for certain properties of interest, every element of $\dnr(p)$ computes an infinite sequence with that desired property if $p$ is sufficiently slow-growing.

If we wish to replace `sufficiently slow-growing' with explicit bounds, we quickly run into problems -- the weak degree of $\dnr(p)$ depends not only on $p$, but also on the particular admissible enumeration $\varphi_\bullet$ used in the definition of $\dnr(p)$, a fact that has been observed in \cite[Remark 10.6]{simpson2005mass}, \cite[\S 7.3]{bienvenu2016deep}, and \cite{miller2020assorted}. In particular, we can show:

\begin{prop} \label{weak degree of dnr_p is ill-defined}
There exist admissible enumerations $\varphi_\bullet$ and $\tilde{\varphi}_\bullet$ such that the weak degree of $\dnr(\lambda n.n)$ differs when defined with respect to $\varphi_\bullet$ and $\tilde{\varphi}_\bullet$.
\end{prop}
\begin{proof}
\cite{bienvenu2016deep} shows that there exist admissible enumerations $\varphi_\bullet$ for which $\dnr(p) \nweakleq \mlr$ (where $\dnr(p)$ is defined with respect to that particular choice of enumeration $\varphi_\bullet$) if and only if $\sum_{n=0}^\infty{p(n)^{-1}} = \infty$.\footnote{This fact has also been independently observed by Greenberg, Miller \cite{miller2020assorted}, and Slaman.} Let $\varphi_\bullet$ be such an admissible enumeration.

We define a new admissible enumeration $\tilde{\varphi}_\bullet$. Given $x \in \mathbb{N}$, define
\begin{equation*}
\tilde{\varphi}_e(x) \simeq \begin{cases} \varphi_n(x) & \text{if $e = 2^n$,} \\ \diverge & \text{otherwise.} \end{cases}
\end{equation*}
It is straight-forward to check that $\tilde{\varphi}_\bullet$ is an admissible enumeration. 

Write $\dnr^{(1)}$ to mean $\dnr$ with respect to the enumeration $\varphi_\bullet$ and $\dnr^{(2)}$ to mean $\dnr$ with respect to the enumeration $\tilde{\varphi}_\bullet$. Our assumption about $\varphi_\bullet$ implies $\dnr^{(1)}(\lambda n. n) \nweakleq \mlr$ while $\dnr^{(1)}(\lambda n. 2^n) \weakleq \mlr$. However, $\dnr^{(2)}(\lambda n. n) \strongeq \dnr^{(1)}(\lambda n. 2^n)$ and hence $\dnr^{(2)}(\lambda n. n) \weakleq \mlr$ and so $\dnr^{(1)}(\lambda n. n)$ and $\dnr^{(2)}(\lambda n. n)$ are not weakly equivalent. 
\end{proof}

With this dependence in mind, if we wish to discuss the weak degrees of the classes $\dnr_p$, then the notation `$\dnr$' should be replaced with one explicitly acknowledging the choice of admissible enumeration implicit in `$\dnr$'. By definition $X \in \dnr$ if and only if $X$ avoids the diagonal $\psi(e) \simeq \varphi_e(e)$, i.e., $X \cap \psi = \emptyset$. In particular, a more precise observation is that $\dnr$ depends only on the choice of $\psi$, where $\psi$ is chosen among the diagonals of admissible enumerations. The constraint that $\psi$ be the diagonal of an admissible enumeration can be lifted to give a more general notion.

\begin{definition}[avoidance]
Suppose $\psi \colonsub \mathbb{N} \to \mathbb{N}$ is a partial recursive function. The class $\avoid^\psi$ is defined by $\avoid^\psi \coloneq \{ X \in \baire \mid X \cap \psi = \emptyset\}$. Additionally, given any recursive $p \colon \mathbb{N} \to (1,\infty)$, we define:
\begin{equation*}
\avoid^\psi(p) \coloneq \{ X \in \baire \mid X \cap \psi = \emptyset \wedge \forall n\qspace (X(n) < p(n))\}.
\end{equation*}
\end{definition}

In this more general framework, the analogs of the diagonals of admissible enumerations are the universal partial recursive funcitons. 

\begin{definition}[universal partial recursive function]
A partial recursive function $\psi \colonsub \mathbb{N} \to \mathbb{N}$ is \textdef{universal} if for every partial recursive function $\theta \colonsub \mathbb{N} \to \mathbb{N}$ there is a total recursive function $f$ such that $\psi \circ f = \theta$.
\end{definition}

This is based on the fact that the diagonal of any admissible enumeration $\varphi_\bullet$ is universal.

\begin{prop} \label{diagonal of admissible numbering is universal}
The diagonal of any admissible enumeration is universal. 
\end{prop}
\begin{proof}
Let $\varphi_\bullet$ be an admissible enumeration and $\psi$ its diagonal, and let $\theta \colonsub \mathbb{N} \to \mathbb{N}$ be any partial recursive function. Let $\chi \colonsub \mathbb{N}^3 \to \mathbb{N}$ be defined by $\chi(e,x,y) \simeq \varphi_e(x)$ for $e,x,y \in \mathbb{N}$. By the Parametrization Theorem, there exists a total recursive function $g \colon \mathbb{N}^2 \to \mathbb{N}$ such that $\varphi_{g(e,x)}(y) \simeq \chi(e,x,y) \simeq \varphi_e(x)$ for all $e,x,y \in \mathbb{N}$. Let $e \in \mathbb{N}$ satisfy $\varphi_e(x) \simeq \theta(x)$ for all $x \in \mathbb{N}$, and let $f\colon \mathbb{N} \to \mathbb{N}$ be defined by $f(x) \coloneq g(e,x)$. Then
\begin{equation*}
(\psi \circ f)(x) \simeq \varphi_{g(e,x)}(g(e,x)) \simeq \varphi_e(x) \simeq \theta(x).
\end{equation*}
As $\theta$ was arbitrary, it follows that $\psi$ is universal.
\end{proof}

\subsection{Linearly Universal Avoidance}

As with $\dnr(p)$, $\avoid^\psi(p)$ still depends on the choice of universal partial recursive function $\psi$. To remove the dependence on a \emph{single} choice of universal partial recursive function we instead choose a \emph{collection} of universal partial recursive functions. 

\begin{definition}
Suppose $\mathcal{C}$ is a collection of partial recursive functions and $p \colon \mathbb{N} \to (1,\infty)$ is recursive. We define
\begin{equation*}
\avoid^\mathcal{C}(p) \coloneq \bigcup_{\psi \in \mathcal{C}}{\avoid^\psi(p)}.
\end{equation*}
\end{definition}

The choice of $\mathcal{C}$ we make follows the convention introduced by \cite{simpson2017turing}, motivated by \cite{bienvenu2016deep} and \cite{miller2020assorted}.

\begin{definition}[linearly universal partial recursive function]
A \textdef{linearly universal partial recursive function} is a partial recursive function $\psi \colonsub \mathbb{N} \to \mathbb{N}$ such that for any partial recursive $\theta \colonsub \mathbb{N} \to \mathbb{N}$ there exist $a,b \in \mathbb{N}$ such that $\psi(ax+b) \simeq \theta(x)$ for all $x \in \mathbb{N}$.
\end{definition}

\begin{definition}
Given a recursive $p \colon \mathbb{N} \to (1,\infty)$ we define $\ldnr(p) \coloneq \avoid^\mathcal{LU}(p)$, where $\mathcal{LU}$ is the family of linearly universal partial recursive functions. 
\end{definition}


\begin{remark}
The approach used in \cite{bienvenu2016deep} and \cite{miller2020assorted} is to use $\avoid^\psi(p)$, where $\psi$ is the diagonal of a `linear enumeration', i.e., an admissible enumeration $\varphi_\bullet$ for which there is a total recursive $\ell \colon \mathbb{N}^2 \to \mathbb{N}$ satisfying the following conditions: 
\begin{enumerate*}[(i)]
\item For all $e,x \in \mathbb{N}$, $\psi(\ell(e,x)) \simeq \varphi_e(x)$ and
\item for each $e \in \mathbb{N}$ there exist $a,b \in \mathbb{N}$ such that $\ell(e,x) \leq ax+b$ for all $x \in \mathbb{N}$.
\end{enumerate*}
\end{remark}

\subsection{Properties of Linearly Universal Partial Recursive Functions}
\label{properties of linearly univeral partial recursive functions section}

Every linearly universal partial recursive function $\psi$ produces an effective enumeration $\varphi_\bullet$ of the partial recursive functions by setting
\begin{equation*}
\varphi_e(x) \simeq \psi((e)_0 x + (e)_1)
\end{equation*}
for $e,x \in \mathbb{N}$, where $(-)_i \coloneq \pi_i \circ (\pi^{(2)})^{-1}$ for $i \in \{0,1\}$ (so $\pi^{(2)}((n)_0,(n)_1) = n$ for all $n \in \mathbb{N}$). This enumeration $\varphi_\bullet$ is admissible.

\begin{prop}[Parametrization Theorem for Linearly Universal Partial Recursive Functions] \label{parametrization theorem for linearly universal partial recursive functions} 
Suppose $\psi$ is a linearly universal partial recursive function and $\theta \colonsub \mathbb{N}^{k+1} \to \mathbb{N}$ is a partial recursive function. There exist elementary recursive functions $A,B \colon \mathbb{N}^k \to \mathbb{N}$ such that 
\begin{equation*}
\psi(A(\mathbf{x})y + B(\mathbf{x})) \simeq \theta(\mathbf{x},y)
\end{equation*}
for all $\mathbf{x} \in \mathbb{N}^k$ and $y \in \mathbb{N}$.
\end{prop}
\begin{proof}
Because $\psi$ is linearly universal, there exists $a,b \in \mathbb{N}$ such that $\psi(a\pi^{(k+1)}(\mathbf{x},y) + b) \simeq \theta(\mathbf{x},y)$ for all $\mathbf{x} \in \mathbb{N}^k$ and $y \in \mathbb{N}$. Thus, defining $A,B\colon \mathbb{N}^k \to \mathbb{N}$ by
\begin{equation*}
A(\mathbf{x}) \coloneq a2^{\pi^{(k)}(\mathbf{x}) + 1} \quad \text{and} \quad B(\mathbf{x}) \coloneq a(2^{\pi^{(k)}(\mathbf{x})} - 1) + b
\end{equation*}
gives
\begin{align*}
\psi(A(\mathbf{x})y + B(\mathbf{x})) & \simeq \psi(a2^{\pi^{(k)}(\mathbf{x})+1}y + a(2^{\pi^{(k)}(\mathbf{x})}-1)+b) \\
& \simeq \psi(a(2^{\pi^{(k)}(\mathbf{x})}(2y+1)-1)+b) \\
& \simeq \psi(a \pi^{(k+1)}(\mathbf{x}, y) + b) \\
& \simeq \theta(\mathbf{x},y)
\end{align*}
for all $\mathbf{x} \in \mathbb{N}^k$ and $y \in \mathbb{N}$.
\end{proof}

\begin{cor} \label{linearly universal partial recursive function enumeration is admissible}
Suppose $\psi$ is a linearly universal partial recursive function. Then the effective enumeration $\varphi_\bullet$ defined by $\varphi_e(n) \coloneq \psi((e)_0 n + (e)_1)$ is admissible.
\end{cor}

\begin{cor}[Recursion Theorem for Linearly Universal Partial Recursive Functions] \label{recursion theorem for linearly universal partial recursive functions}
Suppose $\psi$ is a linearly universal partial recursive function and $\theta \colonsub \mathbb{N}^{k+2} \to \mathbb{N}$ is a partial recursive function. Then there exist $a,b \in \mathbb{N}$ such that
\begin{equation*}
\psi(a\pi^{(k)}(\mathbf{x})+b) \simeq \theta(a,b,\mathbf{x})
\end{equation*}
for all $\mathbf{x} \in \mathbb{N}^k$.
\end{cor}
\begin{proof}
Let $\tilde{\theta} \colonsub \mathbb{N}^{k+1} \to \mathbb{N}$ be defined by 
\begin{equation*}
\tilde{\theta}(c,\mathbf{x}) \simeq \theta((c)_0,(c)_1,\mathbf{x})
\end{equation*}
for all $c \in \mathbb{N}$ and $\mathbf{x} \in \mathbb{N}^k$. \cref{linearly universal partial recursive function enumeration is admissible,,implications of parametrization theorem} show that that there exists an $e \in \mathbb{N}$ such that 
\begin{equation*}
\psi((e)_0\pi^{(k)}(\mathbf{x})+(e)_1) \simeq \varphi_e^{(k)}(\mathbf{x}) \simeq \tilde{\theta}(e,\mathbf{x}) \simeq \theta((e)_0,(e)_1,\mathbf{x})
\end{equation*}
for all $\mathbf{x} \in \mathbb{N}$. Thus, we may let $a = (e)_0$ and $b = (e)_1$.
\end{proof}

Moreover, the diagonal of $\varphi_\bullet$ is linearly universal. 

\begin{prop} \label{diagonal of linearly universal is linearly universal}
If $\psi_0$ is a linearly universal partial recursive function, then the partial function $\psi$ defined by $\psi(e) \simeq \psi_0((e)_0 e+(e)_1)$ for $e\in\mathbb{N}$ is also linearly universal partial recursive.
\end{prop}
\begin{proof}
Suppose $\theta \colonsub \mathbb{N} \to \mathbb{N}$ is a partial recursive function. There are $a$ and $b$ such that $\forall x \qspace (\psi_0(ax+b) \simeq \theta(x))$. For any $x \in \mathbb{N}$, $\pi^{(2)}(0,x) = 2^0(2x+1)-1 = 2x$, so 
\begin{equation*}
\psi(2x) \simeq \psi(\pi^{(2)}(0,x)) \simeq \psi_0(0\cdot x + x) = \psi_0(x).
\end{equation*}
It follows that $\psi$ is linearly universal.
\end{proof}

Together, \cref{linearly universal partial recursive function enumeration is admissible} and \cref{diagonal of linearly universal is linearly universal} allow us to enjoy the benefits of linearly universal partial recursive functions and of admissible enumerations simultaneously, e.g., having access to the particularly nice versions of the Parametrization and Recursion Theorems in the forms of \cref{parametrization theorem for linearly universal partial recursive functions,,recursion theorem for linearly universal partial recursive functions}, respectively.

Another convenient property of linearly universal partial recursive functions is that we may we edit any finite number of values without affecting the linear universality.

\begin{prop} \label{finitely many changes preserves linear universality}
Suppose $\psi,\chi$ are partial recursive functions such that $\psi(n) \simeq \chi(n)$ for almost all $n \in \mathbb{N}$. Then $\psi$ is linearly universal if and only if $\chi$ is linearly universal.
\end{prop}
\begin{proof}
Suppose $\psi$ is linearly universal. Let $N \in \mathbb{N}$ be such that $\psi(x) \simeq \chi(x)$ for all $x \geq N$. Consider the partial recursive function $\theta \colonsub \mathbb{N} \to \mathbb{N}$ defined by
\begin{equation*}
\theta(x) \simeq \begin{cases} \psi(x-N) & \text{if $x \geq N$,} \\ \diverge & \text{otherwise.} \end{cases}
\end{equation*}
Because $\psi$ is linearly universal, there are $a,b \in \mathbb{N}$ such that $\psi(ax+b) \simeq \theta(x)$ for all $x \in \mathbb{N}$. $\theta$ is nonconstant, so $a \neq 0$ and hence $ax+b \geq x$ for all $x \in \mathbb{N}$. Thus, for all $x \in \mathbb{N}$,
\begin{equation*}
\chi(ax+(aN+b)) \simeq \psi(ax+(aN+b)) \simeq \psi(a(x+N)+b) \simeq \theta(x+N) \simeq \psi(x),
\end{equation*}
showing $\chi$ is linearly universal.
\end{proof}

\begin{convention}
Given $p \colonsub \mathbb{N} \to \mathbb{R}$, suppose $p \restrict \mathbb{N}_{\geq a}$ is a total, computable, nondecreasing, unbounded function with image in $(1,\infty)$ for some $a \in \mathbb{N}$. Define $\tilde{p} \colon \mathbb{N} \to (1,\infty)$ by $\tilde{p}(x) \coloneq p(x)$ for $x \geq a$ and $\tilde{p}(x) \coloneq p(a)$ otherwise. Then $\tilde{p}$ is an order function, and we let $\ldnr(p)$ denote the class $\ldnr(\tilde{p})$. This allows us to make sense of something like, e.g., $\ldnr(\log_2)$.
\end{convention}

\subsubsection{Basic Properties of the \texorpdfstring{$\ldnr$}{LUA} Hierarchy}

The regularity of the manner in which linearly universal partial recursive functions express their universality allows us to prove some simple but useful strong and weak reductions.

\begin{prop} \label{ldnr basic facts}
Let $p,q \colon \mathbb{N} \to (1,\infty)$ be recursive functions.
\begin{enumerate}[(a)]
\item If $q(x) = p(ax+b)$ for some $a \in \mathbb{N}_{>0}$ and $b \in \mathbb{N}$ for all $x \in \mathbb{N}$, then $\ldnr(p) \weakeq \ldnr(q)$.
\item If $X \in \ldnr(p)$, $Y \in p^\mathbb{N}$, and $X(x) = Y(x)$ for almost all $x \in \mathbb{N}$, then $Y \in \ldnr(p)$. 
\item If $p \domleq q$, then $\ldnr(q) \strongleq \ldnr(p)$.
\item If for all $a,b \in \mathbb{N}$ we have $q(ax+b) \leq p(x)$ for almost all $x \in \mathbb{N}$ and $\psi_0$ is a partial recursive function, then $\avoid^{\psi_0}(p) \weakleq \ldnr(q)$.
\end{enumerate}
\end{prop}
\begin{proof} \mbox{}
\begin{enumerate}[(a)]
\item Let $\psi$ be a linearly universal partial recursive function. If $X \in \avoid^\psi(p)$, then $X \in \avoid^\psi(q) \subseteq \ldnr(q)$ since $p(x) \leq q(x)$ for all $x \in \mathbb{N}$. Thus, the recursive functional $X \mapsto X$ shows $\ldnr(q) \strongleq \ldnr(p)$. 

Conversely, if $X \in \avoid^\psi(q)$, define $\tilde{\psi}$ by
\begin{equation*}
\tilde{\psi}(y) \simeq \begin{cases} \psi(x) & \text{if $y=ax+b$,} \\ \diverge & \text{otherwise.} \end{cases}
\end{equation*}
Because $\tilde{\psi}(ax+b) \simeq \psi(x)$, it follows that $\tilde{\psi}$ is linearly universal partial recursive. Similarly define $\tilde{X} \in \baire$ by
\begin{equation*}
\tilde{X}(y) \simeq \begin{cases} X(x) & \text{if $y=ax+b$,} \\ 0 & \text{otherwise.} \end{cases}
\end{equation*}
If $X \cap \psi = \emptyset$ and $X$ is $q$-bounded, then $\tilde{X} \cap \tilde{\psi} = \emptyset$ and $\tilde{X}$ is $p$-bounded. Thus, the recursive functional $X \mapsto \tilde{X}$ shows $\ldnr(p) \strongleq \ldnr(q)$. 

\item Let $\psi$ be a linearly universal partial recursive function and suppose $X \in \avoid_p^\psi$. Let $N \in \mathbb{N}$ be such that $X(x) = Y(x)$ for all $x \geq N$ and define $\tilde{\psi}$ by
\begin{equation*}
\tilde{\psi}(x) \simeq \begin{cases} \psi(x) & \text{if $x \geq N$,} \\ \diverge & \text{otherwise,} \end{cases}
\end{equation*}
so $Y \in \avoid^{\tilde{\psi}}(p)$. By \cref{finitely many changes preserves linear universality}, $\tilde{\psi}$ is linearly universal, so $Y \in \ldnr(p)$. 

\item Suppose $p(x) \leq q(x)$ for all $x \geq N$. Given $X \in \baire$, let $\tilde{X}$ be defined by 
\begin{equation*}
\tilde{X}(x) \coloneq \begin{cases} X(x) & \text{if $x \geq N$} \\ c & \text{otherwise,} \end{cases}
\end{equation*}
where $c$ is a rational number such that $1 < c < q(0)$. If $X$ is $p$-bounded, then $\tilde{X}$ is $q$-bounded. (b) above then shows that $\tilde{X} \in \ldnr(q)$. This process defines a total recursive functional $\Psi$, so $\ldnr(q) \strongleq \ldnr(p)$.

\item Let $\psi$ be a linearly universal partial recursive function and suppose $X \in \avoid^\psi(q)$. Let $a,b \in \mathbb{N}$ be such that $\psi(ax+b) \simeq \psi_0(x)$ for all $x \in \mathbb{N}$ and let $\tilde{X}$ be defined by $\tilde{X}(x) \coloneq X(ax+b)$. Then $\tilde{X} \cap \psi_0 = \emptyset$ and $\tilde{X}(x) = X(ax+b) < q(ax+b) \leq p(x)$ shows that $\tilde{X} \in \avoid^{\psi_0}(p)$. 
\end{enumerate}
\end{proof}

Some other general basic reductions we make use of are given below.

\begin{prop} \label{avoidance basic reductions}
Let $\psi \colonsub \mathbb{N} \to \mathbb{N}$ and $p \colon \mathbb{N} \to (1,\infty)$ be given. 
\begin{enumerate}[(a)]
\item Suppose $q \colon \mathbb{N} \to (1,\infty)$ is an order function dominating $p$. Then $\avoid^\psi(q) \strongleq \avoid^\psi(p)$.

\item Suppose $u \colon \mathbb{N} \to \mathbb{N}$ is recursive. Then $\avoid^{\psi \circ u}(p \circ u) \strongleq \avoid^\psi(p)$.
\end{enumerate}
\end{prop}
\begin{proof} \mbox{}
\begin{enumerate}[(a)]
\item Suppose $p(n) \leq q(n)$ for all $n \geq N$. Let $\tau \in \{0,1\}^N$ be any string such that $\tau(n) \nsimeq \psi(n)$ for all $n < N$. Then the recursive functional $X \mapsto \tau \concat \left(X \restrict [N,\infty) \right)$ gives a strong reduction from $\avoid^\psi(q)$ to $\avoid^\psi(p)$.

\item The recursive functional $X \mapsto X \circ u$ gives a strong reduction from $\avoid^\psi(p)$ to $\avoid^{\psi \circ u}(p \circ u)$.

\end{enumerate}
\end{proof}

As our interest lies in $\mathcal{E}_\weak$, we must show that $\weakdeg(\ldnr(p)) \in \mathcal{E}_\weak$ for any recursive $p$. 

\begin{lem} \label{effective enumeration of the linearly universal partial recursive functions}
There is an effective enumeration of the linearly universal partial recursive functions.
\end{lem}
\begin{proof}
Let $\varphi_\bullet$ be an admissible enumeration of the partial recursive functions, let $\psi_0$ be a fixed linearly universal partial recursive function, and let $e_0$ be such that $\psi_0 = \varphi_{e_0}$.

The central observation we make is that for any partial recursive $\psi$, $\psi$ is linearly universal if and only if there are $a,b \in \mathbb{N}$ such that $\psi(ax+b) \simeq \psi_0(x)$ for all $x \in \mathbb{N}$. With this in mind, we modify $\varphi_\bullet$ to produce an enumeration of the linearly universal partial recursive functions as follows: given $a,b,e \in \mathbb{N}$, define $\psi_{\pi^{(3)}(a,b,e)} \colonsub \mathbb{N} \to \mathbb{N}$ by
\begin{equation*}
\psi_{\pi^{(3)}(a,b,e)}(x) \simeq \begin{cases} \psi_0(y) & \text{if $ay+b = x$,} \\ \varphi_e(x) & \text{otherwise.} \end{cases}
\end{equation*}
Then $\psi$ is linearly universal if and only if there exist $a,b,e \in \mathbb{N}$ such that $\psi = \psi_{\pi^{(3)}(a,b,e)}$, so $\psi_\bullet$ gives an effective enumeration of the linearly universal partial recursive functions.
\end{proof}

\begin{prop} \label{ldnr_p is Sigma02}
Suppose $p$ is a recursive function. Then $\ldnr(p)$ is $\Sigma^0_2$. Consequently, $\weakdeg(\ldnr(p)) \in \mathcal{E}_\weak$.
\end{prop}
\begin{proof}
By \cref{effective enumeration of the linearly universal partial recursive functions}, there exists an effective enumeration $\psi_\bullet$ of the linearly universal partial recursive functions. Let $\varphi_\bullet$ be any admissible enumeration; the Parametrization Theorem implies there is a total recursive function $f \colon \mathbb{N} \to \mathbb{N}$ such that $\varphi_{f(e)} = \psi_e$ for all $e \in \mathbb{N}$. Then
\begin{equation*}
X \in \ldnr(p) \equiv \exists e \forall n \forall s \forall m \qspace (\varphi_{f(e),s}(n) \converge = m \to m \neq X(n)) \wedge \forall n (X(n) < p(n))
\end{equation*}
shows that $\ldnr(p)$ is $\Sigma^0_2$. The \nameref{embedding lemma} then implies $\weakdeg(\ldnr(p)) \in \mathcal{E}_\weak$.
\end{proof}

\section{Fast \& Slow-Growing Order Functions}
\label{fast and slow-growing order functions section}

An important dividing line concerning how the growth rate of $p$ determines where in $\mathcal{E}_\weak$ the weak degree of $\ldnr(p)$ falls is whether the series $\sum_{n=0}^\infty{p(n)^{-1}}$ converges or not. 

\begin{definition}[fast \& slow-growing order functions]
Suppose $p \colon \mathbb{N} \to (0,\infty)$ is nondecreasing and computable. $p$ is \textdef{fast-growing} if $\sum_{n=0}^\infty{p(n)^{-1}} < \infty$ and \textdef{slow-growing} otherwise.
\end{definition}

\subsection{Bounding Sequences of Fast \& Slow-Growing Order Functions}

Given a recursive sequence of slow-growing order functions, we can effectively find a slow-growing lower bound (in the sense of $\domleq$). Under some additional conditions on the sequence, we can also effectively find a slow-growing upper bound.

\begin{prop} \label{lower bound of sequence of divergent series}
Suppose $\langle p_k\rangle_{k\in\mathbb{N}}$ is a recursive sequence of slow-growing order functions. 
\begin{enumerate}[(a)]
\item There is a slow-growing order function $q^-$ such that $q^- \domleq p_k$ for all $k \in \mathbb{N}$.
\item Suppose, additionally, that $p_k \domleq p_{k+1}$ for all $k \in \mathbb{N}$. Then there is a slow-growing order funciton $q^+$ such that $p_k \domleq q^+$ for all $k \in \mathbb{N}$.
\end{enumerate}
\end{prop}
\begin{proof} \mbox{}
\begin{enumerate}[(a)]
\item We simultaneously define the values $q^-(n)$ and natural numbers $M_n$ by recursion. We start by setting $q^-(0) \coloneq p_0(0)$ and $M_0 \coloneq 0$.

Given $q^-(0), q^-(1), \ldots, q^-(n)$ and $M_0, M_1, \ldots, M_n$ have been defined, let $M_{n+1}$ equal $M_n+1$ if $M_n+1 \leq \min_{0\leq k \leq M_n+1}{p_k(n+1)}$ and otherwise equal to $M_n$. Then define
\begin{equation*}
q^-(n+1) \coloneq \min\{p_0(n+1), p_1(n+1), \ldots, p_{M_{n+1}}(n+1), M_{n+1}+1\}.
\end{equation*}

We now claim that $q^-$ is a slow-growing order function dominated by each $p_k$.
\begin{description}
\item[Nondecreasing.] Given $n \in \mathbb{N}$, 
\begin{align*}
q^-(n) & = \min\{p_0(n), p_1(n), \ldots, p_{M_n}(n), M_n+1\} \\
& \leq \min\{p_0(n+1), p_1(n+1), \ldots, p_{M_{n+1}}(n+1), M_n+1\} \\
& \leq \min\{p_0(n+1), p_1(n+1), \ldots, p_{M_{n+1}}(n+1), M_{n+1}+1\} \\
& = q^-(n+1)
\end{align*}
as $p_k$ is nondecreasing for each $k$. 

\item[Unbounded.] It suffices to show that $\lim_{n \to \infty}{M_n} = \infty$. Suppose for the sake of a contradiction that $\lim_{n \to \infty}{M_n} < \infty$, so that $M_n$ is eventually constant, say to $M$. For $M_n$ to be eventually constant, it must be the case that $M > \min_{0 \leq k \leq M}{p_k(n+1)}$ for all $n \in \mathbb{N}$. But $p_0$ is unbounded, so $\min_{0 \leq k \leq M}{p_k(n+1)}$ is unbounded as a function of $n$, yielding a contradiction.

\item[Dominated by $p_k$.] Given $k$, there exists $n \in \mathbb{N}$ such that $M_n \geq k$. Then $q^-(n) \leq p_k(n)$ for all $n \geq M_n$. 

\item[Slow-Growing.] As $\sum_{n=0}^\infty{p_0(n)^{-1}} = \infty$ and $q^-$ is dominated by $p_0$, it follows by Direct Comparison that $\sum_{n=0}^\infty{q^-(n)^{-1}} = \infty$.

\item[Recursive.] The uniform recursiveness of the $p_k$'s implies that the simultaneous construction of $q^-$ and the sequence $\langle M_n \rangle_{n\in\mathbb{N}}$ is recursive.

\end{description}

\item We start by recursively defining natural numbers $N_m \in \mathbb{N}$. Let $N_0 = 0$, and given $N_m$ has been defined, define $N_{m+1}$ to be the least natural number greater than $N_m$ such that 
\begin{equation*}
\sum_{n=N_m}^{N_{m+1}-1}{p_m(n)^{-1}} \geq 1
\end{equation*}
and for which $p_m(N_{m+1}-1) \leq p_{m+1}(N_{m+1})$. This is possible because $p_m$ is slow-growing and $p_m \domleq p_{m+1}$.
We define $q^+\colon \mathbb{N} \to (0,\infty)$ as follows: given $n \in \mathbb{N}$, let $m$ be the unique natural number for which $N_m \leq n < N_{m+1}$, and define 
\begin{equation*}
q^+(n) \coloneq p_m(n).
\end{equation*}
We claim that $q^+$ is a slow-growing order function dominating each $p_k$.

\begin{description}
\item[Nondecreasing.] By definition, $q^+$ is nondecreasing on the interval $N_m \leq n < N_{m+1}$ since it agrees with the nondecreasing function $p_m$ on that interval. Thus, to show that $q^+$ is nondecreasing, it suffices to show that $q^+(N_{m+1}-1) \leq q^+(N_{m+1})$ for each $m \in \mathbb{N}$. But by the definition of $N_{m+1}$ and $q^+$, we have $q^+(N_{m+1}-1) = p_m(N_{m+1}-1) \leq p_{m+1}(N_{m+1}) = q^+(N_{m+1})$.

\item[Unbounded.] By definition, $p_0(n) \leq q^+(n)$ for all $n \in \mathbb{N}$. Since $p_0$ is unbounded, it follows that $q^+$ is unbounded.

\item[Slow-Growing.] For each $m \in \mathbb{N}$, by the definition of $\langle N_m\rangle_{m \in \mathbb{N}}$ and $q^+$ we have
\begin{equation*}
\sum_{n=0}^{N_{m}-1}{q^+(n)^{-1}} = \sum_{n=0}^{N_1-1}{p_0(n)^{-1}} + \sum_{n=N_1}^{N_2-1}{p_1(n)^{-1}} + \cdots + \sum_{n=N_{m-1}}^{N_{m}-1}{p_m(n)^{-1}} \geq m.
\end{equation*}
Thus, $\sum_{n=0}^\infty{q^+(n)^{-1}} = \lim_{m \to \infty}{\sum_{n=0}^{N_{m}-1}{q^+(N)^{-1}}} = \lim_{m \to \infty}{m} = \infty$, so $q^+$ is slow-growing.

\item[Dominates $p_k$.] By the definition of $\langle N_m\rangle_{m \in \mathbb{N}}$ and $q^+$, for each $k \in \mathbb{N}$ we have $q^+(n) \geq p_k(n)$ for all $n \geq N_k$, so $q^+ \domgeq p_k$.

\item[Recursive.] The uniform recursiveness of the $p_k$'s implies that the sequence $\langle N_m\rangle_{m \in \mathbb{N}}$ is recursive, and subsequently that the function $q^+$ is recursive.
\end{description}

\end{enumerate}
\end{proof}

Without the additional hypotheses in \cref{lower bound of sequence of divergent series}(b), an upper bound may not exist, however:

\begin{example} \label{max of slow-growing example}
We simultaneously define two slow-growing order functions $p_1,p_2 \colon \mathbb{N} \to (0,\infty)$ and a strictly increasing sequence $\langle N_m \rangle_{m \in \mathbb{N}}$. The role of $\langle N_m \rangle_{m \in \mathbb{N}}$ will be that the behaviors of $p_1$ or $p_2$ will be consistent between $N_m$ and $N_{m+1}-1$, with those behaviors switching upon incrementing $m$. We start by defining $p_1(0) = p_2(0) \coloneq 1$, $N_0 \coloneq 0$, and $N_1 \coloneq 1$. Suppose $N_m$ has been defined and that $p_1(n)$ and $p_2(n)$ have been defined for all $n < N_m$. We split into two cases, depending on whether $m$ is even or not.
\begin{description}
\item[Case 1: $m$ even.] Let $N_{m+1}$ be the least natural number greater than $N_m$ such that $\sum_{n=0}^{N_m-1}{p_1(n)^{-1}} + (N_{m+1}-N_m) \cdot p_1(N_m-1)^{-1} \geq m+1$, then define
\begin{align*}
p_1(n) & \coloneq p_1(N_m-1), \\
p_2(n) & \coloneq n^2,
\end{align*}
for $N_m \leq n < N_{m+1}$.

\item[Case 2: $m$ odd.] Identical to the case where $m$ is even, but with $p_1$ and $p_2$ switched.

\end{description}

In other words, we continually switch between being constant and being equal to the square function, with $p_1$ and $p_2$ having the opposite behavior of the other.
By construction, both $p_1$ and $p_2$ are slow-growing order functions, but $\max\{p_1(n),p_2(n)\} = n^2$ for every $n \in \mathbb{N}_{>0}$, so there is no slow-growing order function $q \colon \mathbb{N} \to (0,\infty)$ which dominates both $p_1$ and $p_2$ since $\sum_{n=1}^\infty{\frac{1}{n^2}} < \infty$.
\end{example}

For the fast-growing case, an upper bound always exists, but the existence of a lower bound requires an additional hypothesis. Unlike in the slow-growing case, this additional hypothesis is not only on the form of the sequence $\langle p_k \rangle_{k \in \mathbb{N}}$, but on the constituent $p_k$'s themselves.

\begin{lem} \label{sum of recursive series is left re}
Suppose $p \colon \mathbb{N} \to (0,\infty)$ is a fast-growing order function. Then $\sum_{n=0}^\infty{p(n)^{-1}}$ is a left r.e.\ real.
\end{lem}
\begin{proof}
$\left\langle \sum_{n=0}^k{p(n)^{-1}} \right\rangle_{k\in\mathbb{N}}$ is a sequence of uniformly recursive reals converging monotonically to $\sum_{n=0}^\infty{p(n)^{-1}}$ from below.
\end{proof}

\begin{prop} \label{lower and upper bounds of sequence of fast-growing functions}
Suppose $\langle p_k\rangle_{k\in\mathbb{N}}$ is a recursive sequence of fast-growing order functions $p_k\colon\mathbb{N} \to (0,\infty)$. 
\begin{enumerate}[(a)]
\item There is a fast-growing order funciton $q^+$ such that $p_k \domleq q^+$ for all $k \in \mathbb{N}$.
\item Suppose, additionally, that $\langle \sum_{n=0}^\infty{p_k(n)^{-1}} \rangle_{k \in \mathbb{N}}$ is a sequence of uniformly recursive reals. Then there is a fast-growing order function $q^-$ such that $q^- \domleq p_k$ for all $k \in \mathbb{N}$ and for which $\sum_{n=0}^\infty{q^-(n)^{-1}}$ is a recursive real.
\end{enumerate}
\end{prop}
\begin{proof} \mbox{}
\begin{enumerate}[(a)]
\item For $n \in \mathbb{N}$ we define
\begin{equation*}
q^+(n) \coloneq \max_{k \leq n}{p_k(n)}.
\end{equation*}
We claim that $q^+$ is a fast-growing order function dominating each $p_k$.

\begin{description}
\item[Nondecreasing.] For all $n \in \mathbb{N}$, that each $p_k$ is nondecreasing implies 
\begin{equation*}
q^+(n) = \max_{k \leq n}{p_k(n)} \leq \max_{k \leq n}{p_k(n+1)} \leq \max_{k \leq n+1}{p_k(n+1)} = q^+(n+1).
\end{equation*}

\item[Unbounded.] By construction, $q^+(n) \geq p_0(n)$ for all $n \in \mathbb{N}$. Because $p_0$ is unbounded, $q^+$ is as well.

\item[Fast-Growing.] By construction, $q^+(n) \geq p_0(n)$ for all $n$, so Direct Comparison shows $\sum_{n=0}^\infty{q^+(n)^{-1}} < \infty$ since $p_0$ is fast-growing.

\item[Dominates $p_k$.] Given $k \in \mathbb{N}$, for all $n \geq k$ we have $q^+(n) = \max_{m \leq n}{p_m(n)} \geq p_k(n)$. Thus, $p_k \domleq q^+$.

\item[Recursive.] The uniform recursiveness of the $p_k$'s immediately shows that $q^+$ is recursive.
\end{description}

\item We start by recursively defining natural numbesr $N_m \in \mathbb{N}$. Let $N_0 \coloneq 0$, and given $N_m$ has been defined, define $N_{m+1}$ to be the least natural number greater than $N_m$ such that
\begin{equation*}
\sum_{k=0}^{m+1}{\sum_{n=N_{m+1}}^\infty{\frac{1}{p_k(n)}}} \leq \frac{1}{2^{m+1}}
\end{equation*}
which exists since $p_k$ is fast-growing for each $k$. 

Now define $q^-$ as follows. Given $n \in \mathbb{N}$, let $m$ be the unique natural number for which $N_m \leq n < N_{m+1}$. Then define
\begin{equation*}
q^-(n) \coloneq \min_{k \leq m}{p_k(n)}.
\end{equation*}
We claim that $q^-$ is a nondecreasing, unbounded, fast-growing function which is dominated by each $p_k$. 

\begin{description}
\item[Nondecreasing.] For $n \in \mathbb{N}$, let $m$ be such that $N_m \leq n < N_{m+1}$. Then
\begin{equation*}
q^-(n+1) = \begin{cases} \min_{k \leq m}{p_k(n+1)} & \text{if $N_m \leq n+1 < N_{m+1}$,} \\ \min_{k \leq m+1}{p_k(n+1)} & \text{if $N_{m+1} \leq n+1$,} \end{cases} 
 \geq \min_{k \leq m}{p_k(n+1)} 
 \geq \min_{k \leq m}{p_k(n)} 
 = q^-(n).
\end{equation*}

\item[Unbounded.] Observe that if $N_m \leq n$, then $p_k(n) \geq m$ for all $k \leq m+1$. Thus, if $N_m \leq n < N_{m+1}$, we have
\begin{equation*}
q^-(n) = \min_{k \leq m}{p_k(n)} \geq m.
\end{equation*}
It follows that $q^-$ is unbounded.

\item[Fast-Growing.] By the definition of $N_m$ for $m \geq 1$,
\begin{equation*}
\sum_{k=0}^m{\sum_{n=N_m}^{N_{m+1}-1}{p_k(n)^{-1}}} \leq \sum_{k=0}^m{\sum_{n=N_m}^\infty{p_k(n)^{-1}}} \leq 2^{-m}.
\end{equation*}
Thus, 
\begin{align*}
\sum_{n=0}^\infty{q^-(n)^{-1}} 
& = \sum_{n=0}^{N_1-1}{q^-(n)^{-1}} + \sum_{m=1}^\infty{\sum_{n=N_m}^{N_{m+1}-1}{\max_{k \leq m}{p_k(n)^{-1}}}} \\
& \leq \sum_{n=0}^{N_1-1}{q^-(n)^{-1}} + \sum_{m=1}^\infty{\sum_{k=0}^m{\sum_{n=N_m}^{N_{m+1}-1}{p_k(n)^{-1}}}} \\
& \leq \sum_{n=0}^{N_1-1}{q^-(n)^{-1}} + \sum_{m=1}^\infty{2^{-m}} \\
& = \sum_{n=0}^{N_1-1}{q^-(n)^{-1}} + 1 \\
& < \infty.
\end{align*}

\item[Dominated by $p_k$.] By construction, for each $k$, $q^-(n) \leq p_k(n)$ for all $n > N_k$.

\end{description}

If the both the functions $p_k$ and the reals $\alpha_k \coloneq \sum_{n=0}^\infty{p_k(n)^{-1}}$ are uniformly recursive, then the function $m \mapsto N_m$ is recursive since $N_{m+1}$ is the least natural number greater than $N_m$ such that
\begin{equation*}
\sum_{k=0}^{m+1}{\left( \alpha_k - \sum_{n=N_m}^{N_{m+1}-1}{p_k(n)^{-1}}\right)} \leq 2^{-m}.
\end{equation*}
The recursiveness of the map $m \mapsto N_m$ then implies that $q^-$ is recursive.

Finally, we show that $\beta \coloneq \sum_{n=0}^\infty{q^-(n)^{-1}}$ is a recursive real. Define, for $i \geq 1$,
\begin{equation*}
\beta_i \coloneq \sum_{n=0}^{N_i-1}{q^-(n)^{-1}} + 2^{-(i-1)}.
\end{equation*}
We claim that $\langle \beta_i\rangle_{i\in\mathbb{N}_{\geq 1}}$ is a recursive sequence of uniformly recursive reals converging monotonically to $\beta$ from above. Since $\lim_{i \to \infty}{\sum_{n=0}^{N_i-1}{q^-(n)^{-1}}} = \beta$ and $\lim_{i \to \infty}{2^{-(i-1)}} = 0$, an argument analogous to the proof that $q^-$ was fast-growing shows $\lim_{i \to \infty}{\beta_i} = \beta$. Additionally, $\langle \beta_i\rangle_{i\in\mathbb{N}_{\geq 1}}$ is nonincreasing:
\begin{equation*}
\beta_{i+1} = \sum_{n=0}^{N_i-1}{q^-(n)^{-1}} + \sum_{n=N_i}^{N_{i+1}-1}{q^-(n)^{-1}} + 2^{-(i+1)} \leq \sum_{n=0}^{N_i-1}{q^-(n)^{-1}} + 2^{-(i+1)} + 2^{-(i+1)} = \sum_{n=0}^{N_i-1}{q^-(n)^{-1}} + 2^{-i} = \beta_i.
\end{equation*}
Thus, $\beta$ is right r.e.\ and hence recursive. 

\end{enumerate}
\end{proof}

\begin{cor} \label{fast-growing multiplicatively infinite lower bound}
Suppose $p\colon \mathbb{N} \to (0,\infty)$ is a fast-growing order function such that $\sum_{n=0}^\infty{p(n)^{-1}}$ is recursive. Then there exists a fast-growing order function $q$ such that $p(n)/q(n) \nearrow \infty$ as $n \to \infty$ and for which $\sum_{n=0}^\infty{q(n)^{-1}}$ is recursive.
\end{cor}
\begin{proof}
Let $\alpha \coloneq \sum_{n=0}^\infty{p(n)^{-1}}$ and let $p_m$ denote the function defined by $p_k(n) \coloneq p(n)/2^k$ for $k,n \in \mathbb{N}$. Note that $\sum_{n=0}^\infty{p_k(n)^{-1}} = 2^k\alpha$ is recursive, and the sequences $\langle p_k\rangle_{k \in \mathbb{N}}$ and $\langle 2^k\alpha \rangle_{k \in \mathbb{N}}$ are uniformly recursive. 

By \cref{lower and upper bounds of sequence of fast-growing functions}(b), there exists a fast-growing order function $q$ such that $q \domleq p_k$ for all $k \in \mathbb{N}$ and where $\sum_{n=0}^\infty{q(n)^{-1}}$ is a recursive real. Moreover, the proof of \cref{lower and upper bounds of sequence of fast-growing functions}(b) shows that there is such a $q$ for which $p(n)/q(n) \nearrow \infty$ as $n \to \infty$.
\end{proof}

\subsection{More about Recursive Sums}

The extra hypothesis that $\sum_{n=0}^\infty{p(n)^{-1}}$ be not only finite but additionally recursive does not hold for all fast-growing order functions $p$, so it is a strictly stronger hypothesis than just being fast-growing:

\begin{example}
Let $\alpha$ be any real in $(0,1)$ which is left r.e.\ but not recursive (e.g., $\alpha = \sum_{\varphi_e(0) \converge}{2^{-e}}$) and let $\langle \alpha_k \rangle_{k \in \mathbb{N}}$ be a recursive sequence of rational numbers converging monotonically to $\alpha$ from below. Without loss of generality, we may assume $\langle \alpha_k \rangle_{k \in \mathbb{N}}$ is strictly increasing. We simultaneously define a strictly increasing recursive function $k \mapsto N_k$ and an order function $p \colon \mathbb{N} \to \mathbb{N}$ with the following additional properties:
\begin{enumerate}[(i)]
\item $N_0 = 0$.
\item For all $k \in \mathbb{N}$, $p$ is constant on $\{N_k,N_k+1,\ldots,N_{k+1}-1\}$.
\item For all $k \in \mathbb{N}$, $\alpha_k \leq \sum_{n < N_{k+1}}{p(n)^{-1}} < \alpha_{k+1}$.
\end{enumerate}

Define $N_0 \coloneq 0$, let $m_0$ be the least positive integer such that $\frac{1}{m_0} < \alpha_1-\alpha_0$, and let $N_1$ be the least natural number for which $\alpha_0 \leq \frac{N_1}{m_0}$. The minimality of $N_1$ and $m$ imply $\frac{N_1}{m_0} < \alpha_1$.

Now suppose $N_{k+1}$ and $p(n)$ for $n < N_{k+1}$ have been defined such that $\alpha_k \leq \sum_{n < N_{k+1}}{p(n)^{-1}} < \alpha_{k+1}$. Let $m_k$ be the least natural number greater than $p(N_{k+1}-1)$ such that $\frac{1}{m_k} < \alpha_{k+2}-\alpha_{k+1}$, define $N_{k+2}$ to be the least natural number for which $\alpha_{k+1} \leq \sum_{n < N_{k+1}}{p(n)^{-1}} + (N_{k+2}-N_{k+1})\cdot \frac{1}{m_k}$, and finally let $p(n) = m_k$ for all $n \in \{N_{k+1},N_{k+1}+1,\ldots,N_{k+2}-1\}$. By definition, $\alpha_{k+1} \leq \sum_{n < N_{k+2}}{p(n)^{-1}}$. To see that $\sum_{n < N_{k+2}}{p(n)^{-1}} < \alpha_{k+2}$, observe that the minimality of $N_{k+2}$ would implies $\sum_{n < N_{k+2}-1}{p(n)^{-1}} < \alpha_{k+1}$, so if $\sum_{n < N_{k+2}}{p(n)^{-1}} \geq \alpha_{k+2}$ we would have $m^{-1} \geq \alpha_{k+2}-\alpha_{k+1}$, contradicting the choice of $m_k$. 

By construction, $p$ is an order function for which $\alpha_k \leq \sum_{n=0}^{N_{k+1}-1}{p(n)^{-1}} < \alpha$ for all $k \in \mathbb{N}$, so $\sum_{n=0}^\infty{p(n)^{-1}} = \alpha$ is left r.e.\ but not recursive.
\end{example}

We cover two results that assist in finding examples of fast-growing order functions $p$ for which $\sum_{n=0}^\infty{p(n)^{-1}}$ is a recursive real. The first shows that if a series converges at least as quickly as a series converging to a recursive real, then the first series converges to a recursive real.

\begin{prop} \label{recursive sum implies lower bounds have recursive sums}
Suppose $f,g \colon \mathbb{N} \to (0,\infty)$ are recursive functions such that $f \domleq g$ and $\sum_{n=0}^\infty{g(n)}$ is a recursive real. Then $\sum_{n=0}^\infty{f(n)}$ is a recursive real.
\end{prop}
\begin{proof}
We assume without loss of generality that $f(n) \leq g(n)$ for all $n \in \mathbb{N}$.

By the Direct Comparison Test, $\sum_{n=0}^\infty{f(n)}$ converges. Let $\alpha_i = \sum_{n=0}^i{f(n)}$, $\alpha = \sum_{n=0}^\infty{f(n)}$, $\beta_i = \sum_{n=0}^i{g(n)}$, and $\beta = \sum_{n=0}^\infty{g(n)}$. Because $f(n) \leq g(n)$ for all $n$, 
\begin{equation*}
\alpha - \alpha_i = \sum_{n=i+1}^\infty{f(n)} \leq \sum_{n=i+1}^\infty{g(n)} = \beta - \beta_i
\end{equation*}
and hence $\alpha \leq \alpha_i + (\beta - \beta_i)$. Then the sequence $\langle \alpha_i + (\beta - \beta_i) \rangle_{i\in\mathbb{N}}$ converges monotonically to $\alpha$ from above, showing that $\alpha$ is right r.e.\ and hence recursive by \cref{sum of recursive series is left re}.
\end{proof}

\begin{cor} \label{faster growing than recursive sum has recursive sum}
If $p$ and $q$ are order functions, $p \domleq q$, and $\sum_{n=0}^\infty{p(n)^{-1}}$ is a recursive real, then $\sum_{n=0}^\infty{q(n)^{-1}}$ is a recursive real.
\end{cor}

Our second tool is a reduction of the recursiveness of $\sum_{n=0}^\infty{p(n)^{-1}}$ to that of an improper integral $\int_0^\infty{\overline{p}(x)^{-1}\dd x}$ for a continuous recursive nondecreasing extension $\overline{p}$ of $p$.

\begin{prop} \label{infinite sum recursive iff improper integral recursive}
Suppose $p \colon \mathbb{N} \to (0,\infty)$ is a fast-growing order function and let $\overline{p} \colon [0,\infty) \to (0,\infty)$ be any continuous recursive nondecreasing extension of $p$. Then $\sum_{n=0}^\infty{p(n)^{-1}}$ is a recursive real if and only if $\int_0^\infty{\overline{p}(x)^{-1}\dd x}$ is a recursive real.
\end{prop}
%

\begin{proof}
We make two initial observations: 
\begin{enumerate}[(1)] 
\item Just as $\sum_{n=0}^\infty{p(n)^{-1}}$ is always left r.e., so is $\int_0^\infty{\overline{p}(x)^{-1}\dd x}$. Thus, for either quantity to be recursive it suffices to show that that quantity is right r.e.
\item The reals $\int_k^{k+1}{\overline{p}(x)^{-1}\dd x}$ are uniformly recursive in $k \in \mathbb{N}$.
\end{enumerate}

Suppose $\int_0^\infty{\overline{p}(x)^{-1}\dd x}$ is a recursive real. By the Integral Test, for each $k \in \mathbb{N}$ we have $\sum_{n=k+1}^\infty{p(n)^{-1}} \leq \int_k^\infty{\overline{p}(x)^{-1}\dd x} \leq \sum_{n=k}^\infty{p(n)^{-1}}$ and hence 
\begin{equation*}
0 \leq \left( \sum_{n=0}^k{p(n)^{-1}} + \int_k^\infty{\overline{p}(x)^{-1}\dd x}\right) - \sum_{n=0}^\infty{p(n)^{-1}} \leq p(k)^{-1}.
\end{equation*}
Thus, 
\begin{equation*}
\left\langle \sum_{n=0}^k{p(n)^{-1}} + \int_0^\infty{\overline{p}(x)^{-1}\dd x} - \int_0^k{\overline{p}(x)^{-1}\dd x} \right\rangle_{k \in \mathbb{N}}
\end{equation*}
is a sequence of uniformly recursive reals converging to $\sum_{n=0}^\infty{p(n)^{-1}}$ from above, hence right r.e.\ by \cref{equivalent characterizations of recursive reals}.

Now suppose $\sum_{n=0}^\infty{p(n)^{-1}}$ is a recursive real. By the Integral Test, for each $k \in \mathbb{N}$ we have $\int_{k+1}^\infty{\overline{p}(x)^{-1}\dd x} \leq \sum_{n=k+1}^\infty{p(n)^{-1}} \leq \int_k^\infty{\overline{p}(x)^{-1}\dd x}$ and hence
\begin{equation*}
0 \leq \left(\sum_{n=k+1}^\infty{p(n)^{-1}} + \int_0^{k+1}{\overline{p}(x)^{-1}\dd x}\right) - \int_0^\infty{\overline{p}(x)^{-1}\dd x} \leq \int_k^{k+1}{\overline{p}(n)^{-1}\dd x} \leq \overline{p}(k)^{-1}.
\end{equation*}
Thus, 
\begin{equation*}
\left\langle \sum_{n=0}^\infty{p(n)^{-1}} + \int_0^{k+1}{\overline{p}(x)^{-1}\dd x} - \sum_{n=0}^k{p(n)^{-1}} \right\rangle_{k \in \mathbb{N}}
\end{equation*}
is a sequence of uniformly recursive reals converging to $\int_0^\infty{\overline{p}(x)^{-1}\dd x}$ from above, hence right r.e.\ by \cref{equivalent characterizations of recursive reals}.
\end{proof}

Using \cref{infinite sum recursive iff improper integral recursive} allows us to show that many of the usual convergent series give recursive sums.

\begin{cor} \label{examples of recursive sums}
Given $k \in \mathbb{N}$ and a recursive real $\alpha \in (1,\infty)$, then 
\begin{equation*}
\sum_{n={}^k2}^\infty{\left( n \cdot \log_2 n \cdot \log_2^2 n \mdots \log_2^{k-1} n \cdot (\log_2^k n)^\alpha\right)^{-1}}
\end{equation*}
is a recursive real.
\end{cor}
\begin{proof}
Define $p \colon \mathbb{N}_{\geq {}^k2} \to \mathbb{R}$ and $\overline{p}\colon \co{{}^k2,\infty} \to \mathbb{R}$ by 
\begin{align*}
p(n) & \coloneq n \cdot \log_2 n \cdot \log_2^2 n \mdots \log_2^{k-1} n \cdot (\log_2^k n)^\alpha, \\
\overline{p}(x) & \coloneq x \cdot \log_2 x \cdot \log_2^2 x \mdots \log_2^{k-1} x \cdot (\log_2^k x)^\alpha,
\end{align*}
for $n \in \mathbb{N}$ and $x \in \co{{}^k2,\infty}$.

For each $k \geq 1$,
\begin{equation*}
\int_{{}^k2}^\infty{\frac{1}{x \cdot \log_2 x \cdot \log_2^2 x \mdots \log_2^{k-1} x \cdot (\log_2^k x)^\alpha}\mathrm{d}x} = (\ln 2)\int_{{}^{k-1}2}^\infty{\frac{1}{u \cdot \log_2 u \mdots \log_2^{k-2} u \cdot (\log_2^{k-1} u)^\alpha}\mathrm{d}u}.
\end{equation*}
so by induction on $k$ we may show
\begin{equation*}
\int_{{}^k2}^\infty{\frac{1}{x \cdot \log_2 x \cdot \log_2^2 x \mdots \log_2^{k-1} x \cdot (\log_2^k x)^\alpha}\mathrm{d}x} = \frac{(\ln 2)^k}{\alpha - 1},
\end{equation*}
so $\int_{{}^k2}^\infty{\overline{p}(x)^{-1}\dd x}$ is a recursive real. Then \cref{infinite sum recursive iff improper integral recursive} shows $\sum_{n={}^k2}^\infty{p(n)^{-1}}$ is a recursive real.
\end{proof}

\begin{cor}
If $p\colon \mathbb{N} \to (0,\infty)$ is an order function such that there is a $k \in \mathbb{N}$ and a recursive real $\alpha \in (1,\infty)$ for which $n \cdot \log_2 n \cdot \log_2^2 n \mdots \log_2^{k-1} n \cdot (\log_2^k n)^\alpha \leq p(n)$ for almost all $n \in \mathbb{N}_{\geq {}^k 2}$, then $\sum_{n=0}^\infty{p(n)^{-1}}$ is a recursive real.
\end{cor}
\begin{proof}
This follows from \cref{examples of recursive sums} and \cref{faster growing than recursive sum has recursive sum}.
\end{proof}

%
%

\subsection{The Fast and Slow-Growing \texorpdfstring{$\ldnr$}{LUA} Hierarchies}
\label{fast and slow-growing ldnr hierarchies subsection}

The dichotomy between fast-growing and slow-growing allows us to split the $\ldnr$ hierarchy into two sub-hierarchies. 

\begin{definition}[fast-growing and slow-growing $\ldnr$ hierarchies]
The \textdef{fast-growing $\ldnr$ hierarchy} is the collection of the classes $\ldnr(p)$ where $p$ is a fast-growing order function, and the \textdef{slow-growing $\ldnr$ hierarchy} is the collection of the classes $\ldnr(q)$ where $q$ is a slow-growing order function.
\end{definition}

The fast-growing $\ldnr$ hierarchy is downwards closed, while the slow-growing $\ldnr$ hierarchy is upwards closed, so the two hierarchies are separated:

\begin{prop} \label{separation of fast-growing and slow-growing ldnr hierarchies}
Suppose $p$ and $q$ are order functions such that $\ldnr(p) \weakleq \ldnr(q)$. 
\begin{enumerate}[(a)]
\item If $p$ is slow-growing, then $q$ is slow-growing.
\item If $q$ is fast-growing, then $p$ is fast-growing.
\end{enumerate}
\end{prop}
\begin{proof}
\cite[Theorem 5.4]{simpson2017turing} (see also \cref{ldnr and mlr}) shows that $\ldnr(q) \weakleq \mlr$ if and only if $q$ is fast-growing. If $q$ is fast-growing, then $\ldnr(p) \weakleq \ldnr(q) \weakleq \mlr$, showing $p$ is fast-growing. 
\end{proof}

The infimum of the slow-growing $\ldnr$ hierarchy lies in $\mathcal{E}_\weak$. 

\begin{prop}
Define
\begin{equation*}
\ldnr_\slow \coloneq \bigcup\{ \ldnr(q) \mid \text{$q$ a slow-growing order function}\}.
\end{equation*}
$\ldnr_\slow$ is $\Sigma^0_3$ and hence $\weakdeg(\ldnr_\slow) \in \mathcal{E}_\weak$.
\end{prop}
\begin{proof}
Let $f$ be as in the proof of \cref{ldnr_p is Sigma02}, so that $f$ is a total recursive function and $\varphi_{f(\bullet)}$ is an enumeration of the linearly universal partial recursive functions. Then
\begin{align*}
X \in \ldnr_\slow & \equiv \exists i \qspace\Bigl( \exists e \forall n \forall s \forall m \forall k \qspace \bigl( \bigl(\varphi_{f(e),s}(n) \converge = m \wedge \varphi_{i,s}(n) \converge = k\bigr) \to \bigl( m\neq X(n) \wedge X(n) < k\bigr)\bigr) \\
& \qquad \wedge \forall n \exists s \qspace \bigl( \varphi_{i,s}(n) \converge \bigr)\Bigr)
\end{align*}
shows that $\ldnr_\slow$ is $\Sigma^0_3$. The \nameref{embedding lemma} then implies $\weakdeg(\ldnr_\slow) \in \mathcal{E}_\weak$. 
\end{proof}

In \cref{bushy tree chapter} we will prove the following result:

\begin{repthm}{ldnr incomparable}
For all order functions $p_1$ and $p_2$, there exists a slow-growing order function $q$ such that $\ldnr(p_1) \nweakleq \ldnr(q) \nweakleq \ldnr(p_2)$. In particular, for any order function $p$, there exists a slow-growing order function $q$ such that $\ldnr(p)$ and $\ldnr(q)$ are weakly incomparable.
\end{repthm}

In particular, it is not the case that for every fast-growing $p$ and slow-growing $q$ that $q \domleq p$. It is interesting to note, however, that there exist slow-growing order functions $q$ such that $\ldnr(p) \weakleq \ldnr(q)$ for all fast-growing order functions $p$ -- in fact, we may take $q = \id_\mathbb{N}$.

\begin{lem} \label{fast-growing dominates identity}
Suppose $p$ is a fast-growing order function. Then $\id_\mathbb{N} \domleq p$.
\end{lem}
\begin{proof}
Suppose $n>0$ satisfies $p(n) \leq n$, and let $m \in \mathbb{N}$ be the largest natural number such that $2^m \leq n$. Then $p(2^m) \leq p(n) \leq n \leq 2^{m+1}$, so 
\begin{equation*}
\frac{1}{2} = \frac{2^m}{2^{m+1}} \leq \frac{2^m}{p(2^m)}.
\end{equation*}
Thus, if $p(n) \leq n$ for infinitely many $n \in \mathbb{N}$, then infinitely many of the terms of the series $\sum_{m=0}^\infty{2^m p(2^m)^{-1}}$ are bounded below by $1/2$, implying the series diverges. But an application of the Cauchy Condensation Test shows that the convergence of the series $\sum_{n=0}^\infty{p(n)^{-1}}$ (since $p$ is fast-growing) implies the convergence of $\sum_{m=0}^\infty{2^m p(2^m)^{-1}}$, a contradiction.
\end{proof}

\begin{prop} \label{fast-growing lies under identity}
$\ldnr(p) \weakleq \ldnr(\id_\mathbb{N})$ for every fast-growing order function $p$. 
\end{prop}
\begin{proof}
This follows immediately from \cref{fast-growing dominates identity}.
\end{proof}

Although there is no fast-growing order function $p$ such that $\ldnr(p) \weakleq \ldnr(q)$ for all slow-growing order functions $q$, to every slow-growing $q$ there is a fast-growing $p$ for which $\ldnr(p) \weakleq \ldnr(q)$.

\begin{prop} \label{fast-growing under each slow-growing}
For every slow-growing order function $q$ there exists a fast-growing order function $p$ such that $\ldnr(p) \weakleq \ldnr(q)$.
\end{prop}
\begin{proof}
Given $q$, define $p$ by setting $p(n) \coloneq q(n) + 2^n$ for each $n \in \mathbb{N}$.
\end{proof}

\section{Depth}
\label{depth section}

The notion of depth is a strengthening of the condition of being negligible.

\begin{definition}[negligibility]
$P \subseteq \cantor$ is \textdef{negligible} if $\lambda(P^{\turingleq}) = 0$. Equivalently, $\lambda(\Psi^{-1}[P])=0$ for every partial recursive functional $\Psi \colonsub \cantor \to \cantor$.
\end{definition}

Depth strengthens this by requiring additional uniformity and was first defined in \cite{bienvenu2016deep} for subsets of $\cantor$, as with negligibility. However, because certain weak degrees in $\mathcal{E}_\weak$ are best represented with subsets of $\baire$, it is more convenient to generalize the definition given by Bienvenu \& Porter. 

\begin{definition}[continuous semimeasure on $\mathbb{N}^\ast$]
A \textdef{continuous semimeasure on $\mathbb{N}^\ast$} is a map $\nu \colon \mathbb{N}^\ast \to [0,1]$ such that $\nu(\langle\rangle) = 1$ and $\sum_{i=0}^\infty{\nu(\sigma \concat \langle i \rangle)} \leq \nu(\sigma)$ for all $\sigma \in T$.

$\nu$ is \textdef{left recursively enumerable}, or \textdef{left r.e.}, if it is left r.e.\ in the usual sense. A left r.e.\ continuous semimeasure $\mathbf{M}$ on $\mathbb{N}^\ast$ is \textdef{universal} if for every left r.e.\ continuous semimeasure $\nu$ on $\mathbb{N}^\ast$ there exists a $c \in \mathbb{N}$ such that $\nu(\sigma) \leq c \cdot \mathbf{M}(\sigma)$ for all $\sigma \in \mathbb{N}^\ast$.
\end{definition}

The left r.e. continuous semimeasures on $\mathbb{N}^\ast$ can be characterized in terms of partial recursive functionals $\Psi \colonsub \cantor \to \baire$.

\begin{prop} \label{levin and zvonkin}
\textnormal{\cite[Theorem 3.1, essentially]{zvonkin1970complexity}}
\mbox{}
\begin{enumerate}[(a)]
\item If $\Psi \colonsub \cantor \to \baire$ is a partial recursive functional, then the map $\nu \colon \mathbb{N}^\ast \to [0,1]$ defined by $\nu(\sigma) \coloneq \lambda(\Psi^{-1}(\sigma)) = \lambda(\{ X \in \cantor \mid \Psi^X \supseteq \sigma\})$ is a left r.e.\ continuous semimeasure on $\mathbb{N}^\ast$.
\item If $\nu$ is a left r.e.\ continuous semimeasure on $\mathbb{N}^\ast$, then there is a partial recursive functional $\Psi \colonsub \cantor \to \baire$ such that $\nu(\sigma) = \lambda(\Psi^{-1}(\sigma)) = \lambda(\{ X \in \cantor \mid \Psi^X \supseteq \sigma\})$ for all $\sigma \in \mathbb{N}^\ast$.
\end{enumerate}
\end{prop}

\begin{remark}
The existence of a universal left r.e.\ continuous semimeasure on $\mathbb{N}^\ast$ can be shown by appropriately modifying proofs for the case of $\{0,1\}^\ast$ (e.g., in \cite[Theorem 3.16.2]{downey2010algorithmic} consider monotone machines of the form $M\colonsub \{0,1\}^\ast \to \mathbb{N}^\ast$).
\end{remark}

\begin{notation}
Given $P \subseteq \baire$ and $n \in \mathbb{N}$, $P \restrict n$ denotes the set $\{ X \restrict n \mid X \in P\}$.
\end{notation}

\begin{definition}[depth]
Let $\mathbf{M}$ be a universal left r.e.\ continuous semimeasure on $\mathbb{N}^\ast$. A mass problem $P \subseteq \baire$ is \textdef{deep} (with respect to $\mathbf{M}$) if it is a r.b. $\Pi^0_1$ class and there exists an order function $r \colon \mathbb{N} \to \mathbb{N}$ such that $\mathbf{M}(P \restrict r(n)) \leq 2^{-n}$ for all $n \in \mathbb{N}$. Such an $r$ is a \textdef{modulus of depth} for $P$.
\end{definition}

Despite the added generality, this more general notion of depth gives the same notion as that of Bienvenu \& Porter.

\begin{lem} \label{depth independent of superset}
Let $\mathbf{M}_{\{0,1\}}$ be a universal left r.e.\ continuous semimeasure on $\{0,1\}^\ast$, and $\mathbf{M}_\mathbb{N}$ is a universal left r.e.\ continuous semimeasure on $\mathbb{N}^\ast$. Given $P \subseteq \{0,1\}^\ast$, $P$ is deep with respect $\mathbf{M}_{\{0,1\}}$ if and only if $P$ is deep with respect to $\mathbf{M}_\mathbb{N}$.
\end{lem}
\begin{proof}
First suppose $P$ is deep with respect to $\mathbf{M}_{\{0,1\}}$, so that there is an order function $r_{\{0,1\}} \colon \mathbb{N} \to \mathbb{N}$ such that $\mathbf{M}_{\{0,1\}}(P \restrict r_{\{0,1\}}(n)) \leq 2^{-n}$ for all $n \in \mathbb{N}$. Let $\nu$ be the restriction of $\mathbf{M}_\mathbb{N}$ to $\{0,1\}^\ast$ and observe that $\nu$ is a left r.e.\ continuous semimeasure on $\{0,1\}^\ast$. By the universality of $\mathbf{M}_{\{0,1\}}$, there exists a $c \in \mathbb{N}$ such that $\mathbf{M}_\mathbb{N}(\sigma) = \nu(\sigma) \leq c \cdot \mathbf{M}_{\{0,1\}}(\sigma)$ for all $\sigma \in \{0,1\}^\ast$. Let $m \in \mathbb{N}$ be such that $c \leq 2^m$ and define $r_\mathbb{N} \colon \mathbb{N} \to \mathbb{N}$ by $r_\mathbb{N}(n) \coloneq r_{\{0,1\}}(n+m)$. Then for each $n \in \mathbb{N}$,
\begin{equation*}
\mathbf{M}_\mathbb{N}(P \restrict r_\mathbb{N}(n)) \leq c \cdot \mathbf{M}_{\{0,1\}}(P \restrict r_{\{0,1\}}(n+m)) \leq c \cdot 2^{-(n+m)} \leq 2^{-n}.
\end{equation*}
Thus, $P$ is deep with respect to $\mathbf{M}_\mathbb{N}$.

Now suppose $P$ is deep with respect to $\mathbf{M}_\mathbb{N}$, so that there is an order function $r_\mathbb{N} \colon \mathbb{N} \to \mathbb{N}$ such that $\mathbf{M}_\mathbb{N}(P \restrict r_\mathbb{N}(n)) \leq 2^{-n}$ for all $n \in \mathbb{N}$. Let $\nu$ be the extension of $\mathbf{M}_{\{0,1\}}$ to $\mathbb{N}^\ast$ by setting $\nu(\sigma) \coloneq 0$ for any $\sigma \notin \{0,1\}^\ast$ and $\nu(\sigma) = \mathbf{M}_{\{0,1\}}(\sigma)$ otherwise. The universality of $\mathbf{M}_\mathbb{N}$ implies there is a $c \in \mathbb{N}$ such that $\nu(\sigma) \leq c \cdot \mathbf{M}_\mathbb{N}(\sigma)$ for all $\sigma \in \mathbb{N}^\ast$. Let $m \in \mathbb{N}$ be such that $c \leq 2^m$ and define $r_{\{0,1\}} \colon \mathbb{N} \to \mathbb{N}$ by $r_{\{0,1\}}(n) \coloneq r_\mathbb{N}(n+m)$. Then for each $n \in \mathbb{N}$,
\begin{equation*}
\mathbf{M}_{\{0,1\}}(P \restrict r_{\{0,1\}}(n)) = \nu(P \restrict r_{\{0,1\}}(n)) \leq c \cdot \mathbf{M}_\mathbb{N}(P \restrict r_\mathbb{N}(n)) \leq c \cdot 2^{-(n+m)} \leq 2^{-n}.
\end{equation*}
Thus, $P$ is deep with respect to $\mathbf{M}_{\{0,1\}}$.
\end{proof}

\begin{remark}
\cref{depth independent of superset} additionally shows that the definition of depth does not depend on the choice of universal left r.e.\ semimeasure on $\mathbb{N}^\ast$.
\end{remark}

The classes $\ldnr(p)$ for $p$ a slow-growing order function provide a plethora of examples of deep r.b.\ $\Pi^0_1$ classes.

\begin{thm} \label{depth of dnr_p^psi}
\textnormal{\cite[Theorem 7.6]{bienvenu2016deep}, \cite[Theorem 2.2, Theorem 4.4]{simpson2017turing}}
Suppose $p \colon \mathbb{N} \to (1,\infty)$ is an order function and $\psi$ is a linearly universal partial recursive function.
\begin{enumerate}[(a)]
\item If $p$ is fast-growing, then $\avoid^\psi(p) \weakleq \mlr$. In particular, $\avoid^\psi(p)$ is non-negligible.
\item If $p$ is slow-growing, then $\avoid^\psi(p)$ is deep.
\end{enumerate}
\end{thm}

\subsection{Depth and Difference Randoms}

One of the principal properties of deep $\Pi^0_1$ classes relates deep $\Pi^0_1$ classes to \emph{difference random} sequences.

\begin{definition}[difference random]
$X \in \cantor$ is \textdef{difference random} if it is Martin-\Lof\ random but incomplete (i.e., $0' \turingnleq X$).
\end{definition}

A theorem of Sacks \cite[Corollary 8.12.2]{downey2010algorithmic} shows that $\lambda(\{0'\}^{\turingleq}) = 0$, so almost all members of $\mlr$ are difference random in the measure-theoretic sense.

\begin{thm} \label{difference randoms cannot compute members of deep pi01 classes}
\textnormal{\cite[Theorem 5.3]{bienvenu2016deep}}
Suppose $P \subseteq \cantor$ is a deep $\Pi^0_1$ class. If $X \in \cantor$ is difference random, then $X$ computes no member of $P$.
\end{thm}

A direct application is the following:

\begin{thm} \label{ldnr and mlr}
\textnormal{\cite[Theorem 5.4]{simpson2017turing}}
Suppose $p \colon \mathbb{N} \to (1,\infty)$ is an order function. 
\begin{enumerate}[(a)]
\item If $p$ is fast-growing, then $\ldnr(p) \weakle \mlr$.
\item If $p$ is slow-growing, then $\ldnr(p)$ and $\mlr$ are weakly incomparable.
\end{enumerate}
\end{thm}
\begin{proof} \mbox{}
\begin{enumerate}[(a)]
\item Suppose $p$ is fast-growing. \cref{depth of dnr_p^psi}(a) shows that for any linearly universal partial recursive $\psi$ that $\avoid^\psi(p) \weakleq \mlr$. As $\avoid^\psi(p) \subseteq \ldnr(p)$, it follows then that $\ldnr(p) \weakleq \mlr$. That this is strict follows from \cite[Theorem 5.11]{greenberg2011diagonally}.

\item Supose $p$ is slow-growing. If $\ldnr(p) \weakleq \mlr$, then in particular there is an $X \in \ldnr(p)$ and a difference random $Y$ such that $X \turingleq Y$. But $X$ being a member of $\ldnr(p)$ means that $X \in \avoid^\psi(p)$ for some linearly universal partial recursive function $\psi$, meaning that $Y$ computes a member of the deep (\cref{depth of dnr_p^psi}(b)) $\Pi^0_1$ class $\avoid^\psi(p)$, contradicting \cref{difference randoms cannot compute members of deep pi01 classes}.

The opposite direction, i.e., that $\mlr \nweakleq \ldnr(p)$, follows from \cite[Theorem 5.11]{greenberg2011diagonally}.

\end{enumerate}
\end{proof}

\subsection{Depth and Strong Reducibility}

We used \cref{recursively bounded pi01 class recursively homeomorphic to pi01 class in cantor space} to show that despite our interests laying in $\cantor$, considering r.b.\ $\Pi^0_1$ classes was safe. We show depth behaves similarly by showing that depth is preserved under recursive homeomorphisms.

\begin{lem} \label{use function defined for recursive functionals from recursively bounded pi01 sets}
Suppose $P$ is a r.b.\ $\Pi^0_1$ class and that $\Psi \colon P \to \baire$ is a recursive functional. Then there exist nondecreasing recursive functions $\psi \colon \mathbb{N}^\ast \to \mathbb{N}^\ast$ and $j \colon \mathbb{N} \to \mathbb{N}$ such that $\Psi(X) \restrict n = \psi(X \restrict j(n))$ for all $X \in P$.
\end{lem}
\begin{proof}
By \cref{extending recursive functionals on recursively bounded pi01 classes}, $\Psi$ extends to a total recursive functional $\tilde{\Psi} \colon \baire \to \baire$. Let $e \in \mathbb{N}$ be an index for $\tilde{\Psi}$ (i.e., so that $\varphi_e^{X}(n) \simeq \tilde{\Psi}(X)(n)$ for all $X \in \baire$ and $n \in \mathbb{N}$). Let $h \colon \mathbb{N} \to (1,\infty)$ be a nondecreasing recursive function for which $P \subseteq h^\mathbb{N}$. 

By the compactness of $h^\mathbb{N}$, for each $n \in \mathbb{N}$ there is a $s_n$ such that $\varphi_{e,s_n}^{X \restrict s_n}(n) \converge$ for all $X \in h^\mathbb{N}$, and such an $s_n$ can be found effectively as a function of $n$, so define $j(n) \coloneq \max\{s_0,s_1,s_2,\ldots,s_n\}$. 

Given $\sigma \in h^\ast$, let $n$ be the largest natural number for which $j(n) \leq |\sigma|$ and define 
\begin{equation*}
\psi(\sigma) \coloneq \left\langle \varphi_{e,j(n)}^{\sigma}(0),\varphi_{e,j(n)}^{\sigma}(1),\ldots,\varphi_{e,j(n)}^{\sigma}(n)\right\rangle.
\end{equation*}
For $\sigma \in \mathbb{N}^\ast \setminus h^\ast$, let $\psi(\sigma) = \sigma$. Then $\psi(X \restrict j(n)) = \tilde{\Psi}(X) \restrict n$ for all $X \in h^\mathbb{N}$, and in particular $\psi(X \restrict j(n)) = \Psi(X) \restrict n$ for all $X \in P$.
\end{proof}

\begin{prop} \label{depth and strong reducibility}
Suppose $P$ is a r.b.\ $\Pi^0_1$ class and that $\Psi \colon P \to \baire$ is a recursive functional. If $\Psi[P]$ is deep, then $P$ is deep.
\end{prop}
\begin{proof}
Let $Q = \Psi[P]$ and let $r$ be a modulus of depth for $Q$. Without loss of generality, assume $\Psi$ is a total recursive functional. 

Let $\psi \colon \mathbb{N}^\ast \to \mathbb{N}^\ast$ and $j \colon \mathbb{N} \to \mathbb{N}$ be as in the proof of \cref{use function defined for recursive functionals from recursively bounded pi01 sets}, so that they are nondecreasing recursive functions  such that $\Psi(X) \restrict n = \psi(X \restrict j(n))$ for all $X \in \baire$ and that if $|\sigma| = j(n)$, then $|\psi(\sigma)| = n$. For $\tau \in \mathbb{N}^\ast$, let $\psi^{-1}_\mathrm{min}[\{\tau\}]$ denote the set of minimal elements of $\psi^{-1}[\{\tau\}]$, and observe that $\psi^{-1}_\mathrm{min}[\{\tau\}] = \psi^{-1}[\{\tau\}] \cap \mathbb{N}^{j(|\tau|)}$. Define $\nu \colon \mathbb{N}^\ast \to [0,1]$ by
\begin{equation*}
\nu(\tau) \coloneq \mathbf{M}(\psi^{-1}[\{\tau\}]) = \sum_{\psi(\sigma) = \tau}{\mathbf{M}(\sigma)}.
\end{equation*}
Assume at the present that $\nu$ is a left r.e.\ continuous semimeasure so that there is a $c \in \mathbb{N}$ such that $N(\tau) \leq c \cdot \mathbf{M}(\tau)$ for all $\tau \in \mathbb{N}^\ast$. Letting $m \in \mathbb{N}$ be such that $c \leq 2^m$, for each $n \in \mathbb{N}$ we have
\begin{equation*}
\mathbf{M}(P \restrict j(r(n+m))) \leq \mathbf{M}(\psi^{-1}\left[ Q \restrict r(n+m)\right]) 
 = \nu(Q \restrict r(n+m)) 
 \leq c \cdot \mathbf{M}(Q \restrict r(n+m)) 
 \leq c \cdot 2^{-(n+m)} 
 \leq 2^{-n}
\end{equation*}
so the function $n \mapsto j(r(n+m))$ is a modulus of depth for $P$.

Now we show that $\nu$ is a left r.e.\ continuous semimeasure. For any $\tau \in \mathbb{N}^\ast$ and $i < i'$, we have $\psi^{-1}[\{\tau \concat \langle i \rangle\}] \cap \psi^{-1}[\{\tau \concat \langle i' \rangle\}] = \emptyset$. Additionally, every element of $\psi^{-1}[\{\tau \concat \langle i \rangle\}]$ has length $j(|\tau|+1)$ (so elements of $\psi^{-1}[\{\tau \concat \langle i \rangle\}]$ are pairwise incompatible) and extends a member of $\psi^{-1}[\{\tau\}]$. Thus, $\sum_{i = 0}^\infty{\mathbf{M}(\psi^{-1}[\{\tau \concat \langle i \rangle\}])} \leq \mathbf{M}(\psi^{-1}[\{\tau\}])$ and so $\nu$ is a continuous semimeasure. For each $\tau \in \mathbb{N}^\ast$, the set $\psi^{-1}[\{\tau\}]$ is recursive since $\psi$ is recursive, nondecreasing, and $|\sigma| = j(n)$ implies $|\psi(\sigma)| = n$. Thus, along with the fact that $\mathbf{M}$ is left r.e.\ we see that $\nu$ is left r.e., as desired.
\end{proof}

\begin{cor}
Suppose $P$ and $Q$ are recursively homeomorphic r.b.\ $\Pi^0_1$ classes. Then $P$ is deep if and only if $Q$ is deep.
\end{cor}

This well-behavedness of depth with partial recursive functionals can be summarized nicely:

\begin{prop} \label{deep pi01 classes form filter under strong reduction}
\textnormal{\cite[Theorem 6.4]{bienvenu2016deep}}
The collection of deep r.b.\ $\Pi^0_1$ classes forms a filter with respect to strong reducibility. I.e., for all r.b.\ $\Pi^0_1$ classes $P$ and $Q$:
\begin{enumerate}[(a)]
\item If $P \strongleq Q$ and $P$ is deep, then $Q$ is deep.
\item If $P$ and $Q$ are deep, the $\langle 0 \rangle \concat P \cup \langle 1 \rangle \concat Q$ is deep.
\end{enumerate}
\end{prop}
\begin{proof}
Let $P$ and $Q$ be r.b.\ $\Pi^0_1$ classes. If $P$ is deep and $P \strongleq Q$, then \cref{depth and strong reducibility} shows that $Q$ is deep. 

Now suppose $P$ and $Q$ are both deep and let $r_0$ and $r_1$ be moduli of depth for $P$ and $Q$, respectively. For $i \in \{0,1\}$, define $\nu_i \colon \mathbb{N}^\ast \to [0,1]$ by $\nu_i(\sigma) \coloneq \mathbf{M}(\langle i \rangle \concat \sigma)$. $\nu_i$ is a left r.e.\ continuous semimeasure, so there exists $c_i \in \mathbb{N}$ such that $\nu_i(\sigma) \leq c_i \cdot \mathbf{M}(\sigma)$ for all $\sigma \in \mathbb{N}^\ast$. Let $m \in \mathbb{N}$ be such that $\max\{c_0,c_1\} \leq 2^{m-1}$, and define $r \colon \mathbb{N} \to \mathbb{N}$ by $r(n) \coloneq \max\{r_0(n+m),r_1(n+m)\}+1$. Then
\begin{align*}
\mathbf{M}([\langle 0 \rangle \concat P \cup \langle 1 \rangle \concat Q] \restrict r(n)) & = \mathbf{M}([\langle 0 \rangle \concat P] \restrict r(n)) + \mathbf{M}([\langle 1 \rangle \concat Q] \restrict r(n)) \\
& = \nu_0(P \restrict (r(n)-1)) + \nu_1(Q \restrict (r(n)-1)) \\
& \leq c_0 \cdot \mathbf{M}(P \restrict r_0(n)) + c_1 \cdot \mathbf{M}(Q \restrict r_1(n)) \\
& \leq 2^{-n-1} + 2^{-n-1} = 2^{-n}.
\end{align*}
\end{proof}

\subsection{Depth and Weak Reducibility}

\cref{depth and strong reducibility} shows that depth is especially well-behaved with respect to strong reducibility. However, depth is not as well-behaved with respect to weak reducibility.

\begin{prop} 
\textnormal{\cite[Theorem 4.7]{bienvenu2016deep}}
For any $\Pi^0_1$ class $P \subseteq \cantor$ there exists a $\Pi^0_1$ class $Q \subseteq \cantor$ which is weakly equivalent to $P$ but not deep.
\end{prop}

There \emph{is} an analog of \cref{deep pi01 classes form filter under strong reduction} if we consider weak degrees of deep $\Pi^0_1$ classes.

\begin{definition}[deep degree in $\mathcal{E}_\weak$]
A weak degree $\mathbf{p} \in \mathcal{E}_\weak$ is a \textdef{deep degree (in $\mathcal{E}_\weak$)} if there exists a deep nonempty $\Pi^0_1$ class $P$ for which $\weakdeg(P) = \mathbf{p}$. 

$P \subseteq \baire$ is \textdef{of deep degree} if $\weakdeg(P)$ is a deep degree in $\mathcal{E}_\weak$.
\end{definition}

\begin{prop} \label{deep degrees form filter}
The collection of deep degrees in $\mathcal{E}_\weak$ forms a filter in $\langle \mathcal{E}_\weak, \leq\rangle$. I.e., for all $\mathbf{p},\mathbf{q} \in \mathcal{E}_\weak$:
\begin{enumerate}[(a)]
\item If $\mathbf{p} \leq \mathbf{q}$ and $\mathbf{p}$ is a deep degree, then $\mathbf{q}$ is a deep degree.
\item If $\mathbf{p}$ and $\mathbf{q}$ are deep degrees, then $\inf\{\mathbf{p},\mathbf{q}\}$ is a deep degree.
\end{enumerate}
\end{prop}
\begin{proof}
Let $P$ and $Q$ be $\Pi^0_1$ classes for which $\mathbf{p} = \weakdeg(P)$ and $\mathbf{q} = \weakdeg(Q)$. 

If $\mathbf{p}$ is a deep degree, then we may assume without loss of generality that $P$ is deep. Then $P \times Q = \{ X \oplus Y \mid X \in P \wedge Y \in Q\}$ is deep by \cref{deep pi01 classes form filter under strong reduction} since $P \strongleq P \times Q$. Because $P \weakleq Q$, $\weakdeg(P \times Q) = \weakdeg(Q) = \mathbf{q}$, so $\mathbf{q}$ is a deep degree. 

Now suppose both $\mathbf{p}$ and $\mathbf{q}$ are deep degrees, and assume without loss of generality that $P$ and $Q$ are both deep. \cref{deep pi01 classes form filter under strong reduction} shows that $\langle 0 \rangle \concat P \cup \langle 1 \rangle \concat Q$ is deep. Because $\inf\{\mathbf{p},\mathbf{q}\} = \weakdeg(P \cup Q) = \weakdeg(\langle 0 \rangle \concat P \cup \langle 1 \rangle \concat Q)$, $\inf\{\mathbf{p},\mathbf{q}\}$ is a deep degree.
\end{proof}

\begin{lem} \label{slow-growing composed with linear is slow-growing}
Suppose $p \colon \mathbb{N} \to (1,\infty)$ is nondecreasing, $a \in \mathbb{N}_{>0}$, and $b \in \mathbb{N}$. Then $\sum_{n=0}^\infty{p(n)^{-1}} < \infty$ if and only if $\sum_{n=0}^\infty{p(an+b)^{-1}} < \infty$.
\end{lem}
\begin{proof}
We may assume without loss of generality that $b=0$. Because $p$ is nondecreasing, 
\begin{equation*}
a \cdot p(a \cdot (n+1))^{-1} \leq \sum_{i=0}^{a-1}{p(an+i)^{-1}} \leq a \cdot p(an)^{-1}, 
\end{equation*}
so
\begin{equation*}
a \cdot \sum_{n=1}^{k+1}{p(an)^{-1}} \leq \sum_{n=0}^{a(k+1)-1}{p(n)^{-1}} \leq a \cdot \sum_{n=0}^k{p(an)^{-1}}.
\end{equation*}
Because $p$ is a positive-valued function, each of the above summations is nondecreasing as a function of $k$, so if $\sum_{n=0}^\infty{p(n)^{-1}} < \infty$, Monotone Convergence would imply $\sum_{n=0}^\infty{p(an)^{-1}} < \infty$ and vice-versa.
\end{proof}

\begin{cor} \label{ldnr p deep degree if p slow-growing}
Suppose $p \colon \mathbb{N} \to (1,\infty)$ is an order function. Then $\ldnr(p)$ is of deep degree if and only if $p$ is slow-growing.
\end{cor}
\begin{proof}
Fix a linearly universal partial recursive function $\psi_0$. Given $a,b \in \mathbb{N}$, let $p_{a,b}$ denote the function defined by $p_{a,b}(x) \coloneq p(ax+b)$ for $x \in \mathbb{N}$. By \cref{slow-growing composed with linear is slow-growing}, the sequence $\langle p_{a,0} \rangle_{a \in \mathbb{N}}$ is a recursive sequence of slow-growing order functions. Moreover, $p_{a,0} \domleq p_{a+1,0}$ for every $a \in \mathbb{N}$, so \cref{lower bound of sequence of divergent series}(b) shows there is a slow-growing order function $q$ such that $p_{a,0} \domleq q$ for all $a \in \mathbb{N}$. For any $a,b \in \mathbb{N}$, we have $p_{a,b} \domleq p_{a+1,0}$, so we have $p_{a,b} \domleq q$ more generally.

Given a linearly universal partial recursive $\psi$, there are $a,b \in \mathbb{N}$ such that $\psi_0(x) \simeq \psi(ax+b)$ for all $x \in \mathbb{N}$. For such $a,b \in \mathbb{N}$, $\avoid^{\psi_0}(p_{a,b}) \strongleq \avoid^\psi(p)$ by \cref{avoidance basic reductions}(b). Because $p_{a,b} \domleq q$, we have $\avoid^{\psi_0}(q) \strongleq \avoid^{\psi_0}(p_{a,b})$ by \cref{avoidance basic reductions}(a). Thus, $\avoid^{\psi_0}(q) \weakleq \ldnr(p)$. $\avoid^{\psi_0}(q)$ is a deep r.b.\ $\Pi^0_1$ class by \cref{depth of dnr_p^psi}, so \cref{deep degrees form filter} implies $\ldnr(p)$ is of deep degree.

If $p$ were fast-growing, then for any linearly universal partial recursive function $\psi$, we would have $\ldnr(p) \strongleq \avoid^\psi(p) \weakleq \mlr$.
\end{proof}

\cref{difference randoms cannot compute members of deep pi01 classes} extends to $P \subseteq \baire$ of deep degree.

\begin{prop} \label{difference randoms cannot compute member of sets of deep degree}
Suppose $P \subseteq \baire$ is of deep degree. If $X \in \cantor$ is difference random, then $X$ computes no member of $P$.
\end{prop}
\begin{proof}
Because $P$ is of deep degree, there exists a deep $\Pi^0_1$ class $Q$ such that $P \weakeq Q$. If $Y \turingleq X$ for some $Y \in P$, then the fact that $P \weakeq Q$ implies there is a $Z \in Q$ such that $Z \turingleq Y \turingleq X$, contradicting \cref{difference randoms cannot compute members of deep pi01 classes}.
\end{proof}

\subsection{Depth for non-r.b.\ \texorpdfstring{$\Pi^0_1$}{Pi01} Sets}

Nothing in our definition of depth necessitates that $P$ be a r.b.\ $\Pi^0_1$ class in order for the definition to make sense. However, there are two reasons for our restriction to only r.b.\ $\Pi^0_1$ classes. The first is that our interest in depth is ultimately relegated to r.b.\ $\Pi^0_1$ subsets of $\baire$. The second is that it is unclear whether `depth' is a useful notion for arbitrary subsets of $\baire$, and if so, whether the verbatim extension of the definition of depth to any subset of $\baire$ provides the `right' definition.

In fact, we can show that extending our definition of depth even to only $\Pi^0_1$ subsets of $\baire$ or $\Pi^0_2$ subsets of $\cantor$ causes us to lose the guarantee of \cref{difference randoms cannot compute members of deep pi01 classes} that no difference random computes a member of a `deep set'.

\begin{lem} \label{reducible to halting problem implies pi02 singleton}
Suppose $X \in \cantor$ and $X \turingleq 0'$. Then $X$ is a $\Pi^0_2$ singleton, i.e., $\{X\}$ is $\Pi^0_2$.
\end{lem}
\begin{proof}
$X\turingleq 0'$ implies $X$ (as a subset of $\mathbb{N}$) is $\Delta^0_2$, so there are recursive predicates $R$ and $S$ such that 
\begin{equation*}
x \in X \iff \forall n \exists m \qspace R(x,n,m) \iff \exists n \forall m \qspace S(x,n,m).
\end{equation*}
Then
\begin{equation*}
\{ X \} = \bigl\{ Y \in \cantor \mid \forall x \qspace \bigl(\bigl(x \in Y \to \forall n \exists m \qspace R(x,n,m)\bigr) \wedge \bigl(\exists n \forall m \qspace S(x,n,m) \to x \in Y\bigr)\bigr)\bigr\}
\end{equation*}
shows that $X$ is a $\Pi^0_2$ singleton.
\end{proof}

\begin{prop} \label{depth cannot be generalized without cost}
There exists a subset $P \subseteq \baire$ which is deep in the sense that there is an order function $r$ such that $\mathbf{M}(P \restrict r(n)) \leq 2^{-n}$ for all $n \in \mathbb{N}$, but for which there are difference random sequences that compute members of $P$. Moreover, $P$ may be taken to either be a $\Pi^0_2$ subset of $\cantor$ or a $\Pi^0_1$ subset of $\baire$.
\end{prop}
\begin{proof}
Let $Q = \{ X \in \cantor \mid \forall n \qspace \apc(X \restrict n) \geq n-c\}$, where $c$ is sufficiently large so that $Q \neq \emptyset$. By the Low Basis Theorem \cite{jockusch1971classes} there is an $A \in Q$ such that $A \turingle 0'$. \cref{reducible to halting problem implies pi02 singleton} then implies $A$ is a $\Pi^0_2$ singleton, but being an incomplete Martin-\Lof\ random sequence means that it is also difference random. Since $A \in Q$, we also know that $\mathbf{M}(A \restrict (n+c)) \leq 2^c \cdot 2^{-(n+c)} = 2^{-n}$ for all $n \in \mathbb{N}$. Thus, $\{A\}$ is deep in the extended sense, but the difference random $A$ computes a member of $\{A\}$.

Being a $\Pi^0_2$ singleton, there is a recursive predicate $R$ such that $A$ is the only sequence $X$ satisfying $\forall n \exists m \qspace R(X,n,m)$. Define $f \colon \mathbb{N} \to \mathbb{N}$ by 
\begin{equation*}
f(n) \coloneq \text{least $m$ such that $\langle A,n,m\rangle \in R$}.
\end{equation*}
Then define $B \in \baire$ by $B(n) \coloneq \pi^{(2)}(A(n),f(n))$. Given $X \in \baire$ and $i \in \{0,1\}$, let $(X)_i$ be defined by $(X)_i(n) = (X(n))_i$ for $n \in \mathbb{N}$, where $(\pi^{(2)}(n_0,n_1))_i  = n_i$. Then $\{B\} = \{ X \mid \forall n \qspace R((X)_0,n,(X)_1(n))\}$ is $\Pi^0_1$ and recursively homeomorphic to $\{A\}$. 

Let $\Psi \colon \baire \to \baire$ be the total recursive functional defined by $\Psi(X) \coloneq (X)_0$ for $X \in \baire$. By \cref{levin and zvonkin}, there is a partial recursive functional $\Psi_0 \colonsub \cantor \to \baire$ such that $\mathbf{M}(\sigma) = \lambda(\Psi_0^{-1}(\sigma))$. Then define $\nu$ to be the left r.e.\ continuous semimeasure corresponding to the partial recursive functional $\Psi \circ \Psi_0$, i.e.,
\begin{equation*}
\nu(\sigma) = \lambda(\{ Z \in \cantor \mid (\Psi \circ \Psi_0)^Z \supseteq \sigma\}).
\end{equation*}
By the universality of $\mathbf{M}$, there is a $c \in \mathbb{N}$ such that $\nu(\sigma) \leq c' \cdot \mathbf{M}(\sigma)$ for all $\sigma \in \{0,1\}^\ast$. Let $m \in \mathbb{N}$ be such that $2^{c'} \leq m$. Then
\begin{align*}
\mathbf{M}(\{B\} \restrict (n+c+m)) & = \lambda(\{ Z \in \cantor \mid \Psi_0^Z \supseteq B \restrict (n+c+m)\}) \\
& \leq \lambda(\{ Z \in \cantor \mid (\Psi \circ \Psi_0)^Z \supseteq A \restrict (n+c+m)\}) \\
& = \nu(A \restrict (n+c+m)) \\
& \leq c' \cdot \mathbf{M}(A \restrict (n+c+m)) \\
& \leq 2^{-n}.
\end{align*}
Thus, $\{B\}$ is a $\Pi^0_1$ subset of $\baire$ which is deep in the extended sense, but the difference random $A$ computes a member of $\{B\}$.
\end{proof}

\clearpage
\chapter{Complexity and Fast-Growing Avoidance}
\label{complexity and avoidance downward relationships chapter}

Looking downward, the $\complex$ and $\ldnr$ hierarchies are closely coupled based on the following reseult of Kjos-Hanssen, Merkle, and Stephan:

\begin{thm*} \label{kjos-hanssen complexity result}
\textnormal{\cite[Theorem 2.3.2]{kjoshanssen2006kolmogorov}}
Suppose $X \in \cantor$. The following are equivalent.
\begin{enumerate}[(i)]
\item $X \in \complex$.
\item There is a total recursive functional $\Psi \colon \cantor \to \baire$ such that $\Psi(X) \in \dnr$.
\end{enumerate}
\end{thm*}

In terms of the mass problems $\complex$ and $\dnr_\rec \coloneq \bigcup\{\dnr(p) \mid \text{$p$ recursive}\}$, \cite[Theorem 2.3.2]{kjoshanssen2006kolmogorov} implies:

\begin{cor} \label{complex weakly equivalent to dnr rec}
$\complex \weakeq \dnr_\rec$.
\end{cor}
\begin{proof}
Suppose $X \in \dnr_\rec$, so that there is an order function $p$ such that $X \in \dnr_p$. Let $\Psi \colon \cantor \to p^\mathbb{N}$ be a recursive homeomorphism. Then $\Psi(\Psi^{-1}(X)) = X$ shows that $\Psi^{-1}(X) \in \complex$ by \cite[Theorem 2.3.2]{kjoshanssen2006kolmogorov}. This shows $\complex \weakleq \dnr_\rec$.

Now suppose $X \in \complex$. By \cite[Theorem 2.3.2]{kjoshanssen2006kolmogorov}, there is a total recursive functional $\Psi \colon \cantor \to \baire$ such that $Y \coloneq \Psi(X) \in \dnr$. \cref{use function defined for recursive functionals from recursively bounded pi01 sets} shows that there exist nondecreasing recursive functions $\psi \colon \{0,1\}^\ast \to \mathbb{N}^\ast$ and $j \colon \mathbb{N} \to \mathbb{N}$ such that $Y \restrict n = \psi(X \restrict j(n))$. Let $p \colon \mathbb{N} \to \mathbb{N}$ be defined by $p(n) \coloneq \max\{ \psi(\sigma)(n) \mid \sigma \in \{0,1\}^{j(n+1)}\}+1$, so that $Y(n) < p(n)$ for all $n \in \mathbb{N}$. $p$ is recursive, showing $Y$ is recursively bounded, i.e., $Y \in \dnr_\rec$. This shows $\dnr_\rec \weakleq \complex$.
\end{proof}

Let $\ldnr_\rec \coloneq \bigcup\{\ldnr(p) \mid \text{$p$ recursive}\}$.

\begin{cor} \label{complex weakly equivalent to ldnrrec}
$\complex \weakeq \ldnr_\rec$.
\end{cor}
\begin{proof}
By \cref{complex weakly equivalent to dnr rec}, it suffices to show that $\ldnr_\rec \weakeq \dnr_\rec$. 
Because \cite[Theorem 2.3.2]{kjoshanssen2006kolmogorov} (and by extension \cref{complex weakly equivalent to dnr rec}) holds with $\dnr$ defined with respect to any admissible enumeration $\varphi_\bullet$, we may assume without loss of generality that $\dnr$ is interpreted with respect to an admissible enumeration $\varphi_\bullet$ whose diagonal $\psi$ is linearly universal, so that $\dnr_\rec \subseteq \ldnr_\rec$ and hence $\ldnr_\rec \weakleq \dnr_\rec$. Conversely, given $X \in \ldnr_\rec$, there is a linearly universal partial recursive function $\tilde{\psi}$ and an order function $p$ such that $X \in \avoid^{\tilde{\psi}}(p)$. Because $\tilde{\psi}$ is linearly universal, there exist $a,b \in \mathbb{N}$ such that $\tilde{\psi}(an+b) \simeq \psi(n)$ for all $n \in \mathbb{N}$. Let $Y \in \baire$ be defined by $Y(n) \coloneq X(an+b)$ for $n \in \mathbb{N}$. Then $Y \in \avoid^\psi(\lambda n. p(an+b)) \subseteq \dnr_\rec$, showing $\dnr_\rec \weakleq \ldnr_\rec$.
\end{proof}

Knowing $\complex \weakeq \ldnr_\rec$ alone does not reveal how the complexity and fast-growing $\ldnr$ hierarchies intertwine (if at all) when going downward, prompting the following questions.

\begin{question} \label{downward ldnr below complex}
Given a sub-identical order function $f \colon \mathbb{N} \to \co{0,\infty}$, is there a fast-growing order function $q \colon \mathbb{N} \to (1,\infty)$ such that $\ldnr(q) \weakleq \complex(f)$?
\end{question}

\begin{question} \label{downward complex below ldnr}
Given a fast-growing order function $p \colon \mathbb{N} \to (1,\infty)$, is there a sub-identical order function $g \colon \mathbb{N} \to \co{0,\infty}$ such that $\complex(g) \weakleq \ldnr(p)$?
\end{question}

We will answer \cref{downward ldnr below complex,,downward complex below ldnr} in the affirmative, proving:

\begin{thm} \label{main downward theorem}
To each sub-identical order function $f\colon \mathbb{N} \to \co{0,\infty}$ there is a fast-growing order function $q\colon \mathbb{N} \to (1,\infty)$ such that $\ldnr(q) \strongleq \complex(f)$, and to each fast-growing order function $p\colon \mathbb{N} \to (1,\infty)$ there is a sub-identical order function $g\colon \mathbb{N} \to \co{0,\infty}$ such that $\complex(g) \strongleq \ldnr(p)$.
\end{thm}

\cref{main downward theorem} will follow as a direct consequence of the following theorems, which additionally provide explicit bounds on $q$ and $g$ in terms of $f$ and $p$, respectively.


\begin{repthm}{complex reals compute ldnr functions specific}
Suppose $f\colon \mathbb{N} \to \co{0,\infty}$ is a sub-identical order function, $k$ is a nonzero natural number, and $\epsilon > 0$ is a rational number. 
Then 
\begin{equation*}
\ldnr\bigl(\lambda n. \exp_2\bigl(f^\inverse(\log_2 n + \log_2^2 n + \cdots + \log_2^{k-1} n + (1+\epsilon) \log_2^k n)+1\bigr)\bigr) \strongleq \complex(f).
\end{equation*}
\end{repthm}


\begin{repthm}{ldnr functions computes complex reals}
Suppose $p \colon \mathbb{N} \to (1,\infty)$ is an order function, and let $r \colon \mathbb{N} \to \mathbb{N}$ be any order function such that $\lim_{n \to \infty}{r(n)/2^n} = \infty$. Then
\begin{equation*}
\complex\Bigl( \bigl(\lambda n. \raisebox{.1em}{\scalebox{.9}{$\displaystyle\sum$}}_{i < r(n)}{\lfloor \log_2 p(i) \rfloor} \bigr)^\inverse\Bigr) \strongleq \ldnr(p).
\end{equation*}
\end{repthm}

Less is known in the opposite direction.

\begin{question} \label{upward ldnr above complex}
Given a sub-identical order function $f \colon \mathbb{N} \to \co{0,\infty}$, is there a fast-growing order function $q \colon \mathbb{N} \to (1,\infty)$ such that $\complex(f) \weakleq \ldnr(q)$?
\end{question}

\begin{question} \label{upward complex above ldnr}
Given a fast-growing order function $p \colon \mathbb{N} \to (1,\infty)$, is there a sub-identical order function $g \colon \mathbb{N} \to \co{0,\infty}$ such that $\ldnr(p) \weakleq \complex(g)$?
\end{question}

While we have no general answer to \cref{upward ldnr above complex}, we will give a partial affirmative answer to \cref{upward complex above ldnr}.

\begin{repthm}{complex hierarchy outpaces fast nice ldnr hierarchy}
If $p\colon \mathbb{N} \to (1,\infty)$ is a fast-growing order function such that $\sum_{n=0}^\infty{p(n)^{-1}}$ is a recursive real, then there exists a convex sub-identical order function $g$ such that $\ldnr(p) \strongleq \complex(g) \neq \mlr$.
\end{repthm}

\cref{complex hierarchy outpaces fast nice ldnr hierarchy} follows from a more general result:

\begin{repthm}{complex hierarchy outpaces fast nice ldnr hierarchy quantified}
Suppose $p\colon \mathbb{N} \to (1,\infty)$ is a fast-growing order function such that $\sum_{n=0}^\infty{p(n)^{-1}}$ is a recursive real. Then for any order function $\tilde{p}\colon \mathbb{N} \to (1,\infty)$ such that $p(n)/\tilde{p}(n) \nearrow \infty$ as $n \to \infty$ and for which $\sum_{n=0}^\infty{\tilde{p}(n)^{-1}}$ is a recursive real, 
\begin{equation*}
\ldnr(p) \strongleq \complex\bigl(\lambda n. \log_2 \tilde{p}\bigl(p^\inverse(2^{n+1})-1\bigr)\bigr).
\end{equation*}
Moreover, such a $\tilde{p}$ exists and for any such $\tilde{p}$ the function $\lambda n. \log_2 \tilde{p}(p^\inverse(2^{n+1})-1)$ is dominated by a convex sub-identical recursive function $g \colon \mathbb{N} \to \co{0,\infty}$.
\end{repthm}


\section{Extracting Fast-Growing Avoidance from Complexity} \label{ldnr below complex section}

In the direction $\ldnr(p) \weakleq \complex(f)$, we have the following:

\begin{thm} \label{complex reals compute ldnr functions specific}
Suppose $f\colon \mathbb{N} \to \co{0,\infty}$ is a sub-identical order function, $k$ is a nonzero natural number, and $\epsilon > 0$ is a rational number. 
Then 
\begin{equation*}
\ldnr\bigl(\lambda n. \exp_2\bigl(f^\inverse(\log_2 n + \log_2^2 n + \cdots + \log_2^{k-1} n + (1+\epsilon) \log_2^k n)+1\bigr)\bigr) \strongleq \complex(f).
\end{equation*}
\end{thm}

We prove \cref{complex reals compute ldnr functions} by deducing it from a more general technical result:

\begin{thm} \label{complex reals compute ldnr functions}
Suppose $f \colon \mathbb{N} \to \co{0,\infty}$ is a sub-identical order function and $h,\hat{h}\colon \mathbb{N} \to \co{0,\infty}$ are order functions such that $\sum_{n=0}^\infty{2^{-h(n)}} < \infty$, $\sum_{n=0}^\infty{2^{-\hat{h}(n)}} < \infty$, and $\lim_{n \to \infty}{(h(n) - \hat{h}(n))} = \infty$. Then
\begin{equation*}
\ldnr\bigl(\lambda n. \exp_2((f^\inverse \circ h)(n)+1)\bigr) \strongleq \complex(f).
\end{equation*}
\end{thm}

The role of the condition $\sum_{n=0}^\infty{2^{-h(n)}} < \infty$ comes from the following result:

\begin{lem} \label{computable bounds on prefix free complexity}
\textnormal{\cite[Lemma 3.12.2]{downey2010algorithmic}}
Suppose $h \colon \mathbb{N} \to \mathbb{R}$ is recursive. The following are equivalent.
\begin{enumerate}[(i)]
\item $\sum_{n=0}^\infty{2^{-h(n)}} < \infty$.
\item There exists a $c \in \mathbb{N}$ such that $\pfc(n) \leq h(n) + c$ for all $n \in \mathbb{N}$.
\end{enumerate}
\end{lem}

\begin{proof}[Proof of \cref{complex reals compute ldnr functions}.]
Suppose $X \in \complex(f)$. Define $Y \in \baire$ by setting
\begin{equation*}
Y(n) \coloneq \str^{-1}(X \restrict (f^\inverse\circ h)(n))
\end{equation*}
for $n \in \mathbb{N}$. 
We claim that $Y \in \ldnr\bigl(\lambda n. \exp_2((f^\inverse\circ h)(n)+1)\bigr)$. Because $X$ is $f$-complex, there is $c _0\in \mathbb{N}$ such that for every $n \in \mathbb{N}$ we have
\begin{equation*}
\pfc(\str Y(n)) = \pfc(X \restrict (f^\inverse\circ h)(n)) \geq f\bigl((f^\inverse\circ h)(n)\bigr) - c_0 \geq h(n) - c_0.
\end{equation*}

Suppose $\psi$ is any partial recursive function. There is $c_1 \in \mathbb{N}$ such that $\pfc(\str \psi(n)) \leq \pfc(\str n) + c_1$ for all $n \in \dom \psi$. By \cref{computable bounds on prefix free complexity} there is an $c_2 \in \mathbb{N}$ such that $\pfc(n) \leq \hat{h}(n) + c_2$ for all $n \in \mathbb{N}$. Thus, for every $n \in \dom \psi$ we have
\begin{equation*}
\pfc(\str \psi(n)) \leq \hat{h}(n) + (c_1+c_2).
\end{equation*}

For every $n \in \dom \psi$ such that $\psi(n) = Y(n)$ we have
\begin{equation*}
h(n) - c_0 \leq \pfc(\str Y(n)) = \pfc(\str \psi(n)) \leq \hat{h}(n) + (c_1+c_2).
\end{equation*}
Because $\lim_{n \to \infty}{(h(n) - \hat{h}(n))} = \infty$, it follows that $|Y \cap \psi|$ is finite. When $\psi$ is linearly universal, \cref{ldnr basic facts}(b) shows that $Y \in \ldnr$.

It remains to put a uniform upper bound on $Y$. By definition, 
\begin{equation*}
Y(n) = X(0) + X(1)\cdot 2 + X(2) \cdot 2^2 + \cdots + X((f^\inverse\circ h)(n)-1) \cdot 2^{(f^\inverse\circ h)(n)-1} + 2^{(f^\inverse\circ h)(n)}
\end{equation*}
so $Y(n) < \exp_2((f^\inverse\circ h)(n)+1)$ and hence $Y \in \ldnr\bigl(\lambda n. \exp_2((f^\inverse\circ h)(n)+1)\bigr)$.
\end{proof}

\begin{cor} \label{complex reals compute ldnr functions recursive sum}
Suppose $f\colon \mathbb{N} \to \co{0,\infty}$ is a sub-identical order function and $h \colon \mathbb{N} \to \mathbb{R}$ is an order function such that $\sum_{n=0}^\infty{2^{-h(n)}}$ is a recursive real. Then
\begin{equation*}
\ldnr\bigl(\lambda n. \exp_2((f^\inverse\circ h)(n)+1)\bigr) \strongleq \complex(f).
\end{equation*}
\end{cor}
\begin{proof}
By \cref{fast-growing multiplicatively infinite lower bound}, there exists a fast-growing, recursive order function $\hat{h}\colon \mathbb{N} \to \mathbb{R}$ such that \\ $\lim_{n \to \infty}{2^{h(n)}/2^{\hat{h}(n)}} = \infty$, or equivalently that $\lim_{n \to \infty}{(h(n) - \hat{h}(n))} = \infty$. \cref{complex reals compute ldnr functions} now applies.
\end{proof}

Using \cref{complex reals compute ldnr functions} and \cref{complex reals compute ldnr functions recursive sum} we can prove the special case of \cref{complex reals compute ldnr functions specific}. 

\begin{proof}[Proof of \cref{complex reals compute ldnr functions specific}.]
Define $h \colon \mathbb{N} \to \co{0,\infty}$ by setting $h(n) \coloneq \log_2 n + \log_2^2 n + \cdots + \log_2^{k-1} n + (1+\epsilon) \log_2^k n$ for $n \geq {}^k 2$ and $h(n) \coloneq \log_2 (2 \cdot {}^22\cdot {}^32 \mdots {}^k2)$ for $n < {}^k 2$. Then $2^{h(n)} = n \cdot \log_2 n \cdot \log_2^2 n \mdots \log_2^{k-2} n \cdot (\log_2^{k-1} n)^{1+\epsilon}$ for all $n \geq {}^k 2$ and $2^{h(n)} = 2 \cdot {}^22\cdot {}^32 \mdots {}^k2$ for $n < {}^k 2$. By \cref{examples of recursive sums}, $\sum_{n=0}^\infty{2^{-h(n)}}$ is a recursive real, so \cref{complex reals compute ldnr functions recursive sum} applies.
\end{proof}

\begin{remark}
Let $\varphi_\bullet$ be an admissible enumeration and take $X \in \complex(f,c)$, where $f$ is a sub-identical order function and $c \in \mathbb{N}$. 
The element $Y \turingleq X$ of $\baire$ defined in the proof of \cref{complex reals compute ldnr functions} is shown to eventually avoid not just linearly universal partial recursive functions, but \emph{all} partial recursive functions. Moreover, at what point $Y$ begins avoiding $\varphi_e$ is predictable: the partial function $\psi(e,n) \simeq \varphi_e(n)$ is partial recursive, so there exists $d \in \mathbb{N}$ (depending on $e$) such that $\pfc(\varphi_e(n)) \leq h(n) + d$.
Then for all $n \in \mathbb{N}$ such that $h(n) - c > \hat{h}(n) + d$ we have $Y(n) \nsimeq \varphi_e(n)$. 
\end{remark}

\begin{example} \label{reduction from delta-complex sequences}
Fix a rational $\delta \in (0,1)$, $k \in \mathbb{N}$, and a rational $\epsilon > 0$. Let $f$ be defined by $f(n) \coloneq \delta n$ for each $n \in \mathbb{N}$, so $f^\inverse(n) = \lceil \frac{1}{\delta} n \rceil$ for each $n \in \mathbb{N}$. Observe that 
\begin{align*}
& \exp_2(\lceil \delta^{-1} \cdot (\log_2 n + \log_2^2 n + \cdots + \log_2^{k} n + (1+\epsilon) \log_2^{k+1} n) \rceil +1) \\
& \quad \leq \exp_2(\delta^{-1} \cdot (\log_2 n + \log_2^2 n + \cdots + \log_2^{k} n + (1+\epsilon) \log_2^{k+1} n) +2) \\
& \quad \leq 4\sqrt[\delta]{n \cdot \log_2 n \mdots \log_2^{k-1} n \cdot (\log_2^k n)^{(1+\epsilon)}}.
\end{align*}
Thus, \cref{complex reals compute ldnr functions specific} implies
$\displaystyle
\ldnr\Bigl(\lambda n. 4\sqrt[\delta]{n \cdot \log_2 n \mdots \log_2^{k-1} n \cdot (\log_2^k n)^{(1+\epsilon)}}\Bigr) \strongleq \complex(\delta).
$
\end{example}

\begin{example} \label{reduction from power-complex sequences}
Fix rational numbers $\alpha \in (0,1)$ and $\epsilon > 0$. Let $f$ be defined by $f(n) \coloneq n^\alpha$ for each $n \in \mathbb{N}$, so $f^\inverse(n) = \lceil n^{1/\alpha} \rceil$ for each $n \in \mathbb{N}$. Observe that
\begin{equation*}
\lceil ((1+\epsilon) \log_2 n)^{1/\alpha} \rceil + 1 \leq ((1+\epsilon) \log_2 n)^{1/\alpha} + 2 
\end{equation*}
so \cref{complex reals compute ldnr functions specific} implies
$\displaystyle
\ldnr\bigl(\lambda n. 4 \exp_2\bigl(\sqrt[\alpha]{(1+\epsilon)\log_2 n}\bigr)\bigr) \strongleq \complex(\lambda n. n^\alpha).
$
\end{example}

\begin{example} \label{logarithmic complex example}
Fix rationals $\beta \in (0,\infty)$ and $\epsilon > 0$. Let $f$ be defined by $f(n) \coloneq \beta \log_2 n$ for each $n \in \mathbb{N}$, so $f^\inverse(n) = \lceil 2^{n/\beta} \rceil$ for each $n \in \mathbb{N}$. Observe that 
\begin{equation*}
\lceil \exp_2((1+\epsilon) (\log_2 n) / \beta) \rceil +1 \leq n^{(1+\epsilon)/\beta} + 2
\end{equation*}
so \cref{complex reals compute ldnr functions specific} implies
$\displaystyle
\ldnr\bigl(\lambda n. 4 \exp_2(n^{(1+\epsilon)/\beta})\bigr) \strongleq \complex(\lambda n. \beta \log_2 n).
$
\end{example}

\section{Extracting Complexity from Fast-Growing Avoidance} \label{complex below ldnr section}

In the direction $\complex(g) \weakleq \ldnr(p)$, we have the following:

\begin{thm} \label{ldnr functions computes complex reals}
Suppose $p \colon \mathbb{N} \to (1,\infty)$ is an order function, and let $r \colon \mathbb{N} \to \mathbb{N}$ be any order function such that $\lim_{n \to \infty}{r(n)/2^n} = \infty$. Then
\begin{equation*}
\complex\Bigl( \bigl(\lambda n. \raisebox{.1em}{\scalebox{.9}{$\displaystyle\sum$}}_{i < r(n)}{\lfloor \log_2 p(i) \rfloor} \bigr)^\inverse\Bigr) \strongleq \ldnr(p).
\end{equation*}
\end{thm}
\begin{proof}
Suppose $X \in \ldnr(p)$, and let $\psi$ be a linearly universal partial recursive function for which $X \in \avoid^\psi(p)$. Define $q \colon \mathbb{N} \to \mathbb{N}$ by setting $q(n) \coloneq \sum_{i<n}{\lfloor \log_2 p(i)\rfloor}$ for each $n \in \mathbb{N}$. 

Define $Y \in \cantor$ to be the unique real for which $X(n) = \sum_{i < \lfloor \log_2 p(n) \rfloor}{Y(q(n)+i)\cdot 2^i}$. We claim that there is $c \in \mathbb{N}$ such that $Y \in \complex((\lambda n. q(c \cdot 2^n))^\inverse)$.

Let $U \colonsub \{0,1\}^\ast \to \{0,1\}^\ast$ be the universal prefix-free machine for which $\pfc = \pfc_U$. Then define $\theta \colonsub \mathbb{N}^3 \to \mathbb{N}$ by setting
\begin{equation*}
\theta(u,v,n) \simeq \sum_{i < \lfloor \log_2 p(u n + v) \rfloor}{U(\str n)(q(u n + v) + i) \cdot 2^i}
\end{equation*}
for $u,v,n \in \mathbb{N}$. 
$\theta$ is partial recursive, so \cref{recursion theorem for linearly universal partial recursive functions} implies there are $a,b \in \mathbb{N}$ such that 
\begin{equation*}
\psi(an+b) \simeq \sum_{i < \lfloor \log_2 p(a n + b) \rfloor}{U(\str n)(q(a n + b) + i) \cdot 2^i}
\end{equation*}
for all $n \in \mathbb{N}$. By hypothesis, $X \cap \psi = \emptyset$, so for all $n \in \mathbb{N}$ we have
\begin{equation*}
\sum_{i < \lfloor \log_2 p(an+b) \rfloor}{Y(q(an+b)+i) \cdot 2^i} \nsimeq \sum_{i < \lfloor \log_2 p(an+b) \rfloor}{U(\str n)(q(an+b)+i) \cdot 2^i}.
\end{equation*}
In other words, for all $n \in \mathbb{N}$ either $U(\str n) \diverge$, $|U(\str n)| < q(an+b+1)$, or $U(\str n)$ is incompatible with $Y \restrict q(an+b+1)$. Note that for any $\sigma \in \{0,1\}^\ast$ we have $\str^{-1} \sigma \leq 2^{|\sigma|+1}$. Thus, for $\sigma \in \{0,1\}^{\leq n}$ we find that either $U(\sigma) \diverge$, $|U(\sigma)| < q(an+b+1)$, or $U(\sigma)$ is incompatible with $Y \restrict q(a \cdot 2^{n+1}+b+1)$. Consequently, for all $n \in \mathbb{N}$ we have
\begin{equation*}
\pfc(Y \restrict q(a \cdot 2^{n+1}+b+1)) > n
\end{equation*}
so $Y \in \complex((\lambda n. q(a \cdot 2^{n+1}+b+1))^\inverse)$. If $\lim_{n \to \infty}{r(n)/2^n} = \infty$, then $\lambda n. q(a \cdot 2^{n+1}+b+1) \domleq \lambda n. (q \circ r)(n)$ and hence $(\lambda n. (q \circ r)(n))^\inverse \domleq (\lambda n. q(a\cdot 2^{n+1} + b+1))^\inverse$. Thus, $Y \in \complex((\lambda n. (q \circ r)(n))^\inverse)$.
\end{proof}

\begin{example} \label{exponential ldnr example}
Fix a rational $\epsilon > 0$, then let $p$ and $r$ be defined by $p(n) \coloneq 2^n$ and $r(n) \coloneq \lfloor 2^{(1+\epsilon)n}\rfloor$ for each $n \in \mathbb{N}$. Let $q$ be as in proof of \cref{ldnr functions computes complex reals}, so that $q(n) = \frac{n(n-1)}{2}$ for all $n \in \mathbb{N}$. Then $(q \circ r)(n) \leq 2^{2(1+\epsilon) n}$ for almost all $n \in \mathbb{N}$. Thus, \cref{ldnr functions computes complex reals} implies
$
\complex(\lambda n. \frac{1}{2+\epsilon}\log_2 n) \strongleq \ldnr(\lambda n. 2^n).
$
for any rational $\epsilon > 0$.
\end{example}

\begin{remark}
Combining \cref{exponential ldnr example,,logarithmic complex example} shows that for any rational $\epsilon > 0$ we have \\
$
\ldnr(\lambda n. \exp_2(n^{2+\epsilon})) \strongleq \complex(\lambda n. \frac{1}{2+\epsilon} \log_2 n) \strongleq \ldnr(\lambda n. 2^n).
$
\end{remark}

\section{Finding Complexity Above Fast-Growing Avoidance} \label{complexity to avoidance upwards}

In one of the upward directions, we can give a partial answer to \cref{upward complex above ldnr}. In addition to being sub-identical, a common additional hypothesis put on order functions $f \colon \mathbb{N} \to \co{0,\infty}$ is that $f$ be convex (see, e.g., \cite{hudelson2013partial} and \cite{hudelson2014mass}).

\begin{definition}[convex]
A nondecreasing function $f \colon \mathbb{N} \to \mathbb{R}$ is \textdef{convex} if $f(n+1) - f(n) \leq 1$ for all $n \in \mathbb{N}$.
\end{definition}

There is a simple characterization of convexity by putting functions into the form $\lambda n. n-j(n)$.

\begin{prop} \label{equivalent characterization of convexity}
Suppose $f \colon \mathbb{N} \to \mathbb{R}$ is a nondecreasing function, and define $j \colon \mathbb{N} \to \mathbb{R}$ by $j(n) \coloneq n - f(n)$ for each $n \in \mathbb{N}$. Then $f$ is convex if and only if $j$ is nondecreasing.
\end{prop}
\begin{proof}
For each $n \in \mathbb{N}$,
\begin{align*}
f(n+1) - f(n) \leq 1 & \iff ((n+1) - j(n+1)) - (n - j(n)) \leq 1 \\
& \iff 1 + j(n) - j(n+1) \leq 1 \\
& \iff j(n) \leq j(n+1).
\end{align*}
\end{proof}

\begin{thm} \label{complex hierarchy outpaces fast nice ldnr hierarchy}
If $p\colon \mathbb{N} \to (1,\infty)$ is a fast-growing order function such that $\sum_{n=0}^\infty{p(n)^{-1}}$ is a recursive real, then there exists a convex sub-identical order function $g$ such that $\ldnr(p) \strongleq \complex(g) \neq \mlr$.
\end{thm}

\cref{complex hierarchy outpaces fast nice ldnr hierarchy} follows immediately from the following more general result, which makes use of the downward result \cref{ldnr functions computes complex reals}. 

\begin{thm} \label{complex hierarchy outpaces fast nice ldnr hierarchy quantified}
Suppose $p\colon \mathbb{N} \to (1,\infty)$ is a fast-growing order function such that $\sum_{n=0}^\infty{p(n)^{-1}}$ is a recursive real. Then for any order function $\tilde{p}\colon \mathbb{N} \to (1,\infty)$ such that $p(n)/\tilde{p}(n) \nearrow \infty$ as $n \to \infty$ and for which $\sum_{n=0}^\infty{\tilde{p}(n)^{-1}}$ is a recursive real, 
\begin{equation*}
\ldnr(p) \strongleq \complex\bigl(\lambda n. \log_2 \tilde{p}\bigl(p^\inverse(2^{n+1})-1\bigr)\bigr).
\end{equation*}
Moreover, such a $\tilde{p}$ exists and for any such $\tilde{p}$ the function $\lambda n. \log_2 \tilde{p}(p^\inverse(2^{n+1})-1)$ is dominated by a convex sub-identical recursive function $g \colon \mathbb{N} \to \co{0,\infty}$.
\end{thm}

It will be convenient to assume that $p$ is strictly increasing. The following lemma shows that we may assume this without loss of generality.

\begin{lem} \label{strictly increasing lower bound with recursive sum}
If $p\colon \mathbb{N} \to (1,\infty)$ is a fast-growing order function such that $\sum_{n=0}^\infty{p(n)^{-1}}$ is a recursive real, then there exists a fast-growing order function $\hat{p}\colon \mathbb{N} \to (1,\infty)$ such that $\sum_{n=0}^\infty{\hat{p}(n)^{-1}}$ is a recursive real, $\hat{p}$ is strictly increasing, and $\hat{p} \domleq p$. 
\end{lem}
\begin{proof}
Let $\langle \epsilon_n \rangle_{n\in\mathbb{N}}$ be any strictly decreasing, recursive sequence of positive rational numbers such that $1 < p(0)-\epsilon_0$. Define $\hat{p} \colon \mathbb{N} \to (1,\infty)$ by setting
$
\hat{p}(n) = p(n) - \epsilon_n
$
for each $n \in \mathbb{N}$.
\begin{description}
\item[$\hat{p}$ is recursive.] Immediate.

\item[$\hat{p}$ is strictly increasing.] Because $\langle \epsilon_n\rangle_{n\in\mathbb{N}}$ is \emph{strictly} decreasing, for every $n \in \mathbb{N}$ we have
\begin{equation*}
\hat{p}(n) = p(n) - \epsilon_n < p(n+1) - \epsilon_{n+1} = \hat{p}(n+1).
\end{equation*}

\item[$\sum_{n=0}^\infty{\hat{p}(n)^{-1}}$ is a recursive real.] We start by observing that
\begin{equation*}
\frac{1}{\hat{p}(n)} = \frac{1}{p(n)} + \frac{1}{p(n)(\epsilon_n^{-1}p(n) - 1)}
\end{equation*}
and that
\begin{equation*}
\frac{1}{p(n)(\epsilon_n^{-1}p(n) - 1)} \leq \frac{1}{p(n)}
\end{equation*}
for all sufficiently large $n$. By \cref{recursive sum implies lower bounds have recursive sums}, the recursiveness of $\sum_{n=0}^\infty{p(n)^{-1}}$ implies that \\ $\sum_{n=0}^\infty{\left(p(n)(\epsilon_n^{-1}p(n) - 1)\right)^{-1}}$ is recursive. Thus, $\sum_{n=0}^\infty{\hat{p}(n)^{-1}}$ converges and
\begin{equation*}
\sum_{n=0}^\infty{\hat{p}(n)^{-1}} = \sum_{n=0}^\infty{p(n)^{-1}} + \sum_{n=0}^\infty{\left(p(n)(\epsilon_n^{-1}p(n) - 1)\right)^{-1}}
\end{equation*}
is recursive, being the sum of two recursive reals.

\item[$\hat{p} \domleq p$.] Immediate since $\epsilon_n > 0$ for each $n \in \mathbb{N}$.

\end{description}

\end{proof}

\begin{lem} \label{inverse function strict inequality equivalence}
For any order function $p\colon \mathbb{N} \to (1,\infty)$ and $n,m \in \mathbb{N}$, $p^\inverse(n) > m$ if and only if $p(m) < n$. 
\end{lem}
\begin{proof}
Straight-forward.
\end{proof}

\begin{proof}[Proof of \cref{complex hierarchy outpaces fast nice ldnr hierarchy quantified}.]
By \cref{strictly increasing lower bound with recursive sum}, we may assume without loss of generality that $p$ is strictly increasing. Moreover, the proof of \cref{strictly increasing lower bound with recursive sum} shows that the property that $\lim_{n \to \infty}{\frac{p(n)}{\tilde{p}(n+3)}} = \infty$ is preserved. (Note that \cref{fast-growing multiplicatively infinite lower bound} implies there is such a $\tilde{p}$.)

Let $\overline{p} \colon \co{0,\infty} \to \co{0,\infty}$ be the continuous extension of $p$ which is defined linearly on the intervals $[n,n+1]$ for $n \in \mathbb{N}$, so its inverse $\overline{p}^{-1} \colon \co{p(0),\infty} \to \co{0,\infty}$ exists.

Define $h$ and $f$ by setting $h(n) \coloneq \log_2 \tilde{p}(n)$ and $f(n) \coloneq \log_2 \tilde{p}(p^\inverse(2^{n+1})-1)$ for each $n \in \mathbb{N}$. Since $\sum_{n=0}^\infty{2^{-h(n)}} = \sum_{n=0}^\infty{\tilde{p}(n)^{-1}}$ is a recursive real, \cref{complex reals compute ldnr functions recursive sum} implies $\ldnr(\lambda n. \exp_2(f^\inverse(h(n))+1)) \strongleq \complex(f)$.
For almost all $n \in \mathbb{N}$, applying \cref{inverse function strict inequality equivalence} shows
\begin{align*}
f^\inverse(h(n)) = f^\inverse(\log_2 \tilde{p}(n)) & = \text{least $m$ such that $f(m) \geq \log_2 \tilde{p}(n)$} \\
& = \text{least $m$ such that $\log_2 \tilde{p}(p^\inverse(2^{m+1})-1) \geq \log_2 \tilde{p}(n)$} \\
& \leq \text{least $m$ such that $p^\inverse(2^{m+1}) -1 \geq n$} \\
& = \text{least $m$ such that $p^\inverse(2^{m+1}) > n$} \\
& = \text{least $m$ such that $p(n) > 2^{m+1}$} \\
& = \text{least $m$ such that $\log_2 p(n) - 1 > m$} \\
& \leq \lfloor \log_2 p(n) - 1 \rfloor +1 \\
& \leq \log_2 p(n).
\end{align*}
Thus, 
\begin{equation*}
\ldnr(p) \strongleq \ldnr(\lambda n. \exp_2(f^\inverse(h(n))+1)) \strongleq \complex(f).
\end{equation*} 

Let $\overline{\tilde{p}}$ be the continuous extension of $\tilde{p}$ which is defined linearly on the intervals $[n,n+1]$ for $n \in \mathbb{N}$. Then, observing that $p^\inverse(n) = \lceil \overline{p}^{-1}(n) \rceil$, we have
\begin{align*}
f(n) & = \log_2 \tilde{p}(p^\inverse(2^{n+1})-1) \\
& = \log_2 \tilde{p}(\lceil \overline{p}^{-1}(2^{n+1}) \rceil - 1) \\
& \leq \log_2 \overline{\tilde{p}}(\overline{p}^{-1}(2^{n+1})) \\
& = \log_2 \left( \overline{p}(\overline{p}^{-1}(2^{n+1})) \frac{\overline{\tilde{p}}(\overline{p}^{-1}(2^{n+1}))}{\overline{p}(\overline{p}^{-1}(2^{n+1}))}\right) \\
& = n + 1 - \log_2 \left(\frac{\overline{p}(\overline{p}^{-1}(2^{n+1}))}{\overline{\tilde{p}}(\overline{p}^{-1}(2^{n+1}))}\right)
\end{align*}
for all $n \in \mathbb{N}$. Because $p(n)/\tilde{p}(n) \nearrow \infty$ as $n \to \infty$, we additionally have $\overline{p}(x)/\overline{\tilde{p}}(x) \nearrow \infty$ as $x \to \infty$, so $\log_2 \left(\frac{\overline{\tilde{p}}(\overline{p}^{-1}(2^{n+1}))}{\overline{p}(\overline{p}^{-1}(2^{n+1}))}\right)-1$ is a nondecreasing, unbounded function of $n$, and hence the function $g \colon \mathbb{N} \to \co{0,\infty}$ defined by 
\begin{equation*}
g(n) \coloneq n + 1 - \log_2 \left(\frac{\overline{p}(\overline{p}^{-1}(2^{n+1}))}{\overline{\tilde{p}}(\overline{p}^{-1}(2^{n+1}))}\right)
\end{equation*}
is convex by \cref{equivalent characterization of convexity}. Since $g$ is also an order function, this completes the proof.

\end{proof}

\begin{proof}[Proof of \cref{complex hierarchy outpaces fast nice ldnr hierarchy}.]
By \cref{complex hierarchy outpaces fast nice ldnr hierarchy quantified}, since such a $\tilde{p}$ exists there is a convex sub-identical order function $g$ such that $\ldnr(p) \strongleq \complex(g)$.

It remains to show that $\complex(g) \neq \mlr$. \cite[Corollary 4.3.5]{hudelson2013partial} shows that there is an $X$ which is strongly $g$-complex (hence $X \in \complex(g)$) such that $\lim_{n \to \infty}{(\apc(X \restrict n) - g(n))} \neq \infty$. Suppose for the sake of a contradiction that $\complex(g) = \mlr$, so that $X$ is Martin-\Lof\ random. Then there is a $c \in \mathbb{N}$ such that $\apc(X \restrict n) \geq n - c$ for all $n \in \mathbb{N}$. $g$ being sub-identical means that $\lim_{n \to \infty}{(n - g(n))} = \infty$, so
\begin{equation*}
\liminf_n{(\apc(X \restrict n) - g(n))} \geq \liminf_n{(n - g(n) - c)} = \infty.
\end{equation*}
This implies $\lim_{n \to \infty}{(\apc(X \restrict n) - g(n))} = \infty$, a contradiction.
\end{proof}

\cref{examples of recursive sums} allows us to answer \cref{upward complex above ldnr} for a nice collection of fast-growing order functions which, at least in an aesthetic sense, approach the boundary between fast-growing and slow-growing: 

\begin{example}
Given $k \in \mathbb{N}$ and a rational $\alpha \in (1,\infty)$, take any natural number $\ell > k$ and any rational $\beta \in (1,\infty)$. Define $p$, $\tilde{p}$ by setting
\begin{align*}
p(n) & \coloneq n \cdot \log_2 n \mdots \log_2^{k-1} n \cdot (\log_2^k n)^\alpha, \\
\tilde{p}(m) & \coloneq m \cdot \log_2 m \mdots \log_2^{\ell-1} m \cdot (\log_2^\ell m)^\beta
\end{align*}
for $n \geq {}^k 2$ (otherwise set $p(n) = 2 \cdot {}^22\cdot {}^32 \mdots {}^k2$) and $m \geq {}^\ell 2$ (otherwise set $\tilde{p}(m) = 2 \cdot {}^22\cdot {}^32 \mdots {}^\ell 2$), respectively. 

\cref{examples of recursive sums} shows both series $\sum_{n=0}^\infty{p(n)^{-1}}$ and $\sum_{n=0}^\infty{\tilde{p}(n)^{-1}}$ converge to recursive reals. Moreover,
\begin{equation*}
\lim_{n \to \infty}{\frac{p(n)}{\tilde{p}(n+3)}} = \lim_{n \to \infty}{\frac{(\log_2^k n)^{\alpha -1}}{\log_2^{k+1}(n+3)  \cdot \log_2^{k+2}(n+3) \mdots \log_2^{\ell-1}(n+3) \cdot (\log_2^\ell(n+3))^\beta}} = \infty.
\end{equation*}
Then \cref{complex hierarchy outpaces fast nice ldnr hierarchy quantified} implies
\begin{equation*}
\ldnr(p) \strongleq \complex(\lambda n. \log_2 \tilde{p}(p^\inverse(2^{n+1})-1)).
\end{equation*}
\end{example}

Although \cref{complex hierarchy outpaces fast nice ldnr hierarchy} answers \cref{upward complex above ldnr} as stated for order functions $p \colon \mathbb{N} \to (1,\infty)$ with recursive sum $\sum_{n=0}^\infty{p(n)^{-1}}$, $g$ being convex sub-identical does not immediately imply that $\complex(g) \weakle \mlr$, though it is necessarily the case that $\complex(g) \neq \mlr$. This prompts a refinement of \cref{upward complex above ldnr}:

\begin{question}  \label{upward complex above ldnr and below mlr}
Given a fast-growing order function $p \colon \mathbb{N} \to (1,\infty)$, is there a sub-identical order function $g \colon \mathbb{N} \to \co{0,\infty}$ such that $\ldnr(p) \weakleq \complex(g) \weakle \mlr$?
\end{question}

By examining the form of $g$ in the proof of \cref{complex hierarchy outpaces fast nice ldnr hierarchy quantified}, we can give a sufficient condition on $p$ for there to be an affirmative answer to \cref{upward complex above ldnr and below mlr}.

\begin{lem} \label{bounded jumps}
Suppose $\alpha > 1$ and $p \colon \mathbb{N} \to (1,\infty)$ is a fast-growing order function. Then there exists a fast-growing order function $\hat{p} \colon \mathbb{N} \to (1,\infty)$ such that $\hat{p} \domleq p$ and $\hat{p}(n+1)/\hat{p}(n) \leq \alpha$ for all $n \in \mathbb{N}$. Moreover, if $\sum_{n=0}^\infty{p(n)^{-1}}$ is a recursive real, then $\hat{p}$ can be chosen so that $\sum_{n=0}^\infty{\hat{p}(n)^{-1}}$ is a recursive real as well.
\end{lem}
\begin{proof}
Without loss of generality we may assume $\alpha$ is rational. We define $\hat{p}$ recursively as follows:
\begin{align*}
\hat{p}(0) & \coloneq p(0), \\
\hat{p}(n+1) & \coloneq \min\{\alpha \hat{p}(n),p(n+1)\}.
\end{align*}
$\hat{p}$ is an order function dominated by $p$, so it just remains to show that $\hat{p}$ is fast-growing and that if $\sum_{n=0}^\infty{p(n)^{-1}}$ is a recursive real then $\sum_{n=0}^\infty{\hat{p}(n)^{-1}}$ is a recursive real.

Define $I \coloneq \{ n \in \mathbb{N} \mid \hat{p}(n) = p(n)\}$. Note that $I$ is a recursive nonempty subset of $\mathbb{N}$. We consider two cases:
\begin{description}
\item[Case 1: $I$ finite.] Let $n_0 = \max I$. By the definition of $\hat{p}$ we then have $\hat{p}(n) = \alpha^{n-n_0}p(n_0)$ for all $n \geq n_0$. Thus,
\begin{equation*}
\sum_{n=0}^\infty{\hat{p}(n)^{-1}} = \sum_{n=0}^{n_0-1}{\hat{p}(n)^{-1}} + \frac{1}{p(n_0)} \sum_{n=n_0}^\infty{\frac{1}{\alpha^{n-n_0}}} = \sum_{n=0}^{n_0-1}{\hat{p}(n)^{-1}} + \frac{1}{p(n_0)}\cdot \frac{\alpha}{\alpha-1} < \infty.
\end{equation*}
Moreover, we quickly see that $\sum_{n=0}^\infty{\hat{p}(n)^{-1}}$ is a recursive real. 

\item[Case 2: $I$ infinite.] Let $\langle n_k \rangle_{k \in \mathbb{N}}$ be the strictly increasing enumeration of $I$. Then
\begin{equation*}
\sum_{n=0}^\infty{\hat{p}(n)^{-1}} = \sum_{k=0}^\infty{\left(1+\frac{1}{\alpha} + \frac{1}{\alpha^2} + \cdots + \frac{1}{\alpha^{n_{k+1}-n_k-1}}\right)\frac{1}{p(n_k)}} \leq \frac{\alpha}{\alpha - 1}\sum_{k=0}^\infty{\frac{1}{p(n_k)}} \leq \frac{\alpha}{\alpha-1}\sum_{n=0}^\infty{p(n)^{-1}} < \infty.
\end{equation*}
Now suppose $\sum_{n=0}^\infty{p(n)^{-1}}$ is a recursive real, so that there is a nondecreasing recursive sequence $\langle N_m\rangle_{m \in \mathbb{N}}$ such that $\sum_{n=N_m}^\infty{p(n)^{-1}} \leq 2^{-m}$ for all $m \in \mathbb{N}$. Let $j$ be minimal such that $\frac{\alpha}{\alpha - 1} \leq 2^j$. We now define a recursive sequence $\langle k_m \rangle_{m \in \mathbb{N}}$ by setting $k_m$ to be the least $k$ such that $N_{m+j} \leq n_k$. Then
\begin{align*}
\sum_{n=n_{k_m}}^\infty{\hat{p}(n)^{-1}} & = \sum_{k=k_m}^\infty{\left(1+\frac{1}{\alpha} + \frac{1}{\alpha^2} + \cdots + \frac{1}{\alpha^{n_{k+1}-n_k-1}}\right)\frac{1}{p(n_k)}} \\
& \leq \frac{\alpha}{\alpha-1}\sum_{k=k_m}^\infty{\frac{1}{p(n_k)}} \\
& \leq \frac{\alpha}{\alpha-1}\sum_{n=n_{k_m}}^\infty{p(n)^{-1}} \\
& \leq \frac{\alpha}{\alpha-1}\sum_{n=N_{m+j}}^\infty{p(n)^{-1}} \\
& \leq \frac{\alpha}{\alpha-1} \frac{1}{2^{m+j}} \\
& \leq \frac{1}{2^m}.
\end{align*}
It follows that $\sum_{n=0}^\infty{\hat{p}(n)^{-1}}$ is recursive.
\end{description}
\end{proof}

\begin{prop} \label{upward complex above ldnr and below mlr partial answer}
Suppose $p \colon \mathbb{N} \to (1,\infty)$ is an order function. If there exists a computable, nondecreasing function $h \colon \co{1,\infty} \to (0,\infty)$ such that the series $\sum_{n=1}^\infty{\frac{1}{n h(n)}}$ and $\sum_{n=0}^\infty{\frac{h(p(n))}{p(n)}}$ converge to recursive reals and $\sup_{x \in \co{1,\infty}}{h(2x)/h(x)} < \infty$, then there exists a convex sub-identical order function $g \colon \mathbb{N} \to \co{0,\infty}$ such that $\ldnr(p) \strongleq \complex(g) \weakle \mlr$.
\end{prop}
\begin{proof}
Suppose there is such a computable, nondecreasing $h \colon \co{1,\infty} \to (0,\infty)$ and let $\tilde{p} \colon \mathbb{N} \to \mathbb{R}$ be defined by $\tilde{p}(n) \coloneq p(n)/h(p(n))$ for each $n \in \mathbb{N}$. By \cref{bounded jumps}, we may assume without loss of generality that $p(n+1)/p(n) \leq 2$ for all $n \in \mathbb{N}$. 

That $p$ and $h$ are nondecreasing and unbounded ($p$ by hypothesis, $h$ because $\sum_{n=1}^\infty{\frac{1}{n h(n)}} < \infty$) implies $p(n)/\tilde{p}(n) = h(p(n)) \nearrow \infty$ as $n \to \infty$. With this choice of $\tilde{p}$, let $g$ be as in the proof of \cref{complex hierarchy outpaces fast nice ldnr hierarchy quantified}, i.e.,
\begin{equation*}
g(n) \coloneq n + 1 - \log_2 \left(\frac{\overline{p}(\overline{p}^{-1}(2^{n+1}))}{\overline{\tilde{p}}(\overline{p}^{-1}(2^{n+1}))}\right)
\end{equation*}
for each $n \in \mathbb{N}$, so \cref{complex hierarchy outpaces fast nice ldnr hierarchy quantified} implies $\ldnr(p) \strongleq \complex(g)$.

It remains to show that $\complex(g) \weakle \mlr$. By \cite[Theorem 5.1]{hudelson2014mass}, it suffices to show that $\sum_{n=0}^\infty{2^{\lceil g(n) \rceil - n}}$ is a recursive real. Because $2^{\lceil g(n) \rceil - n} \leq 2^{g(n) - n + 1}$ for all $n \in \mathbb{N}$, \cref{recursive sum implies lower bounds have recursive sums} shows that it suffices to show that $\sum_{n=0}^\infty{2^{g(n) - n}}$ is a recursive real, or in other words that
\begin{equation*}
\sum_{n=0}^\infty{\frac{\overline{\tilde{p}}(\overline{p}^{-1}(2^{n+1}))}{\overline{p}(\overline{p}^{-1}(2^{n+1}))}} = \sum_{n=1}^\infty{2^n \frac{\overline{\tilde{p}}(\overline{p}^{-1}(2^n))}{(2^n)^2}}
\end{equation*}
is a recursive real. By the Cauchy Condensation Test, $\sum_{n=1}^\infty{2^n \frac{\overline{\tilde{p}}(\overline{p}^{-1}(2^n))}{(2^n)^2}}$ converges if and only if $\sum_{n=0}^\infty{\frac{\overline{\tilde{p}}(\overline{p}^{-1}(n))}{n^2}}$ converges; moreover, a simple analysis of the standard proof of the Cauchy Condesation Test and an appeal to \cref{recursive sum implies lower bounds have recursive sums} shows that we can replace both instances of `converges' with `converges to a recursive real' in the previous statement. Given $x \in \co{1,\infty}$, we compare $\overline{\tilde{p}}(x)$ to $\frac{\overline{p}(x)}{h(\overline{p}(x))}$. Let $n = \lfloor x \rfloor$; then
\begin{align*}
\frac{\overline{\tilde{p}}(x)}{\overline{p}(x)/h(\overline{p}(x))} & = h(\overline{p}(x)) \frac{\left(\frac{p(n+1)}{h(p(n+1))} - \frac{p(n)}{h(p(n))}\right)(x-n) + \frac{p(n)}{h(p(n))}}{(p(n+1) - p(n))(x-n) + p(n)} \\
& = \frac{h(\overline{p}(x))}{h(p(n))} \frac{\left( \frac{h(p(n))}{h(p(n+1))}p(n+1) - p(n)\right)(x-n) + p(n)}{(p(n+1) - p(n))(x-n) + p(n)} \\
& = \frac{h(\overline{p}(x))}{h(p(n))}\left( 1+ \frac{\left( \frac{h(p(n))}{h(p(n+1))}-1\right)p(n+1)(x-n)}{(p(n+1) - p(n))(x-n) + p(n)}\right) \\
& \leq \frac{h(p(n+1))}{h(p(n))}\left( 1 + \frac{p(n+1)}{p(n)}\right) \\
& \leq 3\frac{h(2p(n))}{h(p(n))} \\
& \leq 3 \sup_{y \in \co{1,\infty}}{\frac{h(2y)}{h(y)}} \\
& < \infty.
\end{align*}
Thus, for some positive rational $\alpha$ and almost all $n \in \mathbb{N}$, we have 
\begin{equation*}
\frac{\overline{\tilde{p}}(\overline{p}^{-1}(n))}{n^2} \leq \alpha \frac{\overline{p}(\overline{p}^{-1}(n))}{n^2 h(\overline{p}(\overline{p}^{-1}(n)))} = \alpha \frac{1}{n h(n)}.
\end{equation*}
Since $\sum_{n=1}^\infty{\frac{1}{n h(n)}}$ is a recursive real by hypothesis, \cref{complex hierarchy outpaces fast nice ldnr hierarchy quantified} implies $\sum_{n=1}^\infty{2^n \frac{\overline{\tilde{p}}(\overline{p}^{-1}(2^n))}{(2^n)^2}}$ is a recursive real, completing the proof.
\end{proof}

\begin{example}
Fix a rational $\epsilon > 0$. Then
\begin{equation*}
\ldnr(\lambda n. n (\log_2 n)^{2+\epsilon}) \weakleq \complex(g) \weakle \mlr
\end{equation*}
for some convex sub-identical order function $g \colon \mathbb{N} \to \co{0,\infty}$ as the hypotheses of \cref{upward complex above ldnr and below mlr partial answer} are satisfied with $p \colon \mathbb{N} \to (1,\infty)$ and $h \colon \co{1,\infty} \to (0,\infty)$ defined by
\begin{equation*}
p(n) \coloneq \begin{cases} 2 & \text{if $n=0$ or $n=1$,} \\ n (\log_2 n)^{2+\epsilon} & \text{otherwise,} \end{cases} \quad \text{and} \quad h(x) \coloneq \begin{cases} 1 & \text{if $x \in \co{1,2}$,} \\ (\log_2 x)^{1+\epsilon/2} & \text{otherwise.} \end{cases}
\end{equation*}
\end{example}

The extra hypothesis in \cref{complex hierarchy outpaces fast nice ldnr hierarchy} that $\sum_{n=0}^\infty{p(n)^{-1}}$ be a recursive real prompts the question as to whether a full affirmative answer to \cref{upward complex above ldnr} can be provided.

\begin{question} \label{recursive sum necessary for complex above ldnr question}
Does there exist a fast-growing order function $p$ such that $\sum_{n=0}^\infty{p(n)^{-1}}$ is nonrecursive and for which there is no sub-identical order function $g$ such that $\ldnr(p) \weakleq \complex(g)$?
\end{question}

\clearpage
\chapter{Complexity and Slow-Growing Avoidance}
\label{complexity and avoidance upward relationships chapter}

The results of \cref{complexity and avoidance downward relationships chapter} explore the relationships between the complexity and fast-growing $\ldnr$ hierarchies, giving full affirmative answers to \cref{downward ldnr below complex,,downward complex below ldnr} which examine the downward direction while in the upward direction we only gave a partial affirmative answer to \cref{upward complex above ldnr} and no answer to \cref{upward ldnr above complex}. In this chapter, we address the following generalization of \cref{upward ldnr above complex}.

\begin{question} \label{question ldnr above complex}
Given a sub-identical order function $f$, is there an order function $q$ such that $\complex(f) \weakleq \ldnr(q)$?
\end{question}

In particular, \cref{question ldnr above complex} drops the condition in \cref{upward ldnr above complex} that $q$ be fast-growing. Allowing $q$ to be slow-growing, we give a partial affirmative answer to \cref{question ldnr above complex} for $f$ of the form $\lambda n.n-\sqrt{n}\cdot\Delta(n)$ (and all sub-identical order functions dominated by a function of that form).

\begin{repthm}{sqrt complex from ldnr}
Given an order function $\Delta \colon \mathbb{N} \to [0,\infty)$ such that $\lim_{n \to \infty}{\Delta(n)/\sqrt{n}} = 0$ and any rational $\epsilon \in (0,1)$, 
\begin{equation*}
\complex\bigl(\lambda n. n-\sqrt{n}\cdot \Delta(n)\bigr) \weakleq \ldnr\bigl(\lambda n. \exp_2\bigl((1-\epsilon)\Delta(\log_2\log_2 n)\bigr)\bigr).
\end{equation*}
More generally, $\complex(\lambda n. n - \sqrt{n}\cdot \Delta(n)) \weakleq \ldnr(q)$ for any order function $q$ satisfying
\begin{equation*}
q\left( \exp_2((1-\epsilon)^{-1} \cdot [(n+1)^2 - (n+1) \cdot \Delta((n+1)^2)] \cdot \ell(n)) \right) \leq \ell(n)
\end{equation*}
for almost all $n \in \mathbb{N}$, where $\ell(n) = \exp_2\left((1-\epsilon)[(n+1) \cdot \Delta((n+1)^2) - n \cdot \Delta(n^2)]\right)$. 
\end{repthm}

Our approach in proving \cref{sqrt complex from ldnr} is motivated by techniques used by Greenberg \& Miller to prove a connection between the $\dnr$ hierarchy and effective Hausdorff dimension. The \textdef{effective Hausdorff dimension} of $Y \in \cantor$ is defined by $\dim(Y) \coloneq \sup \{ \delta \mid Y \in \complex(\delta)\}$. 

\begin{thm*} \label{greenberg and miller main theorem}
\textnormal{\cite[Theorem 4.9]{greenberg2011diagonally}}
For all sufficiently slow-growing order functions $q \colon \mathbb{N} \to (0,\infty)$ and all $Z \in \dnr_q$, there is $Y \in \cantor$ such that $Y \turingleq Z$ and $\dim(Y) = 1$.
\end{thm*}

The statement that $\dim(Y) = 1$ is equivalent to the statement that $Y \in \bigcap_{\delta \in (0,1) \cap \mathbb{Q}}{\complex(\delta)}$, so \cite[Theorem 4.9]{greenberg2011diagonally} gives a affirmative answer to \cref{question ldnr above complex} when $f(n) \leq \delta n$ for almost all $n$, where $\delta \in (0,1)$ is rational. 

\cref{sqrt complex from ldnr} improves \cite[Theorem 4.9]{greenberg2011diagonally} in two ways: first by strengthening $\dim(Y) = 1$ to $Y \in \complex(\lambda n. n-\sqrt{n}\cdot\Delta(n))$ and second by replacing `sufficiently slow-growing' with an explicit bound. 

One of the main ideas employed by Greenberg \& Miller is to consider partial randomness in the space $h^\mathbb{N} = \{ X \in \baire \mid \forall n \qspace (X(n) < h(n))\}$ for an order function $h \colon \mathbb{N} \to \mathbb{N}_{\geq 2}$, show that $Z \in \dnr_q$ computes an $X \in h^\mathbb{N}$ for which $\dim^h(X) = 1$ (with respect to effective Hausdorff dimenion in $h^\mathbb{N}$) and then show that $X$ computes a $Y \in \cantor$ with $\dim(Y) = 1$. The benefit of working with randomness in $h^\mathbb{N}$ instead of randomness in $\cantor$ is that when constructing $Y$ we may do so one entry at a time, whereas a direct construction of $X$ would likely require we construct it in segments whose lengths grow as the construction progresses. 

In \cref{partial randomness in hN}, we generalize our notions of partial randomness to $h^\mathbb{N}$. 
In \cref{randomness in hN to cantor space}, we give technical conditions under which partially random elements of $h^\mathbb{N}$ compute partially random elements of $\cantor$. 
In \cref{quantifying greenberg miller proof}, we examine \cite[Theorem 4.9]{greenberg2011diagonally}, using our generalizations and performing a careful analysis of the growth rates inherent to the construction to prove \cref{greenberg miller theorem 4.9 improved}:

\begin{repthm}{greenberg miller theorem 4.9 improved}
For rationals $\alpha \in (1,\infty)$ and $\beta \in (0,1/2)$, we have
\begin{equation*}
\complex(\lambda n. n - \alpha \sqrt{n}\log_2 n) \weakleq \ldnr(\lambda n. (\log_2 n)^\beta).
\end{equation*}
More generally, if $q \colon \mathbb{N} \to \mathbb{N}$ is an order function such that $q(2^{(3/2+\epsilon)n^2}) \leq n+1$ for almost all $n$ and some $\epsilon>0$, then $\complex(\lambda n. n - \alpha \sqrt{n}\log_2 n) \weakleq \ldnr(q)$. 
\end{repthm}

Finally, in \cref{quantifying greenberg miller proof general}, we prove a technical result (\cref{complex below ldnr}) which implies \cref{sqrt complex from ldnr}.

\section{Partial Randomness in \texorpdfstring{$h^\mathbb{N}$}{hN}} \label{partial randomness in hN}

The measure-theoretic structure on $h^\mathbb{N}$ can be defined similarly to that of the fair-coin measure $\lambda$ on $\cantor$:

\begin{definition}
Given a finite prefix-free $S \subseteq h^\ast$, define
\begin{equation*}
\mu_h(\bbracket{S}_h) \coloneq \sum_{\sigma \in S}{\frac{1}{|h^{|\sigma|}|}}.
\end{equation*}
The above assignment defines a premeasure on the collection of finite unions of basic open sets, so more generally we let $\mu_h$ be the outer measure induced by those assignments.
\end{definition}

\begin{convention}
When $h$ is understood, we write $\mu$ for $\mu_h$ and $\bbracket{-}$ for $\bbracket{-}_h$. Additionally, given $\sigma \in h^\ast$ or a finite prefix-free $S \subseteq h^\ast$, we write $\mu(\sigma)$ and $\mu(S)$ for $\mu(\bbracket{\sigma})$ and $\mu(\bbracket{S})$, respectively.
\end{convention}

Let $f \colon \{0,1\}^\ast \to \mathbb{R}$ be a computable function. Two versions of partial randomness in $h^\mathbb{N}$ will be relevant, \emph{$f$-randomness} and \emph{strong $f$-randomness}. 

A quantity that regularly appears is $\mu(\sigma)^{1/|\sigma|}$ for $\sigma \in h^\ast$. 

\begin{definition}
Define $\gamma \colon h^\ast \to [0,1]$ by setting $\gamma(\sigma) \coloneq \mu(\sigma)^{1/|\sigma|}$ when $\sigma \neq \langle\rangle$ and $\gamma(\langle\rangle) \coloneq 1$.
\end{definition}

\begin{remark}
One interpretation of $\gamma(\sigma)$ is as the geometric mean of the conditional probabilities \\ $\prob(X(k) = \sigma(k) \mid X \restrict k = \sigma \restrict k)$ for $k < |\sigma|$. 

Within $\{0,1\}^\ast$ we have $\gamma(\sigma) = 1/2$ for each $\sigma \in \{0,1\}^\ast \setminus \{\langle\rangle\}$.
\end{remark}

\subsection{$f$-randomness and $f$-complexity}

According to the measure-theoretic paradigm, $X \in \cantor$ is Martin-\Lof\ random if no uniformly r.e.\ sequence $\langle S_i \rangle_{i \in \mathbb{N}}$ of subsets of $\cantor$ for which $\lambda(S_i) \leq 2^{-i}$ for each $i \in \mathbb{N}$ covers $X$. A direct translation suggests defining $X \in h^\mathbb{N}$ to be Martin-\Lof\ random in $h^\mathbb{N}$ if no uniformly r.e.\ sequence $\langle S_i \rangle_{i \in \mathbb{N}}$ of subsets of $h^\mathbb{N}$ for which $\mu_h(S_i) \leq 2^{-i}$ for each $i \in \mathbb{N}$ covers $X$. 

More generally, given a recursive function $f \colon \{0,1\}^\ast \to \mathbb{R}$, $X \in \cantor$ is $f$-random if whenever $\langle S_i \rangle_{i \in \mathbb{N}}$ is a uniformly r.e.\ sequence of subsets of $\{0,1\}^\ast$ such that $\sum_{\sigma \in S_i}{2^{-f(\sigma)}} \leq 2^{-i}$ for each $i \in \mathbb{N}$, then $X \notin \bigcap_{i \in \mathbb{N}}{\bbracket{\sigma}_2}$. It is less obvious how to translate this to the realm of $h^\mathbb{N}$. One way would be to use it verbatim (though with $\dom f = h^\ast$ now), but a consequence would be that Martin-\Lof\ randomness in $h^\mathbb{N}$ would correspond to $f(\sigma) \coloneq \log_2 |h^{|\sigma|}|$. Another approach makes use of the observation that $\frac{1}{2} = \lambda(\sigma)^{1/|\sigma|}$ for any $\sigma \in \{0,1\}^\ast \setminus \{\langle\rangle\}$.

\begin{definition}[direct $f$-weight in $h^\ast$]
Suppose $S \subseteq h^\ast$. The \textdef{direct $f$-weight (in $h^\ast$)} $\dwt_f(S)$ of $S$ is defined by
\begin{equation*}
\dwt_f(S) = \dwt_f^h(S) \coloneq \sum_{\sigma \in S}{\gamma(\sigma)^{f(\sigma)}}.
\end{equation*}

\end{definition}

\begin{definition}[$f$-randomness in $h^\mathbb{N}$]
An \textdef{$f$-ML-test (in $h^\mathbb{N}$)} is a uniformly r.e.\ sequence $\langle S_k \rangle_{k \in \mathbb{N}}$ of subsets $S_k \subseteq h^\ast$ such that $\dwt_f(S_k) \leq 2^{-k}$ for all $k \in \mathbb{N}$.

An $f$-ML test $\langle S_k\rangle_{k \in \mathbb{N}}$ \textdef{covers} $X \in h^\mathbb{N}$ if $X \in \bigcap_{k \in \mathbb{N}}{\bbracket{S_k}}$. If $X$ is covered by an $f$-ML-test, then $X$ is said to be \textdef{$f$-null (in $h^\mathbb{N}$)}, and otherwise is \textdef{$f$-random (in $h^\mathbb{N}$)}. 
\end{definition}

Like in $\cantor$, there is an equivalent characterization of $f$-randomness in terms of complexity. The definition of a prefix-free machine $M \colonsub \{0,1\}^\ast \to h^\ast$ in $h^\ast$ is analogous to that of a prefix-free machine $M \colonsub \{0,1\}^\ast \to \{0,1\}^\ast$.

\begin{definition}[prefix-free machines in $h^\ast$]
A \textdef{prefix-free machine} is a partial recursive function $M \colonsub \{0,1\}^\ast \to h^\ast$ such that $\dom M$ is prefix-free. 

A prefix-free machine $U$ is \textdef{universal} if for any prefix-free machine $M$ there exists $\rho \in \{0,1\}^\ast$ such that $U(\rho\concat \tau) \simeq M(\tau)$ for all $\tau \in \{0,1\}^\ast$.
\end{definition}

\begin{prop} \label{universal prefix-free machine exists}
There exists a universal prefix-free machine $U \colonsub \{0,1\}^\ast \to h^\ast$. 
\end{prop}
\begin{proof}
The proof given in \cite[Theorem 6.2.3]{simpson2009computability} easily generalizes to $h^\ast$.
\end{proof}

\begin{definition}[prefix-free complexity in $h^\ast$]
Fix a universal prefix-free machine $U$. Then the \\ \textdef{prefix-free complexity of $\sigma \in h^\ast$} is defined by
\begin{equation*}
\pfc(\sigma) = \pfc_U^h(\sigma) \coloneq \min \{ |\tau| \mid U(\tau) \simeq \sigma\}.
\end{equation*}
\end{definition}

Many of the standard properties or facts about prefix-free complexity in $\{0,1\}^\ast$ continue to hold in $h^\ast$.

\begin{prop} \label{kraft's inequality} \label{kc theorem}
\mbox{}
\begin{enumerate}[(a)]
\item If $U$ and $V$ are universal prefix-free machines, then there exists $c \in \mathbb{N}$ such that $|\pfc_U(\sigma) - \pfc_V(\sigma)| \leq c$ for all $\sigma \in h^\ast$.
\item Kraft's Inequality: $\sum_{\sigma \in h^\ast}{2^{-\pfc(\sigma)}} \leq 1$.
\item KC Theorem: Suppose $\langle d_k, \sigma_k \rangle_{k \in \mathbb{N}}$ is a recursive sequence of pairs $\langle d_k, \sigma_k\rangle \in \mathbb{N} \times h^\ast$ such that $\sum_{k=0}^\infty{2^{-d_k}} \leq 1$. Then there is a recursive sequence $\langle \tau_k \rangle_{k \in \mathbb{N}}$ of binary strings such that $|\tau_k| = d_k$. 

Consequently, there exists $c \in \mathbb{N}$ such that $\pfc(\sigma_k) \leq d_k + c$ for all $k \in \mathbb{N}$.
\end{enumerate}
\end{prop}
\begin{proof} \mbox{}
\begin{enumerate}[(a)]
\item The universality of $U$ implies there is a $\rho \in \{0,1\}^\ast$ such that $U(\rho \concat \sigma) \simeq V(\sigma)$ for all $\sigma \in \{0,1\}^\ast$. Fix $\tau \in h^\ast$. If $\sigma \in \{0,1\}^\ast$ is such $V(\sigma) \simeq \tau$ and $|\sigma| = \pfc_V(\tau)$, then $U(\rho \concat \sigma) \simeq \tau$ shows $\pfc_U(\tau) \leq \pfc_V(\tau) + |\rho|$, or equivalently $\pfc_U(\tau) - \pfc_V(\tau) \leq |\rho|$, with this final inequality being independent of $\tau$. By symmetry, there is a $\rho' \in \{0,1\}^\ast$ such that $\pfc_V(\tau) - \pfc_U(\tau) \leq |\rho'|$ for all $\tau \in h^\ast$, so we may set $c = \max\{|\rho|,|\rho'|\}$.

\item To each $\tau \in h^\ast$ there is a (not necessarily unique) $\sigma \in \dom U$ such that $U(\sigma) = \tau$ and $|\sigma| = \pwt(\tau)$. This observation shows that $\sum_{\tau \in h^\ast}{2^{-\pfc(\tau)}} \leq \sum_{\sigma \in \dom U}{2^{-|\sigma|}}$. Because $\dom U$ is prefix-free, we have 
\begin{equation*}
\sum_{\tau \in h^\ast}{2^{-\pfc(\tau)}} \leq \sum_{\sigma \in \dom U}{2^{-|\sigma|}} \leq 1.
\end{equation*}

\item The proof of \cite[Theorem 3.6.1]{downey2010algorithmic} shows that there exists a recursive sequence $\langle \tau_k \rangle_{k \in \mathbb{N}}$ of pairwise-incompatible binary strings $\tau_k$ with $|\tau_k| = d_k$. 

To show the ``consequently'' statement holds, define $M \colonsub \{0,1\}^\ast \to h^\ast$ by setting $M(\tau_k) = \sigma_k$ for $k \in \mathbb{N}$ and $M(\tau) \diverge$ for all other $\tau$. Then $ M$ is a prefix-free machine, so there is $\rho \in \{0,1\}^\ast$ such that $U(\rho \concat \tau) \simeq M(\tau)$ for all $\tau \in \{0,1\}^\ast$. In particular, $U(\rho \concat \tau_k) = \sigma_k$, so 
\begin{equation*}
\pfc(\sigma_k) \leq |\rho \concat \tau_k| = d_k + |\rho|.
\end{equation*}

\end{enumerate}
\end{proof}

Given $X \in h^\mathbb{N}$, it makes sense to consider how the prefix-free complexity of an initial segment of $X$ grows as a function of length. Within $\cantor$, $X$ is $f$-complex if there is $c \in \mathbb{N}$ such that $\pfc(X \restrict n) \geq f(n) - c$ for all $n \in \mathbb{N}$. It can be shown \cite[Theorem 2.6]{higuchi2014propagation} that $f$-randomness and $f$-complexity in $\cantor$ are equivalent, so a natural question is whether this continues to hold in $h^\mathbb{N}$ once we define `$f$-complexity' in $h^\mathbb{N}$. For the equivalence to go through an additional factor (depending on $h$) must be introduced.

\begin{definition}[$f$-complexity in $h^\mathbb{N}$]
$X \in h^\mathbb{N}$ is \textdef{$f$-complex (in $h^\mathbb{N}$)} if there exists $c \in \mathbb{N}$ such that, for all $n \in \mathbb{N}$,
\begin{equation*}
\pfc(X \restrict n) \geq (\log_{1/2}\gamma(X \restrict n)) \cdot f(X \restrict n) - c.
\end{equation*}

\end{definition}

Adapting \cite[Theorem 2.6]{higuchi2014propagation} yields the equivalence between $f$-randomness and $f$-complexity in $h^\mathbb{N}$. 

\begin{thm} \label{equivalence of f-randomness and f-complexity}
For all $X \in h^\mathbb{N}$, $X$ is $f$-random in $h^\mathbb{N}$ if and only if it is $f$-complex in $h^\mathbb{N}$.
\end{thm}
\begin{proof}
For $i \in \mathbb{N}$, let $S_i = \{ \sigma \in h^\ast \mid \pfc(\sigma) < (\log_{1/2}\gamma(X \restrict n)) \cdot f(\sigma) - i\}$. We claim that $\langle S_i \rangle_{i \in \mathbb{N}}$ forms an $f$-ML test. Indeed, for each $i \in \mathbb{N}$ we have
\begin{equation*}
\dwt_f(S_i) = \sum_{\sigma \in S_i}{\gamma(\sigma)^{f(\sigma)}} < \sum_{\sigma \in S_i}{\gamma(\sigma)^{(\pfc(\sigma) + i) \cdot \log_{\gamma(\sigma)}(1/2)}} = \sum_{\sigma \in S_i}{2^{-\pfc(\sigma)} \cdot 2^{-i}} \leq 2^{-i}.
\end{equation*}
where the final inequality follows from \cref{kraft's inequality}(b). If $X$ is $f$-random, then $\langle S_i \rangle_{i \in \mathbb{N}}$ does not cover $X$, meaning there is an $i \in \mathbb{N}$ such that $\pfc(X \restrict n) \geq (\log_{1/2}\gamma(X \restrict n)) \cdot f(X \restrict n) - i$ for all $n \in \mathbb{N}$ and hence $X$ is $f$-complex.

Conversely, suppose $X$ is not $f$-random, and let $\langle S_i \rangle_{i \in \mathbb{N}}$ be an $f$-ML test covering $X$. Then
\begin{equation*}
\sum_{i = 0}^\infty{\sum_{\sigma \in S_{2i}}{\exp_2\bigl(-\bigl( (\log_{1/2} \gamma(\sigma)) \cdot f(\sigma) - (i+1)\bigl)\bigr)}} = \sum_{i = 0}^\infty{2^i\gamma(\sigma)^{f(\sigma)}} \leq \sum_{i=0}^\infty{2^{i+1} \cdot 2^{-2i}} = \sum_{i=0}^\infty{2^{-i-1}} = 1.
\end{equation*}
Suppose $g_i \colon \mathbb{N} \to S_{2i}$ is a recursive surjection for each $i \in \mathbb{N}$ and let $g \colon \mathbb{N} \to \bigcup_{i \in \mathbb{N}}{S_{2i}}$ be defined by $g(\pi^{(2)}(i,j)) \coloneq g_i(j)$. Write $\sigma_k = g(k)$ and $d_k = \lceil (\log_{1/2} \gamma(\sigma_k)) \cdot f(\sigma_k) - (i+1) \rceil$. Then \cref{kc theorem}(c) implies there exists $c \in \mathbb{N}$ such that 
\begin{equation*}
\pfc(\sigma_k) \leq \lceil (\log_{1/2} \gamma(\sigma_k)) \cdot f(\sigma_k) - (i+1) \rceil + c \leq (\log_{1/2} \gamma(\sigma_k)) \cdot f(\sigma_k) - i + c
\end{equation*}
for all $k \in \mathbb{N}$. Because $\langle S_i \rangle_{i \in \mathbb{N}}$ covers $X$, $\langle S_{2i} \rangle_{i \in \mathbb{N}}$ does as well. Thus, for every $i \in \mathbb{N}$, there exists an $n \in \mathbb{N}$ such that $\pfc(X \restrict n) \leq (\log_{1/2} \gamma(X \restrict n)) \cdot f(X \restrict n) - i + c$, so $X$ is not $f$-complex.
\end{proof}

\begin{cor}
There exists a \textdef{universal $f$-ML test}, i.e., an $f$-ML test $\langle S_i \rangle_{i \in \mathbb{N}}$ such that $X \in h^\mathbb{N}$ is $f$-random if and only if $X$ is not covered by $\langle S_i \rangle_{i \in \mathbb{N}}$.
\end{cor}
\begin{proof}
The proof of \cref{equivalence of f-randomness and f-complexity} shows that setting $S_i = \{ \sigma \in h^\ast \mid \pfc(\sigma) < (\log_{1/2}\gamma(X \restrict n)) \cdot f(\sigma) - i\}$ yields a universal $f$-ML test.
\end{proof}

\begin{cor}
Suppose $f(\sigma) = \tilde{f}(\sigma)$ for almost all $\sigma$. Then $f$-randomness in $h^\mathbb{N}$ is equivalent to $\tilde{f}$-randomness in $h^\mathbb{N}$.
\end{cor}
\begin{proof}
Suppose $f(\sigma) = \tilde{f}(\sigma)$ for all $\sigma \in h^\ast$ for which $|\sigma| > N$. Let $d = \max_{\sigma \in h^\ast, |\sigma| \leq N}{\pfc(\sigma)}$. Then $\pfc(\sigma) \geq (\log_{1/2} \gamma(\sigma)) \cdot f(\sigma) - c$ if and only if $\pfc(\sigma) \geq (\log_{1/2} \gamma(\sigma)) \cdot \tilde{f}(\sigma) - c$ for all $\sigma \in h^\ast$. It follows that $f$-complexity and $\tilde{f}$-complexity are equivalent, and so \cref{equivalence of f-randomness and f-complexity} shows $f$-randomness and $\tilde{f}$-randomness are equivalent.
\end{proof}

Prior to defining $f$-randomness in $h^\mathbb{N}$, an alternate definition was suggested in which the definition of direct $f$-weight was unmodified when passing from $\{0,1\}^\ast$ to $h^\ast$ aside from changing the domain of $f$. Likewise, an alternative definition of $f$-complexity can be given in which the factor $\log_{1/2}{\gamma(\sigma)}$ is removed, more closely resembling $f$-complexity in $\cantor$. 

Passing between these alternative definitions can be done in a uniform manner:

\begin{prop}
Suppose $f,g \colon h^\ast \to \mathbb{R}$ are computable and $g(\sigma) = (\log_{1/2} \gamma(\sigma)) \cdot f(\sigma)$ for all $\sigma \in h^\ast$. 
\begin{enumerate}[(a)]
\item $X \in h^\mathbb{N}$ is $f$-random if and only if there exists no uniformly r.e.\ sequence $\langle S_k \rangle_{k \in \mathbb{N}}$ such that $\sum_{\sigma \in S_k}{2^{-g(\sigma)}} \leq 1/2^k$ for each $k \in \mathbb{N}$ and $X \in \bigcap_{k \in \mathbb{N}}{\bbracket{S_k}}$.

\item $X \in h^\mathbb{N}$ is $f$-complex if and only if there exists a $c \in \mathbb{N}$ such that $\pfc(X \restrict n) \geq g(X \restrict n) - c$ for all $n \in \mathbb{N}$.
\end{enumerate}
\end{prop}
\begin{proof} 
Straight-forward.
\end{proof}

\subsection{Strong $f$-Randomness}

A related variant of partial randomness can be defined similarly. First, from the measure-theoretic paradigm:

\begin{definition}[prefix-free $f$-weight in $h^\ast$]
The \textdef{prefix-free $f$-weight} of a set of strings $S \subseteq h^\ast$ is defined by
\begin{equation*}
\pwt_f(S) \coloneq \sup\{ \dwt_f(A) \mid \text{prefix-free $A \subseteq S$} \}.
\end{equation*}
\end{definition}

\begin{definition}[strong $f$-randomness in $h^\mathbb{N}$]
A \textdef{weak $f$-ML-test (in $h^\mathbb{N}$)} is a uniformly r.e.\ sequence $\langle A_k \rangle_{k \in \mathbb{N}}$ of subsetes $A_k \subseteq h^\ast$ such that $\pwt_f(A_k) \leq 2^{-k}$ for all $k \in \mathbb{N}$. 

A weak $f$-ML-test $\langle A_k\rangle_{k \in \mathbb{N}}$ \textdef{covers} $X \in h^\mathbb{N}$ if $X \in \bigcap_{k \in \mathbb{N}}{\bbracket{A_k}}$. If $X$ is covered by no weak $f$-ML-test, then $X$ is \textdef{strongly $f$-random (in $h^\mathbb{N}$)}.
\end{definition}

Like in $\cantor$, strong $f$-randomness in $h^\mathbb{N}$ is associated with an analog of a priori complexity in $h^\ast$.

\begin{definition}[continuous semimeasure on $h^\ast$]
A \textdef{continuous semimeasure on $h^\ast$} is a function $\nu \colon h^\ast \to [0,1]$ such that $\nu(\langle\rangle)=1$ and for every $\sigma \in h^\ast$,
\begin{equation*}
\nu(\sigma) \geq \sum_{i \in < h(|\sigma|)}{\nu(\sigma\concat\langle i\rangle)}.
\end{equation*}

A continuous semimeasure $\nu$ is \textdef{left recursively enumerable}, or \textdef{left r.e.}, if it is left r.e.\ as a function $h^\ast \to \mathbb{R}$. A left r.e.\ continuous semimeasure $\nu$ is \textdef{universal} if for every left r.e.\ continuous semimeasure $\xi$ on $h^\ast$ there exists $c \in \mathbb{N}$ such that $\xi(\sigma) \leq c\cdot \nu(\sigma)$ for all $\sigma \in h^\ast$.
\end{definition}

\begin{prop}
There exists a universal left r.e.\ semimeasure $\mathbf{M}$ on $h^\ast$.
\end{prop}
\begin{proof}
The proof given in \cite[Theorem 3.16.2]{downey2010algorithmic} easily generalizes to $h^\ast$.
\end{proof}

\begin{definition}[a priori complexity in $h^\ast$]
Fix a universal left r.e.\ semimeasure $\mathbf{M}$. The \textdef{a priori complexity} of a string $\sigma \in h^\ast$ is defined by
\begin{equation*}
\apc(\sigma) = \apc_\mathbf{M}(\sigma) \coloneq -\log_2 \mathbf{M}(\sigma).
\end{equation*}
\end{definition}

Akin to the well-definedness of prefix-free complexity, if $\mathbf{N}$ were another universal left r.e.\ semimeasure, then $\apc_\mathbf{M}$ and $\apc_\mathbf{N}$ differ by at most a constant.

\begin{definition}[strong $f$-complexity in $h^\mathbb{N}$]
$X \in h^\mathbb{N}$ is \textdef{strongly $f$-complex (in $h^\mathbb{N}$)} if there exists a $c \in \mathbb{N}$ such that, for all $n \in \mathbb{N}$, 
\begin{equation*}
\apc(X \restrict n) \geq (\log_{1/2} \gamma(X \restrict n)) \cdot f(X \restrict n) - c.
\end{equation*}
\end{definition}

We will show that strong $f$-complexity is equivalent to strong $f$-randomness. Before doing so, we introduce a third approach to defining strong $f$-randomness/complexity, this time in terms of supermartingales as in the unpredictability paradigm. 

\begin{definition}[supermartingale]
A \textdef{supermartingale (over $h^\ast$)} is a function $d \colon h^\ast \to \co{0,\infty}$ such that
\begin{equation*}
\sum_{i < h(|\sigma|)}{d(\sigma \concat \langle i \rangle)} \leq h(|\sigma|)d(\sigma)
\end{equation*}
for all $\sigma \in h^\ast$. 

A supermartinagle $d$ is \textdef{left recursively enumerable}, or \textdef{left r.e.}, if it is left r.e.\ as a function $h^\ast \to \co{0,\infty}$.
\end{definition}

\begin{definition}[$f$-success]
Suppose $d$ is a left r.e.\ supermartingale and $X \in h^\mathbb{N}$. $d$ is said to \textdef{$f$-succeed on $X$} if
\begin{equation*}
\limsup_n{\left( d(X \restrict n) \cdot \gamma(X \restrict n)^{n - f(X \restrict n)} \right)} = \infty.
\end{equation*}
\end{definition}

The following lemma reveals the close connection between continuous semimeasures $\nu$ and supermartingales $d$ such that $d(\langle\rangle) = 1$.

\begin{lem} \label{relationship between supermartingales and continuous semimeasures}
Given $\nu \colon h^\ast \to [0,1]$, let $d_\nu \colon h^\ast \to \co{0,\infty}$ be defined by $d_\nu(\sigma) \coloneq |h^{|\sigma|}| \cdot \nu(\sigma)$ for $\sigma \in h^\ast$.
\begin{enumerate}[(a)]
\item $\nu$ is left r.e.\ if and only if $d_\nu$ is left r.e.
\item $\nu$ is a continuous semimeasure if and only if $d_\nu$ is a supermartingale.
\item $\nu$ is a universal left r.e.\ continuous semimeasure if and only if $d_\nu$ is a \textdef{universal left r.e.\ supermartinagle}, in the sense that if $d$ were another left r.e.\ supermartingale then there is a $c \in \mathbb{N}$ such that $d(\sigma) \leq c \cdot d_\nu(\sigma)$ for all $\sigma \in h^\ast$.
\end{enumerate}
\end{lem}
\begin{proof} \mbox{}
\begin{enumerate}[(a)]
\item This follows from the fact that $h$ is recursive.
\item Given $\sigma \in h^\ast$, we have
\begin{equation*}
\sum_{i < h(|\sigma|)}{d_\nu(\sigma \concat \langle i \rangle)} = \sum_{i < h(|\sigma|)}{|h^{|\sigma\concat \langle i \rangle|}| \cdot \nu(\sigma \concat \langle i \rangle)} = |h^{|\sigma|+1}| \cdot \sum_{i < h(|\sigma|)}{\nu(\sigma \concat \langle i \rangle)} = h(|\sigma|) \cdot \left( |h^{|\sigma|}| \cdot \sum_{i < h(|\sigma|)}{\nu(\sigma \concat \langle i \rangle)}\right).
\end{equation*}
By comparing the first and last expressions with the definitions of what it means for $d_\nu$ to be a supermartingale or for $\nu$ to be a continuous semimeasure shows that $d_\nu$ is a supermartinagle if and only if $\nu$ is a continuous semimeasure.
\item Straight-forward.
\end{enumerate}
\end{proof}

\begin{lem} \label{pwt inequalities}
Suppose $S$ and $T$ are subsets of $h^\ast$.
\begin{enumerate}[(a)]
\item If $S \subseteq T$, then $\dwt_f(S) \leq \dwt_f(T)$ and $\pwt_f(S) \leq \pwt_f(T)$.
\item $\dwt_f(S \cup T) = \dwt_f(S) + \dwt_f(T) - \dwt_f(S \cap T)$.
\item $\pwt_f(S \cup T) \leq \pwt_f(S) + \pwt_f(T)$, with equality if the strings in $S$ and $T$ are pairwise incompatible.
\end{enumerate}
\end{lem}
\begin{proof} \mbox{}
\begin{enumerate}[(a)]
\item Straight-forward.
\item Straight-forward.
\item If $P \subseteq S \cup T$ is prefix-free, then $P \cap S$ and $P \cap T$ are prefix-free subsets of $S$ and $T$, respectively, so
\begin{equation*}
\dwt_f(P) \leq \dwt_f(P \cap S) + \dwt_f(P \cap T) \leq \pwt_f(S) + \pwt_f(T).
\end{equation*}
Taking the supremum among all prefix-free $P \subseteq S \cup T$ yields $\pwt_f(S \cup T) \leq \pwt_f(S) + \pwt_f(T)$.

If the strings in $S$ and $T$ are pairwise incompatible, then given prefix-free subsets $A \subseteq S$ and $B \subseteq T$, $A \cap B = \emptyset$ and $A \cup B$ is a prefix-free subset of $S \cup T$, so
\begin{equation*}
\dwt_f(A) + \dwt_f(B) = \dwt_f(A \cup B) \leq \pwt_f(S \cup T).
\end{equation*}
Taking the supremum among all prefix-free $A \subseteq S$ and $B \subseteq T$ yields $\pwt_f(S) + \pwt_f(T) \leq \pwt_f(S \cup T)$.
\end{enumerate}
\end{proof}

\begin{thm} \label{equivalent characterizations of strong f-randomness}
Suppose $X \in h^\mathbb{N}$. The following are equivalent.
\begin{enumerate}[(i)]
\item $X$ is strongly $f$-random.
\item $X$ is strongly $f$-complex.
\item $d^h$ does not $f$-succeed on $X$, where $d^h$ is the universal left r.e.\ supermartingale corresponding to $\mathbf{M}$ as in \cref{relationship between supermartingales and continuous semimeasures}.
\item No left r.e.\ supermartingale $f$-succeeds on $X$.
\end{enumerate}
\end{thm}
\begin{proof} \mbox{}
\begin{description}
\item[$(i) \iff (ii)$] Suppose $X$ is strongly $f$-random. Let $S_i = \{ \sigma \in h^\ast \mid \apc(\sigma) < (\log_{1/2}\gamma(\sigma)) \cdot f(\sigma) - i\}$. If $P \subseteq S_i$ is prefix-free, then
\begin{equation*}
\dwt_f(P) = \sum_{\sigma \in P}{\gamma(\sigma)^{f(\sigma)}} \leq \sum_{\sigma \in P}{\gamma(\sigma)^{(\apc(\sigma) + i) \cdot (\log_{\gamma(\sigma)}1/2)}} \leq \frac{1}{2^i} \sum_{\sigma \in P}{\mathbf{M}(\sigma)} \leq \frac{1}{2^i} \mathbf{M}(\langle\rangle) \leq \frac{1}{2^i}.
\end{equation*}
Thus, $\langle S_i \rangle_{i \in \mathbb{N}}$ forms a weak $f$-ML test. Because $X$ is strongly $f$-random, $X$ is not covered by $\langle S_i \rangle_{i \in \mathbb{N}}$ and so there is an $i \in \mathbb{N}$ such that $X \notin \bbracket{S_i}$, i.e., for every $n \in \mathbb{N}$ we have $\apc(X \restrict n) \geq (\log_{1/2} \gamma(X \restrict n)) \cdot f(X \restrict n) - i$, so $X$ is strongly $f$-complex.

If $X$ is not strongly $f$-random, then there is a weak $f$-ML test $\langle S_i \rangle_{i \in \mathbb{N}}$ which covers $X$. Uniformly in $i \in \mathbb{N}$, we let $\nu_i$ be defined by $\nu_i(\sigma) = \pwt_f(\{ \tau \in S_i \mid \tau \supseteq \sigma\})$. $\nu_i$ is a continuous semimeasure; using \cref{pwt inequalities} we have
\begin{align*}
\nu_i(\sigma) & = \pwt_f(\{ \tau \in S_i \mid \tau \supseteq \sigma\}) \\
& \geq \pwt_f(\{ \tau \in S_i \mid \tau \supset \sigma\}) \\
& = \pwt_f\left(\bigcup_{j < h(|\sigma|)}{\{ \tau \in S_i \mid \tau \supseteq \sigma\concat\langle j \rangle\}}\right) \\
& = \sum_{j < h(|\sigma|)}{\pwt_f(\{ \tau \in S_i \mid \tau \supseteq \sigma\concat\langle j\rangle\})} \\
& = \sum_{j < h(|\sigma|)}{\nu_i(\sigma \concat \langle j\rangle)}.
\end{align*}
That $\nu_i$ is left r.e.\ follows from the fact that $S_i$ is r.e.
Observe that for $\tau \in S_i$ we have $\dwt_f(\tau) \leq \nu_i(\tau)$. Because $\nu_i(\langle\rangle) = \pwt_f(S_i) \leq 2^{-i}$ for each $i$, the map $\overline{\nu} \colon h^\ast \to [0,1]$ defined by
\begin{equation*}
\overline{\nu}(\sigma) \coloneq \sum_{i=0}^\infty{2^i \nu_{2i}(\sigma)}
\end{equation*}
for $\sigma \in h^\ast$ is a left r.e.\ semimeasure, and hence there is a $c \in \mathbb{N}$ such that $\overline{\nu}(\sigma) < c \cdot \mathbf{M}(\sigma)$ for all $\sigma \in h^\ast$. Then for $\sigma \in S_{2i}$, we have
\begin{equation*}
\exp_2(i - (\log_{1/2} \gamma(\sigma)) \cdot f(\sigma)) = 2^i \dwt_f(\sigma) \leq 2^i \nu_{2i}(\sigma) \leq \overline{\nu}(\sigma) < c \cdot \mathbf{M}(\sigma) = \exp_2(-( \apc(\sigma) + \log_{1/2} c))
\end{equation*}
and hence
\begin{equation*}
\apc(\sigma) + i + \log_{1/2} c < (\log_{1/2} \gamma(\sigma)) \cdot f(\sigma).
\end{equation*}
Being covered by $\langle S_i \rangle_{i \in \mathbb{N}}$ and hence by $\langle S_{2i} \rangle_{i \in \mathbb{N}}$ as well, $X$ is not strongly $f$-complex.

\item[$(ii) \iff (iii)$] Let $d^h$ be the universal left r.e.\ supermartingale corresponding to $\mathbf{M}$, as in \cref{relationship between supermartingales and continuous semimeasures}. Now observe that for any $X \in h^\mathbb{N}$ and $n \in \mathbb{N}$,
\begin{align*}
d^h(X \restrict n) \cdot \mu(X \restrict n)^{1 - f(X \restrict n)/n} & = \mathbf{M}(X \restrict n) \cdot \mu(X \restrict n)^{-1} \cdot \mu(X \restrict n)^{1-f(X \restrict n)/n} \\
& = \exp_2(-( \apc(X\restrict n) - (\log_{1/2} \gamma(X \restrict n)) \cdot f(X \restrict n))).
\end{align*}
Thus,
\begin{equation*}
\limsup_n{d_0(X \restrict n) \cdot \mu(X \restrict n)^{1-f(X \restrict n)/n}} = \infty \iff \forall c \exists n \qspace (\apc(X \restrict n) < (\log_{1/2} \gamma(X \restrict n)) \cdot f(X \restrict n) - c.
\end{equation*} 
In other words, $d^h$ $f$-succeeds on $X$ if and only if $X$ is not strongly $f$-complex.

\item[$(iii) \iff (iv)$] If no left r.e.\ supermartingale $f$-succeeds on $X$, then in particular $d^h$ does not $f$-succeed on $X$. Conversely, if $d^h$ does not $f$-succeed on $X$, then the universality of $d^h$ shows that no left r.e.\ supermartingale $f$-succeeds on $X$. 
\end{description}
\end{proof}

\begin{cor}
There exists a \textdef{universal weak $f$-ML test}, i.e., a weak $f$-ML test $\langle S_i \rangle_{i \in \mathbb{N}}$ such that $X \in h^\mathbb{N}$ is strongly $f$-random in $h^\mathbb{N}$ if and only if $X$ is not covered by $\langle S_i \rangle_{i \in \mathbb{N}}$.
\end{cor}
\begin{proof}
The proof of \cref{equivalent characterizations of strong f-randomness} shows that letting $S_i = \{ \sigma \in h^\ast \mid \apc(\sigma) < (\log_{1/2}\gamma(\sigma)) \cdot f(\sigma) - i\}$ yields a universal weak $f$-ML test.
\end{proof}

\begin{cor}
Suppose $f(\sigma) = \tilde{f}(\sigma)$ for almost all $\sigma$. Then strong $f$-randomness is equivalent to strong $\tilde{f}$-randomness.
\end{cor}
\begin{proof}
Suppose $f(\sigma) = \tilde{f}(\sigma)$ for all $\sigma \in h^\ast$ for which $|\sigma| > N$. Let $c = \max_{\sigma \in h^\ast, |\sigma| \leq N}{\apc(\sigma)}$. Then for every $i \in \mathbb{N}$, $\apc(\sigma) \geq (\log_{1/2} \gamma(\sigma)) \cdot f(\sigma) - c - i$ if and only if $\apc(\sigma) \geq (\log_{1/2} \gamma(\sigma)) \cdot \tilde{f}(\sigma) - c - i$ for all $\sigma \in h^\ast$. It follows that strong $f$-complexity and strong $\tilde{f}$-complexity are equivalent, and so \cref{equivalent characterizations of strong f-randomness} shows strong $f$-randomness and strong $\tilde{f}$-randomness are equivalent.
\end{proof}

\begin{remark}
All of the above results hold with $\mu$ replaced by any computable measure on $h^\mathbb{N}$ for which $\mu(\sigma) > 0$ for all $\sigma \in h^\ast$, with one adjustment -- a supermartingale $d$ $f$-succeeds on $X$ with respect to $\mu$ if and only if
\begin{equation*}
\limsup_n{\left( d(X \restrict n) \cdot \gamma(X \restrict n)^{-f(X \restrict n)} \cdot |h^n|^{-1} \right)} = \infty.
\end{equation*} 
\end{remark}

\subsection{Relationships between randomness notions}

\cite[Proposition 2.5]{greenberg2011diagonally} and \cite[Theorem 3.5]{higuchi2014propagation} show that (in $\cantor$) if $g$ grows sufficiently faster than $f$, then $g$-randomness implies strong $f$-randomness. This prompts the following question:

\begin{question}
Suppose $h$ is an order function and $f \colon h^\ast \to \co{0,\infty}$ is a nondecreasing computable function such that $\lim_{x \to \infty}{(x - f(x))} = \infty$. For what nondecreasing computable functions $g \colon h^\ast \to \co{0,\infty}$ such that $g$-randomness implies strong $f$-randomness?
\end{question}

Although we will not make use of it, an analog of \cite[Proposition 2.5]{greenberg2011diagonally} and \cite[Theorem 3.5]{higuchi2014propagation} holds for $h^\mathbb{N}$ with the growth rate of $h$ factoring heavily into how much faster $g$ must grow than $f$.

A simplifying assumption we will make is that $f$ is of the form $f \colon h^\ast \to \co{0,\infty}$. 

\begin{prop}
Suppose $f \colon h^\ast \to \mathbb{R}$ is a recursive function. Then there exists a recursive function $\hat{f} \colon h^\ast \to \co{0,\infty}$ such that (strong) $f$-randomness is equivalent to (strong) $\hat{f}$-randomness.
\end{prop}
\begin{proof}
Let $\hat{f}(\sigma) \coloneq \max\{0,f(\sigma)\}$. For $K$ representing either $\pfc$ or $\apc$, $K(\sigma) \geq 0$ for all $\sigma \in h^\ast$, so $K(\sigma) \geq (\log_{1/2} \gamma(\sigma)) \cdot f(\sigma) - c$ implies
\begin{equation*}
K(\sigma) \geq \max\{0, (\log_{1/2} \gamma(\sigma)) \cdot f(\sigma) - c\} \geq (\log_{1/2} \gamma(\sigma)) \cdot \max\{0,f(\sigma)\} - c = (\log_{1/2} \gamma(\sigma)) \cdot \hat{f}(\sigma) - c.
\end{equation*}
Conversely, if $K(\sigma) \geq (\log_{1/2} \gamma(\sigma)) \cdot \hat{f}(\sigma) - c$, then
\begin{equation*}
K(\sigma) \geq (\log_{1/2} \gamma(\sigma)) \cdot \hat{f}(\sigma) - c \geq (\log_{1/2} \gamma(\sigma)) \cdot f(\sigma) - c.
\end{equation*}
This suffices to show that (strong) $f$-randomness is equivalent to (strong) $\hat{f}$-randomness.
\end{proof}

\begin{notation}
Let $\gamma_0 = \gamma(\langle\rangle) = 1$ and $\gamma_n = \gamma(0^n) = |h^n|^{-1/n}$ for $n \geq 1$. 
\end{notation}

\begin{convention}
$f$ will denote a computable, unbounded function of the form $f \colon h^\ast \to \co{0,\infty}$ such that for every $X \in h^\mathbb{N}$ the sequence $\langle \gamma(X \restrict n)^{f(X \restrict n)} \rangle_{n \in \mathbb{N}}$ is eventually decreasing. $g$ and variations of $f$ and $g$ will similarly denote such functions unless otherwise specified.
\end{convention}

\begin{prop} \label{g-randomness to strong f-randomness}
Suppose $f,g \colon h^\ast \to \co{0,\infty}$ are recursive functions and there exists a nondecreasing function $j \colon \mathbb{N} \to \co{0,\infty}$ such that $g(\sigma) \geq f(\sigma) + j(|\sigma|)$ for all $\sigma \in h^\ast$ and for which $\sum_{n=0}^\infty{\gamma_n^{j(n)}} < \infty$. Then $g$-randomness implies strong $f$-randomness.
\end{prop}
\begin{proof}
We start by showing that there is a $c > 0$ such that $\dwt_g(S) \leq c \cdot \pwt_f(A)$ for all $A \subseteq h^\ast$. This allows us to convert a weak $f$-ML test $\langle S_i \rangle_{i \in \mathbb{N}}$ into a $g$-ML test by taking a tail of $\langle S_i \rangle_{i \in \mathbb{N}}$. Noting that $S \cap h^n$ is a prefix-free subset of $S$ for each $n \in \mathbb{N}$, we have
\begin{align*}
\dwt_g(S) = \sum_{\sigma \in S}{\gamma(\sigma)^{g(\sigma)}}
& \leq \sum_{n=0}^\infty{\sum_{\sigma \in S \cap h^n}{\gamma(\sigma)^{f(\sigma) + j(n)}}} \\
& = \sum_{n=0}^\infty{\gamma_n^{j(n)}} \cdot \sum_{\sigma \in S \cap h^n}{\gamma(\sigma)^{f(\sigma)}} \\
& \leq \sum_{n=0}^\infty{\gamma_n^{j(n)}} \cdot \pwt_f(S) \\
& = \left( \sum_{n=0}^\infty{\gamma_n^{j(n)}} \right) \cdot \pwt_f(S).
\end{align*}
Thus, we may let $c = \sum_{n=0}^\infty{\gamma_n^{j(n)}}$, which is finite by hypothesis.
\end{proof}

\begin{cor}
Fix a rational $\epsilon>0$, a $k \in \mathbb{N}_{\geq 1}$, and a computable $f \colon \{0,1\}^\ast \to \co{0,\infty}$. Then for any computable $g \colon \{0,1\}^\ast \to \co{0,\infty}$ satisfying
\begin{equation*}
f(\sigma) + (\log_{\gamma_{|\sigma|}} 1/2) \cdot \bigl( \log_2 |\sigma| + \log_2^2|\sigma| + \cdots + \log_2^{k-1}|\sigma| + (1+\epsilon)\log_2^k|\sigma| \bigr) \leq g(\sigma)
\end{equation*}
for all $\sigma \in h^\ast$, if $X \in h^\mathbb{N}$ is $g$-random, then $X$ is strongly $f$-random.
\end{cor}

\begin{remark}
\cref{g-randomness to strong f-randomness} can be generalized to the case of an arbitrary computable measure $\mu$ for which $\mu(\sigma) > 0$ for all $\sigma \in h^\ast$ by requiring $j$ to satisfy $\sum_{n=0}^\infty{\gamma_n^{h(n)}} < \infty$, where $\gamma_n = \max_{\sigma \in h^n}{\gamma(\sigma)}$. 
\end{remark}

\section{Randomness in \texorpdfstring{$h^\mathbb{N}$}{hN} versus \texorpdfstring{$\cantor$}{\{0,1\}N}} \label{randomness in hN to cantor space}

Algorithmic randomness and complexity is traditionally done within $\cantor$, and our use of partial randomness in $h^\mathbb{N}$ for an order function $h$ is ultimately a tool in proving facts about partial randomness in $\cantor$. To facilitate that, we want to translate randomness in $h^\mathbb{N}$ to randomness in $\cantor$. In general, $h(n)$ is not necessarily a power of two for each $n \in \mathbb{N}$. For that reason, if we wish to relate randomness in $h^\mathbb{N}$ with randomness in $\{0,1\}^\ast$, it is more convenient to pass through $[0,1]$ on the way to $\cantor$ where we may associate a string $\sigma \in h^\ast$ with a closed subinterval of $[0,1]$ having rational endpoints.

\subsection{Randomness in \texorpdfstring{$[0,1]$}{[0,1]} versus \texorpdfstring{$\cantor$}{\{0,1\}N}}

For the translation to and from $[0,1]$ and $\cantor$, the `obvious' approach works. 

\begin{definition}
Suppose $x \in [0,1]$. $\bin(x)$ is the unique infinite binary sequence such that $x = \sum_{i=0}^\infty{\bin(x)(i) \cdot 2^{-i-1}}$ which does not end in an infinite sequence of $1$'s except for in the case where $x=1$.

Suppose $X \in \cantor$. $0.X$ denotes the real number $\sum_{i=0}^\infty{X(i)\cdot 2^{-i-1}}$. Given $\sigma \in \{0,1\}^\ast$, $0.\sigma$ denotes the real number $\sum_{i=0}^{|\sigma|-1}{X(i)\cdot 2^{-i-1}}$.
\end{definition}

The map $X \mapsto 0.X$ is a surjection but not an injection, with $x=0.X = 0.Y$ for distinct $X$ and $Y$ if and only if $x$ is a dyadic rational in $(0,1)$ and $X$ and $Y$ are the binary representations of $x$, one ending in an infinite sequence of $0$'s and the other in an infinite sequence of $1$'s. 

\begin{definition}
Suppose $I \subseteq [0,1]$ is a closed interval with rational endpoints. The \textdef{norm} of $I$ is defined by
\begin{equation*}
|I| \coloneq -\log_2(\max I - \min I) = -\log_2 \lambda(I).
\end{equation*}

$\mathcal{J}$ denotes the set of closed subintervals $I \subseteq [0,1]$ with rational endpoints for which $|I| \in \mathbb{N}$ (equivalently, $\max I - \min I$ is a nonnegative power of $1/2$). Given $I \in \mathcal{J}$, a \textdef{code} for $I$ is a $4$-tuple $\langle a,b,c,d\rangle \in \mathbb{N} \times \mathbb{N}_{\geq 1} \times \mathbb{N} \times \mathbb{N}_{\geq 1}$ where $[a/b,c/d] = I$. 

Given $S \subseteq \mathcal{J}$, write $\bbracket{S} \coloneq \bigcup{S}$. $S \subseteq \mathcal{J}$ is \textdef{recursively enumerable}, or \textdef{r.e.}, if the set of codes of elements of $S$ is recursively enumerable. A sequence $\langle S_i \rangle_{i \in \mathbb{N}}$ of subsets of $\mathcal{J}$ is \textdef{uniformly r.e.} if the set of all $5$-tuples $\langle a,b,c,d,i\rangle$ where $i \in \mathbb{N}$ and $\langle a,b,c,d\rangle$ is a code for an element of $S_i$ is r.e.
\end{definition}

To make sense of direct $f$-weight of an interval $I \in \mathcal{J}$, we will require that $f$ be \textdef{length invariant}, i.e., $|\sigma| = |\tau|$ implies $f(\sigma) = f(\tau)$ for $\sigma,\tau \in h^\ast$. Thus, the map $n \mapsto f(0^n)$ completely characterizes $f$. Considering only the length invariant case allows us to use $f$ regardless of whether we are working within $\cantor$, $h^\mathbb{N}$, or $[0,1]$. 

\begin{convention}
Unless otherwise specified, from this point forward we will assume that the `$f$' in (strong) $f$-randomness is of the form $f \colon \mathbb{N} \to \co{0,\infty}$. 

\end{convention}

\begin{definition}[direct $f$-weight \& $f$-randomness in {$[0,1]$}]
Given $S \subseteq \mathcal{J}$, its \textdef{direct $f$-weight} is defined by 
\begin{equation*}
\dwt_f(S) \coloneq \sum_{I \in S}{2^{-f(|I|)}} = \sum_{I \in S}{(\lambda(I)^{|I|})^{f(|I|)}}.
\end{equation*}

An \textdef{$f$-ML test (in $[0,1]$)} is a uniformly r.e.\ sequence $\langle S_i \rangle_{i \in \mathbb{N}}$ of subsets of $\mathcal{J}$ such that $\dwt_f(S_i) \leq 1/2^i$ for each $i \in \mathbb{N}$. Such an $f$-ML test \textdef{covers} $x \in [0,1]$ if $x \in \bigcap_{i \in \mathbb{N}}{\bbracket{S_i}}$. $x \in [0,1]$ is \textdef{$f$-random (in $[0,1]$)} if no $f$-ML test in $[0,1]$ covers $x$.
\end{definition}

%
%

\begin{lem} \label{unit interval to cantor space lemma}
\textnormal{\cite[Lemma 6.1]{hudelson2014mass}}
There is a partial recursive function $\psi \colonsub \mathbb{N}^4 \to (\{0,1\}^\ast)^2$ such that if $\langle a,b,c,d\rangle$ is a code for $I \in \mathcal{J}$, then $\psi(a,b,c,d) \converge = \langle \sigma,\tau\rangle$ where $I \subseteq \{0.X \mid \sigma \subset X \vee \tau \subset X\}$ and $|\sigma| = |\tau| = |I|$.
\end{lem}

\begin{prop} \label{randomness in cantor space vs unit interval}
\textnormal{\cite[Lemma 6.2, essentially]{hudelson2014mass}}
Suppose $x \in [0,1]$. Then $x$ is $f$-random in $[0,1]$ if and only if $\bin(x)$ is $f$-random in $\cantor$.
\end{prop}
\begin{proof}
Let $\psi$ be as in \cref{unit interval to cantor space lemma}. 

Suppose $I \in \mathcal{J}$ is given, and let $\langle a,b,c,d\rangle$ be a code for $I$. Define $\bin(I) \coloneq \{\sigma,\tau\}$, where $\psi(a,b,c,d) \converge = \langle \sigma,\tau\rangle$, and observe that for any $x \in I$ we have $\bin(x) \in \bin(I)$ and that
\begin{equation*}
\dwt_f(\{\sigma,\tau\}) \leq 2^{-f(|\sigma|)} + 2^{-f(|\tau|)} = 2 \cdot 2^{-f(|I|)} = 2 \dwt_f(\{I\}).
\end{equation*}
Given $S \subseteq \mathcal{J}$, we define $\bin(S) \coloneq \bigcup\{ \bin(I) \mid I \in S\}$. Then $\dwt_f(\bin(S)) \leq 2 \cdot \dwt_f(S)$; the uniformity of the assignment $S \mapsto \bin(S)$ implies that if $S$ is r.e.\ then $\bin(S)$ is r.e., and if $\langle S_i \rangle_{i \in \mathbb{N}}$ is uniformly r.e.\ then $\langle \bin(S_i) \rangle_{i \in \mathbb{N}}$ is uniformly r.e.\ Moreover, if $x \in \bbracket{S}$ then $\bin(x) \in \bbracket{\bin(S)}$, so if $x$ is covered by an $f$-ML test then $\bin(x)$ is covered by an $f$-ML test. Thus, if $\bin(x)$ is $f$-random in $\cantor$ then $x$ is $f$-random in $[0,1]$.

Conversely, given $\sigma \in \{0,1\}^\ast$, let $I_\sigma = \{ 0.X \mid \sigma \subset X\}$. If $\sigma = \langle\rangle$ then $I_\sigma = [0,1]$. Otherwise, $I_\sigma = [0.\sigma\concat \langle 0,0,\ldots\rangle, 0.\sigma\concat \langle 1,1,\ldots\rangle]$. Observe that $|\sigma| = |I_\sigma|$, so $2^{-f(|\sigma|)} = 2^{-f(|I|)}$. Additionally, $I_\sigma = I_\tau$ if and only if $\sigma = \tau$. Given $S \subseteq \{0,1\}^\ast$, let $0.S = \{ I_\sigma \mid \sigma \in S\}$.  Then $\dwt_f(S) = \dwt_f(0.S)$; the uniformity in the assignment $S \mapsto 0.S$ implies that if $S$ is r.e.\ then $0.S$ is r.e., and if $\langle S_i \rangle_{i \in \mathbb{N}}$ is uniformly r.e.\ then $\langle 0.S_i \rangle_{i \in \mathbb{N}}$ is uniformly r.e.\ Moreover, if $X \in \bbracket{S}$ then $0.X \in \bbracket{0.S}$, so if $X$ is covered by an $f$-ML test in $\cantor$, then $0.X$ is covered by an $f$-ML test in $[0,1]$. Thus, if $x$ is $f$-random in $[0,1]$ then $\bin(x)$ is $f$-random in $\cantor$.
\end{proof}

\subsection{Randomness in \texorpdfstring{$h^\mathbb{N}$}{hN} versus \texorpdfstring{$[0,1]$}{[0,1]}}

The map $X \mapsto 0.X$ from $\cantor$ to $[0,1]$ can be described in a different way. With $I_\sigma$ as in the proof of \cref{randomness in cantor space vs unit interval}, $0.X$ is the unique element of $\bigcap_{n \in \mathbb{N}}{I_{X \restrict n}}$. Said another way, $[0,1]$ is split into two intervals of length $1/2$ corresponding to $\langle 0 \rangle$ and $\langle 1\rangle$, each of those intervals are split into two intervals of length $1/4$ corresponding to $\langle 0,0\rangle$, $\langle 0,1\rangle$, $\langle 1,0\rangle$, and $\langle 1,1\rangle$, and so on, then we take the intersection of the intervals corresponding to the initial segments of $X$ to get $0.X$. 

Repeating this methodology for $h^\ast$ produces closed subintervals of $[0,1]$ with rational endpoints, but not necessarily elements of $\mathcal{J}$. For that reason, we let $\mathcal{I}$ be the set of \emph{all} closed intervals $I \subseteq [0,1]$ with rational endpoints. 

\begin{definition}
Given $\sigma \in h^\ast$, define $\pi^h(\sigma) \coloneq [k/|h^{|\sigma|}|,(k+1)/|^{|\sigma|}|]$, where $k = \sum_{i=0}^{|\sigma|-1}{\sigma(i) \cdot |h^i|}$; in other words, $\pi^h(\langle\rangle) = [0,1]$ and, for $\sigma \in h^\ast$ and $0 \leq i < h(|\sigma|)$, $\pi^h(\sigma \concat \langle i \rangle)$ is the $i$-th subinterval of $\pi^h(\sigma)$ after splitting $\pi^h(\sigma)$ into $h(|\sigma|)$-many consecutive closed subintervals of equal length $1/|h^{|\sigma|+1}|$. 

The map $\pi^h \colon h^\mathbb{N} \to [0,1]$ is then defined by setting, for $X \in h^\mathbb{N}$,
\begin{equation*}
\pi^h(X) \coloneq \text{unique element of $\bigcap_{n \in \mathbb{N}}{\pi^h(X \restrict n)}$}.
\end{equation*}
\end{definition}

\begin{lem}
$\pi^h$ is a measure-preserving surjection of $h^\mathbb{N}$ onto $[0,1]$. However, $\pi^h$ is not injective, and for distinct $X,Y \in h^\mathbb{N}$, $\pi^h(X) = \pi^h(Y)$ if and only if there is $\sigma \in h^\ast$ such that $\{X,Y\} = \{\sigma \concat \langle 0, 0, \ldots\rangle, \sigma \concat \langle h(|\sigma|)-1,h(|\sigma|+1)-1,\ldots \rangle\}$. 
\end{lem}
\begin{proof}
Straight-forward.
\end{proof}

For an interval $I \in \mathcal{I}$, we wish to consider $f(|I|)$, although $|I| \in \mathbb{N}$ only for $I \in \mathcal{J}$, requiring the following convention:

\begin{convention}
Given $f \colon \mathbb{N} \to \co{0,\infty}$, we implicitly extend $f$ to a function $\co{0,\infty} \to \co{0,\infty}$ by letting $f(x) = (f(\lfloor x \rfloor+1) - f(\lfloor x \rfloor))(x - \lfloor x \rfloor) + f(\lfloor x \rfloor)$. 
\end{convention}

We extend the definition of $\dwt_f$ to $\mathcal{I}$ and $\mathcal{P}(\mathcal{I})$ in the obvious manner, and the definitions of $f$-ML tests and by extension $f$-randomness can likewise be extended. We will term these extended definitions by adding the prefix `extended', as in, ``\emph{$x \in [0,1]$ is extended $f$-random in $[0,1]$ if no extended $f$-ML test covers $x$.}''

An additional assumption we must make regards $f$ and the sequence $\langle f(n)/n\rangle_{n \in \mathbb{N}_{\geq 1}}$. Later we will strengthen this assumption further.

\begin{convention}
Given $f$, we assume $\langle f(n)/n\rangle_{n \in \mathbb{N}_{\geq 1}}$ is nondecreasing. As such, the function $x \in (0,\infty) \mapsto f(x)/x \in \co{0,\infty}$ is nondecreasing as well.
\end{convention}

\begin{notation}
Let $s \colon \mathbb{N} \to \co{0,\infty}$ be the unique nondecreasing computable function such that $s(0) = 1$ and for which $|h^n| = 2^{n\cdot s(n)}$ for all $n \in \mathbb{N}$. Consequently, $|\pi^h(\sigma)| = n\cdot s(n)$ for all $\sigma \in h^n$. 
\end{notation}

\begin{prop} \label{solovay f-random in unit interval implies solovay f-random in hN}
For any $S \subseteq h^\ast$, the set $\pi^h[S] = \{ \pi^h(\sigma) \mid \sigma \in S\}$ satisfies $\dwt_f(\pi^h[S]) \leq \dwt_f(S)$. Moreover, if $S$ is r.e.\ then $\pi^h[S]$ is r.e.
\end{prop}
\begin{proof}

Because $\pi^h$ is measure preserving, $\mu_h(\sigma) = \lambda(\pi^h(\sigma))$. Given $\sigma \in h^n$, $n \leq n \cdot s(n)$ implies $f(|\sigma|)/|\sigma| \leq f(|\pi^h(\sigma)|)/|\pi^h(\sigma)|$, and consequently $s(n)\cdot f(n) \leq f(n\cdot s(n))$. In particular,
\begin{equation*}
2^{- f(|\pi^h(\sigma)|)} = 2^{- f(n\cdot s(n))} \leq 2^{-s(n)\cdot f(n)} = (2^{-n \cdot s(n)})^{f(n)/n} = \gamma(\sigma)^{f(|\sigma|)}.
\end{equation*}
Thus, $\dwt_f(\pi^h[S]) \leq \dwt_f(S)$. That $\pi^h[S]$ is r.e.\ if $S$ is r.e.\ is immediate.
\end{proof}

\begin{cor}
If $\pi^h(X) \in [0,1]$ is extended $f$-random in $[0,1]$, then $X \in h^\mathbb{N}$ is $f$-random in $h^\mathbb{N}$.
\end{cor}
\begin{proof}
The uniformity of the assignment $S \mapsto \pi^h[S]$ implies that if $\langle S_i \rangle_{i \in \mathbb{N}}$ is a uniformly r.e.\ sequence of subsets of $h^\ast$, then $\langle \pi^h[S_i]\rangle_{i \in \mathbb{N}}$ is a uniformly r.e.\ sequence of subsets of $[0,1]$. With \cref{solovay f-random in unit interval implies solovay f-random in hN}, it follows that if $\langle S_i \rangle_{i \in \mathbb{N}}$ is an $f$-ML test in $h^\mathbb{N}$ then $\langle \pi^h[S_i] \rangle_{i \in \mathbb{N}}$ is an extended $f$-ML test in $[0,1]$.
\end{proof}

The proof of \cref{solovay f-random in unit interval implies solovay f-random in hN} suggests that if wish to convert a extended $f$-ML test in $[0,1]$ into an $f$-ML test in $h^\mathbb{N}$ then we want to pull intervals in $\mathcal{I}$ back into strings in $h^\ast$. However, the map $\pi^h \colon h^\ast \to \mathcal{I}$ is not surjective, so given $I \in \mathcal{I}$ we must instead cover $I$ with intervals of the form $\pi^h(\sigma)$ for $\sigma \in h^\ast$. This procedure must be sufficiently regular for an extended $f$-ML test in $[0,1]$ to be pulled back to a $g$-ML test in $h^\mathbb{N}$ for some appropriate $g$.

\begin{definition}
Given $f \domleq g$, then we say that the regularity condition \textdef{$(\ast)(g,f)$ holds for $h$} if
\begin{equation}
\sup_{n \in \mathbb{N}}{\frac{\exp_{h(n-1)}(1-f(n\cdot s(n)) \cdot (n \cdot s(n))^{-1})}{\exp_2(s(n) \cdot g(n) - f(n \cdot s(n)))}} < \infty. \tag*{$(\ast)(g,f)$}
\end{equation}
\end{definition}

\begin{remark}
In \cite{greenberg2011diagonally}, the regularity condition $(\ast)(g,f)$ is simplified by the fact that $f$ is linear, and hence $\frac{f(n\cdot s(n))}{n\cdot s(n)}$ simplifies into an expression independent of $n$ or $s(n)$.
\end{remark}

\begin{prop} \label{pulling intervals back}
Suppose $(\ast)(g,f)$ holds for $h$ and let 
\begin{equation*}
\alpha = 3 \cdot \sup_{n \in \mathbb{N}}{\frac{\exp_{h(n-1)}(1-f(n\cdot s(n)) \cdot (n \cdot s(n))^{-1})}{\exp_2(s(n) \cdot g(n) - f(n \cdot s(n)))}}.
\end{equation*}
Then for each $I \in \mathcal{I}$ there exists $\tilde{I} \subseteq h^\ast$ such that $I \subseteq \bigcup{\pi^h[\tilde{I}]}$ and $\dwt_g(\tilde{I}) \leq \alpha \cdot 2^{-f(|I|)}$. Moreover, $\tilde{I}$ can be uniformly computed from a code for $I$.
\end{prop}
\begin{proof}
We start by setting notation. For each $I \in \mathcal{I}$, let $n_I$ be the unique $n \geq 1$ such that $|h^n|^{-1} < \lambda(I) \leq |h^{n-1}|^{-1}$ and $k_I$ be the greatest integer $k$ such that $k/|h^{n_I}| \leq \lambda(I)$. Then $k_I < h(n_I-1)$ and there is a set $\hat{I} \subseteq \pi^h[h^{n_I}]$ computable from a code of $I$ of size $\leq k_I +2$ such that $I \subseteq \bigcup{\hat{I}}$, namely the intervals in $\pi^h[h^{n_I}]$ intersecting $I$ nontrivially (i.e., intervals which intersect $I$ at more than just an endpoint).

Suppose $I \in \mathcal{I}$ is given, and let $n = n_I$ and $k = k_I$. Because $k/|h^n| \leq \lambda(I)$, we have $|h^n|^{-1}/\lambda(I) \leq k$. Given $J \in \hat{I}$, $\lambda(J) \leq \lambda(I)$ and so $|I| \leq |J| = n \cdot s(n)$. Then
\begin{align*}
\dwt_g(\tilde{I}) = \sum_{\sigma \in \tilde{I}}{\gamma_n^{-g(n)}} & \leq (k+2) 2^{-s(n)\cdot g(n)} \\
& \leq  3k \cdot \exp_2(-\bigl( s(n) \cdot g(n) - f(n \cdot s(n)) \bigr)) \cdot (\lambda(I)/k)^{f(n \cdot s(n)) \cdot (n\cdot s(n))^{-1}} \\
& \leq 3 \cdot \frac{\exp_k(1-f(n\cdot s(n)) \cdot (n \cdot s(n))^{-1})}{\exp_2(s(n)\cdot g(n) - f(n \cdot s(n)))} \cdot \lambda(I)^{f(n\cdot s(n)) \cdot (n \cdot s(n))^{-1}} \\
& \leq 3 \cdot \frac{\exp_{h(n-1)}(1-f(n\cdot s(n)) \cdot (n \cdot s(n))^{-1})}{\exp_2(s(n)\cdot g(n) - f(n \cdot s(n)))} \cdot \lambda(I)^{f(|I|)/|I|} \\
& \leq \alpha \cdot 2^{-f(|I|)}.
\end{align*}
\end{proof}

\begin{cor} \label{g-random in hN to f-random in unit interval}
Suppose $(\ast)(g,f)$ holds for $h$ and let $X \in h^\mathbb{N}$. If $X$ is $g$-random in $h^\mathbb{N}$ then $\pi^h(X)$ is generalized $f$-random in $[0,1]$.
\end{cor}
\begin{proof} 
Let $\alpha$ be as in the statement of \cref{pulling intervals back}. Given $I \in \mathcal{I}$, let $\tilde{I}$ and $\hat{I}$ be as in the statement and proof of \cref{pulling intervals back}. Given $S \subseteq \mathcal{I}$ r.e., let $\hat{S} = \bigcup\{\hat{I} \mid I \in S\}$ and $\tilde{S} = \bigcup\{ \tilde{I} \mid I \in S\}$. Then $\tilde{S}$ is r.e.\ and $\dwt_g(\tilde{S}) = \dwt_g(\hat{S}) \leq \alpha \cdot \dwt_f(S)$.

Suppose for the sake of a contradiction that $\pi^h(X)$ is not generalized $f$-random in $[0,1]$, and so let $\langle S_i \rangle_{i \in \mathbb{N}}$ be a generalized $f$-ML test covering $\pi^h(X)$. Let $m \in \mathbb{N}$ satisfy $\alpha \leq 2^m$. That $\tilde{I}$ can be computed uniformly from a code for $I$ implies $\langle \tilde{S}_{i+m}\rangle_{i \in \mathbb{N}}$ is uniformly r.e.\ Then $\langle \tilde{S}_{i+m} \rangle_{i \in \mathbb{N}}$ is a $g$-ML test covering $X$, contradicting the hypothesis that $X$ is $g$-random in $h^\mathbb{N}$. 

\end{proof}

\begin{cor}
Suppose $x \in [0,1]$. Then $x$ is $f$-random in $[0,1]$ if and only if $x$ is generalized $f$-random in $[0,1]$.
\end{cor}
\begin{proof}
Being generalized $f$-random in $[0,1]$ clearly implies being $f$-random in $[0,1]$.

In the opposite direction, suppose $x$ is $f$-random in $[0,1]$, so that $\bin(x)$ is $f$-random in $\cantor$ by \cref{randomness in cantor space vs unit interval}. With $h(n) \coloneq 2$ for all $n \in \mathbb{N}$ we have $s(n) = 1$, so $|\pi^h(\sigma)| = |\sigma|$ for all $\sigma \in h^\ast = \{0,1\}^\ast$. The condition $(\ast)(g,f)$ for $h$ is then the statement that
\begin{equation*}
\sup_{n \in \mathbb{N}}{\frac{\exp_2(1-f(n)/n)}{\exp_2(g(n)-f(n))}} < \infty.
\end{equation*}
Then we may observe that $(\ast)(f,f)$ holds for $h$, and so \cref{g-random in hN to f-random in unit interval} implies $\pi^h(\bin(x)) = x$ is generalized $f$-random in $[0,1]$.
\end{proof}

\begin{cor} \label{translation of randomness from hN to cantor space}
Suppose $\lim_{n \to \infty}{\frac{s(n-1)}{s(n)}} = 1$ and $\epsilon>0$. If $f(n) = n - j(n)$ and $g(n) = n - (1-\epsilon) \frac{j(n\cdot s(n))}{s(n)}$ then $(\ast)(g,f)$ holds for $h$. Consequently, if $X$ is $g$-random in $h^\mathbb{N}$ then $\pi^h(X)$ is $f$-random in $[0,1]$.
\end{cor}
\begin{proof}
If $|h^n| = n \cdot s(n)$, then $h(n-1) = |h^n|/|h^{n-1}| = 2^{n\cdot s(n) - (n-1)\cdot s(n-1)}$. Then
\begin{align*}
\log_2\left(\frac{\exp_{h(n-1)}(1-f(n\cdot s(n)) \cdot (n \cdot s(n))^{-1})}{\exp_2(s(n) \cdot g(n) - f(n \cdot s(n)))}\right) & = (n\cdot s(n) - (n-1)\cdot s(n-1)) \cdot \left(1-\frac{f(n \cdot s(n))}{n\cdot s(n)}\right) \\
& \quad - s(n)\cdot g(n) + f(n\cdot s(n)) \\
& = n\cdot s(n) - (n-1)\cdot s(n-1) - f(n \cdot s(n)) \\
& \quad + \frac{(n-1)\cdot s(n-1)}{n\cdot s(n)}\cdot (n\cdot s(n) - j(n\cdot s(n))) \\
& \quad - s(n)\cdot (n- (1-\epsilon) \frac{j(n\cdot s(n))}{s(n)}) + f(n\cdot s(n)) \\
& = \left( (1-\epsilon) - \frac{n-1}{n}\frac{s(n-1)}{s(n)}\right)\cdot j(n\cdot s(n)).
\end{align*}
Because $\lim_{n \to \infty}{\frac{s(n-1)}{s(n)}} = 1$ by hypothesis, for all sufficiently large $n$ we have $1-\epsilon < \frac{n-1}{n}\frac{s(n-1)}{s(n)}$ and hence $\left( (1-\epsilon) - \frac{n-1}{n}\frac{s(n-1)}{s(n)}\right)\cdot j(n\cdot s(n)) < 0$. It follows that $(\ast)(g,f)$ holds for $h$.
\end{proof}

\begin{remark}
The condition that $\lim_{n \to \infty}{\frac{s(n-1)}{s(n)}} = 1$ is equivalent to $\lim_{n \to \infty}{\frac{\log_2 |h^n|}{\log_2 |h^{n-1}|}} = 1$, which is equivalent to $\lim_{n \to \infty}{\frac{\log_2 h(n-1)}{\log_2 |h^n|}} = 0$.
\end{remark}

\subsection{Improving Greenberg \& Miller's Conclusion} \label{improving greenberg and miller's conclusion section}

The motivating theorem \cite[Theorem 4.9]{greenberg2011diagonally} proceeds by showing that for any sufficiently slow-growing $\dnr$ function $Z$, $Z$ computes an $X \in h^\mathbb{N}$ which is $(\lambda n.\delta n)$-random in $h^\mathbb{N}$ for each rational $\delta < 1$, where $h(n) = (n+1)\cdot 2^n$. By showing that $\lim_{n \to \infty}{\frac{\log_2 |h^n|}{\log_2 |h^{n-1}|}} = 1$ and noting that $\frac{\delta n \cdot s(n)}{s(n)} = \delta n$, \cref{translation of randomness from hN to cantor space} shows that $\pi^h(X)$ is $(\lambda n.\delta n)$-random in $[0,1]$ (and hence $\bin(\pi^h(X))$ is $(\lambda n.\qspace \delta n)$-random in $\cantor$) for each rational $\delta < 1$. To arrange for $X$ being $(\lambda n.\delta n)$-random in $h^\mathbb{N}$ for each rational $\delta < 1$, $X$ is constructed entry by entry so that $d^h(X \restrict n) \leq n!$ for all $n$, where $d^h$ is a fixed universal r.e.\ supermatingale $d^h \colon h^\ast \to \co{0,\infty}$. By carefully examining the relevant calculations and using the full power of \cref{translation of randomness from hN to cantor space}, it can be shown that $\pi^h(X)$ exhibits more partial randomness than just having effective Hausdorff dimension $1$.

\begin{thm} \label{improved greenberg and miller's randomness conclusion}
Let $h(n) \coloneq (n+1) \cdot 2^n$.
\begin{enumerate}[(a)]
\item $\lim_{n \to \infty}{\frac{\log_2 |h^n|}{\log_2 |h^{n-1}|}} = 1$.

\item Suppose $X \in h^\mathbb{N}$ satisfies $d^h(X \restrict n) \leq n!$ for all $n \in \mathbb{N}$. Then $X$ is $(\lambda n. n - \beta \log_2 n)$-random in $h^\mathbb{N}$ for all $\beta > 2$. Consequently, $\pi^h(X)$ is $(\lambda n. n-\alpha \sqrt{n}\log_2 n)$-random in $[0,1]$ for all $\alpha > 1$.

\end{enumerate}
\end{thm}
\begin{proof} \mbox{}
\begin{enumerate}[(a)]
\item To aid in finding the corresponding $s(n)$, recall Stirling's Approximation:

\begin{lem}[Stirling's Approximation] \mbox{}
\begin{equation*}
\lim_{n \to \infty}{\frac{n!}{\sqrt{2\pi n}(n/e)^n}} = \lim_{n \to \infty}{\frac{n!}{\exp_2(n\log_2 n + n\log_2 e + \frac{1}{2}\log_2 n + \log_2\sqrt{2\pi})}} = 1.
\end{equation*}
\end{lem}

Then $\log_2 |h^n| \approx n \cdot \left( \frac{n-1}{2} + \log_2 n + \log_2 e + \frac{1}{2}\log_2 n^{1/n} + \log_2 (2\pi)^{2/n}\right)$, so $\lim_{n \to \infty}{\frac{s(n-1)}{s(n)}} = 1$. 

\item We start by computing $d(X \restrict n) \mu_h(X \restrict n)^{1-g(n)/n}$:
\begin{align*}
d(X \restrict n) \mu_h(X \restrict n)^{1-g(n)/n} & \leq n! \cdot (n! 2^{n(n-1)/2})^{\frac{g(n)}{n}-1} \\
& = (n!)^{\frac{g(n)}{n}} \cdot 2^{\left(\frac{g(n)}{n}-1\right) \cdot \frac{n(n-1)}{2}} \\
& \approx \exp_2\bigg(\frac{g(n)}{n}\left(n\log_2 n + n\log_2 e + \frac{1}{2}\log_2 n + \log_2\sqrt{2\pi}\right) \\
& \qquad ~~~~~  + \left(\frac{g(n)}{n}-1\right) \cdot \frac{n(n-1)}{2}\bigg) \\
& = \exp_2\left(g(n)\left( \frac{n-1}{2} + \log_2\left( n^{1+1/n}\cdot e \cdot \sqrt[2n]{2\pi}\right)\right) - \frac{n(n-1)}{2}\right).
\end{align*}
We want this last expression to be bounded above, so there must be a $c \in \mathbb{N}$ for which 
\begin{equation*}
g(n) \left( n-1+2\log_2\left(n^{1+1/n}\cdot e \cdot \sqrt[2n]{2\pi}\right)\right) \leq n^2-n+c.
\end{equation*}
Writing $g(n) = n - \tilde{j}(n)$, we find that if
\begin{equation*}
\tilde{j}(n) \geq \frac{2\log_2\left(n^{1+1/n} \cdot e \cdot \sqrt[2n]{2\pi}\right)-\frac{c}{n}}{1-\frac{1}{n}+\frac{2}{n}\log_2\left(n^{1+1/n} \cdot e \cdot \sqrt[2n]{2\pi}\right)}
\end{equation*}
then $X$ is $g$-random. Hence, for any $\beta > 2$, $X$ is $g$-random for $g(n) = n - \beta \cdot \log_2 n$.

By \cref{translation of randomness from hN to cantor space}, if $X$ is $\left(\lambda n.n -  (1-\epsilon)\frac{j(n \cdot s(n))}{s(n)}\right)$-random in $h^\mathbb{N}$ then $\pi^h(X)$ is $\left(\lambda n.n-j(n)\right)$-random in $[0,1]$. To show that $\pi^h(X)$ is $\left(\lambda n.n - \alpha \sqrt{n} \log_2 n\right)$-random in $[0,1]$ for any $\alpha > 1$, it suffices to show that there is $\beta > 2$ and $\epsilon>0$ such that $n - \beta \log_2 n \geq n - (1-\epsilon) \frac{\alpha \sqrt{n\cdot s(n)} \log_2 (n \cdot s(n))}{s(n)}$ for all sufficiently large $n$, or equivalently that $\beta \log_2 n \leq (1-\epsilon) \frac{\alpha \sqrt{n\cdot s(n)} \log_2 (n \cdot s(n))}{s(n)}$ for all sufficiently large $n$. Using the approximations $\frac{1}{2}n \leq s(n) \leq (\frac{1}{2}+\delta) n$ for $\delta > 0$, we have
\begin{align*}
(1-\epsilon) \frac{\alpha\sqrt{n\cdot s(n)} \log_2 (n \cdot s(n))}{s(n)} & \geq (1-\epsilon) \frac{\alpha\sqrt{n \cdot \frac{1}{2} n} \log_2 (n \cdot \frac{1}{2} n)}{(\frac{1}{2}+\delta) n} \\
& \geq 2(1-\epsilon) \frac{\sqrt{1/2}}{\sqrt{1/2+\delta}} \alpha \log_2 n - (1-\epsilon)\frac{\sqrt{1/2}}{\sqrt{1/2+\delta}}\alpha.
\end{align*}
Thus, it suffices for there to be $\beta > 2$, $\epsilon>0$, and $\delta > 0$ such that $\beta < 2(1-\epsilon) \frac{\sqrt{1/2}}{\sqrt{1/2 + \delta}}\alpha$. This is possible whenever $\alpha > 1$.

\end{enumerate}
\end{proof}


\section{Quantifying the Reduction of Avoidance to Complexity -- Preliminary Case} \label{quantifying greenberg miller proof}

\cref{improved greenberg and miller's randomness conclusion} addresses the question of the degree of partial randomness we may extract within the proof of \cite[Theorem 4.9]{greenberg2011diagonally} and serves as a precursor to a more general result putting a lower bound on which order functions $f$ have slow-growing order functions $q$ for which $\complex(f) \weakleq \ldnr(q)$. Our precursor to addressing the growth rate of such $q$ is the following refinement of \cite[Theorem 4.9]{greenberg2011diagonally}:

\begin{thm} \label{greenberg miller theorem 4.9 improved}
For rationals $\alpha \in (1,\infty)$ and $\beta \in (0,1/2)$, we have
\begin{equation*}
\complex(\lambda n. n - \alpha \sqrt{n}\log_2 n) \weakleq \ldnr(\lambda n. (\log_2 n)^\beta).
\end{equation*}
More generally, if $q \colon \mathbb{N} \to \mathbb{N}$ is an order function such that $q(2^{(3/2+\epsilon)n^2}) \leq n+1$ for almost all $n$ and some $\epsilon>0$, then $\complex(\lambda n. n - \alpha \sqrt{n}\log_2 n) \weakleq \ldnr(q)$. 
\end{thm}

We start by fixing some notation and definitions. 

\begin{definition}
For $a,b,c \in \mathbb{N}$, the class $P_a^{b,c}$ is defined by
\begin{equation*}
P_a^{b,c} \coloneq \{ F\colon\mathbb{N} \to [a]^b \mid \forall n \forall j < c \qspace (j \in \dom \varphi_n \to \varphi_n(j) \notin F(n))\}
\end{equation*}
where $[a]^b \coloneq \{ S \subseteq \{0,1,2,\ldots,a-1\} \mid |S| = b\}$.

In particular, $P_a^{1,c} = \{ F\colon\mathbb{N} \to a \mid \forall n \forall j < c \qspace (j \in \dom \varphi_n \to \varphi_n(j) \neq F(n))\}$.
\end{definition}

Given $a \in \mathbb{N}$, $P^{1,1}_a = \{ X \in a^\mathbb{N} \mid \forall n \qspace (F(n) \nsimeq \varphi_n(0))\}$. By the Parametrization Theorem, there is a total recursive $f \colon \mathbb{N}^2 \to \mathbb{N}$ such that $\varphi_{f(e,x)}(y) \simeq \varphi_e(x)$ for all $e,x,y \in \mathbb{N}$. Then given $X \in P_a^{1,1}$, the sequence $Y \in a^\mathbb{N}$ defined by $Y(n) \coloneq X(f(n,n))$ is a member of $\dnr(a)$. Conversely, given $Y \in \dnr(a)$, the sequence $X \in a^\mathbb{N}$ defined by $X(n) \coloneq Y(f(n,0))$ is a member of $P_a^{1,1}$. It is also relevant to observe that in both directions, each entry of the output sequence depends on only a single entry of the input sequence. Moreover, this one-to-one correspondence is uniform in $a$.

\begin{convention}
$\dnr(a)$ will be identified with $P_a^{1,1}$.
\end{convention}

\cite[Corollary 4.6]{greenberg2011diagonally} shows that $P_{ca}^{1,c} \strongleq \dnr(a)$, uniformly in $a,c \in \mathbb{N}$. In order to analyze a related result of Khan \cite[Theorem 6.3]{khan2013shift} (see \cref{shift complexity and avoidance section}), in 2020 Simpson performed a detailed analysis of this strong reduction with an eye towards generalization and to put an explicit and uniform bound on the number of entries of an element of $\dnr(a)$ are needed to compute a given bit of the corresponding element of $P_{ca+b}^{b+1,c}$.

\begin{prop} \label{greenberg miller reduction}
Uniformly in $a,b,c \in \mathbb{N}$, there is a recursive functional $\Psi \colon \dnr(a) \to P_{ca+b}^{b+1,c}$ and a recursive function $U \colon \mathbb{N} \to \mathcal{P}_\fin(\mathbb{N})$ such that for every $X \in \dnr(a)$, $\Psi(X)(n)$ depends only on $X \restrict U(n)$. Moreover, $|U(n)| \leq c \binom{ca+b}{a}$ for all $n \in \mathbb{N}$.
\end{prop}
\begin{proof}
With the identification of $\dnr(a)$ with $P_a^{1,1}$, the reduction $P_{ca+b}^{b+1,c} \strongleq \dnr(a)$ will result from a sequence of strong reductions
\begin{equation*}
P_{ca+b}^{b+1,c} \strongleq P_{a+d}^{d+1,1} \strongleq P_a^{1,1},
\end{equation*}
where $d = (c-1)a+b$. These reductions result from the following lemmas:

\begin{lem} \label{Pabc lemma 1}
$P_{a+d}^{d+1,c} \strongleq P_a^{1,c}$, uniformly in $a \geq 2$, $d \geq 0$, and $c \geq 1$.
\end{lem}
\begin{proof}
The partial function $\theta \colonsub \mathbb{N}^\ast \times \mathbb{N}^2 \to \mathbb{N}$ defined by
\begin{equation*}
\theta(\sigma,n,j) \simeq \min\{ i < |\sigma| \mid \varphi_n(j) \converge = \sigma(i) \}
\end{equation*}
for each $\sigma \in \mathbb{N}^\ast$ and $n,j \in \mathbb{N}$ 
is partial recursive, so there exists a total recursive function $f \colon \mathbb{N}^\ast \times \mathbb{N} \to \mathbb{N}$ such that $\varphi_{f(\sigma,n)}(j) \simeq \theta(\sigma,n,j)$ for all $\sigma \in \mathbb{N}^\ast$ and $n,j \in \mathbb{N}$. Given $S \subseteq \mathbb{N}$ with $|S| = a$, let $\sigma_S \in \mathbb{N}^a$ be the enumeration of $S$ in increasing order. 

Suppose $X \in P_a^{1,c}$. Recursively in $d$, we define $F_d \in P_{a+d}^{d+1,c}$. 
\begin{description}
\item[Base Case.] For $d=0$, $F_0 = X$.

\item[Induction Step.] Given $F_d \in P_{a+d}^{d+1,c}$ has been defined, let $F_{d+1}(n) = F_d(n) \cup \{ \sigma_S(X(f(\sigma_S,n)))\}$ for each $n \in \mathbb{N}$, where $S = (a+d+1) \setminus F_d(n)$ (note that $|S|=a$). Because $X \in P_a^{1,c}$, for all $j < c$
\begin{equation*}
X(f(\sigma_S,n)) \nsimeq \varphi_{f(\sigma_S,n)}(j) \simeq \theta(\sigma_S,n,j) \simeq \min\{ i < a \mid \varphi_n(j) \converge = \sigma_S(i)\}.
\end{equation*}
Thus, $\sigma_S(X(f(\sigma_S,n))) \nsimeq \varphi_n(j)$ for all $j < c$, and so $F_{d+1} \in P_{a+d+1}^{d+2,c}$.
\end{description}
\end{proof}

\begin{lem} \label{Pabc lemma 2}
Let $d = (c-1)a+b$. Then $P_{ca+b}^{c+b,c+e} \strongleq P_{a+d}^{d+1,e+1}$, uniformly in $a \geq 2$, $c \geq 1$, $b \geq 0$, and $e \geq 0$.
\end{lem}
\begin{proof}
Suppose $F \in P_{a+d}^{d+1,e+1}$. Because $a+d = ca+b$, for every $n \in \mathbb{N}$ and $j < e+1$ we have $\varphi_n(j) \notin F(n) \in [ca+b]^{d+1}$. The partial function $\theta \colonsub \mathbb{N}^3 \to \mathbb{N}$ defined by
\begin{equation*}
\theta(n,j,y) \simeq \varphi_n(j)
\end{equation*}
for each $n,j,y \in \mathbb{N}$ 
is partial recursive, so there exists a total recursive function $g \colon \mathbb{N}^2 \to \mathbb{N}$ such that $\varphi_{g(n,j)}(y) \simeq \varphi_n(j)$ for all $n,j,y \in \mathbb{N}$. In particular, $\varphi_n(j) \simeq \varphi_{g(n,j)}(0) \notin F(g(n,j))$ for all $n,j \in \mathbb{N}$. 

Define $H(n) \coloneq F(n) \cap \bigcap_{i < c-1}{F(g(n,e+i+1))}$, so that for all $j < e+1$ we have $\varphi_n(j) \notin F(n) \supseteq H(n)$, and for all $i < c-1$ we have $\varphi_n(e+i+1) \notin F(g(n,e+i+1)) \supseteq H(n)$. Thus, for every $j < c+e$ we have $\varphi_n(j) \notin H(n)$. The only obstacle to $H$ being a member of $P_{ca+b}^{c+b,c+e}$ is that $H$ need not be of size $c+b$. However, as long as $|H(n)| \geq c+b$ for every $n \in \mathbb{N}$ then we can let $G(n)$ consist of the first $c+b$ elements of $H(n)$. To that effect,
\begin{align*}
|(ca+b) \setminus H(n)| & = \left| \bigl[ (ca+b) \setminus F(n) \bigr] \cup \bigcup_{i < c-1}{\bigl[ (ca+b) \setminus F(g(n,e+i+1)) \bigr]} \right| \\
& \leq c \cdot \bigl( (ca+b) - (d+1) \bigr) \\
& = c(a-1)
\end{align*}
so $|H(n)| \geq (ca+b) - c(a-1) = b+c$. With $G(n)$ consisting of the first $c+b$ elements of $H(n)$, we have $G \in P_{ca+b}^{c+b,c+e}$.
\end{proof}

In the proof of \cref{Pabc lemma 1}, for each $n \in \mathbb{N}$ $F_d(n)$ depends only on the values $X(f(\sigma_S,n))$ for certain $S \in [a+d]^a$. In the proof of \cref{Pabc lemma 2}, $H(n)$ (and hence $G(n)$) depends on $F(n)$ and $F(g(n,e+i+1))$ for $i < c-1$. Thus, in the reduction $X \in P_a^{1,1} \mapsto F_d \in P_{a+d}^{d+1,1} \mapsto G \in P_{ca+b}^{b+1,c}$, $G(n)$ is determined by $X \restrict U(n)$, where
\begin{equation*}
U(n) = \{ f(\sigma_S,n) \mid S \in [a+d]^a \} \cup \{ f(\sigma_S,g(n,i+1)) \mid S \in [a+d]^a, i < c-1\}
\end{equation*}
and $|U(n)| \leq c \binom{a+d}{a} = c \binom{ca+b}{a}$.
\end{proof}

\begin{proof}[Proof of \cref{greenberg miller theorem 4.9 improved}.]
Let $h(n) = (n+1)2^n$ and let $d = d^h$ be a universal left r.e. supermartingale for $h^\mathbb{N}$. As $d$ is left r.e., uniformly in $\sigma \in h^\ast$ we can simultaneously and uniformly approximate $d(\sigma \concat \langle i\rangle)$ from below for all $i < h(n)$. Thus, there is a total recursive function $\sigma \mapsto m_\sigma$ such that for all $\sigma \in h^\ast$ and $x<2^{|\sigma|}$, $\varphi_{m_\sigma}(x)\downarrow = i$ if and only if $\sigma \concat \langle i\rangle$ is the $x$-th immediate successor $\tau$ of $\sigma$ found with respect to the aforementioned procedure with $d(\tau) > (n+1)!$. 

Let $\#\colon h^\ast \to \mathbb{N}$ be the inverse of the enumeration of $h^\ast$ according to the shortlex ordering. In particular, for $\sigma \in h^n$,
\begin{equation*}
\#(\sigma) \leq |h^0| + |h^1| + \cdots + |h^n| = \sum_{i=0}^n{i!\cdot 2^{i(i-1)/2}}.
\end{equation*}
By potentially modifying our enumeration $\varphi_0,\varphi_1,\varphi_2,\ldots$ of partial recursive functions, we can assume without loss of generality that $m_\sigma = 2\#(\sigma)$. Fix $\epsilon>0$ and let $m_n^\ast = 2^{(1/2+\epsilon)n^2}$, so that
\begin{equation*}
1+\sup\{m_\sigma \mid \sigma \in h^n\} \leq m_n^\ast,
\end{equation*}
which follows from the computations:
\begin{align*}
1 + 2 \cdot \#(\sigma) & \leq 1 + 2\left( |h^0| + |h^1| + \cdots + |h^n| \right) \\
& = 1 + 2\sum_{i=0}^n{\left(i! \cdot 2^{i(i-1)/2}\right)} \\
& \leq n\cdot \left(n! \cdot 2^{n(n-1)/2}\right) \\
& \approx \exp_2(n^2/2 + n\log_2 n + n(\log_2 e - 1/2) + (1/2)\log_2 n + \log_2\sqrt{2\pi} +1) \\
& < 2^{(1/2+\epsilon)n^2}.
\end{align*}

\cref{greenberg miller reduction} shows that, uniformly in $n$, there is a recursive functional $\Psi_n \colon \dnr(n+1) \to P_{h(n)}^{1,2^n}$ and recursive function $U_n \colon \mathbb{N} \to \mathcal{P}_\fin(\mathbb{N})$ such that for any $Z \in \dnr(n+1)$ and $i \in \mathbb{N}$, $Z \restrict U_n(i)$ determines $\Psi_n(Z)(i)$ and $|U_n(i)| \leq 2^n \binom{(n+1)2^n}{n+1}$. 

We are principally interested in initial segments $\rho$ of elements of $P_{h(n)}^{1,2^n}$ of length $m_n^\ast$ (in fact, we are only concerned with the values at the inputs $m_\sigma$ for $\sigma \in h^n$), so that:
\begin{enumerate}[(1)]
\item $\rho(m_\sigma) < h(n) = h(|\sigma|)$.
\item For all $x < 2^n$, if $\varphi_{m_\sigma}(x) \converge$, then $\rho(m_\sigma) \neq \varphi_{m_\sigma}(x)$.
\end{enumerate}
Define $U \colon \mathbb{N} \to \mathcal{P}_\fin(\mathbb{N})$ by $U(n) \coloneq \bigcup_{i < m_n^\ast}{U_n(i)}$ for each $n \in \mathbb{N}$ and subsequently define $\overline{u} \colon \mathbb{N} \to \mathbb{N}$ recursively by
\begin{align*}
\overline{u}(0) & \coloneq 0, \\
\overline{u}(n+1) & \coloneq \overline{u}(n) + |U(n)|.
\end{align*}
Finally, define $\psi \colonsub \mathbb{N} \to \mathbb{N}$ by letting
\begin{equation*}
\psi(\overline{u}(n) + j) \simeq \varphi_{\text{$j$-th element of $U(n)$}}(0)
\end{equation*}
for each $n \in \mathbb{N}$ and $j < |U(n)|$. 
By construction, for any $Z \in \avoid^\psi(n+1)$, $Z \restrict \overline{u}(n+1)$ can be used to compute an initial segment of an element of $P_{h(n)}^{1,2^n}$ of length $m_n^\ast$, and this is uniform in $n$.

If $p\colon\mathbb{N} \to \mathbb{N}$ is an order function satisfying
\begin{equation*}
p(\overline{u}(n+1)) \leq n+1
\end{equation*}
for all $n \in \mathbb{N}$, then uniformly in $n$ and $Z \in \avoid^\psi(p)$, $Z \restrict \overline{u}(n+1)$ can be used to compute an initial segment of an element of $P_{h(n)}^{1,2^n}$ of length $m_n^\ast$. Given $Z \in \avoid^\psi(p)$, define $G \in \baire$ by setting the value of $G(m_\sigma)$ according to this uniform process for each $\sigma \in h^\ast$; for $n$ not of the form $m_\sigma$ (which can be recursively checked) set $G(n) \coloneq 0$. Then define $X \in h^\mathbb{N}$ recursively by
\begin{align*}
X(0) & \coloneq G(m_{\langle\rangle}), \\
X(n+1) & \coloneq G(m_{\langle X(0),X(1),\ldots,X(n)\rangle}).
\end{align*}
In particular, for all $x < 2^n$, if $\varphi_{m_{X \restrict n}}(x) \downarrow = i$ (which is equivalent to $X \restrict n \concat \langle i\rangle$ being the $x$-th immediate successor of $X\restrict n$ found such that $d(X\restrict n \concat \langle i \rangle) \geq (n+1)!$) then $X(n+1) = G(m_{X\restrict n}) \neq i$. We make the following observation: if $d(\sigma) \leq n!$ for some $\sigma \in h^n$, then there at most $2^n$ many immediate successors $\tau$ of $\sigma$ such that $d(\tau) \geq (n+1)!$ since
\begin{equation*}
n! \geq d(\sigma) \geq \frac{1}{h(n)}\sum_{k < h(n)}{d(\sigma\concat\langle k\rangle)} = \frac{1}{(n+1)2^n}\sum_{k < h(n)}{d(\sigma\concat\langle k\rangle)}.
\end{equation*}
Thus, by induction on $n$ we find that $d(X \restrict n) \leq n!$ for all $n \in \mathbb{N}$. \cref{improved greenberg and miller's randomness conclusion} then implies that $Y = \bin(\pi^h(X)) \in \cantor$ is $f$-random.

It remains to finish the analysis of $q$. Suppose $q \colon \mathbb{N} \to \mathbb{N}$ is an order function such that for each $a,b \in \mathbb{N}$ we have $q(an+b) \leq p(n)$ for almost all $n$. Then $\avoid^\psi(p) \weakleq \ldnr(q)$. So if $q$ is such that for all $a,b \in \mathbb{N}$ we have $q(a\overline{u}(n+1)+b) \leq n+1$ for almost all $n \in \mathbb{N}$, then $\complex(f) \weakleq \ldnr(q)$. We show that if $q(2^{(3/2+\epsilon)n^2}) \leq n+1$ for almost all $n \in \mathbb{N}$, then $q$ satisfies this aforementioned condition.

First, we must find an upper bound of $\overline{u}$. Recall that $|U_n(i)| \leq 2^n\binom{(n+1)2^n}{n+1}$ for all $n,i \in \mathbb{N}$, so
\begin{align*}
|U(n)| & \leq \sum_{i < m_n^\ast}{|U_n(i)|} \\
& \leq \sum_{i < m_n^\ast}{2^n \cdot \binom{(n+1)2^n}{n+1}} \\
& = m_n^\ast \cdot 2^n \cdot \binom{(n+1)2^n}{n+1} \\
& \leq 2^{(1/2+\epsilon)n^2}\cdot \binom{(n+1) \cdot 2^n}{n+1} \\
& \leq 2^{(1/2+\epsilon)n^2}\cdot \exp_2(n^2+(1+\log_2 e)n + \log_2 e) \\
& = \exp_2((3/2+\epsilon)n^2+(1+\log_2 e)n+\log_2 e).
\end{align*}
By slightly increasing $\epsilon$, we find $|U(n)| \leq 2^{(3/2+\epsilon)n^2}$. 
Next, $\overline{u}(n+1) = \sum_{m \leq n}{|U(m)|} \leq (n+1)|U(n)|$, so $a \cdot \overline{u}(n+1)+b \leq a (n+1)|U(n)| + b \leq 3a \cdot n \cdot |U(n)|$ for all sufficiently large $n$. As before, a slight increase in $\epsilon$ allows us to absorb the $3an$ term, so that 
\begin{equation*}
a\overline{u}(n+1) + b \leq 2^{(3/2+\epsilon)n^2}.
\end{equation*}

Thus, if $q(2^{(3/2+\epsilon)n^2}) \leq n+1$, then 
\begin{equation*}
q(a\overline{u}(n+1)+b) \leq q(2^{(3/2+\epsilon)n^2}) \leq n+1.
\end{equation*}

If $\alpha$ is a positive rational less than $1/2$ then for almost all $n$ we have $(\log_2(2^{(3/2+\epsilon)n^2}))^\alpha \leq n+1$, proving the `in particular' statement.
\end{proof}

\begin{remark}
We assumed without loss of generality that $m_\sigma = 2 \cdot \#(\sigma)$ by choosing an appropriate enumeration of the partial recursive functions. In general, if $\theta \colon \mathbb{N} \to \mathbb{N}$ is a total recursive, injective function with recursive coinfinite image (as in the case of $n \mapsto 2n$), then for any total recursive function $g \colon \mathbb{N} \to \mathbb{N}$ there is an admissible enumeration $\tilde{\varphi}_0,\tilde{\varphi}_1,\ldots$ such that $\tilde{\varphi}_{\theta(e)} \simeq \varphi_{g(e)}$ for all $e \in \mathbb{N}$.

\end{remark}

\section{Quantifying the Reduction of Avoidance to Complexity -- General Case} \label{quantifying greenberg miller proof general}

Within the proof of \cref{greenberg miller theorem 4.9 improved}, how much does our result depend on the particular choice of $h$? 
\begin{description}
\item[Use I.] For $\sigma \in h^\ast$, we defined $m_\sigma = 2\#(\sigma)$ (where $\#(\sigma)$ is the enumeration of $h^\ast$ according to the shortlex ordering) and assumed without loss of genearlity that $\varphi_{m_\sigma}(x)\downarrow = k$ if and only if $\sigma\concat \langle k\rangle$ is the $x$-th immediate successor $\tau$ of $\sigma$ such that $d(\tau) \geq (n+1)!$.

\item[Use II.] The reduction $P_{h(n)}^{1,2^n} \strongleq \dnr(n+1)$, uniform in $n$, is used to define the recursive function $\overline{u}\colon\mathbb{N} \to \mathbb{N}$ and the partial recursive function $\psi\colon\mathbb{N} \to \mathbb{N}$ such that, uniformly in $n$, $Z \restrict \overline{u}(n+1)$ can be used to compute an initial segment of an element of $P_{h(n)}^{1,2^n}$ of length $m_n^\ast$ given any $Z \in \avoid^\psi(n+1)$.

\item[Use III.] For $p\colon\mathbb{N} \to (1,\infty)$ a recursive order function satisfying $p(\overline{u}(n+1)) \leq n+1$, a $Z \in \avoid^\psi(p)$ computes an $X \in h^\mathbb{N}$ such that $d(X\restrict n) \leq n!$ for all $n$, based on the observation that if $d(\sigma) \leq n!$ for some $\sigma \in h^n$, then there are at most $2^n$ many immediate successors $\tau$ of $\sigma$ such that $d(\tau) \geq (n+1)!$. 

\item[Use IV.] Using the fact that $d(X\restrict n) \leq n!$ for all $n$, show that $d(X\restrict n)\mu^h(X\restrict n)^{1-g(n)/n}$ is bounded above for any $g$ of the form $g(n) = n - \beta\log_2 n$ for $\beta > 2$. Hence, $X \in h^\mathbb{N}$ is (strongly) $g$-random in $h^\mathbb{N}$.

\item[Use V.] Because $\lim_{n \to \infty}{\frac{\log_2 |h^n|}{\log_2 |h^{n-1}|}} = 1$ and for each $\alpha > 1$ there are $\beta > 2$ and $\epsilon > 0$ such that $\beta \log_2 n \leq (1-\epsilon) \left(\alpha\sqrt{n \cdot s(n)} \log_2 (n \cdot s(n))\right)\cdot s(n)^{-1}$ for almost all $n$ (where $2^{n\cdot s(n)} = |h^n|$ for all $n \in \mathbb{N}$), it follows that $Y=\bin(\pi^h(X))$ is $(\lambda n.n-\alpha\sqrt{n} \log_2 n)$-random in $\cantor$.
\end{description}

By analyzing the necessary and sufficient conditions for an $h$ of the form $h(n) = k(n)\cdot \ell(n)$ to satisfy the properties corresponding to each Use and subsequently choosing $k$ and $\ell$ satisfying those conditions with specific $f$ and $g$ in mind, we can prove the following technical result:

\begin{thm} \label{complex below ldnr}
Suppose $j \colon \mathbb{N} \to (1,\infty)$ is an order function such that $\lim_{n \to \infty}{j(n)/n} = 0$ and for which the function $f \colon \mathbb{N} \to (1,\infty)$ defined by $f(n) \coloneq n-j(n)$ for $n \in \mathbb{N}$ is an order function. Given $s \colon \mathbb{N} \to \co{1,\infty}$ and a rational $\epsilon > 0$, define 
\begin{align*}
\tilde{j}(n) & \coloneq (1-\epsilon)\frac{j(s(n)\cdot n)}{s(n)}, \\
\ell(n) & \coloneq \exp_2(j(s(n+1)\cdot (n+1))-j(s(n)\cdot n))^{1-\epsilon}, \\
h(n) & \coloneq \exp_2(s(n+1)\cdot (n+1) - s(n)\cdot n).
\end{align*} 
If there are $s\colon \mathbb{N} \to [1,\infty)$ and rational $\epsilon > 0$ such that \begin{enumerate*}[(i)] \item $\im h \subseteq \mathbb{N}$, \item $\lim_{n \to \infty}{s(n+1)/s(n)} = 1$, and \item $\tilde{j}$ is an order function, \end{enumerate*} then $\complex(f) \weakleq \ldnr(q)$ for any order function $q \colon \mathbb{N} \to \co{0,\infty}$ such that, for almost all $n \in \mathbb{N}$,
\begin{equation*}
q\left( \exp_2((1-\epsilon)^{-1} \cdot [s(n+1) \cdot (n+1) - j(s(n+1) \cdot (n+1))] \cdot \ell(n)) \right) \leq \ell(n).
\end{equation*}
\end{thm}

For the remainder of this subsection, we use the following notation:

\begin{notation}
$k$ and $\ell$ denote recursive functions $\mathbb{N} \to \co{1,\infty}$. $h$, $K$, $L$, and $H$ are defined by, for $n \in \mathbb{N}$, $h(n) \coloneq k(n) \cdot \ell(n)$ and
\begin{equation*}
H(n) \coloneq h(0) \cdot h(1) \mdots h(n-1), \quad K(n) \coloneq k(0) \cdot k(1) \mdots k(n-1), \quad L(n) \coloneq \ell(0) \cdot \ell(1) \mdots \ell(n-1).
\end{equation*}
\end{notation}

It is most convenient for $h$ to take image in $\mathbb{N}$ so that $|h^n| = H(n)$, so we make the following conventions:

\begin{convention}
$k$ and $\ell$ will always be such that $h$ is an order function of the form $\mathbb{N} \to \mathbb{N}$. As a result, $|h^n| = H(n) = K(n) \cdot L(n)$.

Additionally, if $\alpha \in (1,\infty)$, then $\dnr(\alpha) = \{ X \in \baire \mid \forall n \qspace (X(n) < \alpha \wedge X(n) \nsimeq \varphi_n(n))\} = \dnr(\lceil \alpha \rceil)$. 
\end{convention}

\begin{notation}
$d^h$ is a fixed universal left r.e. supermartingale on $h^\mathbb{N}$.
\end{notation}

Among the explicit uses of $h$ in the proof of \cref{greenberg miller theorem 4.9 improved}, Uses I, II, and III only depended on writing $h$ in the form $h(n) = k(n) \cdot \ell(n) = 2^n \cdot (n+1)$ and observing that if $d^h(\sigma) \leq n!$ for some $\sigma \in h^n$, then there are at most $2^n$ many immediate successors $\tau$ of $\sigma$ such that $d^h(\tau) \geq (n+1)!$. This last observation holds in general:

\begin{prop} \label{general h bound on simple immediate extensions}
For all $n \in \mathbb{N}$ and all $\sigma \in h^n$ such that $d^h(\sigma) \leq L(n)$, there are at most $k(n)$-many immediate extensions $\tau$ of $\sigma$ such that $d^h(\tau) > L(n+1)$.
\end{prop}
\begin{proof}
Suppose for the sake of a contradiction that there are more than $k(n)$ immediate extensions $\tau$ of $\sigma$ such that $d^h(\tau) > L(n+1)$. Thus,
\begin{equation*}
\sum_{i < h(n)}{d^h(\sigma\concat \langle i\rangle)} > k(n) \cdot L(n+1) = h(n) \cdot L(n) \geq h(n) \cdot d^h(\sigma)
\end{equation*}
which contradicts the fact that $d^h$ is a supermartingale.
\end{proof}

Now we turn our attention to Use IV.

\begin{prop} \label{general h randomness criterion}
Let $g\colon \mathbb{N} \to \co{0,\infty}$ be an order function. Then $X \in h^\mathbb{N}$ is $g$-random in $h^\mathbb{N}$ if the following two conditions hold:
\begin{enumerate}[(i)]
\item $L(n) \leq K(n)^{\frac{n-g(n)}{g(n)}}$ for almost all $n$.
\item $d^h(X \restrict n) \leq L(n)$ for almost all $n$.
\end{enumerate}
\end{prop}
\begin{proof}
It suffices to show that $d^h$ does not $g$-succeed on $X$, i.e., that $\limsup_n{d^h(X\restrict n)\cdot |h^n|^{\frac{g(n)}{n}-1}} < \infty$. 
For all sufficiently large $n$,
\begin{align*}
L(n) \leq K(n)^{\frac{n}{g(n)}-1} & \iff L(n) \leq K(n)^{\frac{1-g(n)/n}{g(n)/n}} \\
& \iff L(n)^{g(n)/n} \leq K(n)^{1-g(n)/n} \\
& \iff \frac{L(n)^{g(n)/n}}{K(n)^{1-g(n)/n}} \leq 1 \\
& \iff \frac{L(n)^{\frac{g(n)}{n}-1+1}}{K(n)^{-(\frac{g(n)}{n}-1)}} \leq 1 \\
& \iff L(n) (K(n)L(n))^{\frac{g(n)}{n}-1} \leq 1 \\
& \iff L(n) |h^n|^{\frac{g(n)}{n}-1} \leq 1 \\
& \implies d^h(X\restrict n) \cdot |h^n|^{\frac{g(n)}{n}-1} \leq 1.
\end{align*}
Thus, $\limsup_n{d^h(X\restrict n)\cdot |h^n|^{\frac{g(n)}{n}-1}} < \infty$, as desired. 
\end{proof}

Given $f(n) = n - j(n)$, \cref{translation of randomness from hN to cantor space} suggests we find a $g$ of the form $g(n) = n - (1-\epsilon) \frac{j(s(n) \cdot n)}{s(n)}$ for some $\epsilon>0$, where $|h^n| = H(n) = 2^{s(n) \cdot n}$. \cref{general h randomness criterion} suggests that if we wish for the $X \in h^\mathbb{N}$ we construct to be (strongly) $g$-random, then we might as well start with $K$ and define $L$ by
\begin{equation*}
L(n) \coloneq K(n)^{\frac{n-g(n)}{g(n)}}.
\end{equation*}
$k$ and $\ell$ are then defined by $k(n) \coloneq K(n+1)/K(n)$ and $\ell(n) \coloneq L(n+1)/L(n)$ for $n \in \mathbb{N}$. Note that $H(n) = K(n)L(n) = K(n)^{n/g(n)}$ for all $n \in \mathbb{N}$. 




\begin{lem} \label{difference with n is unbounded}
If $t \colon \mathbb{N} \to \co{0,\infty}$ is an order function, then the function $\colon \mathbb{N} \to \co{0,\infty}$ defined by $r(n) \coloneq (n+1) \cdot t(n+1) - n \cdot t(n)$ for $n \in \mathbb{N}$ is a recursive function which dominates an order function.
\end{lem}
\begin{proof}
That $r$ is recursive is immediate. For all $n \in \mathbb{N}$,
\begin{equation*}
r(n) = (n+1) \cdot t(n+1) - n \cdot t(n) = n \cdot (t(n+1) - t(n)) + t(n+1).
\end{equation*}
Since $t$ is nondecreasing, $n \cdot (t(n+1) - t(n)) \geq 0$, and so $t(n+1) \leq r(n)$ for all $n \in \mathbb{N}$. The function $\tilde{t} \colon \mathbb{N} \to \co{0,\infty}$ defined by $\tilde{t}(n) \coloneq t(n+1)$ for $n \in \mathbb{N}$ is an order function such that $\tilde{t} \domleq r$.
\end{proof}


\begin{prop} \label{construction of h given j}
Let $j$, $f$, $s$, $\epsilon$, and $\tilde{j}$ satisfy the conditions of \cref{complex below ldnr}. Let $g(n) \coloneq n - \tilde{j}(n)$ and $K(n) = 2^{s(n) \cdot g(n)}$.
\begin{enumerate}[(a)]
\item For all $n \in \mathbb{N}$, $k(n), \ell(n), h(n) \geq 1$.
\item $\ell$ and $h$ dominate order functions.
\item If $X \in h^\mathbb{N}$ is such that $d^h(X \restrict n) \leq L(n)$ for almost all $n$, then $X$ is strongly $g$-random in $h^\mathbb{N}$.
\item If $X \in h^\mathbb{N}$ is $g$-random in $h^\mathbb{N}$, then $\pi^h(X) \in [0,1]$ is $f$-random in $[0,1]$.
\end{enumerate}
\end{prop}
\begin{proof} 
The functions $K,L,H,k,\ell,h$ can all be written as powers of two whose exponents involve $s$, $g$, and $\tilde{j}$:
\begin{align*}
K(n) & = 2^{s(n) \cdot g(n)}, & k(n) & = \exp_2(s(n+1) \cdot g(n+1) - s(n) \cdot g(n)),\\ 
L(n) & = 2^{s(n) \cdot \tilde{j}(n)}, & \ell(n) & = \exp_2(s(n+1) \cdot \tilde{j}(n+1) - s(n) \cdot \tilde{j}(n)), \\
H(n) & = 2^{s(n) \cdot n}, & h(n) & = \exp_2(s(n+1) \cdot (n+1) - s(n) \cdot n).
\end{align*}
Condition (iii) of \cref{complex below ldnr} implies $\tilde{j}$ is an order function. The hypothesis that $j(n) \leq n$ for all $n \in \mathbb{N}$ implies $\tilde{j}(n) \leq n$ for all $n \in \mathbb{N}$ as well, so $g$ is a nondecreasing function such that $g(n) \leq n$ for all $n \in \mathbb{N}$. Condition (i) implies $\im h \subseteq \mathbb{N}$.

\begin{enumerate}[(a)]
\item That $s$, $g$, and $\tilde{j}$ are all nondecreasing implies that $k$, $\ell$, and $h$ are bounded below by $1$.

\item \cref{difference with n is unbounded} shows that $h$ and $\ell$ each dominate order functions.

\item This is simply \cref{general h randomness criterion}.

\item Using Condition (ii) of \cref{complex below ldnr}, this is simply \cref{translation of randomness from hN to cantor space}.

\end{enumerate}
\end{proof}

It remains to generalize the construction of $X$ in \cref{greenberg miller theorem 4.9 improved} and establish our bounds on $q$ as a function of $j$.

\begin{proof}[Proof of \cref{complex below ldnr}.]
Fix a rational $\epsilon>0$ and observe that fulfillment of the conditions on $s$, $j$, and $\epsilon$ do not depend the value of $\epsilon$.

$d^h$ is left r.e., so uniformly in $\sigma \in h^\ast$ we can simultaneously and uniformly approximate $d(\sigma \concat \langle i\rangle)$ from below for all $i < h(n)$. 
Uniformly in $\sigma \in h^n$, let $m_\sigma$ be such that for all $x<k(n)$, $\varphi_{m_\sigma}(x)\downarrow = i$ if and only if $\sigma \concat \langle i\rangle$ is the $x$-th immediate successor $\tau$ of $\sigma$ found with respect to the aforementioned procedure such that $d(\tau) > L(n)$. 

Let $\#\colon h^\ast \to \mathbb{N}$ be the inverse of the enumeration of $h^\ast$ according to the shortlex ordering. In particular, for almost all $n \in \mathbb{N}$ and all $\sigma \in h^n$,
\begin{equation*}
\#(\sigma) \leq |h^0| + |h^1| + \cdots + |h^n| = \sum_{i=0}^n{2^{s(i) \cdot i}} < (n+1) \cdot 2^{s(n) \cdot n}.
\end{equation*}
By potentially modifying our enumeration $\varphi_0,\varphi_1,\varphi_2,\ldots$ of partial recursive functions, we can assume without loss of generality that $m_\sigma = 2\#(\sigma)$. Let $m_n^\ast = (n+1) \cdot 2^{s(n) \cdot n+1}$, so that $1+\sup\{m_\sigma \mid \sigma \in h^n\} \leq m_n^\ast$ for almost all $n$.

Let $\Psi_n\colon \dnr(\ell(n)) \to P_{h(n)}^{1,k(n)}$ be uniformly recursive functionals realizing the reductions $P_{h(n)}^{1,k(n)} \strongleq \dnr(\ell(n))$ from \cref{greenberg miller reduction}, and let $U_n\colon\mathbb{N} \to \mathcal{P}_\mathrm{fin}(\mathbb{N})$ be the associated recursive functions (so that $|U_n(i)| \leq k(n)\binom{h(n)}{\ell(n)}$). We are principally interested in initial segments $\rho$ of elements of length $m_n^\ast$ (in fact, we are only concerned with the values at the inputs $m_\sigma$ for $\sigma \in h^n$), so that:
\begin{enumerate}[(1)]
\item $\rho(m_\sigma) < h(n) = h(|\sigma|)$.
\item For all $x < k(n)$, if $\varphi_{m_\sigma}(x)\converge$, then $\rho(m_\sigma) \neq \varphi_{m_\sigma}(x)$.
\end{enumerate}
Define $U\colon\mathbb{N} \to \mathcal{P}_\mathrm{fin}(\mathbb{N})$ by setting $U(n) \coloneq \bigcup_{i < m_n^\ast}{U_n(i)}$ for $n \in \mathbb{N}$ and subsequently define $\overline{u} \colon \mathbb{N} \to \mathbb{N}$ by
\begin{align*}
\overline{u}(0) & \coloneq 0, \\
\overline{u}(n+1) & \coloneq \overline{u}(n) + |U(n)|.
\end{align*}
Finally, define $\psi\colonsub\mathbb{N} \to \mathbb{N}$ by letting, for $n \in \mathbb{N}$ and $j < |U(n)|$,
\begin{equation*}
\psi(\overline{u}(n)+j) \simeq \varphi_\text{$j$-th element of $U(n)$}(0).
\end{equation*}
By construction, for any $Z \in \avoid^\psi(\ell(n))$, $Z\restrict \overline{u}(n+1)$ can be used to compute an initial segment of an element of $P_{h(n)}^{1,k(n)}$ of length $m_n^\ast$, and this is uniform in $n$. 

If $p\colon\mathbb{N} \to \mathbb{N}$ is a recursive order function satisfying
\begin{equation*}
p(\overline{u}(n+1)) \leq \ell(n)
\end{equation*}
for almost all $n$, then uniformly in $n$ and $Z \in \avoid^\psi(p)$, $Z\restrict \overline{u}(n+1)$ can be used to compute an initial segment of $P_{h(n)}^{1,k(n)}$ of length $m_n^\ast$. Given $Z \in \avoid^\psi(p)$, define $G\colon\mathbb{N} \to \mathbb{N}$ by setting the value of $G(m_\sigma)$ according to this uniform process for each $\sigma \in h^\ast$; for $n$ not of the form $m_\sigma$ (which can be recursively checked), set $G(n) \coloneq 0$. Then define $X \in h^\mathbb{N}$ recursively by
\begin{align*}
X(0) & \coloneq G(m_{\langle\rangle}), \\
X(n+1) & \coloneq G(m_{\langle X(0), X(1),\ldots X(n)\rangle}).
\end{align*}

We claim that $d^h$ does not $g$-succeed on $X$. We start by showing that $d^h(X\restrict n) \leq L(n)$ by induction. For $n=0$, this follows since $d^h(X\restrict n) = d^h(\langle\rangle) =1$. Now suppose for our induction hypothesis that $d(X\restrict n) \leq L(n)$. By construction, for $x < k(n)$, if $\varphi_{m_{X\restrict n}}(x) \converge = i$ then $X(n+1) = G(m_{X\restrict n}) \neq i$; in combination with the induction hypothesis and \cref{general h bound on simple immediate extensions}, it follows that $d(X\restrict (n+1)) \leq L(n+1)$. By our definition of $L$ and \cref{general h randomness criterion}, $d^h$ does not $g$-succeed on $X$. Equivalently, $X$ is (strongly) $g$-random in $h^\mathbb{N}$.

\cref{construction of h given j} shows that $\pi^h(X) = Y \in [0,1]$ is $f$-random. In other words, $\complex(f) \weakleq \avoid^\psi(p)$. By \cref{construction of h given j}, $\ell$ dominates an order function, and hence there exists a recursive order function $p$ satisfying $p(\overline{u}(n+1)) \leq \ell(n)$. If $q \colon \mathbb{N} \to \mathbb{N}$ is a recursive order function such that for all $a,b \in \mathbb{N}$ we have $q(an+b) \leq p(n)$ for almost all $n$, then 
\begin{equation*}
\complex(f) \weakleq \avoid^\psi(p) \weakleq \ldnr(q).
\end{equation*}

To get a more explicit condition on $q$, we find an upper bound for $a\overline{u}(n+1)+b$. For all $n$ and $i$, 
\begin{equation*}
|U_n(i)| \leq k(n)\binom{h(n)}{\ell(n)} \leq k(n)\left(\frac{h(n)e}{\ell(n)}\right)^{\ell(n)}  = e^{\ell(n)}k(n)^{\ell(n)+1} = \exp_2((\log_2k(n) + \log_2 e)\ell(n) + \log_2k(n)).
\end{equation*}
Thus, for all $n$,
\begin{align*}
|U(n)| & \leq \sum_{i < m_n^\ast}{|U_n(i)|} \\
& \leq m_n^\ast \cdot \max_i{U_n(i)} \\
& \leq (n+1) \cdot 2^{s(n) \cdot n+1} \cdot \exp_2((\log_2k(n) + \log_2 e)\ell(n) + \log_2k(n)) \\
& \leq \exp_2\left(\log_2 H(n)+(\log_2k(n)+\log_2e)\ell(n) + \log_2k(n)+\log_2(n+1) + 1\right).
\end{align*}
For any $a$ and $b$ and almost all $n$,
\begin{align*}
\log_2 \bigl( a\overline{u}(n+1)+b \bigr) & = \log_2 \bigl( a\sum_{m \leq n}{|U(m)|}+b \bigr) \\
& \leq \log_2 \bigl( a(n+1) \cdot |U(n)| + b \bigr) \\
& \leq \log_2 \bigl( 3an \cdot |U(n)| \bigr) \\
& \leq \log_2 H(n)+(\log_2k(n)+\log_2e)\ell(n) + \log_2k(n)+2\log_2(n+1) +\log_2(3a)+1.
\end{align*}
Substituting $\log_2 H(n)$ and $\log_2 k(n)$ with expressions in terms of $s$, $g$, and $\tilde{j}$ gives
\begin{align*}
\log_2\bigl( a\overline{u}(n+1) + b \bigr) & = s(n) \cdot n + \bigl( s(n+1) \cdot g(n+1) - s(n) \cdot g(n) +\log_2 e \bigr)\cdot \ell(n) + s(n+1) \cdot g(n+1) \\
& \quad - s(n) \cdot g(n) + 2 \log_2 (n+1) + \log_2(3a)+1 \\
& \leq 2s(n) \cdot g(n+1) + \bigl( s(n+1) \cdot g(n+1) - s(n) \cdot g(n) +\log_2 e \bigr)\cdot \ell(n) + s(n+1) \cdot g(n+1) \\
& \quad - s(n) \cdot g(n) + s(n) \cdot g(n+1) + \log_2(3a)+1 \\
& \leq s(n+1) \cdot g(n+1) \cdot (\ell(n)+4) \\
& = s(n+1) \cdot \left( n+1 - (1-\epsilon) \frac{j(s(n+1)\cdot (n+1))}{s(n+1)}\right) \cdot (\ell(n)+4) \\
& \leq \left( s(n+1) \cdot (n+1) - (1-\epsilon)j(s(n+1)\cdot (n+1))\right) \cdot (\ell(n)+4) \\
& \leq \frac{1}{1-\epsilon}\left( s(n+1) \cdot (n+1) - j(s(n+1) \cdot (n+1))\right) \cdot \ell(n) \\
& = \frac{1}{1-\epsilon} f(s(n+1)\cdot(n+1)) \cdot \ell(n)
\end{align*}
for almost all $n$, where the final line follows from the fact that $\lim_{n \to \infty}{j(n)/n} = 0$. Thus, if $q$ satisfies
\begin{equation*}
q\left(\exp_2((1-\epsilon)^{-1} \cdot f(s(n+1) \cdot (n+1)) \cdot \ell(n))\right) \leq \ell(n)
\end{equation*}
then $\complex(f) \weakleq \ldnr(q)$.
\end{proof}

\begin{example} \label{greenberg miller comparison example}
Suppose $j(n) = \sqrt{n} \log_2 n$ and let $s(n) = n$. Simplifying $\ell$ gives
\begin{align*}
\ell(n) & = \exp_2\left(2(1-\epsilon) \left(\sqrt{(n+1)^2} \log_2 (n+1) - \sqrt{n^2} \log_2 n\right)\right) \\
& = \exp_2\left(2(1-\epsilon) \log_2 \left( (n+1) \cdot \left(1+\frac{1}{n}\right)^n\right)\right) \\
& = \left( (n+1) \cdot \left(1+\frac{1}{n}\right)^n\right)^{2(1-\epsilon)}.
\end{align*}
This provides the bounds 
\begin{equation*}
(n+1)^{2(1-\epsilon)} \leq \ell(n) \leq (n+1)^{2(1-\epsilon)} \cdot e^{2(1-\epsilon)}
\end{equation*}
and consequently
\begin{align*}
(1-\epsilon)^{-1} \cdot f( (n+1)^2) \cdot \ell(n) & \leq (1-\epsilon)^{-1} \cdot [(n+1)^2 - 2(n+1) \log_2 (n+1)] \cdot (n+1)^{2(1-\epsilon)} \cdot e^{2(1-\epsilon)} \\
& \leq \frac{e^{2(1-\epsilon)}}{1-\epsilon} \cdot (n+1)^{2(2 - \epsilon)} \\
& \leq (n+1)^{4 - 2\epsilon}.
\end{align*}
Thus, \cref{complex below ldnr} implies $\complex(\lambda n. n - \sqrt{n} \log_2 n) \weakleq \ldnr(q)$ whenever 
\begin{equation*}
q\left(\exp_2((n+1)^{2(2 - \epsilon)})\right) \leq (n+1)^{2(1-\epsilon)}.
\end{equation*}
In particular, we may take $q(n) \coloneq \left(\log_2 n\right)^\beta$ for any $\beta < 1/2$. In other words, whereas \cref{greenberg miller theorem 4.9 improved} shows that $\complex(\lambda n. n - \alpha \sqrt{n} \log_2 n) \weakleq \ldnr(\lambda n. (\log_2 n)^\beta)$ for any $\alpha > 1$ and $\beta < 1/2$, \cref{complex below ldnr} shows that for the same order functions $q$, we in fact are able to compute $(\lambda n. n - \sqrt{n}\log_2 n)$-complex sequences. 
\end{example}

\cref{greenberg miller comparison example} can be generalized further to address functions of the form $f(n) = n - \sqrt{n} \cdot \Delta(n)$:

\begin{thm} \label{sqrt complex from ldnr}
Given an order function $\Delta \colon \mathbb{N} \to [0,\infty)$ such that $\lim_{n \to \infty}{\Delta(n)/\sqrt{n}} = 0$ and any rational $\epsilon \in (0,1)$, 
\begin{equation*}
\complex\bigl(\lambda n. n-\sqrt{n}\cdot \Delta(n)\bigr) \weakleq \ldnr\bigl(\lambda n. \exp_2\bigl((1-\epsilon)\Delta(\log_2\log_2 n)\bigr)\bigr).
\end{equation*}
More generally, $\complex(\lambda n. n - \sqrt{n}\cdot \Delta(n)) \weakleq \ldnr(q)$ for any order function $q$ satisfying
\begin{equation*}
q\left( \exp_2((1-\epsilon)^{-1} \cdot [(n+1)^2 - (n+1) \cdot \Delta((n+1)^2)] \cdot \ell(n)) \right) \leq \ell(n)
\end{equation*}
for almost all $n \in \mathbb{N}$, where $\ell(n) = \exp_2\left((1-\epsilon)[(n+1) \cdot \Delta((n+1)^2) - n \cdot \Delta(n^2)]\right)$. 
\end{thm}
\begin{proof}
Let $s(n) \coloneq n$ and $j(n) \coloneq \sqrt{n} \cdot \Delta(n)$ for all $n \in \mathbb{N}$. We show that the conditions of \cref{complex below ldnr} are fulfilled with these choices of $s$ and $j$:
\begin{enumerate}[(i)]
\item $(n+1)^2 - n^2 = 2n+1$ and $2^{2n+1} \in \mathbb{N}$ for all $n \in \mathbb{N}$. 
\item Immediate.
\item $\frac{j(n^2)}{n} = \frac{\sqrt{n^2} \cdot \Delta(n^2)}{n} = \Delta(n^2)$ shows that the function $n \mapsto j(s(n) \cdot n)/s(n)$ is an order function. 
\end{enumerate}

The condition given by \cref{complex below ldnr} requires that for some $\epsilon > 0$ and almost all $n$
\begin{equation*}
q\left( \exp_2((1-\epsilon)^{-1} \cdot [(n+1)^2 - (n+1)\cdot \Delta( (n+1)^2)] \cdot \ell(n)) \right) \leq \ell(n),
\end{equation*}
where $\ell(n) = \exp_2((1-\epsilon)((n+1) \cdot \Delta((n+1)^2) - n \cdot \Delta(n^2)))$. Rearranging the exponent of $\ell(n)$ gives a simple lower bound:
\begin{equation*}
\log_2 \ell(n) = (1-\epsilon) \cdot \left( \Delta((n+1)^2) + n \cdot (\Delta((n+1)^2) - \Delta(n^2))\right) \geq (1-\epsilon) \cdot \Delta((n+1)^2) \geq (1-\epsilon) \Delta(n^2).
\end{equation*}
Concerning the exponent of $q$'s argument, we have the following: for almost all $n$,
\begin{equation*}
(1-\epsilon) \cdot [(n+1)^2 - (n+1) \cdot \Delta((n+1)^2)] \cdot \ell(n) \leq (1-\epsilon) \cdot (n+1)^2 \cdot \exp_2\left((1-\epsilon) \cdot (n+1) \cdot \Delta((n+1)^2)\right) \leq 2^{n^2}.
\end{equation*}
Thus, $\complex(\lambda n. n - \sqrt{n} \cdot \Delta(n)) \weakleq \ldnr(q)$ if $q(\exp_2 \exp_2 n^2) \leq \exp_2((1-\epsilon) \Delta(n^2))$ and hence if the following stronger condition holds:
\begin{equation*}
q(n) \leq \exp((1-\epsilon) \Delta(\log_2 \log_2 n)).
\end{equation*}
\end{proof}

\begin{remark}
Setting $q(n) \coloneq \exp_2((1-\epsilon)\Delta(\log_2\log_2 n))$ can be very inefficient; when $\Delta(n) = \log_2 n$, this gives $q(n) = \left(\log_2 \log_2 n \right)^{1-\epsilon}$, significantly slower than the lower bound on $q$ established in \cref{greenberg miller comparison example}. 

A better bound can be given for well-behaved $\Delta$ which are dominated by $\log_2$. Suppose $\Delta \domleq \log_2$ and that 
\begin{equation*}
c \coloneq \lim_{n \to \infty}{n[ \Delta((n+1)^2) - \Delta(n^2)]} < \infty.
\end{equation*}
Then we may give the following lower and upper bounds for $\ell$: for almost all $n$,
\begin{equation*}
\exp_2((1-\epsilon) \Delta((n+1)^2)) \leq \ell(n) \leq \exp_2((1-\epsilon)\Delta((n+1)^2) + c) \leq 2^c \cdot (n+1)^{2(1-\epsilon)}.
\end{equation*}
Consequently, $\complex(\lambda n. n- \sqrt{n}\cdot \Delta(n)) \weakleq \ldnr(q)$ whenever $q$ is such that, for almost all $n$,
\begin{equation*}
q(n) \leq \exp_2\left((1-\epsilon)\Delta\left(\left[ (1-\epsilon) \cdot 2^{-c} \cdot \log_2 n\right]^{1/(2-\epsilon)}\right)\right).
\end{equation*}
\end{remark}

\section{Open Questions} \label{complexity and slow-growing ldnr question section}

\cref{complex below ldnr,,sqrt complex from ldnr} only provide partial answers to \cref{question ldnr above complex}, leaving that question open in general. 

It is unclear what the full extent of the coverage of \cref{complex below ldnr} is, suggesting the following question:

\begin{question} \label{what does complex below ldnr apply to}
For what sub-identical order functions $f$ is there a computable function $s \colon \mathbb{N} \to [1,\infty)$ satisfying the conditions of \cref{complex below ldnr}? 

E.g., does such an $s$ exist when $f = \lambda n.n - \log_2 n$, or even when $f = \lambda n. n - \sqrt[3]{n}$?
\end{question}

Questions about the optimality of $q$ in \cref{complex below ldnr,,sqrt complex from ldnr} also remain:

\begin{question} \label{quantify upper bound on q}
For a given sub-identical order function $f$, can we provide an upper bound on how slowly $q$ must grow for $\complex(f) \weakleq \ldnr(q)$ to hold?
\end{question}

An affirmative answer to \cref{upward ldnr above complex} would put a strong bound on how slow-growing $q$ must be.

\begin{question} \label{quantify upper bound on q fast-growing}
Is there a fast-growing order function $q$ such that $\complex(\lambda n. n-\sqrt{n}) \weakleq \ldnr(q)$?
\end{question}

Currently, there is little in the direction of answering \cref{upward ldnr above complex} or \cref{quantify upper bound on q fast-growing}. In fact, the following specific -- seemingly tame -- instance of \cref{upward ldnr above complex} remains open:

\begin{question} \label{fast-growing ldnr above half-random}
Does there exist a fast-growing order function $p$ for which $\complex(1/2) \weakleq \ldnr(p)$?
\end{question}

\clearpage
\chapter{Generalized Shift Complexity}
\label{generalized shift complexity chapter}

$\complex(f)$ is non-negligible for every order function $f$ such that $\limsup_n{(f(n)-n)} < \infty$. One notion of complexity with corresponding weak degrees that lies in the deep region of $\mathcal{E}_\weak$ is that of \emph{shift} complexity.

\begin{definition}[shift complexity]
$X \in \cantor$ is\ldots
\begin{description}
\item[\ldots] \textdef{$\langle\delta,c\rangle$-shift complex} (where $\delta \in (0,1)$ and $c \in \mathbb{N}$) if $\pfc(\tau) \geq \delta |\tau| - c$ for every substring $\tau$ of $X$. The set of all $\langle\delta,c\rangle$-shift complex sequences is denoted by $\shiftcomplex(\delta,c)$.
\item[\ldots] \textdef{$\delta$-shift complex} if $X$ is $\langle\delta,c\rangle$-shift complex for some $c \in \mathbb{N}$. The set of all $\delta$-shift complex sequences is denoted by $\shiftcomplex(\delta)$.
\item[\ldots] \textdef{shift complex} if $X$ is $\delta$-shift complex for some $\delta \in (0,1)$. The set of all shift complex sequences is denoted by $\shiftcomplex$.
\end{description}
\end{definition}

For $f$ an order function satisfying $\limsup_n{(f(n)-n)} < \infty$, we know that $\complex(f) \neq \emptyset$ since $\lambda(\complex(f)) = 1$. In contrast, it is nontrivial that $\shiftcomplex(\delta) \neq \emptyset$. The existence of a $1/3$-shift complex sequence was first shown by Durand, Levin, \& Shen in \cite[Lemma 1]{durand2001complex} in an investigation about complex tilings. Later, in a paper of the same name, Durand, Levin, \& Shen proved more generally the existence of a $\delta$-shift complex sequence for each $\delta \in (0,1)$ (\cite[Lemma 1]{durand2008complex}) by explicitly constructing such a $\delta$-shift complex real segment by segment, using prefix-free symmetry of information at each stage of the construction. In \cite[Proposition 2]{rumyantsev2006forbidden}, Rumyantsev \& Ushakov gave a probabilistic approach to the existence of $\delta$-shift complex sequences, using the Lovasz Local Lemma to prove that arbitrarily long finite $\delta$-shift complex strings (defined in the obvious way) exist and then appealing to compactness to show that a $\delta$-shift complex sequence exists. Yet another approach by Miller \& Khan (\cite[Corollary 2.4]{miller2012two}, \cite[Corollary 3.3]{khan2013shift}) proves existence from the perspective of subshifts. 

In \cite{rumyantsev2011everywhere}, Rumyantsev observes that the proof in \cite[Lemma 1]{durand2008complex} actually gives a stronger existence result:

\begin{thm*} \label{existence of strongly shift complex reals}
\textnormal{\cite[Lemma 1, essentially]{durand2001complex} \cite[Lemma 2]{rumyantsev2011everywhere}}
For each $\delta \in (0,1)$, there exists $X \in \cantor$ and $c \in \mathbb{N}$ such that 
\begin{equation*}
\pfc(\langle X(k),X(k+1),\ldots,X(k+n-1)\rangle \mid k,n) \geq \delta n - c
\end{equation*}
for all $k,n \in \mathbb{N}$.
\end{thm*}

In light of the above theorem and the results \cref{rumyantsev nonnegligible} and \cref{rumyantsev negligible}, we propose the following definition:

\begin{notation}
Suppose $X \in \cantor$ and $i\leq j$ are natural numbers. Define $X(\co{i,j})$ to be the string $\langle X(i), X(i+1),\ldots, X(j-1)\rangle$ of length $j-i$.
\end{notation}

\begin{definition}[strong shift complexity]
$X \in \cantor$ is\ldots
\begin{description}
\item[\ldots] \textdef{strongly $\langle\delta,c\rangle$-shift complex} (where $\delta \in (0,1)$ and $c \in \mathbb{N}$) if $\pfc(X(\co{k,k+n}) \mid k,n) \geq \delta n - c$ for all $k, n \in \mathbb{N}$. The set of all strongly $\langle\delta,c\rangle$-shift complex sequences is denoted by $\stronglyshiftcomplex(\delta,c)$.
\item[\ldots] \textdef{strongly $\delta$-shift complex} if $X$ is strongly $\langle\delta,c\rangle$-shift complex for some $c \in \mathbb{N}$. The set of all strongly $\delta$-shift complex sequences is denoted by $\stronglyshiftcomplex(\delta)$.
\item[\ldots] \textdef{strongly shift complex} if $X$ is strongly $\delta$-shift complex for some $\delta \in (0,1)$. The set of all strongly shift complex sequences is denoted by $\stronglyshiftcomplex$.
\end{description}
\end{definition}

In \cref{shift complexity as a mass problem section}, we examine the mass problems $\shiftcomplex(\delta)$, showing that $\weakdeg(\shiftcomplex(\delta)) \in \mathcal{E}_\weak$ and examining where in $\mathcal{E}_\weak$ the weak degrees $\weakdeg(\shiftcomplex(\delta))$ lie. To that end, we prove:

\begin{repthm}{delta-c-shift complex reals form deep pi01 class}
\textnormal{\cite[Theorem 3, essentially]{rumyantsev2011everywhere}}
$\shiftcomplex(\delta,c)$ is a deep $\Pi^0_1$ class for all rational $\delta \in (0,1)$ and $c \in \mathbb{N}$.
\end{repthm}

Our proof of \cref{delta-c-shift complex reals form deep pi01 class} follows the proof of \cite[Theorem 3]{rumyantsev2011everywhere} and the remark of Bievenu \& Porter that Rumyantsev's proof exhibits the uniformity necessary to prove depth instead of just negligibility, but providing further detail. Another known fact about $\weakdeg(\shiftcomplex(\delta))$ is the following:

\begin{thm*}
\textnormal{\cite[Theorem 6.3]{khan2013shift}}
For each rational $\delta \in (0,1)$ there is an order function $h$ such that $\shiftcomplex(\delta) \weakleq \dnr(h)$.
\end{thm*}

We improve \cite[Theorem 6.3]{khan2013shift} by replacing $\dnr$ with $\ldnr$ and providing explicit bounds:

\begin{repthm}{quantified Khan shift complex theorem}
Given rational numbers $0 < \delta < \alpha < 1$, define $\pi \colon \mathbb{N} \to (0,\infty)$ by $\pi(n) \coloneq 2^{(\alpha - \delta) n}$. Then $\shiftcomplex(\delta) \weakleq \ldnr(q)$ for any order function $q$ such that $q(2^{(n+1)\cdot \pi(n)}) \leq \pi(n)$ for almost all $n \in \mathbb{N}$.
\end{repthm}

\begin{repcor}{quantified Khan shift complex theorem explicit example bound}
Fix a rational $\epsilon > 0$. For all rational $\delta \in (0,1)$ we have $\shiftcomplex(\delta) \weakleq \ldnr(\lambda n. (\log_2 n)^{1-\epsilon})$.
\end{repcor}

In \cref{generazlied shift complexity section}, we provide generalizations of the notions of $\delta$-shift complexity and strong $\delta$-shift complexity in analogy with that of $\complex(f)$, giving rise to the class $\shiftcomplex(f)$ and $\stronglyshiftcomplex(f)$ for $f$ a sub-identical order function. 

In \cref{generalized shift complexity depth section}, we give a detailed proof of a result of Rumyantsev (\cite[Theorem 4]{rumyantsev2011everywhere}) that shows that for certain sufficiently slow-growing order functions $f$, $\shiftcomplex(f)$ is non-negligible. This is strengthened to build a relationship between $\shiftcomplex(f)$ and $\complex(g)$ for another sub-identical order function $g$. 

\begin{repthm}{strong rumyantsev theorem 4}
Suppose $f$ and $g$ are sub-identical order functions such that $\sum_{m=0}^\infty{f(2^m)/2^m} < \infty$ and there is a recursive sequence $\langle \epsilon_m \rangle_{m \in \mathbb{N}}$ of positive rationals for which $\sum_{m=0}^\infty{\epsilon_m} < \infty$ and such that 
\begin{equation*}
\liminf_m{\frac{g(2^m\epsilon_m) - f(2 \cdot 2^m)}{m}} > 1.
\end{equation*}
Then $\shiftcomplex(f) \strongleq \complex(g)$.
\end{repthm}

We also show that for sufficiently slow-growing order functions $f$ there is an order function $g$ such that $\shiftcomplex(f) \strongleq \complex(g)$ and which is \textdef{sublinear} (i.e., $\lim_{n \to \infty}{g(n)/n} = 0$):

\begin{repthm}{sub-identical complex strongly computes shift complex}
Suppose $f$ is a sub-identical order function such that $\sum_{m=0}^\infty{\frac{f(2^m)}{2^m}}$ converges to a recursive real. Then there is a sublinear order function $g$ such that $\shiftcomplex(f) \strongleq \complex(g)$.
\end{repthm}

In contrast, Rumyantsev has shown (\cite[Theorem 5]{rumyantsev2011everywhere}) that for every order function $f$, $\stronglyshiftcomplex(f)$ is negligible. We give a careful and detailed presentation of Rumyantsev's proof to give a slightly stronger result:

\begin{repthm}{rumyantsev negligible}
$\stronglyshiftcomplex(f,c)$ is a deep $\Pi^0_1$ class for any order function $f$ satisfying $\limsup_n{\frac{f(n)}{n}} < 1$ and any $c \in \mathbb{N}$.
\end{repthm}

\section{\texorpdfstring{$\delta$}{delta}-Shift Complexity as a Mass Problem}
\label{shift complexity as a mass problem section}

The weak degrees of $\shiftcomplex(\delta,c)$, $\shiftcomplex(\delta)$, and $\shiftcomplex$ all lie in $\mathcal{E}_\weak$.

\begin{prop} \label{shift complex classes are pi01}
Suppose $\delta \in (0,1)$ is rational and $c \in \mathbb{N}$. Then $\shiftcomplex(\delta,c)$ is $\Pi^0_1$, uniformly in $\delta$, $c$. Consequently, $\weakdeg(\shiftcomplex(\delta,c))$ (for $c$ sufficiently large), $\weakdeg(\shiftcomplex(\delta))$, and $\weakdeg(\shiftcomplex)$ all lie in $\mathcal{E}_\weak$.
\end{prop}
\begin{proof}
Let $U$ be the universal prefix-free machine for which $\pfc = \pfc_U$ and then let $e \in \mathbb{N}$ be such that $\varphi_e(\str^{-1} \sigma) \simeq \str^{-1} U(\sigma)$ for all $\sigma \in \{0,1\}^\ast$. Then 
\begin{equation*}
\shiftcomplex(\delta,c) = \{ X \in \cantor \mid \forall n \forall k \forall \sigma \forall s \qspace ( \varphi_{e,s}(\str^{-1} \sigma) \converge = X \restrict n \to |\sigma| > \delta n - c)\}
\end{equation*}
shows $\shiftcomplex(\delta,c)$ is $\Pi^0_1$. For all sufficiently large $c$, $\shiftcomplex(\delta,c)$ is nonempty, so for such $c$ we have $\weakdeg(\shiftcomplex(\delta,c)) \in \mathcal{E}_\weak$. 

$\shiftcomplex(\delta,c)$ being $\Pi^0_1$ (uniformly in $\delta$ and $c$) implies $\shiftcomplex(\delta) = \bigcup_{c \in \mathbb{N}}{\shiftcomplex(\delta,c)}$ and $\shiftcomplex = \bigcup_{n,c \in \mathbb{N}}{\shiftcomplex(2^{-n},c)}$ are both $\Sigma^0_2$, so  the \nameref{embedding lemma} implies $\weakdeg(\shiftcomplex(\delta))$ and $\weakdeg(\shiftcomplex)$ both lie in $\mathcal{E}_\weak$.
\end{proof}

The same is true of $\stronglyshiftcomplex(\delta,c)$, $\stronglyshiftcomplex(\delta)$, and $\stronglyshiftcomplex$.

\begin{prop} \label{strongly shift complex classes are pi01}
Suppose $\delta \in (0,1)$ is rational and $c \in \mathbb{N}$. Then $\stronglyshiftcomplex(\delta,c)$ is $\Pi^0_1$, uniformly in $\delta$, $c$. Consequently, $\weakdeg(\stronglyshiftcomplex(\delta,c))$ (for $c$ sufficiently large), $\weakdeg(\stronglyshiftcomplex(\delta))$, and $\weakdeg(\stronglyshiftcomplex)$ all lie in $\mathcal{E}_\weak$.
\end{prop}
\begin{proof}
Let $U$ be the universal oracle prefix-free machine for which conditional prefix-free complexity is defined with respect to. Let $e \in \mathbb{N}$ be such that $\varphi_e^\tau(\str^{-1} \sigma) \simeq \str^{-1} U^{\overline{\tau}}(\sigma)$ for all $\sigma \in \{0,1\}^\ast$. Then 
\begin{equation*}
\stronglyshiftcomplex(\delta,c) = \{ X \in \cantor \mid \forall n \forall k \forall \sigma \forall s \qspace ( \varphi_{e,s}^{\str \pi^{(2)}(n,k)}(\str^{-1} \sigma) \converge = X \restrict n \to |\sigma| > \delta n - c)\}
\end{equation*}
shows $\stronglyshiftcomplex(\delta,c)$ is $\Pi^0_1$. For all sufficiently large $c$, $\stronglyshiftcomplex(\delta,c)$ is nonempty, so for such $c$ we have $\weakdeg(\stronglyshiftcomplex(\delta,c)) \in \mathcal{E}_\weak$. 

$\stronglyshiftcomplex(\delta,c)$ being $\Pi^0_1$ (uniformly in $\delta$ and $c$) implies $\stronglyshiftcomplex(\delta) = \bigcup_{c \in \mathbb{N}}{\stronglyshiftcomplex(\delta,c)}$ and $\stronglyshiftcomplex = \bigcup_{n,c \in \mathbb{N}}{\stronglyshiftcomplex(2^{-n},c)}$ are both $\Sigma^0_2$, so  the \nameref{embedding lemma} implies $\weakdeg(\stronglyshiftcomplex(\delta))$ and $\weakdeg(\stronglyshiftcomplex)$ both lie in $\mathcal{E}_\weak$.
\end{proof}

\subsection{Shift Complexity and Depth} \label{shift complexity and depth subsection}

Unlike $\complex(\delta,c)$, $\shiftcomplex(\delta,c)$ is a deep $\Pi^0_1$ class for every computable $\delta \in (0,1)$ and $c \in \mathbb{N}$. Our proof is essentially a more detailed presentation of the proof given by Rumyantsev in \cite[Theorem 3, essentially]{rumyantsev2011everywhere} plus the uniformity observation by Bienvenu \& Porter given in \cite{bienvenu2016deep}.

\begin{thm} \label{delta-c-shift complex reals form deep pi01 class}
\textnormal{\cite[Theorem 3, essentially]{rumyantsev2011everywhere}}
$\shiftcomplex(\delta,c)$ is a deep $\Pi^0_1$ class for all rational $\delta \in (0,1)$ and $c \in \mathbb{N}$.
\end{thm}

To prove \cref{delta-c-shift complex reals form deep pi01 class}, we make use of the following probabilistic lemma:

\begin{lem} \label{rumyantsev lemma 6}
\textnormal{\cite[Lemma 6]{rumyantsev2011everywhere}}
Suppose $\delta \in (0,1)$ is rational. For every rational $\epsilon>0$ and $n_0 \in \mathbb{N}$ there exist natural numbers $n<N$ and random variables $\mathcal{A}_n, \mathcal{A}_{n+1}, \ldots, \mathcal{A}_N$ such that
\begin{enumerate}[(i)]
\item $\mathcal{A}_i$ is a subset of $\{0,1\}^i$ of size at most $2^{\delta i}$,
\item for every string $\sigma \in \{0,1\}^N$, the probability that $\sigma$ has no substring in $\bigcup_{i=n}^N{\mathcal{A}_i}$ is less than $\epsilon$, and
\item $n\geq n_0$.
\end{enumerate}
Moreover, the natural numbers $n,N \in \mathbb{N}$ and the (probability distributions of the) random variables $\mathcal{A}_n, \mathcal{A}_{n+1}, \\ \ldots, \mathcal{A}_N$ can be found effectively as functions of $\epsilon$ and $n_0$.
\end{lem}
\begin{proof}
Let $m\geq 2$ be such that $\delta > \frac{1}{m}$. We define natural numbers 
\begin{equation*}
n=n_1 < n_2 < \cdots < n_m=N
\end{equation*} 
satisfying the following properties:
\begin{enumerate}[(i)]
\item $n_k$ divides $n_{k+1}$ for all $k \in \{1,\ldots,m-1\}$. 
\item $\delta n_k \in \mathbb{N}$ for all $k \in \{1,\ldots,m\}$. (Assuming (i) holds, it suffices for $\delta n$ to be a natural number.)
\end{enumerate}
When $i$ is not of the form $n_k$ for $k \in \{1,\ldots,m\}$, we will define $\mathcal{A}_i$ to take the constant value $\emptyset$, leaving only $\mathcal{A}_{n_k}$ to define. Roughly speaking, we will define the random variables $\mathcal{A}_{n_k}$ to address strings $\sigma$ which exhibit relatively few substrings of increasingly greater lengths.

Suppose $\sigma$ is a string of length $p$, $q$ divides $p$, and $\alpha \in (0,1)$ satisfies $\alpha q \in \mathbb{N}$. If $\sigma$ has less than $2^{\alpha q}$ substrings of length $q$, then we may encode $\sigma$ in the following way. Writing $\sigma = \sigma_1\concat \sigma_2 \concat \cdots \concat \sigma_{p/q}$ where $|\sigma_k| = q$, our hypothesis implies $|\{\sigma_1,\sigma_2,\ldots,\sigma_{p/q}\}| < 2^{\alpha q}$ (our hypothesis is much stronger than this, but is the natural hypothesis for the remaining arguments in the proof). Let $\rho$ be a string encoding, in lexicographical order, the distinct elements of $\{\sigma_1,\sigma_2,\ldots,\sigma_{p/q}\}$; for exactness, we may encode a finite sequence $\nu_1,\ldots,\nu_n$ of strings by
\begin{equation*}
\nu_1^0 \concat \langle 1,1 \rangle \concat \nu_2^0 \concat \langle 1,1\rangle \concat \cdots \concat \nu_n^0 \concat \langle 1,1\rangle
\end{equation*}
where $\nu_i^0$ denotes the string $\langle \nu_i(0), 0, \nu_i(1), 0, \nu_i(2), 0, \ldots,\nu_i(|\nu_i|-1), 0\rangle$. Then we may encode $\sigma$ by
\begin{equation*}
\tau = \tau_1 \concat \tau_2 \concat \cdots \concat \tau_{p/q} \concat \rho
\end{equation*}
where $\tau_i$ is the binary representation of the index of $\sigma_i$ in $\rho$ (regarding $\rho$ as a finite sequence of strings). It is convenient to ensure $\tau$ has a length depending only on $p$ and $q$, so we appropriately pad each $\tau_i$ by $0$'s to ensure that $|\tau_i|= \alpha q$ ($\rho$ encodes at most $2^{\alpha q}-1$ strings and hence requires at most $\alpha q$ bits to describe) and appropriately pad $\rho$ by $1$'s to ensure that $|\rho| = 2(q+1)(2^{\alpha q}-1)$ (the length of $\rho$ if it encodes the maximum number of $2^{\alpha q}-1$ strings). Thus, to each such string $\sigma$ of length $p$ we associate with it a unique string $\tau$ of length $\alpha p + 2(q+1)(2^{\alpha q}-1)$. 

Supposing $n_k$ has already been defined and $n_k$ divides $i$, we say that a string $\sigma$ of length $i$ is \emph{$k$-sparse} if it has less than $2^{\frac{m-k}{m}n_k}$ substrings of length $n_k$. As observed above, we may encode a $k$-sparse $\sigma$ with a string of length $\frac{m-k}{m}|\sigma|+2(n_k+1)(2^{\frac{m-k}{m}n_k}-1)$. For simplicity, we write $c_{k+1} \coloneq 2(n_k+1)(2^{\frac{m-k}{m}n_k}-1)$. 

We may now define $n_k$ and the associated random variable $\mathcal{A}_{n_k}$ for $k \in \{1,\ldots,m\}$. 
\begin{description}
\item[$k=1$.] Let $n_1=n$ be the least natural number greater than $n_0$ such that $\delta n \in \mathbb{N}$ and
\begin{equation*}
\left(1-\frac{1}{2^{n/m}}\right)^{2^{\delta n}} < \epsilon.
\end{equation*}
Note that such an $n$ exists as
\begin{equation*}
\left(1-\frac{1}{2^{n/m}}\right)^{2^{\delta n}} = \left(1-\frac{1}{2^{n/m}}\right)^{2^{n/m}\cdot 2^{(\delta - 1/m)n}} = \left(\left(1-\frac{1}{2^{n/m}}\right)^{2^{n/m}}\right)^{2^{(\delta-1/m)n}} \approx \left(\frac{1}{e}\right)^{2^{(\delta-1/m)n}}
\end{equation*}
with both the approximation getting tighter and the final expression tending toward $0$ as $n \to \infty$. 

Then $\mathcal{A}_{n_1}=\mathcal{A}_n$ is defined to be randomly chosen uniformly among all subsets of $\{0,1\}^n$ of size $2^{\delta n}$ (i.e., each such subset of $\{0,1\}^n$ of size $2^{\delta n}$ has an equal probability $\binom{2^n}{2^{\delta n}}$ of being the value of $\mathcal{A}_n$). 

\item[$1<k<m$.] Suppose $n_{k-1}$ has been defined. Then $n_k$ is the least multiple of $n_{k-1}$ such that
\begin{equation*}
\left(1-\frac{1}{2^{n_k/m+c_k}}\right)^{2^{\delta n_k}} = \left(\left(1-\frac{1}{2^{n_k/m+c_k}}\right)^{2^{n_k/m+c_k}}\right)^{2^{(\delta -1/m)n_k-c_k}} < \epsilon.
\end{equation*}
(That such an $n_k$ exists is analogous to the case where $k=1$.)

Then $\mathcal{A}_{n_k}$ is defined to be randomly chosen uniformly among all subsets of $\{0,1\}^{n_k}$ of size $2^{\delta n_k}$ which consist only of $(k-1)$-sparse strings. 

\item[$k=m$.] Let $n_k=n_m=N$ to be the least multiple of $n_{m-1}$ such that $\frac{1}{m}N+c_m < \delta N$ (by hypothesis, $1/m<\delta$ and $c_m$ is constant with respect to $N$, so there is such an $N$). 

Then $\mathcal{A}_N$ is defined to be constantly equal to the set of all $(m-1)$-sparse strings of length $N$ (note that an $(m-1)$-sparse string is described uniquely by a string of length $\frac{m-(m-1)}{m}N+c_m=\frac{1}{m}N+c_m$, so $|\mathcal{A}_N| \leq 2^{N/m+c_m} < 2^{\delta N}$).

\end{description}

Finally, we show that for every $\sigma \in \{0,1\}^N$,
\begin{equation*}
\Prob(\textsf{$\mathcal{A}_i$ has no substring of $\sigma$ for all $i \in \{n,\ldots,N\}$}) < \epsilon.
\end{equation*}
It suffices to show that $\Prob(\textsf{$\mathcal{A}_i$ has no substring of $\sigma$}) < \epsilon$ for at least one $i \in \{n,\ldots,N\}$. 
\begin{description}
\item[Case 1:] Suppose $\sigma$ is not $1$-sparse, so that $\sigma$ has at least $2^{\frac{m-1}{m} n}$ substrings of length $n$. Because of the definition of the output distribution of $\mathcal{A}_n$, the probability that $\sigma$ has no substring in $\mathcal{A}_n$ is at most the probability that $\sigma$ has no substring among $2^{\delta n}$ strings chosen uniformly and indepedently at random (the latter probability may be higher because we allow duplicates). The independence and uniformity of those $2^{\delta n}$ random choices means that the latter probability is equal to $\Prob(\textsf{$\tau \in \{0,1\}^n$ is not a substring of $\sigma$})^{2^{\delta n}}$. $\Prob(\textsf{$\tau \in \{0,1\}^n$ is not a substring of $\sigma$})$ is at most the probability that a random $\tau \in \{0,1\}^n$ (chosen uniformly) is not in a set of size $2^{\frac{m-1}{m} n}$. Thus,
\begin{align*}
& \Prob(\textsf{$\mathcal{A}_n$ has no substring of $\sigma$}) \\
& \quad \leq \Prob(\textsf{$2^{\delta n}$ random strings of length $n$ are not substrings of $\sigma$}) \\
& \quad \leq \Prob(\textsf{random string of length $n$ is not substring of $\sigma$})^{2^{\delta n}} \\
& \quad \leq \left( 1-\frac{2^{\frac{m-1}{m}n}}{2^n}\right)^{2^{\delta n}} \\
& \quad = \left( 1- \frac{1}{2^{n/m}}\right)^{2^{\delta n}} \\
& \quad < \epsilon.
\end{align*}

\item[Case 2:] Suppose $\sigma$ is $1$-sparse but is not $k$-sparse for some $k\in \{2,\ldots,m-1\}$. Assume $k$ is minimal with that property, so $\sigma$ is not $k$-sparse (and hence has at least $2^{\frac{m-k}{m}n_k}$ substrings of length $n_k$) but is $(k-1)$-sparse. As in Case 1, we have 
\begin{align*}
& \Prob(\textsf{$\mathcal{A}_{n_k}$ has no substring of $\sigma$}) \\
& \quad \leq \Prob(\textsf{$2^{\delta n_k}$ random $(k-1)$-sparse strings of length $n_k$ are not substrings of $\sigma$}) \\
& \quad \leq \Prob(\textsf{random $(k-1)$-sparse string of length $n_k$ is not substring of $\sigma$})^{2^{\delta n_k}}.
\intertext{The probability $\Prob(\textsf{random $(k-1)$-sparse string of length $n_k$ is not substring of $\sigma$})$ is of the form \newline $\Prob(E\setminus F)$, where $E = \{ \tau \in \{0,1\}^{n_k} \mid \text{$\tau$ is $(k-1)$-sparse}\}$ and $F = \{ \tau \in \{0,1\}^{n_k} \mid \text{$\tau$ is a substring of $\sigma$}\}$. Observe that $F \subseteq E$: if a substring of $\sigma$ is not $(k-1)$-sparse, then it contains at least $2^{\frac{m-(k-1)}{m}n_{k-1}}$ substrings of length $n_{k-1}$, and hence $\sigma$ does as well, contrary to the hypothesis that $\sigma$ is $(k-1)$-sparse. Because $E$ is finite, we have $\Prob(E \setminus F) = 1-\frac{|F|}{|E|}$. To get an upper bound on $\Prob(E\setminus F)$, it suffices to have an upper bound on $|E|$ and a lower bound on $|F|$. Thus,}
& \Prob(\textsf{random $(k-1)$-sparse string of length $n_k$ is not substring of $\sigma$})^{2^{\delta n_k}} \\
& \quad \leq \left( 1- \frac{2^{\frac{m-k}{m}n_k}}{2^{\frac{m-k+1}{m}n_k+c_k}}\right)^{2^{\delta n_k}} \\
& \quad = \left( 1- \frac{1}{2^{n_k/m+c_k}}\right)^{2^{\delta n_k}} \\
& \quad < \epsilon.
\end{align*}

\item[Case 3:] Suppose $\sigma$ is $k$-sparse for all $k \in \{1,\ldots,m-1\}$. In particular, $\sigma$ is $(m-1)$-sparse and so an element of $\mathcal{A}_N$. Thus, $\Prob(\textsf{$\mathcal{A}_N$ has no substring of $\sigma$}) = 0 < \epsilon$.

\end{description}
\end{proof}

\begin{proof}[Proof of \cref{delta-c-shift complex reals form deep pi01 class}.]

With $\mathbf{M}$ a universal left r.e.\ continuous semimeasure on $\{0,1\}^\ast$, by \cref{levin and zvonkin} there is a partial recursive functional $\Psi$ such that for every $\sigma \in \{0,1\}^\ast$, 
\begin{equation*}
\mathbf{M}(\sigma) = \lambda(\Psi^{-1}(\sigma)) = \lambda(\{ X \in \cantor \mid \Psi^X \supseteq \sigma\}).
\end{equation*}

Say that $Y \in \{0,1\}^\ast \cup \cantor$ \emph{avoids} a $k$-tuple of finite sets of strings $\langle A_1,A_2,\ldots,A_k \rangle$ if $|Y| \geq \max\{ |\sigma| \mid \sigma \in \bigcup_{i=1}^k{A_i}\}$ and no substring of $Y$ is an element of $\bigcup_{i=1}^k{A_i}$. Finite sets of strings in $\{0,1\}^\ast$ are implicitly \godel\ numbered by some recursive bijection $\mathcal{P}_\fin(\{0,1\}^\ast) \to \mathbb{N}$.

Our approach, roughly, involves us finding natural numbers $n<N$ and sets $A_n,A_{n+1},\ldots,A_N$ satisfying the following conditions:
\begin{enumerate}[(i)]
\item $A_i$ is a subset of $\{0,1\}^i$ of size at most $2^{\delta i}$ for each $i \in \{n,n+1,\ldots,N\}$ and
\item $\lambda(\{ X \in \cantor \mid \text{$\Psi(X)$ avoids $A_n,A_{n+1},\ldots,A_N$}\}) < \epsilon$.
\end{enumerate}
The claim is then that each element $\tau$ of $\bigcup_{i=n}^N{A_i}$ satisfies $\pfc(\tau) < \delta |\tau|-c$, so $\Psi(X)$ being $\langle\delta,c\rangle$-shift complex implies that $\Psi(X)$ avoids the sets $A_n,A_{n+1},\ldots,A_N$. Then 
\begin{equation*}
\lambda(\{ X \in \cantor \mid \Psi(X) \in \shiftcomplex(\delta,c)\}) \leq \lambda(\{ X \in \cantor \mid \text{$\Psi(X)$ avoids $A_n,A_{n+1},\ldots,A_N$} \}) < \epsilon.
\end{equation*} 

There are two issues that we must work around, the first being that $\Psi(X)$ need not be an element of $\cantor$, and the second of which is that an element $\tau$ of $\bigcup_{j=n}^N{A_j}$ need not necessarily satisfy $\pfc(\tau) < \delta |\tau| - c$.

We start with addressing the first issue. Our use of avoidance only requires that $|\Psi^X| \geq N$. \cref{rumyantsev lemma 6} shows that as a recursive function of $\langle n_0,m\rangle \in \mathbb{N}^2$, we can find natural numbers $n < N$ and random variables $\mathcal{A}_n,\mathcal{A}_{n+1},\ldots,\mathcal{A}_N$ such that:
\begin{enumerate}[(i)]
\item $n \geq n_0$,
\item $\mathcal{A}_k$ is a subset of $\{0,1\}^k$ of size at most $2^{\delta k /3}$ (the use of $\delta/3$ instead of $\delta$ will become apparent when dealing with the second issue) for each $k \in \{n,n+1,\ldots,N\}$, and
\item for every string $\sigma \in \{0,1\}^N$, the probability that $\sigma$ has no substring in $\bigcup_{i=n}^N{\mathcal{A}_i}$ is less than $2^{-m}$.
\end{enumerate}
Fix $n_0$ and $m$ and let $n<N$ and $\mathcal{A}_n,\mathcal{A}_{n+1},\ldots,\mathcal{A}_N$ be as above. 

Let $S \coloneq \{ X \in \cantor \mid |\Psi^X| \geq N\}$.
$S$ is $\Sigma^0_1$, so there exists a recursive sequence $\langle \sigma_n\rangle_{n\in\mathbb{N}}$ of pairwise incompatible strings such that $S = \bigcup_{n\in\mathbb{N}}{\bbracket{\sigma_n}_2}$. Let $\alpha \coloneq \lambda(S)$, so that $\langle \lambda(\bigcup_{k \leq n}{\bbracket{\sigma_k}_2}) \rangle_{n \in \mathbb{N}}$ is a recursive sequence converging monotonically to $\alpha$ from below. For $i \in \mathbb{N}$, let $\alpha_i \coloneq i \cdot 2^{-m}/3$, and let $i_0$ be the largest $i$ for which $\alpha_i < \alpha$, so that $\alpha - \alpha_{i_0} < 2^{-(m+1)}$. Finally, define
\begin{equation*}
\tilde{S} \coloneq \bigcup_{k \leq p}{\bbracket{\rho_k}_2}
\end{equation*}
where $p$ is the smallest natural number for which $\lambda(\tilde{S}) \geq \alpha_{i_0}$. $\tilde{S}$ is a recursive subset of $S$ and $\lambda(S \setminus \tilde{S}) < 2^{-(m+1)}$. By virtue of being a subset of $S$, $|\Psi^X| \geq N$ for all $X \in \tilde{S}$. $\tilde{S}$ is recursive and an index for $\tilde{S}$ can be computed from $n_0$, $m$, and $i_0$. Define a probability measure $\mu$ on $\{0,1\}^N$ by setting
\begin{equation*}
\mu(\{\sigma\}) \coloneq \lambda(\tilde{S})^{-1} \cdot \lambda(\{X\in \cantor \mid \text{$X \in \tilde{S}$ and $\Psi^X \supseteq \sigma$}\})
\end{equation*}
for each $\sigma \in \{0,1\}^N$. $\mu$ is a computable measure and an index for $\mu$ can be computed from $n_0$, $m$, and $i_0$.

Let $\nu$ be the probability measure on $\mathcal{P}(\{0,1\}^n) \times \mathcal{P}(\{0,1\}^{n+1}) \times \cdots \times \mathcal{P}(\{0,1\}^N)$ defined by
\begin{equation*}
\nu(\{ \langle A_n,A_{n+1},\ldots,A_N\rangle \}) \coloneq \Prob(\textsf{$\mathcal{A}_i = A_i$ for each $i\in\{n,n+1,\ldots,N\}$})
\end{equation*}
(i.e., the joint probability distribution made up of the output distributions of the random variables $\mathcal{A}_n,\mathcal{A}_{n+1}, \\ \ldots,\mathcal{A}_N$). Write
\begin{equation*}
E \coloneq \{ \langle X, \langle A_n,\ldots,A_N \rangle \rangle \in \cantor \times \prod_{i=n}^N{\mathcal{P}(\{0,1\}^i)} \mid \text{$X \in \tilde{S}$ and $\Phi^X$ avoids $A_n,\ldots,A_N$}\}.
\end{equation*}
By Fubini's Theorem,
\begin{align*}
\int{\left(\int{\chi_E(X,(A_n,\ldots,A_N)) \dd \lambda}\right)\dd \nu} & = (\lambda \times \nu)(E) \\
& = \int{\left(\int{\chi_E(X,(A_n,\ldots,A_N))\dd \nu}\right)\dd \lambda} \\
& \leq \int{2^{-(m+1)} \dd \lambda} \\
& = 2^{-(m+1)}.
\end{align*}
If $\int{\chi_E(X,(A_n,\ldots,A_N))\dd \lambda} \geq 2^{-(m+1)}$ for every $\langle A_n,\ldots,A_N\rangle \in \prod_{j=n}^N{\mathcal{P}(\{0,1\}^j)}$, we reach a contradiction. Thus, there is a least one tuple $\langle A_n,\ldots,A_N\rangle \in \prod_{i=n}^N{\mathcal{P}(\{0,1\}^n)}$ with the desired property, i.e., that
\begin{equation*}
\lambda(\{ X \in \cantor \mid \text{$X \in \tilde{S}$ and $\Psi^X$ avoids $A_n,\ldots,A_N$}\}) < 2^{-(m+1)},
\end{equation*}
and hence
\begin{equation*}
\lambda(\{ X \in \cantor \mid \text{$X \in S$ and $\Psi^X$ avoids $A_n,\ldots,A_N$}\}) < 2^{-(m+1)} + 2^{-(m+1)} = 2^{-m}.
\end{equation*}

Let $c_1 \in \mathbb{N}$ be such that $\pfc(n) \leq 2\log_2 n + c_1$ for all $n \in \mathbb{N}$. The recursiveness of $\tilde{S}$ allows us to effectively find such a tuple $\langle A_n,\ldots,A_N\rangle$ for which $\lambda(\{ X \in \cantor \mid \text{$X \in \tilde{S}$ and $\Psi^X$ avoids $\langle A_n,\ldots,A_N \rangle$}\}) < 2^{-(m+1)}$. As noted before, an index for $\tilde{S}$ can be found effectively from $n_0$, $m$, and $i$. As such, there is $c_1'$ such that 
\begin{equation*}
\pfc(A_n,\ldots,A_N) \leq \pfc(n_0) + \pfc(m) + \pfc(i) + c_2'.
\end{equation*}
Although $i$ cannot be found recursively from $n_0$ and $m$ in general, we regardless have the bound have the bound $i \leq (2^{-m}/3)^{-1} = 3 \cdot 2^m$. Thus,
\begin{equation*}
\pfc(i) \leq \max_{0 \leq k \leq \lfloor 3 \cdot 2^m\rfloor}{\pfc(k)} \leq \max_{0\leq k \leq \lfloor 3/\epsilon\rfloor}{(2\log_2 k+c_1)} \leq 2m + 2\log_2 3 + c_1.
\end{equation*}
Let $c_2 = 2\log_2 3 + c_1 + c_2'$, so that for every $\langle n_0,m\rangle$ we have
\begin{equation*}
\pfc(A_n,\ldots,A_N) \leq \pfc(n_0) + \pfc(m) + 2m + c_2.
\end{equation*}

Now we address the second issue. Let $c_3 \in \mathbb{N}$ be such that, where $\tau_i$ is the $i$-th element of $A_{|\tau_i|}$ ordered lexicographically, 
\begin{equation*}
\pfc(\tau_i) \leq \pfc(|\tau_i|) + \pfc(i) + \pfc(A_n,A_{n+1},\ldots,A_N) + c_3.
\end{equation*}
Then for any $\tau \in \bigcup_{j=n}^N{A_j}$ we have
\begin{align*}
\pfc(\tau) & \leq \pfc(|\tau|) + \pfc(\text{index of $\tau$ in $A_{|\tau|}$}) + \pfc(A_n,\ldots,A_N) + c_2 \\
& \leq (2\log_2 |\tau| + c_1) + (2\log_2 (\text{index of $\tau$ in $A_{|\tau|}$}) + c_1) + (\pfc(n_0) + \pfc(m) + 2m + c_2) + c_3 \\
& \leq (2\log_2 |\tau| + c_1)  + ((2/3)\delta |\tau| + c_1) + (2\log_2 n_0 + 2 \log_2 m + 2m + 2c_1 + c_2) + c_3 \\
& = \frac{2}{3}\delta |\tau| + 2\log_2 |\tau| + 2\log_2 n_0 + 2\log_2 m + 2m + (4c_1 + c_2 + c_3).
\end{align*}
Note that $d \coloneq 4c_1 + c_2 + c_3$ is independent of $\langle n_0,m\rangle$. For $|\tau|$ sufficiently large, $\frac{2}{3}\delta |\tau| + 2\log_2 |\tau| + 2\log_2 n_0 + 2\log_2 m + 2m + d < \delta |\tau| - c$. Thus, define $n_0 = n_0(m)$ to be the least $n$ such that $\frac{2}{3}\delta n + 4\log_2 n + 2\log_2 m + 2m + d < \delta n - c$. Finally, define $r\colon \mathbb{N} \to \mathbb{N}$ by $r(m) \coloneq N(n_0(m),m)$. 

If $\Psi(X)$ is $\langle\delta,c\rangle$-shift complex, then $\Psi(X) = \Psi^X \in \cantor$ (therefore $X \in S$) and $\pfc(\tau) \geq \delta |\tau| - c$ for every substring $\tau$ of $\Psi(X)$. 
Then
\begin{align*}
\mathbf{M}(\shiftcomplex(\delta,c) \restrict r(m)) & = \lambda(\{X\in\cantor \mid \text{$|\Psi^X| \geq r(m)$ and $\Psi^X\restrict r(m)$ is $\langle\delta,c\rangle$-shift complex}\}) \\
& \leq \lambda(\{ X\in\cantor \mid \text{$|\Psi^X| \geq r(m)$ and $\Psi^X\restrict r(m)$ avoids $A_n,\ldots,A_N$}\}) \\
& < \frac{1}{2^m}.
\end{align*}
Hence, $\shiftcomplex(\delta,c)$ is deep.
\end{proof}

\begin{cor} \label{difference randoms do not compute shift complex sequences}
No difference random computes a shift complex sequence. Consequently, $\shiftcomplex \weaknleq \mlr$.
\end{cor}

\begin{remark}
$X \in \cantor$ is \emph{Kurtz random} if $X \notin P$ for any $\Pi^0_1$ class $P$ with $\lambda(P) = 0$. \cite[Theorem 6.7]{khan2013shift} shows that for every $\delta \in (0,1)$ there is a $Y \in \shiftcomplex(\delta)$ such that $Y$ computes no Kurtz random. Every Martin-\Lof\ random sequence is Kurtz random, so this shows $\mlr \weaknleq \shiftcomplex(\delta)$ for every $\delta \in (0,1)$.
\end{remark}

\subsection{Shift Complexity and Avoidance} \label{shift complexity and avoidance section}

\cite[Theorem 6.3]{khan2013shift} shows that for each rational $\delta \in (0,1)$ there is an order function $h$ such that $\shiftcomplex(\delta) \weakleq \dnr(h)$. Using the same techniques as we employed in order to quantify the growth rate of $q$ in \cref{greenberg miller theorem 4.9 improved}, we can similarly strengthen \cite[Theorem 6.3]{khan2013shift} by giving explicit bounds.

\begin{thm} \label{quantified Khan shift complex theorem}
Given rational numbers $0 < \delta < \alpha < 1$, define $\pi \colon \mathbb{N} \to (0,\infty)$ by $\pi(n) \coloneq \exp_2((\alpha - \delta) n)$. Then $\shiftcomplex(\delta) \weakleq \ldnr(q)$ for any order function $q$ such that $q(\exp_2((n+1)\cdot \pi(n))) \leq \pi(n)$ for almost all $n \in \mathbb{N}$.
\end{thm}

We may also find an order function $q$ which works for all $\delta \in (0,1)$:

\begin{cor} \label{quantified Khan shift complex theorem explicit example bound}
Fix a rational $\epsilon > 0$. For all rational $\delta \in (0,1)$ we have $\shiftcomplex(\delta) \weakleq \ldnr(\lambda n. (\log_2 n)^{1-\epsilon})$.
\end{cor}
\begin{proof}
Define $q$ by $q(n) \coloneq (\log_2 n)^{1-\epsilon}$ for $n \geq 2$ and $q(0)=q(1)=1$. 
Let $\alpha$ be any rational such that $\delta < \alpha < 1$. Then 
\begin{equation*}
(\log_2 2^{(n+1)\cdot \pi(n)})^{1-\epsilon} = \bigl( (n+1)\cdot 2^{(\alpha - \delta) \cdot n}\bigr)^{1-\epsilon} = (n+1)^{1-\epsilon} \cdot 2^{(\alpha - \delta) \cdot (1-\epsilon) \cdot n} < 2^{(\alpha - \delta) \cdot n}
\end{equation*}
for almost all $n$, so $\shiftcomplex(\delta) \weakleq \ldnr(q)$ by \cref{quantified Khan shift complex theorem}.
\end{proof}

\begin{proof}[Proof of \cref{quantified Khan shift complex theorem}.]
The main idea of the proof follows that of Khan \& Miller in their proof of \cite[Theorem 6.3]{khan2013shift}.

Let $B_n \coloneq \{ \sigma \in \{0,1\}^n \mid \pfc(\sigma) < \delta n\}$ and $\mathcal{S}_n \coloneq \{ S_n \subseteq \{0,1\}^n \mid B_n \subseteq S_n \wedge |S_n| \leq 2^{\alpha \cdot n}\}$. Fix an admissible enumeration $\varphi_\bullet$ and recall that we previously defined
\begin{equation*}
P_a^{b,c} \coloneq \{ F\colon\mathbb{N} \to [a]^b \mid \forall n \forall j < c \qspace (j \in \dom \varphi_n \to \varphi_n(j) \notin F(n))\}
\end{equation*}
where $[a]^b \coloneq \{ S \subseteq \{0,1,2,\ldots,a-1\} \mid |S| = b\}$. 
As in \cref{quantifying greenberg miller proof}, we identify $\dnr(a)$ with $P_a^{1,1}$.

\cref{greenberg miller reduction} shows that, uniformly in $n$, there is a recursive functional $\Psi_n \colon \dnr(\pi(n)) \to \\ P_{2^n}^{2^n-2^{\alpha\cdot n}+1, \pi(n)}$ and recursive function $U_n \colon \mathbb{N} \to \mathcal{P}_\fin(\mathbb{N})$ such that for any $X \in \dnr(\pi(n))$ and $i \in \mathbb{N}$, $X \restrict U_n(i)$ determines $\Psi_n(X)(i)$ and $|U_n(i)| \leq 2^{\delta n}\cdot \binom{2^n}{\pi(n)}$. Given $X \in \dnr(\pi(n))$, $G \coloneq \Psi_n(X)$ is a function $G \colon \mathbb{N} \to [2^n]^{2^n-2^{\alpha \cdot n}+1}$ such that $\varphi_m(j) \notin G(m)$ for all $j < 2^{\delta n}$ and all $m \in \mathbb{N}$. $B_n$ is of cardinality $|B_n| \leq 2^{\delta n}$ and uniformly recursively enumerable, so let $s \colon \mathbb{N} \to \mathbb{N}$ be a primitive recursive function for which $B_n = \{ \varphi_{s(n)}(j) \mid j < 2^{\delta n}\}$ for all $n \in \mathbb{N}$. Then $B_n \cap G(s(n)) = \emptyset$, so letting $S_n \coloneq \{0,1\}^n \setminus G(s(n)) \in \mathcal{S}_n$. Moreover, this process is uniform in $n$ and depends only on $X \restrict U_n(s(n))$.

Define $U \colon \mathbb{N} \to \mathcal{P}_\fin(\mathbb{N})$ by $U(n) \coloneq U_n(s(n))$ for each $n \in \mathbb{N}$ and subsequently define $\overline{u} \colon \mathbb{N} \to \mathbb{N}$ recursively by
\begin{align*}
\overline{u}(0) & \coloneq 0, \\
\overline{u}(n+1) & \coloneq \overline{u}(n) + |U(n)|.
\end{align*}
Finally, define $\psi \colonsub \mathbb{N} \to \mathbb{N}$ by letting
\begin{equation*}
\psi(\overline{u}(n)+j) \simeq \varphi_{\text{$j$-th element of $U(n)$}}(0)
\end{equation*}
for each $n \in \mathbb{N}$ and $j < |U(n)|$.
By construction, uniformly in $n \in \mathbb{N}$ and $X \in \avoid^\psi(\pi(n))$, $X \restrict \overline{u}(n+1)$ can be used to compute an element of $S_n \in \mathcal{S}_n$.

If $p \colon \mathbb{N} \to \mathbb{N}$ is an order function satisfying
\begin{equation*}
p(\overline{u}(n+1)) \leq \pi(n)
\end{equation*}
for all $n \in \mathbb{N}$, then uniformly in $n$ and $X \in \avoid^\psi(p)$, $X \restrict \overline{u}(n+1)$ can be used to compute an element of $S_n \in \mathcal{S}_n$. In other words, for such order functions $p$, we have $\prod_{n \in \mathbb{N}}{\mathcal{S}_n} \weakleq \avoid^\psi(p)$. \cite[Corollary 6.9]{khan2013shift} shows that $\shiftcomplex(\delta) \weakleq \prod_{n \in \mathbb{N}}{\mathcal{S}_n}$, so $\shiftcomplex(\delta) \weakleq \avoid^\psi(p)$.

Suppose $q \colon \mathbb{N} \to \mathbb{N}$ is an order function such that for each $a,b \in \mathbb{N}$ we have $q(an+b) \leq p(n)$ for almost all $n \in \mathbb{N}$.Then $\avoid^\psi(p) \strongleq \ldnr(q)$. In other words, if $q$ is such that for all $a,b \in \mathbb{N}$ we have $q(a\overline{u}(n+1)+b) \leq \pi(n)$ for almost all $n \in \mathbb{N}$, then $\shiftcomplex(\delta) \weakleq \ldnr(q)$. We show that if $q(\exp_2((n+1) \cdot \pi(n))) \leq \pi(n)$ for almost all $n$, then $q$ satisfies this aforementioned condition.

First, we find an upper bound of $\overline{u}$. Recall that $|U_n(i)| \leq 2^{\delta n}\cdot \binom{2^n}{\pi(n)}$ for all $n,i \in \mathbb{N}$, so $|U(n)| = |U_n(s(n))| \leq 2^{\delta n}\cdot \binom{2^n}{\pi(n)}$. Next, 
$\overline{u}(n+1) = \sum_{m \leq n}{|U(m)|} \leq (n+1)\cdot |U(n)|$, so for almost all $n \in \mathbb{N}$,
\begin{equation*}
a\cdot \overline{u}(n+1)+b \leq a(n+1) \cdot |U(n)| + b \leq 3a \cdot n \cdot 2^{\delta n}\cdot \binom{2^n}{\pi(n)} \leq 3a \cdot n \cdot 2^{\delta n}\cdot (2^n)^{\pi(n)} \leq 2^{(n+1) \cdot \pi(n)}.
\end{equation*}
Thus, if $q(2^{(n+1)\cdot \pi(n)}) \leq \pi(n)$ for almost all $n$, then $\shiftcomplex(\delta) \weakleq \ldnr(q)$.
\end{proof}

Because $\shiftcomplex(\delta,c)$ is deep for each rational $\delta \in (0,1)$ and $c \in \mathbb{N}$, if $\shiftcomplex(\delta) \weakleq \ldnr(q)$ then $q$ must be slow-growing. In \cref{SC not weakly below ldnr_slow}, we shall show that $\shiftcomplex \weaknleq \ldnr_\slow$, showing that there are is a slow-growing order function $q$ and an $X \in \ldnr(q)$ such that $X$ computes no shift complex sequence.

\section{Generalized Shift Complexity}
\label{generazlied shift complexity section}

Just as $\complex(f)$ for $f$ a sub-identical order function generalizes the case of $\complex(\delta)$ for $\delta \in \oc{0,1}$, we can replace the map $\tau \mapsto \delta |\tau|$ with any sub-identical order function $\tau \mapsto f(|\tau|)$. Sequences $X$ satisfying $\pfc(\tau) \geq f(|\tau|) - c$ for all substrings $\tau$ of $X$ or $\pfc(X (\co{k,k+n}) \mid k,n) \geq f(|\tau|) - c$ for all $k,n\in\mathbb{N}$ have been considered by Rumyantsev \cite{rumyantsev2011everywhere}, but not give any explicit name to those properties, nor does there seem to be any existing terminology in the literature in that direction. Thus, we propose the following definitions:  

Let $f \colon \mathbb{N} \to \co{0,\infty}$ be an order function.

\begin{definition}[$f$-shift complexity]
$X \in \cantor$ is \ldots
\begin{description}
\item[\ldots] \textdef{$\langle f,c\rangle$-shift complex} if $\pfc(\tau) \geq f(|\tau|) - c$ for every substring $\tau$ of $X$. The set of all $\langle f,c\rangle$-shift complex sequences is denoted by $\shiftcomplex(f,c)$.
\item[\ldots] \textdef{$f$-shift complex} if $X$ is $\langle f,c\rangle$-shift complex for some $c \in \mathbb{N}$. The set of all $f$-shift complex sequences is denoted by $\shiftcomplex(f)$.
\item[\ldots] \textdef{generalized shift complex} if $X$ is $f$-shift complex for some order function $f$.
\end{description}
\end{definition}

\begin{definition}[strong $f$-shift complexity]
$X \in \cantor$ is\ldots
\begin{description}
\item[\ldots] \textdef{strongly $\langle f,c\rangle$-shift complex} if $\pfc(X(\co{k,k+n}) \mid k,n) \geq f(n) - c$ for all $k, n \in \mathbb{N}$. The set of all strongly $\langle f,c\rangle$-shift complex sequences is denoted by $\stronglyshiftcomplex(f,c)$.
\item[\ldots] \textdef{strongly $f$-shift complex} if $X$ is strongly $\langle f,c\rangle$-shift complex for some $c \in \mathbb{N}$. The set of all strongly $f$-shift complex sequences is denoted by $\stronglyshiftcomplex(f)$.
\item[\ldots] \textdef{generalized strongly shift complex} if $X$ is strongly $f$-shift complex for some order function $f$.
\end{description}
\end{definition}

We start by addressing the existence or nonexistence of (strongly) $f$-shift complex sequences. If $\delta \coloneq \limsup_n{\frac{f(n)}{n}} < 1$, then $\stronglyshiftcomplex(f) \neq \emptyset$ since every strongly $\delta$-shift complex sequence is strongly $f$-shift complex. If $\delta \geq 1$, on the other hand, then we can show $\shiftcomplex(f) = \emptyset$.

\begin{prop}
Suppose $f \colon \mathbb{N} \to \co{0,\infty}$ is an order function. Let $\delta \coloneq \limsup_n{\frac{f(n)}{n}}$.
\begin{enumerate}[(a)]
\item If $\delta < 1$, then $\stronglyshiftcomplex(f) \neq \emptyset$.
\item If $\delta \geq 1$, then $\shiftcomplex(f) = \emptyset$.
\end{enumerate}
\end{prop}
\begin{proof}
Suppose $\delta \geq 1$. If $X$ is $f$-shift complex, then $\limsup_n{\frac{X \restrict n}{n}} \geq 1$. But if $X$ is $f$-shift complex, then $0^n$ cannot be a substring of $X$ for infinitely many $n$ since there is $c \in \mathbb{N}$ such that $\pfc(0^n) \leq 2\log_2 n + c$ and $\limsup_n{\frac{f(n)}{n}} \geq 1$ implies that for infinitely many $n$ we have $f(n) > \frac{1}{2}n > 2 \log_2 n + c$. \cite[Proposition 3.6]{khan2013shift} shows that if $X$ does not contain every string as a substring, then $\limsup_n{\frac{X \restrict n}{n}} < 1$, giving a contradiction.
\end{proof}

Just as $\shiftcomplex(\delta,c)$ and $\stronglyshiftcomplex(\delta,c)$ for a recursive $\delta \in (0,1)$ and $c \in \mathbb{N}$ are $\Pi^0_1$, $\shiftcomplex(f,c)$ and $\stronglyshiftcomplex(f,c)$ are $\Pi^0_1$, and consequently $\weakdeg(\shiftcomplex(f,c))$, $\weakdeg(\shiftcomplex(f))$, $\weakdeg(\stronglyshiftcomplex(f,c))$, and $\weakdeg(\stronglyshiftcomplex(f))$ lie in  $\mathcal{E}_\weak$ for any order function satisfying $\limsup_n{\frac{f(n)}{n}} < 1$.

\begin{prop} \label{generalized shift complex classes are pi01}
Suppose $f$ is an order function and $c \in \mathbb{N}$. Then $\shiftcomplex(f,c)$ and $\stronglyshiftcomplex(f,c)$ are $\Pi^0_1$, both uniformly in $\delta$, $c$. Consequently, $\weakdeg(\shiftcomplex(f,c))$ (for $c$ sufficiently large), $\weakdeg(\shiftcomplex(f))$,  $\weakdeg(\stronglyshiftcomplex(f,c))$, and $\weakdeg(\stronglyshiftcomplex(f))$ all lie in $\mathcal{E}_\weak$ if $\limsup_n{\frac{f(n)}{n}} < 1$.
\end{prop}
\begin{proof}
Analogous to \cref{shift complex classes are pi01} and \cref{strongly shift complex classes are pi01}.
\end{proof}

\section{Generalized Shift Complexity and Depth}
\label{generalized shift complexity depth section}

The depth of $\shiftcomplex(\delta,c)$ for any recursive $\delta \in (0,1)$ suggests the following question.

\begin{question} \label{depth of f-shift complex question}
Suppose $f\colon \mathbb{N} \to \co{0,\infty}$ is an order function. Under what conditions on $f$ is $\shiftcomplex(f,c)$ deep?
\end{question}

\begin{question} \label{depth of strong f-shift complex question}
Suppose $f\colon \mathbb{N} \to \co{0,\infty}$ is an order function. Under what conditions on $f$ is $\stronglyshiftcomplex(f,c)$ deep?
\end{question}

Partial answers to \cref{depth of f-shift complex question} exist, showing non-negligibility for sufficiently slow-growing, well-behaved order functions $f$. On the other hand, \cref{depth of strong f-shift complex question} can be completely answered: $\stronglyshiftcomplex(f,c)$ is always deep.

\begin{thm}  \label{rumyantsev nonnegligible}
\textnormal{\cite[Theorem 4]{rumyantsev2011everywhere}}
Suppose $\langle a_m\rangle_{m\in\mathbb{N}}$ is a recursive sequence of nonnegative rational numbers such that $\sum_{m=0}^\infty{a_m} < \infty$. Define $f\colon\mathbb{N} \to \co{0,\infty}$ by $f(0) \coloneq 0$ and $f(n) \coloneq a_{\lfloor \log_2 n \rfloor} n$ for $n \in \mathbb{N}_{>0}$. Then $\shiftcomplex(f)$ is non-negligible.
\end{thm}
\begin{proof}
We may assume without loss of generality that $a_m2^m \in\mathbb{N}$ for all $m$: Given $m \in \mathbb{N}$, consider the sequence $\langle\frac{\lceil a_m2^m \rceil}{2^m}\rangle_{m \in \mathbb{N}}$. As $0 \leq \lceil a_m2^m\rceil - a_m2^m \leq 1$, it follows that $0 \leq \frac{\lceil a_m2^m \rceil}{2^m} - a_m \leq \frac{1}{2^m}$. In particular, $\sum_{m=0}^\infty{a_m}$ converges if and only if $\sum_{m=0}^\infty{\frac{\lceil a_m2^m \rceil}{2^m}}$ converges. Letting $g\colon\mathbb{N} \to \mathbb{R}$ be defined by $g(n)\coloneq \frac{\lceil a_{\lfloor log_2 n\rfloor}2^{\lfloor \log_2 n\rfloor} \rceil}{2^{\lfloor \log_2 n \rfloor}} n$, we have $\shiftcomplex(g) \subseteq \shiftcomplex(f)$. Thus, non-negligibility of $\shiftcomplex(g)$ implies non-negligibility of $\shiftcomplex(f)$, so by potentially replacing $\langle a_m\rangle_{m \in \mathbb{N}}$ with $\langle \frac{\lceil a_m2^m \rceil}{2^m}\rangle_{m\in\mathbb{N}}$, we may assume that $a_m2^m \in \mathbb{N}$ for each $m \in \mathbb{N}$. 

Define $\langle b_m \rangle_{m \in \mathbb{N}}$ by setting $b_m \coloneq 2a_m + \frac{m^2}{2^m}$ for $m \in \mathbb{N}$. Note that $\sum_{m=0}^\infty{b_m}$ converges if and only if $\sum_{m=0}^\infty{a_m}$ converges, and that $2^mb_m \in \mathbb{N}$ for each $m \in \mathbb{N}$. The role of $\langle b_m\rangle_{m \in \mathbb{N}}$ will be to account for several subtleties arising later.

We define a real $\Psi(X)$ in stages.
\begin{description}
\item[Stage $s=0$.] Let $m_0$ be the least natural number $m$ for which $\sum_{k=m}^\infty{b_k} \leq 1$. Because $\sum_{m=0}^\infty{\frac{m^2}{2^m}} = 6$, $m_0 > 0$. Split $\mathbb{N}$ into arithmetic progressions with constant difference $2^{m_0}$, i.e., consider the $2^{m_0}$ sequences $\langle i+k2^{m_0}\rangle_{k \in \mathbb{N}}$ for $0 \leq i < 2^{m_0}$. 

For $0 \leq i < b_{m_0}2^{m_0}$ and $k \in \mathbb{N}$, define 
\begin{equation*}
\Psi(X)(i+k2^{m_0}) \coloneq X(i).
\end{equation*}

Note that $2^{m_0} - b_{m_0}2^{m_0} = (1-b_{m_0})2^{m_0} > 0$ of the $2^{m_0}$ arithmetic progressions with constant difference $2^{m_0}$ remain. 

\item[Stage $s>0$.] Suppose $(1-(b_{m_0}+b_{m_0+1}+ \cdots + b_{m_0+s-1})) \cdot 2^{m_0+s-1}$ arithmetic progressions with constant difference $2^{m_0+s-1}$ remain. These yield $N=(1-(b_{m_0}+b_{m_0+1}+\cdots+b_{m_0+s-1})) \cdot 2^{m_0+s}$ arithmetic progressions with constant difference $2^{m_0+s}$ (an arithmetic progression $\langle a+bk \rangle_{k\in\mathbb{N}}$ with constant difference $b$ can be split into two arithmetic progressions $\langle a+2bk \rangle_{k\in\mathbb{N}}$ and $\langle a+b+2bk \rangle_{k\in\mathbb{N}}$ with constant difference $2b$). Let $0 \leq i_0 < i_1 < \cdots < i_{N-1} < 2^{m_0+s}$ be such that $\langle i_j+k2^{m_0+s} \rangle_{k \in \mathbb{N}}$ enumerates those $N$ arithmetic progressions with constant difference $2^{m_0+s}$ which remain as $j$ ranges over $\{0,1,2,\ldots, N-1\}$. 

$\sum_{m=m_0}^\infty{b_m} \leq 1$ implies $N-b_{m_0+s} \cdot 2^{m_0+s} = (1-(b_{m_0}+b_{m_0+1}+\cdots + b_{m_0+s-1}+b_{m_0+s})) \cdot 2^{m_0+s} > 0$. Thus, for $0 \leq j < b_{m_0+s} \cdot 2^{m_0+s}$ and $k \in \mathbb{N}$, define
\begin{equation*}
\Psi(X)(i_j+k2^{m_0+s}) \coloneq X(j).
\end{equation*}

Note that $N - b_{m_0+s} \cdot 2^{m_0+s} = (1-(b_{m_0}+b_{m_0+1}+\cdots + b_{m_0+s-1}+b_{m_0+s}))\cdot 2^{m_0+s}$ arithmetic progressions with constant difference $2^{m_0+s}$ remain.
\end{description}

$\Psi(X)$ is an element of $\cantor$, as at Stage $s$, $\Psi(X)$ is defined at (among other indices) the least index at which $\Psi(X)$ was undefined previously. The above procedure is uniformly recursive in $X$, so $\Psi$ is a total recursive functional.

We must now show that $\pfc(\Psi(X)(\co{k,k+n})) \geq a_{\lfloor \log_2 n\rfloor}n - c$ for some $c \in \mathbb{N}$ and all $k,n \in \mathbb{N}$.

Suppose $k,m \in \mathbb{N}$ and $m_0 \leq m$, and consider the substring $\Psi(X)(\co{k,k+2^m})$. By the construction above, from $\Psi(X)(\co{k,k+2^m})$, $k \bmod 2^m$, and $m$, one can recursively recover $X \restrict b_m2^m$, so there is a $c_1 \in \mathbb{N}$ such that 
\begin{equation*}
\pfc(\Psi(X)(\co{k,k+2^m})) + \pfc(k\bmod 2^m) + \pfc(m) \geq^+ \pfc(X \restrict b_m2^m)
\end{equation*}
for all $k,m \in \mathbb{N}$. 
If $X$ is Martin-\Lof\ random, then there is a $c_2 \in \mathbb{N}$ such that $\pfc(X \restrict b_m2^m) \geq b_m2^m - c_2$ for all $m \in \mathbb{N}$, and consequently
\begin{equation*}
\pfc(\Psi(X)(\co{k,k+2^m})) + \pfc(k\bmod 2^m) + \pfc(m) \geq b_m2^m - (c_1+c_2)
\end{equation*}
for all $k,m \in \mathbb{N}$.

Now we use the definition of $\langle b_m \rangle_{m\in\mathbb{N}}$ to make our desired conclusion about $\langle a_m \rangle_{m\in\mathbb{N}}$. Noting that
\begin{align*}
\pfc(k \bmod 2^m) & \leq \max_{0 \leq k < 2^m}{\pfc(k)} \leq 2m + c_3, \\
\pfc(m) & \leq 2\log_2 m + c_4,
\end{align*}
for some $c_3,c_4 \in \mathbb{N}$ and all $k,m \in \mathbb{N}$, then for $m \geq m_0$ and $k \in \mathbb{N}$ we have
\begin{align*}
\pfc(\Psi(X)(\co{k,k+2^m})) & \geq b_m2^m - \pfc(k\bmod 2^m) - \pfc(m) - (c_1+c_2) \\
& \geq (2a_m+m^2/2^m)2^m - 2m - 2\log_2 m - (c_1+c_2+c_3+c_4) \\
& = a_m2^{m+1} + m^2 - 2m - 2\log_2 m - (c_1+c_2+c_3+c_4).
\end{align*}
$m^2-2m-2\log_2 m \geq -3$ for all $m \in \mathbb{N}$, so with $c = c_1+c_2+c_3+c_4+3$, for all $k$ and all $m \in \mathbb{N}$ we have
\begin{equation*}
\pfc(\Psi(X)(\co{k,k+2^m})) \geq a_m2^{m+1} - c \geq a_m2^m - c.
\end{equation*}
In other words, if $\tau$ is a substring of $\Psi(X)$ whose length $n = |\tau|$ is a power of $2$, then
\begin{equation} \label{rumyantsev power of two equation}
\pfc(\tau) \geq 2a_{\log_2 n}n - c \geq a_{\log_2 n} n - c. \tag*{($\ast$)}
\end{equation}

Now suppose $\tau$ is an arbitrary substring of $\Psi(X)$. Let $\tau'$ be the longest initial segment of $\tau$ whose length is a power of $2$, say $2^m$. Let $n = |\tau|$. Note that $\lfloor \log_2 n \rfloor = m$ and $n \leq 2^{m+1}$. Because $\tau'$ can be found recursively from $\tau$ and using \cref{rumyantsev power of two equation}, there is a $c_5 \in \mathbb{N}$ independent of $\tau$ such that
\begin{equation*}
\pfc(\tau) \geq \pfc(\tau') - c_5 \geq a_m2^{m+1} - (c+c_5) \geq a_{\lfloor \log_2 n \rfloor}n - (c+c_5).
\end{equation*}
Since the term $c+c_5$ is independent of the choice of $\tau$, we find that $\Psi(X) \in \shiftcomplex(f)$.

As $X$ was an arbitrary Martin-\Lof\ random real, it follows that the Turing upward closure of $\shiftcomplex(f)$ has measure $1$, hence $\shiftcomplex(f)$ is non-negligible.
\end{proof}

We shall continue to use the total recursive functional $\Psi$ defined in the proof of \cref{rumyantsev nonnegligible}, so we give an illustrative example:

\begin{example}
$\shiftcomplex(\lambda n.\sqrt{n})$ is non-negligible, applying \cref{rumyantsev nonnegligible} with $a_m \approx \frac{1}{2^{m/2}}$ (technically, $a_m = \frac{1}{2^{\lfloor n/2 \rfloor - 2}}$, so that $a_{\lfloor \log(n) \rfloor}n \geq \sqrt{n}$ for $n > 0$). 

We walk through the proof of \cref{rumyantsev nonnegligible} in this case. Let $a_m = \frac{\lceil \sqrt{2^m}\rceil}{2^m}$. For simplicity, we ignore some of the technical adjustments indicated in the proof. In this case, $m_0=4$, i.e., $\sum_{m=4}^\infty{a_m} \leq 1$. Given $X \in \cantor$, the construction of $\Psi(X)$ proceeds as follows: 

At Stage $0$, we split $\mathbb{N}$ into arithmetic progressions of constant difference $2^4=16$, of which there are $2^4=16$ such arithmetic progressions. We then take the first $\frac{1}{2^{4/2}}2^4 = 2^2 = 4$ such arithmetic progressions (i.e., $\langle i+16k\rangle_{k\in\mathbb{N}}$ for $i\in \{0,1,2,3\}$) and for $k \in \mathbb{N}$ set 
\begin{align*}
\Psi(X)(0 + 16k) & = X(0), \\
\Psi(X)(1+16k) & = X(1), \\
\Psi(X)(2+16k) & = X(2), \\
\Psi(X)(3+16k) & = X(3).
\end{align*}
In particular, at Stage $0$ we define $\Psi(X)$ at $\frac{1}{4} = \frac{1}{2^{4/2}}$ of its inputs. 

There are $16-4 = 8$ arithmetic progressions of constant difference $16$ remaining, namely $\langle i+16k \rangle_{k\in\mathbb{N}}$ for $i \in \{4,5,\ldots,15\}$. This yields $2\cdot 8 = 16$ arithmetic progressions of constant difference $2\cdot 16 = 32$, namely $\langle i+32k \rangle_{k\in\mathbb{N}}$ for $i \in \{4,5,\ldots, 15\} \cup \{20,21,\ldots,31\}$. At Stage $1$, we now take the first $\lceil \frac{1}{2^{5/2}}2^5\rceil = \lceil 2^{5/2} \rceil = 6$ of these arithmetic progressions (i.e., $\langle i+32k \rangle_{k\in\mathbb{N}}$ for $i \in \{4,5,6,7,8,9\}$) and for $k \in \mathbb{N}$ set 
\begin{align*}
\Psi(X)(4+32k) & = X(0), \\
\Psi(X)(5+32k) & = X(1), \\
& ~\vdots \\
\Psi(X)(9+32k) & = X(5).
\end{align*}
In particular, at Stage $1$, we define $\Psi(X)$ at $\frac{6}{32} = \frac{3}{16}$ of its inputs, so that $\Psi(X)$ has been defined at $\frac{1}{4}+\frac{3}{16} = \frac{7}{16}$ of its inputs in total up to this point.

There are $16-6 = 10$ arithmetic progressions of constant difference $32$ remaining, namely $\langle i+32k \rangle_{k\in\mathbb{N}}$ for $i \in \{10,11,\ldots,15\} \cup \{20,21,\ldots,31\}$. This yields $2\cdot 10 = 20$ arithmetic progressions of constant difference $2\cdot 32 = 64$, namely $\langle i+64k \rangle_{k\in\mathbb{N}}$ for $i \in \{10,11,\ldots,15\} \cup \{20,21,\ldots,31\} \cup \{42,43,\ldots,47\} \cup \{52,53,\ldots,63\}$. At Stage $2$, we take the first $\frac{1}{2^{6/2}}2^6 = 2^{6/2} = 8$ of these arithmetic progressions (i.e., $\langle i+64k \rangle_{k\in\mathbb{N}}$ for $i \in \{10,11,12,13,14,15,20,21\}$) and for $k \in \mathbb{N}$ set 
\begin{align*}
\Psi(X)(10+64k) & = X(0), \\
\Psi(X)(11+64k) & = X(1), \\
& ~\vdots \\
\Psi(X)(21+64k) & = X(7)
\end{align*}
In particular, at Stage $2$, we define $\Psi(X)$ at $\frac{8}{64}=\frac{1}{8}$ of its inputs, so that $\Psi(X)$ has been defined at $\frac{1}{4}+\frac{3}{16}+\frac{1}{8}$ of its inputs in total up to this point.

The definition of $\Psi(X)$ continues in this way forever. At Stage $s$, we define $\Psi(X)$ at $\frac{\lceil \sqrt{2^{4+s}}\rceil}{2^{4+s}} \approx \frac{1}{\sqrt{2^{4+s}}}$ of its inputs. Worth noting is that although $\sum_{m=4}^\infty{\frac{\lceil \sqrt{2^m}\rceil}{2^m}} \leq \sum_{m=4}^\infty{\frac{1}{\sqrt{2}^m}} + \sum_{m=4}^\infty{\frac{1}{2^m}} = \frac{1/4}{1-1/\sqrt{2}} + \frac{1}{8} \approx 0.98 < 1$, $\Psi(X)(i)$ \emph{is} defined for every $i \in \mathbb{N}$, since at each Stage $s$ we define $\Psi(X)$ at (among other indices) the least index at which $\Psi(X)$ was undefined previously. 
\end{example}

\begin{remark}
The negligibility of $\shiftcomplex$ implies that there are $X$ which are $\lambda n.\sqrt{n}$-shift complex which are not $\delta$-shift complex for any $\delta \in (0,1)$.
\end{remark}

\begin{remark}
In the proof of \cref{rumyantsev nonnegligible}, at Stage $s+1$ we encode $X\restrict b_{s+1}2^{m_0+s+1}$ even though $X\restrict b_s2^{m_0+s}$ has already been encoded, so that in a certain sense we continually retread old ground. A more conservative approach would be to encode $X \restrict [b_02^{m_0}+b_12^{m_0+1}+\cdots+b_s2^{m_0+s},b_02^{m_0}+b_12^{m_0+1}+\cdots+b_s2^{m_0+s}+b_{s+1}2^{m_0+s+1})$ at Stage $s+1$, so that for some $c \in \mathbb{N}$ and all $k,m \in \mathbb{N}$,
\begin{align*}
\pfc(\Psi(X)(\co{k,k+2^m})) & \geq a_02^{m_0} + a_12^{m_0+1} + \cdots + a_m2^m - c, \\
& \geq \left(\frac{a_0}{2^{m-m_0}} + \frac{a_1}{2^{m-m_0-1}} + \cdots + \frac{a_m}{2^0}\right)2^m - c.
\end{align*}
Let $c_m = \left(\frac{a_0}{2^{m-m_0}} + \frac{a_1}{2^{m-m_0-1}} + \cdots + \frac{a_m}{2^0}\right)$, defined for $m \geq m_0$. Then 
\begin{align*}
\sum_{k=0}^{m-m_0}{c_{m_0+k}} & = c_{m_0} + c_{m_0+1} + \cdots + c_m \\
& = a_0 + \left(\frac{a_0}{2} + a_1\right) + \left(\frac{a_0}{4} + \frac{a_1}{2} + a_2\right) + \cdots + \left(\frac{a_0}{2^{m-m_0}} + \frac{a_1}{2^{m-m_0-1}} + \cdots + a_m\right) \\
& = \sum_{k=0}^m{\left(2-\frac{1}{2^{m-m_0+1-k}}\right)a_k}
\end{align*}
It follows that $\sum_{k=0}^\infty{c_{m_0+k}}$ converges. In particular, this `conservative' approach does not produce any stronger general conclusion. 
\end{remark}

\subsection{Relating Generalized Shift Complexity and Complexity} \label{relating generalized shift complexity and complexity subsection}

The total recursive functional $\Psi$ defined in the proof of \cref{rumyantsev nonnegligible} can be used to provide stronger results about the location of $\weakdeg(\shiftcomplex(f))$ in $\mathcal{E}_\weak$ for order functions $f$ satisfying $\sum_{m=0}^\infty{\frac{f(2^m)}{2^m}} < \infty$.

\begin{thm} \label{strong rumyantsev theorem 4}
Suppose $f,g \colon \mathbb{N} \to \co{0,\infty}$ are sub-identical order functions such that $\sum_{m=0}^\infty{f(2^m)/2^m} < \infty$ and for which there is a recursive sequence $\langle \epsilon_m \rangle_{m \in \mathbb{N}}$ of positive rationals such that $\sum_{m=0}^\infty{\epsilon_m} < \infty$ and
\begin{equation*}
\liminf_m{\frac{g(2^m\epsilon_m) - f(2 \cdot 2^m)}{m}} > 1.
\end{equation*}
Then $\shiftcomplex(f) \strongleq \complex(g)$.
\end{thm}
\begin{proof}
Suppose $\langle \epsilon_m \rangle_{m\in\mathbb{N}}$ satisfies the hypotheses of the theorem. Let $\alpha$ be a rational number such that $1 < \alpha < \liminf_m{\frac{g(2^m\epsilon_m) - f(2\cdot 2^m)}{m}}$, so that $g(2^m\epsilon_m) - \alpha m \geq f(2\cdot 2^m)$ for almost all $m \in \mathbb{N}$.

Now define $\Psi$ as in the proof of \cref{rumyantsev nonnegligible}, using $\langle a_m \rangle_{m\in\mathbb{N}} = \langle \epsilon_m \rangle_{m\in\mathbb{N}}$. As in that proof, there is a $c_1$ such that for all $k \in \mathbb{N}$ and $m \geq m_0$, we have
\begin{equation*}
\pfc(\Psi(X)(\co{k,k+2^m})) + \pfc(k \bmod 2^m) + \pfc(m) \geq \pfc(X \restrict 2^m\epsilon_m) - c_1.
\end{equation*}
If $X \in \complex(g)$, then there is a $c_2$ such that $\pfc(X \restrict 2^m\epsilon_m) \geq g(2^m\epsilon_m) - c_2$ for all $m \in \mathbb{N}$. Consequently, for all $k,m \in \mathbb{N}$ we have
\begin{equation*}
\pfc(\Psi(X)(\co{k,k+2^m})) + \pfc(k \bmod 2^m) + \pfc(m) \geq g(2^m\epsilon_m) - (c_1+c_2).
\end{equation*}
Additionally, noting that $1 < \frac{1+\alpha}{2}$, there are $c_3,c_4 \in \mathbb{N}$ such that
\begin{align*}
\pfc(k \bmod 2^m) & \leq \max_{0 \leq k < 2^m}{\pfc(k)} \leq (\tfrac{1+\alpha}{2})m + c_3, \\
\pfc(m) & \leq 2\log_2 m + c_4,
\end{align*}
for all $m \in \mathbb{N}$, so that for all $k,m \in \mathbb{N}$ we have
\begin{equation*}
\pfc(\Psi(X)(\co{k,k+2^m})) \geq g(2^m\epsilon_m) - \left( (\tfrac{1+\alpha}{2}) \cdot m - 2\log_2 m\right) - (c_1+c_2+c_3+c_4).
\end{equation*}
For sufficiently large $m$, $(\frac{1+\alpha}{2}) \cdot m + 2\log_2 m \leq \alpha m$, and hence $\alpha m - \left( (\frac{1+\alpha}{2})m - 2\log_2 m\right)$ is bounded from below, say by $c_5$. Let $c = c_1+c_2+c_3+c_4+c_5$. Then for all $k$ and all $m \in \mathbb{N}$,
\begin{equation*}
\pfc(\Psi(X)(\co{k,k+2^m})) \geq f(2\cdot 2^m) - c \geq f(2^m) - c.
\end{equation*}
In other words, if $\tau$ is a substring of $\Psi(X)$ whose length $|\tau|$ is a power of $2$, then 
\begin{equation} \label{general rumyantsev theorem 4 equation}
\pfc(\tau) \geq f(2\cdot |\tau|) - c \geq f(|\tau|) - c. \tag*{($\ast$)}
\end{equation}

Now suppose $\tau$ is an arbitrary substring of $\Psi(X)$. Let $\tau'$ be the longest initial segment of $\tau$ whose length is a power of $2$, say $2^m$. Let $n=|\tau|$. Note that $\lfloor \log_2 n\rfloor = m$ and $n \leq 2^{m+1}$. Because $\tau'$ can be found recursively from $\tau$, that $f$ is monotonic, and using \ref{general rumyantsev theorem 4 equation}, there is a $c_6 \in \mathbb{N}$ independent of $\tau$ such that 
\begin{equation*}
\pfc(\tau) \geq \pfc(\tau') - c_6 \geq f(2\cdot |\tau'|) - (c+c_6) = f(2^{m+1}) - (c+c_6) \geq f(n) - (c+c_6).
\end{equation*}
\end{proof}

\begin{remark}
The requirement that $\sum_{m=0}^\infty{\frac{f(2^m)}{2^m}}$ converges is necessary for \cref{strong rumyantsev theorem 4} to yield any useful information -- the existence of a convergent series $\sum_{m=0}^\infty{\epsilon_m}$ such that $2^m\epsilon_m - f(2^{m+1}) \geq 0$ (so $g$ is the identity function) implies $\sum_{m=0}^\infty{\frac{f(2^{m+1})}{2^m}}$ converges, which is equivalent to the convergence of $\sum_{m=0}^\infty{\frac{f(2^m)}{2^m}} = f(1)+ \frac{1}{2}\sum_{m=0}^\infty{\frac{f(2^{m+1})}{2^m}}$.
\end{remark}

Despite the technicality of the condition in \cref{strong rumyantsev theorem 4}, we can deduce several nice relationships:

\begin{cor} \label{partial randomness and shift complexity}
Suppose $f\colon \mathbb{N} \to \co{0,\infty}$ is a sub-identical order function such that $\sum_{m=0}^\infty{\frac{f(2^m)}{2^m}}$ converges. Then for every rational $\delta \in (0,1]$, $\shiftcomplex(f) \strongleq \complex(\delta)$.
\end{cor}
\begin{proof}
Define $g$ by setting $g(n) \coloneq \delta n$ for $n \in \mathbb{N}$ and define $\langle \epsilon_m \rangle_{m \in \mathbb{N}}$ by setting $\epsilon_m \coloneq \frac{1}{\delta}\left(\frac{f(2^{m+1})}{2^m} + \frac{m^2}{2^m}\right)$ for $m \in \mathbb{N}$. Then
\begin{align*}
\lim_{m \to \infty}{\frac{g(2^m\epsilon_m) - f(2^{m+1})}{m}} & = \lim_{m \to \infty}{\frac{f(2^{m+1}) + m^2 - f(2^{m+1})}{m}} \\
& = \lim_{m \to \infty}{m} \\
& = \infty \\
& > 1,
\end{align*}
so \cref{strong rumyantsev theorem 4} shows $\complex(\alpha) = \complex(g) \stronggeq \shiftcomplex(f)$.
\end{proof}

\begin{cor} \label{complex and shift-complex example 1}
Suppose $0< \alpha < \beta \leq 1$ are rational numbers. Then $\shiftcomplex(\lambda n. n^\alpha) \strongleq \complex(\lambda n. n^\beta)$.
\end{cor}
\begin{proof}
Define $f$, $g$, and $\langle \epsilon_m \rangle_{m \in \mathbb{N}}$ by setting $f(n) \coloneq n^\alpha$, $g(n) \coloneq n^\beta$, and $\epsilon_m \coloneq m^{-1/\beta}$ for each $n, m \in \mathbb{N}$. Then $\sum_{m=0}^\infty{\frac{f(2^m)}{2^m}} = \sum_{m=0}^\infty{\frac{1}{m^\beta}} < \infty$ and
\begin{align*}
\lim_{m \to \infty}{\frac{g(2^m\epsilon_m) - f(2^{m+1})}{m}} & = \lim_{m \to \infty}{2^{\beta \cdot m}/m - 2^{\alpha \cdot m}}{m} \\
& = \lim_{m \to \infty}{\frac{2^{\beta \cdot m} - 2^{\alpha \cdot m}m^2}{m^2}} \\
& = \infty \\
& > 1,
\end{align*}
so \cref{strong rumyantsev theorem 4} shows $\shiftcomplex(\lambda n. n^\alpha) \strongleq \complex(\lambda n. n^\beta)$.
\end{proof}

\begin{cor} \label{complex and shift-complex example 2}
Suppose $0 < \alpha + 1 < \beta$ and rational numbers. Then 
\begin{equation*}
\shiftcomplex(\lambda n. n/(\log_2 n)^\beta) \strongleq \complex(\lambda n. n/(\log_2 n)^\alpha).
\end{equation*}
\end{cor}
\begin{proof}
Define $f$, $g$, and $\langle \epsilon_m \rangle_{m \in \mathbb{N}}$ by setting $f(n) \coloneq n/(\log_2 n)^\beta$, $g(n) \coloneq n/(\log_2 n)^\alpha$, and $\epsilon_m \coloneq 1/m(\log_2 m)^2$ for each $m,n \in \mathbb{N}$. Then
\begin{align*}
\lim_{m \to \infty}{\frac{g(2^m\epsilon_m) - f(2^{m+1})}{m}} & = \lim_{m \to \infty}{\frac{1}{m}\left(\frac{2^m m^{-1} (\log_2 m)^{-2}}{(\log_2 (2^m m^{-1} (\log_2 m)^{-2}))^\alpha} - \frac{2^{m+1}}{m^\beta}\right)} \\
& = \lim_{m \to \infty}{2^m\left( \frac{1}{m^{\alpha+2}(1-(\log_2 m)/m - 2(\log_2 \log_2 m)/m)^\alpha(\log_2 m)^2} - \frac{2}{m^{\beta+1}}\right)} \\
& = \infty \\
& > 1,
\end{align*} 
so \cref{strong rumyantsev theorem 4} shows $\shiftcomplex(\lambda n. n/(\log_2 n)^\beta) \strongleq \complex(\lambda n. n/(\log_2 n)^\alpha)$.
\end{proof}

\subsection{Extracting Generalized Shift Complexity from Sublinear Complexity} \label{extracting generalized shift complexity from sublinear complexity subsection}

\cref{complex and shift-complex example 1,,complex and shift-complex example 2} show for certain sufficiently slow-growing and well-behaved order functions $f$ there is another order function $g$ such that $\complex(g) \stronggeq \shiftcomplex(f)$ which is sublinear. We can show that this holds more generally given that $\sum_{m=0}^\infty{\frac{f(2^m)}{2^m}}$ not only converges, but converges to a recursive real.

\begin{thm} \label{sub-identical complex strongly computes shift complex}
Suppose $f\colon \mathbb{N} \to \co{0,\infty}$ is a sub-identical order function such that $\sum_{m=0}^\infty{\frac{f(2^m)}{2^m}}$ converges to a recursive real. Then there is a sublinear order function $g\colon \mathbb{N} \to \co{0,\infty}$ such that $\shiftcomplex(f) \strongleq \complex(g)$.
\end{thm}

To prove \cref{sub-identical complex strongly computes shift complex}, we start by making several observations concerning recursive series of positive rational numbers.

\begin{lem} \label{convergent series gap}
Suppose $\langle \epsilon_m \rangle_{m\in\mathbb{N}}$ is a recursive sequence of positive rational numbers such that $\sum_{m=0}^\infty{\epsilon_m}$ converges to a recursive real. Then there exists a nondecreasing, recursive sequence $\langle \gamma_m \rangle_{m\in\mathbb{N}}$ of positive integers such that $\sum_{m=0}^\infty{\epsilon_m\gamma_m}$ converges to a recursive real and $\lim_{m \to \infty}{\gamma_m} = \infty$.
\end{lem}
\begin{proof}
Let $\epsilon \coloneq \sum_{m=0}^\infty{\epsilon_m}$, which is a recursive real by hypothesis. 
Recursively define a strictly increasing sequence $\langle n_k \rangle_{k\in\mathbb{N}}$ by setting $n_0 \coloneq 0$ and, given $n_k$ has been defined, let $n_{k+1}$ be the largest $n > n_k$ such that $\sum_{m=n_k+1}^n{\epsilon_m} \leq \frac{\epsilon}{2^{k+1}}$. Note, then, that $\sum_{m=n_k+1}^n{(k+1)\epsilon_m} \leq \frac{(k+1)\epsilon}{2^{k+1}}$.

We define sequences of positive integers $\langle \gamma_m \rangle_{m\in\mathbb{N}}$ and $\langle \delta_m \rangle_{m\in\mathbb{N}}$ recursively. Start by setting $\gamma_m \coloneq \delta_m =1$ for $0 = n_0 \leq m < n_1$, let $\gamma_{n_1} \coloneq 1$, and let $\delta_{n_1}$ be the largest integer such that $\sum_{m=n_0}^{n_1}{\epsilon_m\delta_m} \leq \epsilon/2$ -- because $\sum_{m=n_0}^{n_1}{\epsilon_m} \leq \epsilon/2$ by our definition of $n_1$, we know $\delta_{n_1} \geq 1$. Analogously, given $\gamma_m$ and $\delta_m$ have been defined for $m \leq n_k$, let $\gamma_m \coloneq \delta_m \coloneq k+1$ for $n_k +1 \leq m < n_{k+1}$, let $\gamma_{n_{k+1}} = k+1$, and let $\delta_{n_{k+1}}$ be the largest integer such that $\sum_{m=0}^{n_{k+1}}{\epsilon_m\delta_m} \leq \sum_{j=1}^{k+1}{\frac{j}{2^j}}$ -- because $\sum_{m=n_k+1}^n{(k+1)\epsilon_m} \leq \frac{(k+1)\epsilon}{2^{k+1}}$ by our definition of $n_{k+1}$, it follows that $\delta_{n_{k+1}} \geq k+1$.

By construction, $\sum_{m=0}^\infty{\epsilon_m\delta_m}$ converges to the recursive real $\sum_{k=1}^\infty{\frac{k\epsilon}{2^k}} = 2\epsilon$ (here, we use the fact that $\lim_{k \to \infty}{\epsilon_{n_k}} = 0$ so that the difference $\sum_{j=1}^{k+1}{\frac{j}{2^j}} - \sum_{m=0}^{n_{k+1}}{\epsilon_m\delta_m}$ is made arbitrarily small as $k \to \infty$). Because $\epsilon$ is a recursive real, the above construction is also recursive, so the sequences $\langle \gamma_m \rangle_{m\in\mathbb{N}}$ and $\langle \delta_m \rangle_{m\in\mathbb{N}}$ are recursive. As observed above, $\gamma_m, \delta_m \geq k+1$ for $n_k+1 \leq m \leq n_{k+1}$, so $\lim_{m \to \infty}{\delta_m} = \infty$.

By construction, $0 < \epsilon_m\gamma_m \leq \epsilon_m\delta_m$ for all $m \in \mathbb{N}$, so the convergence of $\sum_{m=0}^\infty{\epsilon_m\delta_m}$ to the recursive real $2\epsilon$ implies that $\sum_{m=0}^\infty{\epsilon_m\gamma_m}$ also converges to a recursive real by \cref{recursive sum implies lower bounds have recursive sums}. Finally, $\langle \gamma_m \rangle_{m\in\mathbb{N}}$ is nondecreasing by construction.
\end{proof}

\begin{proof}[Proof of \cref{sub-identical complex strongly computes shift complex}.]
First, $\sum_{m=0}^\infty{\frac{f(2^m)}{2^m}}$ converging to a recursive real implies that $\sum_{m=0}^\infty{\frac{f(2^{m+1})}{2^m}} = 2\left( \sum_{m=0}^\infty{\frac{f(2^{m+1})}{2^{m+1}}} - f(1)\right)$ converges to a recursive real. Then $\sum_{m=0}^\infty{\frac{f(2^{m+1})+1}{2^m}} = \sum_{m=0}^\infty{\frac{f(2^{m+1})}{2^m}} + 2$ converges to a recursive real, so \cref{recursive sum implies lower bounds have recursive sums} implies that $\sum_{m=0}^\infty{\frac{\lceil f(2^{m+1})\rceil}{2^m}}$ converges to a recursive real. 

Applying \cref{convergent series gap} to $\langle \lceil f(2\cdot 2^m)\rceil/2^m \rangle_{m\in\mathbb{N}}$ yields a nondecreasing, recursive sequence $\langle \gamma_m \rangle_{m\in\mathbb{N}}$ of positive integers tending towards infinity such that $\sum_{m=0}^\infty{\frac{\lceil f(2^{m+1}) \rceil}{2^m}\gamma_m}$ converges. Note that the proof of \cref{convergent series gap} shows that we may assume without loss of generality that $\gamma_m \leq m+1$ for all $m \in \mathbb{N}$.

Define $\epsilon_m = \frac{\lceil f(2^{m+1}) \rceil}{2^m}\gamma_m + \frac{2m^2}{2^m}\gamma_m$. By the definition of $\langle \gamma_m \rangle_{m\in\mathbb{N}}$, $\sum_{m=0}^\infty{\frac{\lceil f(2^{m+1}) \rceil}{2^m}\gamma_m}$ converges. Because $\gamma_m \leq m+1$, $\sum_{m=0}^\infty{\frac{2m^2}{2^m}\gamma_m}$ also converges. Thus, $\sum_{m=0}^\infty{\epsilon_m}$ converges. Additionally, both $\langle 2^m\epsilon_m \rangle_{m\in\mathbb{N}}$ and $\langle 2^m\epsilon_m\gamma_m^{-1} \rangle_{m\in\mathbb{N}}$ are strictly increasing, recursive sequences of natural numbers. 

Define $g\colon\mathbb{N} \to \co{0,\infty}$ by setting
\begin{equation*}
g(n) \coloneq \frac{2^{m+1}\epsilon_{m+1}\gamma_{m+1}^{-1} - 2^m\epsilon_m\gamma_m^{-1}}{2^{m+1}\epsilon_{m+1} - 2^m\epsilon_m}(n - 2^m\epsilon_m) + 2^m\epsilon_m\gamma_m^{-1}
\end{equation*}
for $2^m\epsilon_m \leq n < 2^{m+1}\epsilon_{m+1}$. For $0 \leq n < \epsilon_0$, let $g(n) \coloneq \gamma_0^{-1}n$. In other words, we set $g(0) \coloneq 0$ and $g(2^m\epsilon_m) \coloneq 2^m\epsilon_m\gamma_m^{-1}$ for $m \in \mathbb{N}$, then take $g$ to be defined linearly between consecutive points in the sequence $\langle 0,0 \rangle, \langle \epsilon_0,\epsilon_0\gamma_0^{-1} \rangle, \langle 2\epsilon_1,2\epsilon_2\gamma_1^{-1} \rangle,\ldots$. We make the following observations:
\begin{itemize}
\item By the definition of $g$ above, $g$ is recursive.

\item Because the sequences $\langle 2^m\epsilon_m \rangle_{m\in\mathbb{N}}$ and $\langle 2^m\epsilon_m\gamma_m^{-1} \rangle_{m\in\mathbb{N}}$ are both strictly increasing, $g$ is nondecreasing. 

\item To show that $\lim_{n \to \infty}{\frac{g(n)}{n}} = 0$, the piecewise linearity of $g$ means that it suffices to show that $\lim_{m \to \infty}{\frac{g(2^m\epsilon_m)}{2^m\epsilon_m}} = 0$. Indeed,
\begin{equation*}
\lim_{m \to \infty}{\frac{g(2^m\epsilon_m)}{2^m\epsilon_m}} = \lim_{m \to \infty}{\frac{2^m\epsilon_m\gamma_m^{-1}}{2^m\epsilon_m}} = \lim_{m \to \infty}{\gamma_m^{-1}} = 0.
\end{equation*}

\item $g(2^m\epsilon_m) \geq f(2^{m+1}) + 2m$:
\begin{align*}
g(2^m\epsilon_m) & = 2^m\epsilon_m\gamma_m^{-1} \\
& = \lceil f(2^{m+1}) \rceil \gamma_m \gamma_m^{-1} + 2m^2\gamma_m \gamma_m^{-1} \\
& \geq \lceil f(2^{m+1}) \rceil + 2m \\
& \geq f(2^{m+1}) + 2m.
\end{align*}
Thus, $g$ is an order function and 
\begin{equation*}
\liminf_m{\frac{g(2^m\epsilon_m) - f(2\cdot 2^m)}{m}} \geq 2 > 1.
\end{equation*}
\cref{strong rumyantsev theorem 4} then implies $\complex(g) \stronggeq \shiftcomplex(f)$.
\end{itemize}
\end{proof}

\subsection{Strong Shift Complexity and Depth} \label{strong shift complexity and depth subsection}

Although there exist order functions $f$ for which $\shiftcomplex(f)$ is negligible, the situation for generalized strong shift complexity is more favorable with respect to depth, completely answering \cref{depth of strong f-shift complex question}.

\begin{thm} \label{rumyantsev negligible}
\textnormal{\cite[Theorem 5, essentially]{rumyantsev2011everywhere}}
$\stronglyshiftcomplex(f,c)$ is a deep $\Pi^0_1$ class for any order function $f$ satisfying $\limsup_n{\frac{f(n)}{n}} < 1$ and any $c \in \mathbb{N}$.
\end{thm}
\begin{proof}
With $\mathbf{M}$ a universal left r.e.\ continuous semimeasure on $\{0,1\}^\ast$, by \cref{levin and zvonkin} there is a partial recursive functional $\Psi$ such that for every $\sigma \in \{0,1\}^\ast$, 
\begin{equation*}
\mathbf{M}(\sigma) = \lambda(\Psi^{-1}(\sigma)) = \lambda(\{ X \in \cantor \mid \Psi^X \supseteq \sigma\}).
\end{equation*}

For a fixed $n \in \mathbb{N}$, say $Y \in \{0,1\}^\ast \cup \cantor$ \emph{avoids} a $\vec{\tau} = \langle \tau_0,\tau_1,\ldots,\tau_{N-1}\rangle$ if $|Y| \geq N+n-1$ and $Y(\co{k,k+n}) \neq \tau_k$ for all $k \in \{0,1,2,\ldots,N-1\}$. 

Our approach, roughly, involves us finding, for a fixed length $n$, a sequence $\vec{\tau}$ of `forbidden' strings $\tau_0,\tau_1,\ldots,\tau_{N-1}$ (where $N=N(n,m)$) of length $n$ such that 
\begin{equation*}
\lambda(\{ X \in \cantor \mid \text{$\Psi(X)$ avoids $\vec{\tau}$}\}) < 2^{-m}.
\end{equation*}
By showing that such a $\vec{\tau}$ can be chosen so that $\pfc(\tau_k \mid k,n)$ is bounded for all $k,n$, it will follow that it is exceeded by $f(n)-c$ for all sufficiently large $n$. Hence, if an output $Y$ of $\Psi$ is strongly $f$-shift complex, then it avoids $\vec{\tau}$, and $\lambda(\Psi^{-1}[\stronglyshiftcomplex(f,c)]) < 2^{-m}$.

As in the proof of \cref{delta-c-shift complex reals form deep pi01 class}, an arbitrary output of $\Psi$ need not be in $\cantor$. However, also like in \cref{delta-c-shift complex reals form deep pi01 class}, we only need $|\Psi^X| \geq N+n-1$ for the notion of avoiding $\vec{\sigma} \in (\{0,1\}^n)^N$ to be well-defined. As such, define
\begin{equation*}
S \coloneq \{ X \in \cantor \mid |\Psi^X| \geq N+n-1\}
\end{equation*}
where $N=N(n,m)$ is the smallest natural number for which
\begin{equation*}
\left(1-\frac{1}{2^n}\right)^N < 2^{-(m+1)}.
\end{equation*}
$S$ is $\Sigma^0_1$, so there exists a recursive sequence $\langle \sigma_n \rangle_{n\in\mathbb{N}}$ of pairwise incompatible strings such that $S = \bigcup_{n \in \mathbb{N}}{\bbracket{\sigma_n}}$.  Let $\alpha = \lambda(S)$, so that $\langle \lambda\left(\bigcup_{k\leq n}{\bbracket{\rho_k}}\right) \rangle_{n \in \mathbb{N}}$ converges monotonically to $\alpha$ from below. Let $i_0$ be the largest natural number $i$ such that $\alpha_i \coloneq i \cdot 2^{-m}/3 < \alpha$, so that $\alpha - \alpha_{i_0} < 2^{-(m+1)}$. Then let $p$ be the smallest natural number such that $\lambda\left(\bigcup_{k\leq p}{\bbracket{\sigma_k}}\right) \geq \alpha_{i_0}$ and finally define
\begin{equation*}
\tilde{S} \coloneq \bigcup_{k \leq p}{\bbracket{\sigma_k}}.
\end{equation*}

$\tilde{S}$ is a recursive subset of $S$ and $\lambda(S\setminus \tilde{S}) < 2^{-(m+1)}$. By virtue of being a subset of $S$, $|\Psi^X| \geq N+n-1$ for all $X \in S^\ast$. Define $\mu$ on $\{0,1\}^{N+n-1}$ by
\begin{equation*}
\mu(\tau) \coloneq \lambda(\tilde{S})^{-1} \cdot \lambda(\{ X \in \cantor \mid \text{$X \in \tilde{S}$ and $\Psi^X \supseteq \tau$}\}).
\end{equation*}
$\mu$ is a probability measure on $\{0,1\}^{N+n-1}$, and the recursiveness of $\tilde{S}$ implies $\mu$ is computable. 

For a fixed $Y \in \{0,1\}^{\geq N+n-1}$, the probability that $\vec{\tau} \in (\{0,1\}^n)^N$ is avoided by $Y$ (taken with respect to the uniform probability measure $\nu$ on $(\{0,1\}^n)^N$, or equivalently the $N$-fold product of the uniform probability measures on $\{0,1\}^n$) is equal to
\begin{equation*}
\left( 1- \frac{1}{2^n}\right)^N < \frac{\epsilon}{2}.
\end{equation*}
Write
\begin{equation*}
E \coloneq \{ \langle X,\vec{\tau} \rangle \in \cantor \times (\{0,1\}^n)^N \mid \text{$X \in \tilde{S}$ and $\Phi^X$ avoids $\vec{\tau}$}\}.
\end{equation*}
By Fubini's Theorem,
\begin{align*}
\int{\left(\int{\chi_E(X,\vec{\tau})\dd \lambda}\right)\dd \nu} & = (\lambda \times \nu)(E) \\
& = \int{\left(\int{\chi_E(X,\vec{\sigma})\dd \nu}\right)\dd \lambda} \\
& = \int{(1-2^{-n})^N\dd \lambda} \\
& = (1-2^{-n})^N \\
& < 2^{-(m+1)}.
\end{align*}
If $\int{\chi_E(X,\vec{\tau})\dd \lambda} \geq 2^{-(m+1)}$ for every $\vec{\tau} \in (\{0,1\}^n)^N$, we reach a contradiction. Thus, there is a least one sequence $\vec{\tau} \in (\{0,1\}^n)^N$ with the desired property, i.e., for which
\begin{equation*}
\lambda(\{ X \in \cantor \mid \text{$X \in \tilde{S}$ and $\Psi^X$ avoids $\vec{\tau}$}\}) < 2^{-(m+1)},
\end{equation*}
and hence 
\begin{equation*}
\lambda(\{ X \in \cantor \mid \text{$X \in S$ and $\Psi^X$ avoids $\vec{\tau}$}\}) < 2^{-(m+1)} + 2^{-(m+1)} = 2^{-m}.
\end{equation*}

Let $c_1 \in \mathbb{N}$ be such that $\pfc(\sigma) \leq 2 \log_2 |\sigma| + c_1$ for all $\sigma \in \{0,1\}^\ast$. An index for $\tilde{S}$ can be found effectively as a function of $n$, $m$, and $i_0$, and such an index can be used to compute a $\vec{\tau} = \langle \tau_0,\tau_1,\ldots,\tau_{N-1} \rangle$ for which $\lambda(\{X \in \cantor \mid \text{$X \in \tilde{S}$ and $\Psi^X$ avoids $\vec{\tau}$}\}) < 2^{-(m+1)}$. As such, there is a constant $c_2 \in \mathbb{N}$ (independent of $k$, $n$, $m$, and $i_0$) such that $\pfc(\tau_k \mid k,n,i_0) \leq \pfc(m) + c_2$ for all $k$, $n$, and $i_0$. Although $i_0$ cannot in general be found recursively as a function of $\langle n,m\rangle$, we have the bound $i_0 \leq (2^{-m}/3)^{-1} = 3 \cdot 2^m$. Thus,
\begin{equation*}
\pfc(i_0) \max_{0\leq i \leq 3 \cdot 2^m}{\pfc(i)} \leq \max_{0\leq i \leq 3 \cdot 2^m}{(2\log_2 i+c_1)} \leq 2m + 2\log_2 3 + c_1.
\end{equation*}
Finally, there is $c_3 \in \mathbb{N}$ such that for all $\sigma, \rho_1,\rho_2,\rho_3 \in \{0,1\}^\ast$ we have $\pfc(\sigma \mid \rho_1,\rho_2) \leq \pfc(\sigma \mid \rho_1,\rho_2,\rho_3) + \pfc(\rho_3) + c_3$. Thus,
\begin{align*}
\pfc(\tau_k \mid k,n) & \leq \pfc(\tau_k \mid k,n,i_0) + \pfc(i_0) + c_3 \\
& \leq (\pfc(m) + c_2) + 2m + 2\log_2 3 + c_1 + c_3 \\
& \leq (2\log_2 m + c_1 + c_2) + 2m + 2\log_2 3 + c_1 + c_3 \\
& = 2m + 2\log_2 m + (2\log_2 3 + 2c_1 + c_2 + c_3).
\end{align*}
Let $d = 2\log_2 3 + 2c_1 + c_2 + c_3$ and let $n = n(m)$ be the least such that 
\begin{equation*}
f(n) - c > 2m + 2\log_2 m + d.
\end{equation*}

Suppose $\Psi(X)$ is strongly $(f,c)$-shift complex, so that $\Psi^X = \Psi(X) \in \cantor$ (therefore $X \in S$) and $\pfc(\Psi(X)(\co{k,k+n}) \mid k,n) \geq f(n)-c$ for all $n,k$. For a sufficiently large $n$, 
\begin{equation*}
\pfc(\Psi(X)([k,k+n)) \mid k,n) \geq f(n)-c > \pfc(\sigma_k \mid k,n)
\end{equation*}
so $\Psi(X)([k,k+n)) \neq \sigma_k$, hence $\Psi(X)$ avoids $\vec{\sigma}$. 
Define $r \colon \mathbb{N} \to \mathbb{N}$ by $r(m) \coloneq N(n(m),m) + n(m) - 1$, which is a recursive function. Then 
\begin{align*}
\mathbf{M}(\stronglyshiftcomplex(f,c) \restrict r(m)) & \leq \lambda(\{ X \in \cantor \mid \text{$|\Psi^X| \geq r(m)$ and $\Psi^X \restrict r(m)$ is strongly $(f,c)$-shift complex}\}) \\
& \leq \lambda(\{ X \in \cantor \mid \text{$X \in S$ and $\Psi^X$ avoids $\vec{\tau}$}\}) \\
& < 2^{-m}.
\end{align*}
Hence, $\stronglyshiftcomplex(f,c)$ is deep.
\end{proof}

\section{Open Questions}

There are a great many open questions concerning shift complexity and generalized shift complexity. Just focusing on the classes $\shiftcomplex(\delta)$ and $\stronglyshiftcomplex(\delta)$, basic questions about their weak degrees remain open, such as separating $\shiftcomplex(\alpha)$ and $\shiftcomplex(\beta)$ when $\alpha \neq \beta$:

\begin{question} \label{separating delta-shift complex degrees}
Are there rational numbers $0 < \alpha < \beta < 1$ such that $\shiftcomplex(\alpha) \weakeq \shiftcomplex(\beta)$? $\shiftcomplex(\alpha) \weakle \shiftcomplex(\beta)$?
\end{question}

\begin{question} \label{separating strongly delta-shift complex degrees}
Are there rational numbers $0 < \alpha < \beta < 1$ such that $\stronglyshiftcomplex(\alpha) \weakeq \stronglyshiftcomplex(\beta)$? $\stronglyshiftcomplex(\alpha) \weakle \stronglyshiftcomplex(\beta)$?
\end{question}

The relationship between the classes $\shiftcomplex(\alpha)$ and $\stronglyshiftcomplex(\beta)$ is especially unclear. We know that $\shiftcomplex(\alpha) \strongleq \stronglyshiftcomplex(\alpha)$, but not much more than that is known.

\begin{question} \label{separating shift complexity from strong shift complexity}
Is there a rational $\delta \in (0,1)$ such that $\shiftcomplex(\delta) \weakeq \stronglyshiftcomplex(\delta)$?
\end{question}

If \cref{separating shift complexity from strong shift complexity} is answered in the negative (so $\shiftcomplex(\delta) \weakle \stronglyshiftcomplex(\delta)$ for all $\delta \in (0,1)$), we might turn to comparing $\stronglyshiftcomplex(\alpha)$ and $\shiftcomplex(\beta)$ for some $ 0 < \alpha < \beta < 1$.

\begin{question} \label{relationship between shift complex and strongly shift complex}
Given a rational $0 < \beta < 1$, does $\stronglyshiftcomplex(\alpha) \weakleq \shiftcomplex(\beta)$ hold for all rationals $0 < \alpha < \beta$? If not, for \emph{some} $\alpha < \beta$? 
\end{question}

Although $\shiftcomplex(\delta,c)$ is deep for every rational $\delta \in (0,1)$ and $c \in \mathbb{N}$, it is unclear if $\shiftcomplex(\delta)$ is of deep degree -- while $\complex(\alpha,c) \weakleq \complex(\beta)$ for any $0 < \alpha < \beta$ such that $\complex(\alpha,c) \neq \emptyset$, it is unclear if this holds when complexity is replaced with shift complexity. 

\begin{question} \label{are unions of delta-shift complex classes deep}
Are $\shiftcomplex(\delta)$ or $\stronglyshiftcomplex(\delta)$ of deep degree in $\mathcal{E}_\weak$ for every rational $\delta \in (0,1)$?
\end{question}

More generally, we may ask similar questions with generalized shift complexity:

\begin{question} \label{separating generalized shift complexity}
Given order functions $f\colon \mathbb{N} \to \co{0,\infty}$ and $g\colon \mathbb{N} \to \co{0,\infty}$, when do we have $\shiftcomplex(f) \weakle \shiftcomplex(g)$? 
\end{question}

Partial answers to \cref{separating generalized shift complexity} exist. For example, \cref{complex and shift-complex example 1} allows us to separate the weak degrees $\weakdeg(\shiftcomplex(\lambda n. n^\alpha))$:

\begin{prop}
If $0 < \alpha < \beta < 1$ are rational numbers, then $\shiftcomplex(\lambda n. n^\alpha) \weakle \shiftcomplex(\lambda n. n^\beta)$.
\end{prop}
\begin{proof}
\cref{complex and shift-complex example 1} shows
\begin{equation*}
\shiftcomplex(\lambda n. n^\alpha) \weakleq \complex(\lambda n. n^{(\alpha+\beta)/2}) \weakle \complex(\lambda n. n^\beta) \weakleq \shiftcomplex(\lambda n. n^\beta).
\end{equation*}
\end{proof}

A variant of \cref{separating generalized shift complexity} can be phrased as an existence question:

\begin{question}
Given an order function $f\colon \mathbb{N} \to \co{0,\infty}$ for which $\limsup_n{\frac{f(n)}{n}} < 1$, for which order functions $g\colon \mathbb{N} \to \co{0,\infty}$ can we guarantee that $\shiftcomplex(f) \setminus \shiftcomplex(g) \neq \emptyset$?
\end{question}

The depth and/or negligibility of $\shiftcomplex(f,c)$ for $f$ satisfying $\limsup_n{\frac{f(n)}{n}} = 0$ is only partially addressed by \cref{sub-identical complex strongly computes shift complex}.

\begin{question} \label{depth of generalized shift complex classes}
Does there exist an order function $f\colon \mathbb{N} \to \co{0,\infty}$ such that $\limsup_n{\frac{f(n)}{n}} = 0$ but $\shiftcomplex(f,c)$ is deep for all $c \in \mathbb{N}$? Negligible?
\end{question}

A particular instance of \cref{depth of generalized shift complex classes} asked by Rumyantsev is for $f$ defined by $f(n) \coloneq n/\log_2 n$, where $\sum_{m=0}^\infty{\frac{f(2^m)}{2^m}} = \sum_{m=0}^\infty{ \frac{2^m/m}{2^m}} = \sum_{m=0}^\infty{\frac{1}{m}} = \infty$ and hence not addressed by \cref{sub-identical complex strongly computes shift complex}. 

There are still open questions about the relationship between (generalized) shift complexity and (generalized) strong shift complexity with the slow-growing $\ldnr$ hierarchy. 

\begin{question} \label{comparing delta-shift complex to slow-growing ldnr}
Does there exist any rational $\delta \in (0,1)$ and $c \in \mathbb{N}$ such that $\ldnr_\slow \weakleq \shiftcomplex(\delta,c)$?
\end{question}

Likewise, there are still many open questions about the relationship between generalized shift complexity and complexity.

\begin{question} \label{comparing generalized shift complex to complex}
For what order functions $f\colon \mathbb{N} \to \co{0,\infty}$ is there a sub-identical order function $g \colon \mathbb{N} \to \co{0,\infty}$ such that $\shiftcomplex(f) \weakleq \complex(g)$?
\end{question}

\clearpage
\chapter{Avoidance -- Slow-Growing versus Fast-Growing}
\label{bushy tree chapter}

A result of Khan \& Miller using bushy tree forcing shows that the classes $\dnr(p)$ do not cleanly stack one atop the other as $p$ ranges through the order functions:

\begin{thm*}
\textnormal{\cite[Theorem 3.11]{khan2017forcing}}
Given any order function $p\colon \mathbb{N} \to (1,\infty)$, there is an order function $q\colon \mathbb{N} \to (1,\infty)$ such that $\dnr(p)$ and $\dnr(q)$ are weakly incomparable.
\end{thm*}

Our aim in this chapter is to lift this result to the $\ldnr$ hierarchy with further generality and adding the guarantee that $q$ can be chosen to be slow-growing:

\begin{repthm}{ldnr incomparable}
For all order functions $p_1\colon \mathbb{N} \to (1,\infty)$ and $p_2\colon \mathbb{N} \to (1,\infty)$, there exists a slow-growing order function $q\colon \mathbb{N} \to (1,\infty)$ such that $\ldnr(p_1) \nweakleq \ldnr(q) \nweakleq \ldnr(p_2)$. In particular, for any order function $p \colon \mathbb{N} \to (1,\infty)$, there exists a slow-growing order function $q\colon \mathbb{N} \to (1,\infty)$ such that $\ldnr(p)$ and $\ldnr(q)$ are weakly incomparable.
\end{repthm}

In \cref{bushy trees section} we cover the necessary combinatorial tools, those of $k$-bushy trees and the associated notions of being $k$-big or $k$-small above a string. 

In \cref{avoidance of individual universal partial recursive functions section}, we prove the following variant of \cref{ldnr incomparable} where avoidance is taken with respect to individual partial recursive functions. 

\begin{repthm}{slow-growing q incomparable to given p}
Suppose $p_1\colon \mathbb{N} \to (1,\infty)$ and $p_2\colon (1,\infty)$ are order functions, $u\colon \mathbb{N} \to \mathbb{N}$ is a strictly increasing order function, $\psi_1$ and $\psi_2$ are universal partial recursive functions, and $\psi_3$ and $\psi_4$ are partial recursive functions. Then there exists a order function $q\colon \mathbb{N} \to (1,\infty)$ such that $q \circ u$ is slow-growing and
\begin{equation*}
\avoid^{\psi_1}(p_1) \nweakleq \avoid^{\psi_3}(q) \quad \text{and} \quad \avoid^{\psi_2}(q \circ u) \nweakleq \avoid^{\psi_4}(p_2).
\end{equation*} 
\end{repthm}

In \cref{avoidance of well-behaved families of universal partial recursive functions section}, we deduce from \cref{slow-growing q incomparable to given p} a generalization in which $\avoid^{\psi_1}(p_1)$ and $\avoid^{\psi_2}(q \circ u)$ are replaced with $\avoid^{\mathcal{C}_1}(p_1)$ and $\avoid^{\mathcal{C}_2}(q \circ u)$, respectively, where $\mathcal{C}_1$ and $\mathcal{C}_2$ may be are families of universal partial recursive functions satisfying a condition we term ``translationally bounded from above'', of which the family of linearly universal partial recursive functions is included. 

\begin{repthm}{stronger main theorem for classes}
Suppose $p_1\colon \mathbb{N} \to (1,\infty)$ and $p_2\colon \mathbb{N} \to (1,\infty)$ are order functions, $u \colon \mathbb{N} \to \mathbb{N}$ is a strictly increasing order function, $\mathcal{C}_1$ and $\mathcal{C}_2$ are nonempty families of universal partial recursive functions which are each translationally bounded from above, and $\psi_1$ and $\psi_2$ are partial recursive functions. Then there exists an order function $q\colon \mathbb{N} \to (1,\infty)$ such that $q \circ u$ is slow-growing and for which
\begin{equation*}
\avoid^{\mathcal{C}_1}(p_1) \nweakleq \avoid^{\psi_1}(q) \quad \text{and} \quad \avoid^{\mathcal{C}_2}(q \circ u) \nweakleq \avoid^{\psi_2}(p_2).
\end{equation*}
\end{repthm}

\cref{ldnr incomparable} is then an easy consequence of \cref{stronger main theorem for classes}. 

In \cref{implications for ldnr-slow section}, we use \cref{ldnr incomparable} to deduce the following implications concerning $\ldnr_\slow$.

\begin{repthm}{ldnr-slow not deep}
$\ldnr_\slow$ is not of deep degree.
\end{repthm}

\begin{repthm}{slow-growing hierarchy has no minimum}
There is no order function $q\colon \mathbb{N} \to (1,\infty)$ such that $\ldnr_\slow \weakeq \ldnr(q)$.
\end{repthm}

\begin{repthm}{SC not weakly below ldnr_slow}
$\shiftcomplex \weaknleq \ldnr_\slow$.
\end{repthm}

In \cref{replacing slow-growing with depth section}, we prove the following variant of \cref{stronger main theorem} where $q$ being slow-growing is strengthened to $\avoid^{\psi_2}(q \circ u)$ being of deep degree.

\begin{repthm}{stronger main theorem deep}
Suppose $p_1\colon \mathbb{N} \to (1,\infty)$ and $p_2 \colon \mathbb{N} \to (1,\infty)$ are order functions, $u\colon \mathbb{N} \to \mathbb{N}$ is a strictly increasing order function, $\psi_1$ and $\psi_2$ are universal partial recursive functions, and $\psi_3$ and $\psi_4$ are partial recursive functions. Then there exists a order function $q\colon \mathbb{N} \to (1,\infty)$ such that $\avoid^{\psi_2}(q \circ u)$ is of deep degree and
\begin{equation*}
\avoid^{\psi_1}(p_1) \nweakleq \avoid^{\psi_3}(q) \quad \text{and} \quad \avoid^{\psi_2}(q \circ u) \nweakleq \avoid^{\psi_4}(p_2).
\end{equation*} 
\end{repthm}

\section{Bushy Trees}
\label{bushy trees section}

The main tool we use in the proof of \cref{ldnr incomparable} and \cref{stronger main theorem} is bushy tree forcing, which was developed by Kumabe in 1993 to answer affirmatively a question of Sacks which asked if there were $\dnr$ reals of minimal Turing degree. The proof was never published, but a draft was in private circulation in 1996. The technique in its current form was introduced in a 2009 publication by Kumabe and Lewis to give a simplified version of Kumabe's original proof.

Bushy tree forcing has since become a standard tool when working with $\dnr$ and has been described \cite{bienvenu2016diagonally} as, ``the canonical forcing notion used in the study of $\dnr$ functions.''

\begin{definition}[$n$-bushy above $\sigma$]
Suppose $\sigma \in \mathbb{N}^\ast$. A tree $T \subseteq \mathbb{N}^\ast$ is \textdef{$n$-bushy above $\sigma$} if every element of $T$ is compatible with $\sigma$ and every $\tau \in T$ which extends $\sigma$ is either a leaf or else has at least $n$ immediate extensions in $T$. $\sigma$ is known as the \textdef{stem} of $T$.
\end{definition}


Subsets $A$ of $\mathbb{N}^\ast$ for which we have many ways of extending $\sigma$ to an element of $A$ are `big'. Otherwise, they are `small'.

\begin{definition}[$n$-big, $n$-small above $\sigma$]
Suppose $\sigma \in \mathbb{N}^\ast$. A set $B \subseteq \mathbb{N}^\ast$ is \emph{$n$-big above $\sigma$} if there is a finite tree $T$ which is $n$-bushy tree above $\sigma$ and for which its leaves lie in $B$. 

If $B$ is not $n$-big above $\sigma$, then $B$ is said to be \emph{$n$-small above $\sigma$}.
\end{definition}

Our arguments are based on the idea that there are `bad' sets of strings which we wish to avoid. If we can ensure that those `bad' sets of strings are sufficiently small, then we can construct a real $X \in \baire$ none of whose initial segments lie in those `bad' sets.

In \cite{khan2017forcing} several fundamental combinatorial lemmas are identified which are reproduced below, sometimes with minor modifications suited to our needs. 

\begin{lem}[Concatenation Property] \label{concatenation property}
\textnormal{\cite[Lemma 2.6]{khan2017forcing}}
Suppose $A \subseteq \mathbb{N}^\ast$ is $n$-big above $\sigma$ and $\langle A_\tau\rangle_{\tau \in A}$ is a family of subsets of $\mathbb{N}^\ast$ indexed by $A$. If $A_\tau$ is $n$-big above $\tau$ for every $\tau \in A$, then $\bigcup_{\tau \in T}{A_\tau}$ is $n$-big above $\sigma$.
\end{lem}
\begin{proof}
For each $\tau \in A$, let $T_\tau$ be a finite $n$-bushy tree above $\tau$ all of whose leaves lie in $A_\tau$. Let $T$ be a finite $n$-bushy tree above $\sigma$ all of whose leaves lie in $A$. Define $\hat{T}$ to be the tree obtained by taking the union of $T$ with the trees $T_\tau$ where $\tau$ is a leaf of $T$ in $A$. We claim that $\hat{T}$ is $n$-bushy above $\sigma$. Because every string in $\hat{T}$ extends an element of $T$, it follows that every string in $\hat{T}$ extends $\sigma$. Now suppose $\rho$ is a string in $\hat{T}$ extending $\sigma$ and which is not a leaf. We consider three cases:
\begin{description}
\item[Case 1:] If $\rho$ is a member of $T$ and not a leaf of $T$, then the fact that $T$ is $n$-bushy above $\sigma$ implies it has at least $n$ immediate extensions in $T \subseteq \hat{T}$. 
\item[Case 2:] If $\rho$ is a leaf of $T$, then $\rho = \tau$ for some $\tau \in A$; $\tau$ extends itself (improperly), so $T_\tau$ being $n$-bushy above $\tau$ implies it is either a leaf of $T_\tau$ (and hence of $\hat{T}$) or it has at least $n$ immediate extensions in $T_\tau \subseteq \hat{T}$.
\item[Case 3:] If $\rho$ is not a member of $T$, then $\rho \in T_\tau$ for some $\tau \in A$. The argument then follows exactly as in Case 2.
\end{description}
It only remains to show that the leaves of $\hat{T}$ lie in $\bigcup_{\tau \in A}{A_\tau}$. Indeed, the leaves of $\hat{T}$ are exactly the leaves of $A_\tau$ for all $\tau \in A$, which each lie in $A_\tau$, respectively, and hence in $\bigcup_{\tau \in A}{A_\tau}$.
\end{proof}

\begin{lem}[Smallness Preservation Property] \label{smallness preservation property}
\textnormal{\cite[Lemma 2.7]{khan2017forcing}}
Suppose that $B, C \subseteq \mathbb{N}^\ast$, $m, n \in \mathbb{N}$, and $\sigma \in \mathbb{N}^\ast$. If $B$ is $m$-small above $\sigma$ and $C$ is $n$-small above $\sigma$, then $B \cup C$ is $(n+m-1)$-small above $\sigma$.
\end{lem}
\begin{proof}
Suppose for the sake of a contradiction that $B \cup C$ is not $(n+m-1)$-small above $\sigma$, i.e., $(n+m-1)$-big above $\sigma$. Then there exists a finite tree $T$ which is $(n+m-1)$-bushy above $\sigma$ all of whose leaves lie in $B \cup C$. We will label each element of $T$ by either a `B' or a `C', starting with leaves and working our way towards the stem $\sigma$. Label a leaf of $T$ `B' if it lies in $B$ and `C' otherwise. If an extension $\tau$ of $\sigma$ in $T$ has not been labeled but all of its proper extensions have been, then label $\tau$ `B' if at least $m$ of its immediate successors have `B' label, and `C' otherwise (by the pigeonhole principle, there must be at least $n$ immediate successors labeled `C'). Continue in this way until $\sigma$ itself has been labeled.

If $\sigma$ has been labeled `B', then the set $T_\mathrm{B}$ of all extensions of $\sigma$ (along with the initial segments of $\sigma$) labeled `B' is a finite tree which is $m$-bushy above $\sigma$. Otherwise, the set $T_\mathrm{C}$ of all extensions of $\sigma$ (along with the initial segments of $\sigma$) labeled `C' is $n$-bushy above $\sigma$. In either case, we reach a contradiction.
\end{proof}

\begin{lem}[Small Set Closure Property] \label{small set closure property}
\textnormal{\cite[Lemma 2.8, essentially]{khan2017forcing}}
Suppose $B \subseteq \mathbb{N}^\ast$ is $k$-small above $\sigma$. Let $C = \{ \tau \in \mathbb{N}^\ast \mid \text{$B$ is $k$-big above $\tau$}\}$. Then $C$ is $k$-small above $\sigma$ and is \textdef{$k$-closed}, i.e., if $C$ is $k$-big above a string $\rho$, then $\rho \in C$.

Moreover, the upward closure of $C$ is $k$-small above $\sigma$ and $k$-closed.
\end{lem}
\begin{proof}
Suppose for the sake of a contradiction that $C$ is $k$-big above $\sigma$. Then, by \cref{concatenation property}, $B$ is $k$-big above $\sigma$, yielding a contradiction. 

The same reasoning can be applied to show that if $C$ is $k$-big above a string $\rho$, then $B$ is $k$-big above $\rho$ and hence $\rho \in C$.

Now consider the upward closure $C^\uparrow \coloneq \{ \rho \in \{0,1\}^\ast \mid \exists \tau \in C \qspace (\tau \subseteq \rho)\}$ of $C$. The following lemma shows that $C^\uparrow$ is similarly $k$-small above $\sigma$ and $k$-closed.	

\begin{lem} 
Suppose $B \subseteq \mathbb{N}^\ast$ and $\sigma \in \mathbb{N}^\ast$. Then $B$ is $k$-big above $\sigma$ if and only if its upward closure $B^\uparrow$ is $k$-big above $\sigma$.
\end{lem}
\begin{proof}
If $B$ is $k$-big above $\sigma$, then any finite $k$-bushy tree $T$ above $\sigma$ realizing this also shows that $B^\uparrow$ is $k$-big above $\sigma$.

Conversely, suppose $B^\uparrow$ is $k$-big above $\sigma$, and let $T$ be a finite $k$-bushy tree above $\sigma$ whose leaves are within $B^\uparrow$. Let
\begin{equation*}
\tilde{T} = \{\tau \in T \mid \tau \in B^\uparrow \wedge \tau \restrict (|\tau|-1) \notin B^\uparrow\}.
\end{equation*}
$\tilde{T}$ is a tree, as if $\tau \in T$ and $\sigma \subsetneq \tau$, then $\sigma$ is an element of $T \setminus B^\uparrow$, and $B^\uparrow$ being upward closed implies no initial segment of $\sigma$ is in $B^\uparrow$, so $\sigma \in \tilde{T}$.  
\end{proof}

\end{proof}

\begin{definition}[$k$-Closure]
If $B \subseteq \mathbb{N}^\ast$ is $k$-small above $\sigma$, then its \textdef{$k$-closure} is the upward closure of the set $\{ \tau \in \mathbb{N}^\ast \mid \text{$B$ is $k$-big above $\tau$}\}$.
\end{definition}

In addition to the above lemmas, we also collect a series of facts which either follow quickly from those above lemmas or else follow immediately from the definitions.

\begin{lem} \label{basic bushy lemmas}
Suppose $\sigma \in \mathbb{N}^\ast$ and $B,C \subseteq \mathbb{N}^\ast$ are given.
\begin{enumerate}[(a)]
\item If $B$ is $\ell$-big above $\sigma$ and $k < \ell$, then $B$ is $k$-small above $\sigma$.
\item If $B$ is $k$-small above $\sigma$ and $k < \ell$, then $B$ is $\ell$-small above $\sigma$.
\item If $B = B_1 \cup B_2 \cup \cdots \cup B_n$ is $n\cdot k$-big above $\sigma$ and $n,k > 0$, then there exists $i \in \{1,2,\ldots,n\}$ such that $B_i$ is $k$-big above $\sigma$.
\item If $B \subseteq C$ and $B$ is $k$-big above $\sigma$, then $C$ is $k$-big above $\sigma$.
\item If $B \subseteq C$ and $C$ is $k$-small above $\sigma$, then $B$ is $k$-small above $\sigma$.
\item If $B$ is $k$-small above $\sigma$ and $k$-closed and $C$ is $k$-big above $\sigma$, then there exists a $\tau \in C \setminus B$ which extends $\sigma$.
\end{enumerate}
\end{lem}
\begin{proof} \mbox{}
\begin{enumerate}[(a)]
\item If $T$ is a finite tree which is $\ell$-bushy above $\sigma$ and all of whose leaves lie in $B$, then $T$ is $k$-bushy above $\sigma$. Thus, $B$ is $k$-big above $\sigma$.

\item If $B$ is $\ell$-big above $\sigma$, then (a) above shows that $B$ is $k$-big above $\sigma$, a contradiction.

\item Suppose for the sake of a contradiction that $B_i$ is $k$-small above $\sigma$ for every $i \in \{1,2,\ldots,n\}$. By repeated applications of \cref{smallness preservation property} we find that $B$ is $n \cdot k - (n-1) = (n\cdot (k-1) + 1)$-small above $\sigma$. $n\cdot (k-1) + 1 < n\cdot k$, so (a) above gives a contradiction.

\item A finite $k$-bushy tree above $\sigma$ whose leaves are in $B$ is a finite $k$-bushy tree above $\sigma$ whose leaves are in $C \supseteq B$. 

\item If $B$ is $k$-big above $\sigma$, then (d) above implies $C$ is $k$-big above $\sigma$, a contradiction.

\item Suppose for the sake of a contradiction that there is no $\tau \in C \setminus B$ extending $\sigma$. Because $C$ is $k$-big above $\sigma$, there exists a $k$-bushy tree $T$ above $\sigma$ all of whose leaves lie in $C$. But every leaf of $T$ is an extension of $\sigma$ in $C$, which by hypothesis implies it lies in $B$, so $T$ is a $k$-bushy tree above $\sigma$ all of whose leaves lie in $B$, contradicting the hypothesis that $B$ is $k$-small above $\sigma$.

\end{enumerate}
\end{proof}

\section{Avoidance of Individual Universal Partial Recursive Functions}
\label{avoidance of individual universal partial recursive functions}
\label{avoidance of individual universal partial recursive functions section}

We wish to use what Khan \& Miller term `basic' bushy tree forcing \cite{khan2017forcing} -- in which we approximate our generic real with finite strings -- to prove the following result.

\begin{thm} \label{ldnr incomparable}
For all order functions $p_1\colon \mathbb{N} \to (1,\infty)$ and $p_2\colon \mathbb{N} \to (1,\infty)$, there exists a slow-growing order function $q\colon \mathbb{N} \to (1,\infty)$ such that $\ldnr(p_1) \nweakleq \ldnr(q) \nweakleq \ldnr(p_2)$. In particular, for any order function $p\colon \mathbb{N} \to (1,\infty)$, there exists a slow-growing order function $q \colon \mathbb{N} \to (1,\infty)$ such that $\ldnr(p)$ and $\ldnr(q)$ are weakly incomparable.
\end{thm}

While our aim is to prove \cref{ldnr incomparable} -- a statement about $\ldnr$, i.e., avoidance of the \emph{family} of linearly universal partial recursive functions -- we first establish the case of avoidance of individual universal partial recursive functions. 

\begin{thm} \label{slow-growing q incomparable to given p} \label{stronger main theorem}
Suppose $p_1\colon \mathbb{N} \to (1,\infty)$ and $p_2\colon \mathbb{N} \to (1,\infty)$ are order functions, $u\colon \mathbb{N} \to \mathbb{N}$ is a strictly increasing order function, $\psi_1$ and $\psi_2$ are universal partial recursive functions, and $\psi_3$ and $\psi_4$ are partial recursive functions. Then there exists a order function $q\colon \mathbb{N} \to (1,\infty)$ such that $q \circ u$ is slow-growing and
\begin{equation*}
\avoid^{\psi_1}(p_1) \nweakleq \avoid^{\psi_3}(q) \quad \text{and} \quad \avoid^{\psi_2}(q \circ u) \nweakleq \avoid^{\psi_4}(p_2).
\end{equation*} 
\end{thm}

In order to make use of tools like the Parametrization and Recursion Theorems, we start by proving the following technical result, in which the instances of $\avoid^{\psi_1}(p_1)$ and $\avoid^{\psi_2}(q \circ u)$ are replaced with $\dnr(p_1)$ and $\dnr(q \circ u)$, respectively, where $\dnr$ may be defined with respect to any admissible enumeration. 

\begin{thm} \label{main technical result}
Suppose $p_1\colon \mathbb{N} \to (1,\infty)$ and $p_2\colon \mathbb{N} \to (1,\infty)$ are order functions, $u\colon \mathbb{N} \to \mathbb{N}$ is a strictly increasing order function, and $\psi_1$, $\psi_2$ are partial recursive functions. Then there exists an order function $q\colon \mathbb{N} \to (1,\infty)$ such that $q \circ u$ is slow-growing and
\begin{equation*}
\dnr(p_1) \nweakleq \avoid^{\psi_1}(q) \quad \text{and} \quad \dnr(q \circ u) \nweakleq \avoid^{\psi_2}(p_2).
\end{equation*}
\end{thm}

\cref{main technical result} shows that we can construct our desired order function $q$ in such a way that we address the (potentially faster-growing) function $q \circ u$ concurrently, no matter how fast-growing $u$ may be. Within the context of deriving \cref{stronger main theorem}, $u$ encapsulates the `translation' from one universal function to another. This intuition will be helpful later when we examine the situation of (sufficiently well-behaved) families universal partial recursive functions in the place of $\psi_1$ and $\psi_2$ within \cref{stronger main theorem}, including in particular the family of linearly universal partial recursive functions.


\begin{proof}[Proof of \cref{stronger main theorem}.]
Let $\varphi_\bullet$ be the admissible enumeration with which $\dnr$ is defined with respect to, and let $\psi$ be its diagonal.

Because $\psi_1$ is universal, there exists a total recursive function $u_1\colon \mathbb{N} \to \mathbb{N}$ such that $\psi_1 \circ u_1 = \psi$. $u_1$ is unbounded (as $\psi = \psi_1 \circ u_1$ is universal) but not necessarily nondecreasing, so let $\tilde{u}_1$ be defined by $\tilde{u}_1(x) \coloneq \max_{i \leq x}{u_1(x)}$. Then $p_1 \circ \tilde{u}_1$ is an order function such that $p_1(x) \leq p_1(\tilde{u}_1(x))$ for all $x \in \mathbb{N}$, so 
\begin{equation*}
\avoid^{\psi}(p_1 \circ \tilde{u}_1) \strongleq \avoid^{\psi}(p_1 \circ u_1) \strongleq \avoid^{\psi_1}(p_1).
\end{equation*}
If $q$ is found such that $\avoid^\psi(p_1 \circ \tilde{u}_1) \nweakleq \avoid^{\psi_3}(q)$, then $\avoid^{\psi_1}(p_1) \nweakleq \avoid^{\psi_3}(q)$. 

Likewise, $\psi_2$ being universal implies there is a total recursive function $u_2\colon \mathbb{N} \to \mathbb{N}$ such that $\psi_2 \circ u_2 = \psi$. With $\tilde{u}_2(x) \coloneq \max_{i \leq x}{u_2(x)}$, for any order function $q$ we have 
\begin{equation*}
\avoid^\psi(q \circ u \circ \tilde{u}_2) \strongleq \avoid^\psi(q \circ u \circ u_2) \strongleq \avoid^{\psi_2}(q \circ u).
\end{equation*}
If $q\colon \mathbb{N} \to (1,\infty)$ is found such that $\avoid^\psi(q \circ u \circ \tilde{u}) \nweakleq \avoid^{\psi_4}(p_2)$, then $\avoid^{\psi_2}(q \circ u) \nweakleq \avoid^{\psi_4}(p_2)$. By potentially replacing $\tilde{u}_2$ with $\tilde{u}_2 + \id_\mathbb{N}$, we may assume without loss of generality that $\tilde{u}_2$ is strictly increasing, so that $u \circ \tilde{u}_2$ is also strictly increasing.

Applying \cref{main technical result} to the order functions $p_1 \circ \tilde{u}_1$, $p_2$, and $u \circ \tilde{u}_2$ and partial recursive functions $\psi_3$, $\psi_4$ yields an order function $q\colon \mathbb{N} \to (1,\infty)$ such that $\avoid^\psi(p_1 \circ \tilde{u}_1) \nweakleq \avoid^{\psi_3}(q)$ and $\avoid^\psi(q \circ u \circ \tilde{u}) \nweakleq \avoid^{\psi_4}(p_2)$, which by our above observation is enough to complete the proof. 
\end{proof}


Consider the special case of \cref{main technical result} where $p_1 = p_2 = p$ and $\psi_1 = \psi_2 = \psi$ is the diagonal of the admissible enumeration $\varphi_\bullet$ with respect to which $\dnr$ is defined. $q$ must be a slow-growing function which cannot dominate $p$ and also cannot be dominated by $p$. This suggests an approach to defining $q$ in which $q$ alternates between two phases, one in which $q$ grows slowly in comparison to $p$ and one in which $q$ grows fastly in comparison to $p$. To ensure $q \circ u$ is slow-growing, whenever we enter a slow-growing phase we stay within that phase long enough to work towards $\sum_{n=0}^\infty{q(n)^{-1}}$ diverging. 

In fact, our construction of $q$ will alternate between two actions. The first action involves keeping $q$ constant for a sufficiently long time so that the bushy tree forcing arguments within $p_2^\ast$ go through. In the second action, $q$ makes a sudden jump so that $q$ passes certain watermarks infinitely often, allowing the bushy tree forcing arguments within $q^\ast$ to go through. By staying constant sufficiently long, we can ensure that $q \circ u$ is slow-growing.

\begin{proof}[Proof of \cref{main technical result}.]
Suppose $p_1$ and $p_2$ are order functions and $\psi_1$, $\psi_2$ are partial recursive functions. Let $\varphi_\bullet$ be an admissible enumeration and let $\psi$ be its diagonal. 
Let $\langle \Gamma_i \rangle_{i\in \mathbb{N}}$ be an effective enumeration of all partial recursive functions $\Gamma_i \colonsub \mathbb{N}^\ast \times \mathbb{N} \to \mathbb{N}$ as in \cref{characterizations of partial recursive functional}(i), and assume without loss of generality that if $\Gamma_i^\tau(n) \converge$ then it does so within $|\tau|$-many steps, so that the predicate $R(i,\tau,n) \equiv \Gamma_i^\tau(n) \converge$ is recursive. Given $X \in \baire$, $\Gamma_i^X$ is the partial function defined by $\Gamma_i^X(n) \simeq m$ if and only if there exists $s \in \mathbb{N}$ such that $\Gamma_i^{X \restrict s}(n) \converge = m$.

\paragraph*{Step 1: Defining auxiliary functions.}

We will define total recursive functions $\theta_1 \colon \mathbb{N}^\ast \times \mathbb{N}^2 \to \mathbb{N}$ and $\theta_2 \colon \mathbb{N}^\ast \times \mathbb{N}^3 \to \mathbb{N}$ which will be used in the definition of $q$. In anticipation of using the Recursion Theorem, one of the arguments of $\theta_1$ and $\theta_2$ will be reserved for an index $e$ which will eventually be an index of $q$.

Define the partial recursive function $\chi_1 \colonsub \mathbb{N}^\ast \times \mathbb{N}^3 \to \mathbb{N}$ so that on input $\langle \sigma, i, e, x\rangle$, $\chi_1$ searches for a finite tree $T$ such that:
\begin{enumerate}[(i)]
\item For every $\tau \in T$ and every $j \leq |\tau|$, $\varphi_e(j) \converge$ and $\tau(j) < \varphi_e(j)$ when $j < |\tau|$.
\item $T$ is $k$-bushy above $\sigma$ for some $k < \varphi_e(|\sigma|)$.
\item $\Gamma_{i-1}^\tau(x)$ converges to a common value $j < p_1(x)$ for every leaf $\tau$ of $T$. (For the case of $i=0$, set $\Gamma_{-1} = \Gamma_0$.)
\end{enumerate}
$\chi_1(\sigma,i,e,x)$ is equal to that common value of $\Gamma_i^\tau(x)$ for the first such tree $T$ found, whenever such a tree exists. $\chi_1$ is partial recursive, so by the Parametrization Theorem, there exists a total recursive function $\theta_1 \colon \mathbb{N}^\ast \times \mathbb{N}^2 \to \mathbb{N}$ such that $\varphi_{\theta_1(\sigma,i,e)}(x) \simeq \chi_1(\sigma,i,e,x)$ for all $\sigma \in \mathbb{N}^\ast$ and $i,e,x \in \mathbb{N}$. 

Similarly, the partial recursive function $\chi_2 \colonsub \mathbb{N}^\ast \times \mathbb{N}^4 \to \mathbb{N}$ is defined so that on input $\langle \sigma, i, e, x, N\rangle$, $\chi_2$ first attempts to compute $\varphi_e(N)$, followed by verifying that $N \in \im {u}$ (denoting by $n$ the unique element of ${u}^{-1}[\{N\}]$), then finds the least $k > n$ such that $p_2(k) \geq (\varphi_e(N)+1)\cdot p_2(n)$, and finally searches for a finite tree $T$ such that:
\begin{enumerate}[(i)]
\item For every $\tau \in T$ and every $j < |\tau|$, $\tau(j) < p_2(j)$.
\item $T$ is $p_2(n)$-bushy above $\sigma$.
\item $\Gamma_i^\tau(x)$ converges to a common value $j < \varphi_e(N)$ for every leaf $\tau$ of $T$.
\end{enumerate}
$\chi_2(\sigma,i,e,x,N)$ is equal to that common value of $\Gamma_i^\tau(x)$ for the first such tree $T$ found, whenever such a tree exists. $\theta_2 \colon \mathbb{N}^\ast \times \mathbb{N}^3 \to \mathbb{N}$ is then a total recursive function for which $\varphi_{\theta_2(\sigma,i,e,N)}(x) \simeq \chi_2(\sigma,i,e,x,N)$ for all $\sigma \in \mathbb{N}^\ast$ and $i,e,x,N \in \mathbb{N}$. We may assume without loss of generality that $\theta_2(\sigma,i,e,N) \geq |\sigma|$ for all $\sigma \in \mathbb{N}^\ast$ and $i,e,N \in \mathbb{N}$.

\paragraph*{Step 2: Defining $q$.} 

Having defined $\theta_1$ and $\theta_2$, we are in a position to define $q$ by way of a partial recursive function $Q \colonsub \mathbb{N}^2 \to \mathbb{N}$. To aid in its construction, we simultaneously define three other partial recursive functions $s, i, N \colonsub \mathbb{N}^2 \to \mathbb{N}$ --- $s$ takes values in $\{0,1\}$, indicating which type of action we perform next, while $i$ and $N$ keep track of our progress for one of those two actions. Write $\overline{\theta}_1(e,x) \coloneq \max_{j < x, \sigma \in (\varphi_e)^x}{\theta_1(\sigma,j,e)}$, assuming $\varphi_e(j) \converge$ for all $j < x$.
\begin{description}
\item[$Q(e,0)$.] Define $Q(e,0) \coloneq 3$ and $s(e,0) = i(e,0) = N(e,0) \coloneq 0$.

\item[$Q(e,x)$ ($x>1$).] On input $\langle e,x \rangle$, $Q$ attempts to compute $\varphi_e(j)$, $Q(e,j)$, $s(e,j)$, $i(e,j)$, and $N(e,i(e,j))$ for each $j < x$. If and when it has done so successfully, the computation proceeds in one of two ways depending on the value of $s(e,x-1)$.
\begin{description}
\item[Case 1: $s(e,x-1)=0$.] Compute $\overline{\theta}_1(e,x) \coloneq \max_{j < x, \sigma \in (\varphi_e)^x}{\theta_1(\sigma,j,e)}$, and then set 
\begin{equation*}
Q(e,x) \coloneq p_1(\overline{\theta}_1(e,x))\cdot Q(e,x-1) + Q(e,x-1) + 2.
\end{equation*}
Additionally set $s(e,x) \coloneq 1$ and $i(e,x) \coloneq i(e,x-1)$.

\item[Case 2: $s(e,x-1)=1$.] Write $i \coloneq i(e,x-1)$ and $N_i \coloneq N(e,i)$. If $N_i \notin \im u$, then $Q(e,x)$, $s(e,x)$, and $i(e,x)$ all diverge. Otherwise, let $n_i \coloneq {u}^{-1}(N_i)$. As in the definition of $\chi_2$, let $k > n_i$ be the least natural number such that $p_2(k) \geq (\varphi_e(N_i) + 1) \cdot p_2(n_i)$ ($k$ is defined because $\varphi_e(N_i) \converge$ and $N_i \in \im {u}$). Let $M$ be the least natural number such that 
\begin{enumerate*}[(I)] 
\item $M > \max_{\sigma \in p_2^k}{{u}(\theta_2(\sigma,i,e,N_i))}$, 
\item $M \in \im({u})$,
\item and $\sum_{N_i < {u}(j) < M}{Q(e,N_i)^{-1}} \geq 1$.
\end{enumerate*}

If the inequality $N_i < x < M$ fails, then $Q(e,x)$, $s(e,x)$, and $i(e,x)$ all diverge. Otherwise:
	\begin{description}
	\item[Subcase 1.] If $x < M-1$, then set $Q(e,x) \coloneq Q(e,x-1)$, $s(e,x) \coloneq s(e,x-1)$, and $i(e,x) \coloneq i(e,x-1)$. 
	
	\item[Subcase 2.] If $x = M-1$, then set $Q(e,x) \coloneq Q(e,x-1)$, $s(e,x) \coloneq 0$, $i(e,x) \coloneq i(e,x-1)+1$, and $N(e,i(e,x)) \coloneq M$.
	
	\end{description}

\end{description}
\end{description}

By the Recursion Theorem, there exists an $e \in \mathbb{N}$ such that $Q(e,x) \simeq \varphi_e(x)$ for all $x \in \mathbb{N}$. From the construction of $Q$, if $Q(e,j)$ is defined for all $j < x$ then $Q(e,x)$ is defined; along with the fact that $Q(e,0)$ is defined, it follows that $\varphi_e$ is a total recursive function. The construction of $N$ also ensures that $N(e,i) \in \im u$ for all $i \in \mathbb{N}$. We will write
\begin{equation*}
\begin{array}{c}
	\begin{array}{c c c c c}
		q \coloneq \varphi_e, & s(x) \coloneq s(e,x), & i(x) \coloneq i(e,x), & N_i \coloneq N(e,i), & n_i \coloneq u^{-1}(N_i),
	\end{array} \vspace{2pt} \\
	\begin{array}{c c c}
		\theta_1(\sigma,i) \coloneq \theta_1(\sigma,i,e), & \overline{\theta}_1(x) \coloneq \max\limits_{j<x, \sigma \in q^x}{\theta_1(\sigma,j,e)}, & \theta_2(\sigma,i) \coloneq \theta_2(\sigma,i,e,N_i).
	\end{array}
\end{array}
\end{equation*}
Each of these are total recursive functions.

Both $s(x) = 0$ and $s(x) = 1$ occur for infinitely many $x \in \mathbb{N}$. That Case 1 ($s(x-1) = 0$) occurs infinitely often implies that $q$ is unbounded. That Case 2 ($s(x-1)=1$) occurs infinitely often implies that $q \circ {u}$ is slow-growing. In both Case 1 and Case 2, the definition of $Q$ enforces that $q(x) \leq q(x+1)$ for all $x \in \mathbb{N}$, so $q$ and $q \circ {u}$ are slow-growing order functions.

\paragraph*{Step 3: Showing $\avoid^\psi(p_1) \nweakleq \avoid^{\psi_1}(q)$.} 

We use basic bushy tree forcing. Let $\mathbb{P}$ be the set of all pairs $\langle \sigma,B \rangle \in q^\ast \times \mathcal{P}(q^\ast)$ where $\sigma \neq \langle\rangle$ and $B$ is $k$-small above $\sigma$, $k$-closed, and upward closed for some $k \leq q(|\sigma|-1)$. $\langle \tau, C \rangle$ \emph{extends} $\langle \sigma, B \rangle$ if $\sigma \subseteq \tau$ and $B \subseteq C$. 
For $i \in \mathbb{N}$, let $\mathcal{D}_i$ denote the set of $\langle \sigma,B \rangle \in \mathbb{P}$ such that for all $X \in \bbracket{\sigma}_q \setminus \bbracket{B}_q$, $\Gamma_i^X \notin \avoid^\psi(p_1)$. 

\begin{description}
\item[Claim 1.] For each $m$, $\mathcal{T}_m = \{ \langle \sigma,B \rangle \in \mathbb{P} \mid |\sigma| \geq m\}$ is dense open in $\mathbb{P}$.

\begin{proof}
$\mathcal{T}_m$ is clearly open in $\mathbb{P}$. To show that $\mathcal{T}_m$ is dense in $\mathbb{P}$, let $C = \{ \tau \in q^\ast \mid |\tau| \geq m\}$. For any string $\sigma$, $C$ is $k$-big above $\sigma$ if and only if $k \leq q(|\sigma|)$. In particular, $C$ is $k$-big above $\sigma$ for all $k \leq q(|\sigma|-1)$.

Suppose $\langle \sigma, B \rangle \in \mathbb{P}$; let $k \leq q(|\sigma|-1)$ be such that $B$ is $k$-small above $\sigma$ and $k$-closed. If $|\sigma| \geq m$, then we are done. Otherwise, let $\tau$ be any string in $C \setminus B$ extending $\sigma$. Because $B$ is $k$-closed and $\tau \notin B$, $B$ is $k$-small above $\tau$. Then $\langle \tau, B \rangle$ is an extension of $\langle \sigma,B \rangle$ in $\mathcal{T}_m$. 
\end{proof}

\item[Claim 2.] If $\mathcal{G}$ is any filter on $\mathbb{P}$, then for all $\langle \sigma, B \rangle \in \mathcal{G}$, $X_\mathcal{G} \coloneq \bigcup\{ \tau \in q^\ast \mid \exists C \subseteq q^\ast (\langle \tau,C \rangle  \in \mathcal{G})\}$ has no initial segment in $B$.

\begin{proof}
Suppose otherwise, so that there is $\tau \in B$ which is an initial segment of $X_\mathcal{G}$. By the definition of $X_\mathcal{G}$, there must be a $\langle \rho',C' \rangle \in \mathcal{G}$ such that $\rho'$ extends $\tau$. Let $\langle \rho, C \rangle$ be a common extension of $\langle \rho',C' \rangle$ and $\langle \sigma,B \rangle$. Because $B$ is upward-closed, $\rho \in B$. But $B \subseteq C$, so $\rho \in C$ and hence $C$ is $k$-big above $\rho$ for every $k$, a contradiction.
\end{proof}

\item[Claim 3.] For all $i \in \mathbb{N}$, $\mathcal{D}_i$ is dense open in $\mathbb{P}$. 

\begin{proof}
$\mathcal{D}_i$ is clearly open in $\mathbb{P}$, so it must remains to show that $\mathcal{D}_i$ is dense in $\mathbb{P}$.

Suppose $\langle \sigma,B \rangle \in \mathbb{P}$; let $k \leq q(|\sigma|-1)$ be such that $B$ is $k$-small above $\sigma$. By potentially extending to an element of $\mathcal{T}_m$ for an appropriate $m$, we can assume that $|\sigma| > i$ and that in the computation of $q(|\sigma|)$, Case 2 occurs. Let $x = |\sigma|$ and define
\begin{equation*}
A = \{ \tau \in q^\ast \mid \Gamma_i^\tau(\theta_1(\sigma,i)) \converge < p_1(\theta_1(\sigma,i)) \}.
\end{equation*}

There are two cases:
\begin{description}
\item[Case I.] Suppose $A$ is $p_1(\theta_1(\sigma,i))\cdot q(x-1)$-small above $\sigma$. Let $c = p_1(\theta_1(\sigma,i))\cdot q(x-1) + k - 1$, so $A \cup B$ is $c$-small above $\sigma$ by \cref{smallness preservation property}. Let $C$ be the $c$-closure of $A \cup B$ (note that although this $c$-closure is taken in $\mathbb{N}^\ast$, $A \cup B$ being $c$-big above $\tau$ implies $\tau \in q^\ast$, so $C \subseteq q^\ast$). By definition, $C$ is $c$-small above $\sigma$, $c$-closed, and upward closed. 

Since $q(x) = p(\theta_1(\sigma,i))\cdot q(x-1) + q(x-1) + 2 > c$ and $q$ is nondecreasing, $\{ \tau \in q^{x+1} \mid \sigma \subseteq \tau\}$ is $c$-big above $\sigma$. Thus, there is a string $\tau$ in $q^{x+1} \setminus C$ extending $\sigma$. Because $C$ is $c$-closed, it is $c$-small above $\tau$, and hence $\langle \tau,C \rangle \in \mathbb{P}$ (here we are also using the fact that $C$ is upward closed). Because $A \subseteq C$, $\langle \tau,C \rangle \in \mathcal{D}_i^\psi$ by virtue of $\Gamma^g_i$ being partial for any $g \in \bbracket{\tau}_q \setminus \bbracket{C}_q$. Because $B \subseteq C$ and $\sigma \subseteq \tau$, $\langle \tau,C \rangle$ extends $\langle \sigma,B \rangle$.

\item[Case II.] If $A$ is $p_1(\theta_1(\sigma,i)) \cdot q(x-1)$-big above $\sigma$, then \cref{basic bushy lemmas}(c) implies $\{ \tau \in q^\ast \mid \Gamma^\tau_i(\theta_1(\sigma,i)) \converge = k\}$ is $q(x-1)$-big above $\sigma$ for some $k < p_1(\theta_1(\sigma,i))$. This implies $\varphi_{\theta_1(\sigma,i)}(\theta_1(\sigma,i))$ is defined. Let $\tau$ be an extension of $\sigma$ in $q^\ast \setminus B$ such that $\Gamma_i^\tau(\theta_1(\sigma,i)) = \varphi_{\theta_1(\sigma,i)}(\theta_1(\sigma,i))$. Then $\langle \tau,B \rangle$ is an element of $\mathcal{D}_i$ extending $\langle \sigma,B \rangle$.

\end{description}

\end{proof}

\end{description}

Let $B_{\avoid^{\psi_1}(q)}$ be the set of all strings in $q^\ast$ which cannot be extended to an element of $\avoid^{\psi_1}(q)$. Let $\sigma_0$ be a string of length $1$ such that $\sigma_0(0) \nsimeq \psi_1(0)$, if $\psi_1(0) \converge$. $B_{\avoid^{\psi_1}(q)}$ is $2$-small above $\sigma_0$, upward closed, and $2$-closed, so $\langle \sigma_0,B_{\avoid^{\psi_1}(q)} \rangle \in \mathbb{P}$. Let $\mathcal{G}$ be a filter containing $\langle \sigma_0,B_{\avoid^{\psi_1}(q)} \rangle$ which meets $\mathcal{T}_m$ and $\mathcal{D}_i$ for all $m,i \in \mathbb{N}$. Claim 1 shows that $X_\mathcal{G} \in \prod{q}$, Claim 2 shows that $X_\mathcal{G} \in \avoid^{\psi_1}(q)$, and Claim 3 shows that $X_\mathcal{G}$ computes no element of $\avoid^\psi(p_1)$. In other words, 
\begin{equation*}
\avoid^\psi(p_1) \nweakleq \avoid^{\psi_1}(q).
\end{equation*}

\paragraph*{Step 4: Showing $\avoid^\psi(q \circ {u}) \nweakleq \avoid^{\psi_2}(p_2)$.} 

We define a sequence $\langle \sigma_i,B_i \rangle_{i \in \mathbb{N}}$ such that:
\begin{enumerate}[(i)]
\item $\sigma_i \in p_2^{n_i} \setminus B_i$.
\item $B_i \subseteq p_2^\ast$ is $p_2(n_i)$-small above $\sigma_i$.
\item For all $i \in \mathbb{N}$, $\sigma_i \subseteq \sigma_{i+1}$ and $B_i \subseteq B_{i+1}$.
\end{enumerate}
Let $B_1 = B_{\avoid^{\psi_2}(p_2)}$ and $\sigma_1$ an arbitrary element of $p_2^{n_1} \setminus B_1$. (Note that $p_2(1) \geq 2$ and $B_1$ is $2$-small above $\sigma_1$, so in particular $p_2(1)$-small above $\sigma_1$.) Suppose $\sigma_i, B_i$ have been constructed. Let $k$ be as in Case 2 of the construction of $Q$ 
and let $\rho$ be an extension of $\sigma_i$ in $p_2^k \setminus B_i$ ($k \geq n_i$, so $p_2^k$ is $p_2(n_i)$-big above $\sigma_i$). For $j < (q \circ {u})(\theta_2(\rho,i))$, let
\begin{equation*}
A_j = \{ \tau \in p_2^\ast \mid \Gamma_{i-1}^\tau(\theta_2(\rho,i)) \converge = j\}.
\end{equation*}
We have two cases, depending upon whether $A_j$ is $p_2(n_i)$-big above $\rho$ for some $j$ or not.
\begin{description}
\item[Case 1.] If $A_j$ is $p_2(n_i)$-big above $\rho$ for some $j$, then $\varphi_{\theta_2(\rho,i)}(\theta_2(\rho,i))$ is defined. In that case, let $j' = \varphi_{\theta_2(\rho,i)}(\theta_2(\rho,i))$, so there is a $\tau \in A_{j'} \setminus B_i$ extending $\rho$ such that 
\begin{equation*}
\Gamma_{i-1}^\tau(\theta_2(\rho,i)) = j' = \varphi_{\theta_2(\rho,i)}(\theta_2(\rho,i)).
\end{equation*}
Let $B_{i+1} = B_i$ and let $\sigma_{i+1}$ be any extension of $\tau$ in $p_2^{n_{i+1}} \setminus B_{i+1}$.

\item[Case 2.] If $A_j$ is $p_2(n_i)$-small above $\rho$ for all $j$, then $\bigcup_j{A_j}$ is $(p_2(n_i) \cdot (q\circ {u})(\theta_2(\rho,i)) - (q \circ {u})(\theta_2(\rho,i)) + 1)$-small above $\rho$. Let 
\begin{equation*}
c = p_2(n_i) \cdot ( (q \circ {u})(\theta_2(\rho,i)) + 1) - (q \circ {u})(\theta_2(\rho,i))
\end{equation*}
Then $C = \bigcup_j{A_j} \cup B_i$ is $c$-small above $\rho$. 

We claim that $(q \circ {u})(n_i) = (q \circ {u})(\theta_2(\rho,i))$. First we note the inequality $k = |\rho| \leq \theta_2(\rho,i)$ by the definition of $\theta_2$, so that $N_i < u(k) = u(|\rho|) \leq \theta_2(\rho,i)$. By the definition of $N_{i+1}$, we additionally have $\theta_2(\rho,i) < N_{i+1}$. The definition of $q$ then gives $q(N_i) = q({u}(\theta_2(\rho,i)))$, or equivalently $(q \circ {u})(n_i) = (q \circ {u})(\theta_2(\rho,i))$. 

As a result, 
\begin{align*}
p_2(k) & \geq p_2(n_i) \cdot ((q \circ {u})(n_i)+1) \\
& = p_2(n_i) \cdot ((q \circ {u})(\theta_2(\rho,i))+1) \\
& \geq c,
\end{align*}
so that $C$ is $p_2(k)$-small above $\rho$. Because $k \leq n_{i+1}$, $C$ is also $p_2(n_{i+1})$-small above $\rho$. Finally, let $B_{i+1} = C$ and let $\sigma_{i+1}$ be any extension of $\rho$ in $p_2^{n_{i+1}} \setminus B_{i+1}$.
\end{description}

This completes the construction of the sequence $\langle \sigma_i,B_i \rangle_{i\in\mathbb{N}}$. Let $Y \coloneq \bigcup_{i \in \mathbb{N}}{\sigma_i}$. By construction, $Y \in \avoid^{\psi_2}(p_2)$.  Observe that within the above construction, falling into Case 1 at stage $i$ implies $\Gamma_i^Y$ is not a member of $\avoid^\psi(q)$. In other words,
\begin{equation*}
\avoid^\psi(q \circ {u}) \nweakleq \avoid^{\psi_2}(p_2).
\end{equation*}

This completes the proof.

\end{proof}

\section{Avoidance of Well-Behaved Families of Universal Partial Recursive Functions}
\label{avoidance of well-behaved families of universal partial recursive functions section}

Now we turn to families of universal partial recursive functions, with the prototypical example being the family of linearly universal partial recursive functions.

\begin{definition}[translation]
If $\psi_1$ and $\psi_2$ are universal partial recursive functions, then a \emph{translation} from $\psi_1$ to $\psi_2$ is a total recursive function $u$ such that $\psi_1 \circ u = \psi_2$. 
\end{definition}

\begin{definition}[translationally bounded]
Fix a universal partial recursive function $\psi_0$. We say that a family $\mathcal{C} \subseteq \baire$ of universal partial recursive functions is \emph{translationally bounded} if there exists an order function $U \colon \mathbb{N} \to \mathbb{N}$ such that for every $\psi \in \mathcal{C}$ there is a translation from $\psi$ to $\psi_0$ which is dominated by $U$. 
\end{definition}

\begin{remark}
Being translationally bounded from above does not depend on the choice of universal partial recursive function $\psi_0$: Suppose $\mathcal{C}$ is translationally bounded with respect to $\psi_0$, witnessed by $U$, and $\psi_1$ is another universal partial recursive function. Let $v$ be a translation from $\psi_1$ to $\psi_0$, and let $\overline{v}$ be the order function defined by $\overline{v}(x) \coloneq \max_{i \leq x}{v(i)}$. Then $U \circ \overline{v}$ witnesses $\mathcal{C}$ being translationally bounded with respect to $\psi_1$.
\end{remark}

\begin{example}
The family $\mathcal{LU}$ of linearly universal partial recursive functions is translationally bounded, witnessed by the order function $U = \lambda n. n^2$. The same function also shows the family of the diagonals of linear admissible enumerations of the partial recursive functions is also translationally bounded.
\end{example}

\begin{example}
Suppose $\mathcal{F}$ is a set of total recursive functions such that there is a recursive function $U$ dominating every member of $\mathcal{F}$. Temporarily say that a universal partial recursive function $\psi$ is \emph{$\mathcal{F}$-universal} if for every partial recursive function $\theta \colonsub \mathbb{N} \to \mathbb{N}$ there exists a $u \in \mathcal{F}$ such that $\psi \circ u = \theta$. Then the family $\mathcal{C}$ of all $\mathcal{F}$-universal partial recursive functions is translationally bounded, witnessed by $U$. 

Particular examples include the set of linear functions, the set of recursive linearly bounded functions, or the set of primitive recursive functions.
\end{example}

\begin{lem} \label{lower bound of dnr wrt class}
Suppose $\mathcal{C}$ is a nonempty family of universal partial recursive functions which is translationally bounded, witnessed by the order function $U\colon\mathbb{N} \to \mathbb{N}$. Then for all order functions $p\colon \mathbb{N} \to (1,\infty)$ and all universal partial recursive functions $\psi_0$, $\avoid^{\psi_0}(p \circ U) \weakleq \avoid^\mathcal{C}(p)$.
\end{lem}
\begin{proof}
Let $X$ be an element of $\avoid^\mathcal{C}(p)$, so that there is a $\psi \in \mathcal{C}$ such that $X \in \avoid^\psi(p)$. By hypothesis, there exists a total recursive function $u\colon \mathbb{N} \to \mathbb{N}$ such that $\psi \circ u = \psi_0$ and which is dominated by $U$. Then $\avoid^{\psi_0}(p \circ U) \strongleq \avoid^{\psi_0}(p \circ u) \strongleq \avoid^\psi(p)$, so $X$ computes a member of $\avoid^{\psi_0}(p \circ U)$.
\end{proof}

\begin{thm} \label{stronger main theorem for classes}
Suppose $p_1\colon \mathbb{N} \to (1,\infty)$ and $p_2\colon \mathbb{N} \to (1,\infty)$ are order functions, $u\colon \mathbb{N} \to \mathbb{N}$ is a strictly increasing order function, $\mathcal{C}_1$ and $\mathcal{C}_2$ are nonempty families of universal partial recursive functions which are each translationally bounded, and $\psi_1$ and $\psi_2$ are partial recursive functions. Then there exists an order function $q\colon \mathbb{N} \to (1,\infty)$ such that $q \circ u$ is slow-growing and for which
\begin{equation*}
\avoid^{\mathcal{C}_1}(p_1) \nweakleq \avoid^{\psi_1}(q) \quad \text{and} \quad \avoid^{\mathcal{C}_2}(q \circ u) \nweakleq \avoid^{\psi_2}(p_2).
\end{equation*}
\end{thm}
\begin{proof}
Let $\psi$ be the diagonal of an acceptable enumeration of the partial recursive functions. By hypothesis, there are order functions $U_1\colon \mathbb{N} \to \mathbb{N}$ and $U_2\colon \mathbb{N} \to \mathbb{N}$ witnessing the fact that $\mathcal{C}_1$ and $\mathcal{C}_2$ are translationally bounded, respectively. Applying \cref{main technical result} to  $p_1 \circ U_1$, $p_2$, and $u \circ U_2$ yields an order function $q\colon \mathbb{N} \to (1,\infty)$ such that $q \circ u \circ U_2$ is slow-growing (and hence $q \circ u$ as well), $\avoid^\psi(p_1 \circ U_1) \nweakleq \avoid^{\psi_1}(q)$, and $\avoid^\psi(q \circ u \circ U_2) \nweakleq \avoid^{\psi_2}(p_2)$. By \cref{lower bound of dnr wrt class}, $\avoid^\psi(p_1 \circ U_1) \weakleq \avoid^{\mathcal{C}_1}(p_1)$ and $\avoid^\psi(q \circ u \circ U_2) \weakleq \avoid^{\mathcal{C}_2}(q \circ u)$, so $\avoid^\psi(p_1 \circ U_1) \nweakleq \avoid^{\psi_1}(q)$ implies $\avoid^{\mathcal{C}_1}(p_1) \nweakleq \avoid^{\psi_1}(q)$ and $\avoid^\psi(q \circ u \circ  U_2) \nweakleq \avoid^{\psi_2}(p_2)$ implies $\avoid^{\mathcal{C}_2}(q \circ u) \nweakleq \avoid^{\psi_2}(p_2)$.
\end{proof}

\cref{ldnr incomparable} is then an easy consequence:

\begin{proof}[Proof of \cref{ldnr incomparable}.]
Let $p_1$ and $p_2$ be order functions, $u \coloneq \id_\mathbb{N}$, $\mathcal{C}_1 = \mathcal{C}_2 \coloneq \mathcal{LU}$, and $\psi_1 = \psi_2 \coloneq \psi$ any linearly universal partial recursive function. By \cref{stronger main theorem for classes}, there is a slow-growing order function $q\colon \mathbb{N} \to (1,\infty)$ such that $\ldnr(p_1) \nweakleq \avoid^\psi(q)$ and $\ldnr(q) \nweakleq \avoid^\psi(p_2)$. Since $\avoid^\psi(q) \subseteq \ldnr(q)$ and $\avoid^\psi(p_2) \subseteq \ldnr(p_2)$, we find that $\ldnr(p_1) \nweakleq \ldnr(q) \nweakleq \ldnr(p_2)$.
\end{proof}

\section{Implications for \texorpdfstring{$\ldnr_\slow$}{LUA_slow}}
\label{implications for ldnr-slow section}

\cref{ldnr incomparable} allows us to make several deductions concerning $\ldnr_\slow$. The first is that $\ldnr_\slow$ is not of deep degree:

\begin{thm} \label{ldnr-slow not deep}
$\ldnr_\slow$ is not of deep degree. 
\end{thm}
\begin{proof}
Suppose for the sake of a contradiction that $\ldnr_\slow$ is weakly equivalent to a deep r.b.\ $\Pi^0_1$ class $P$. We may assume without loss of generality that $P \subseteq \cantor$. Because $P$ is deep, taking the inverse of a modulus of depth for $P$ gives us an order function $r\colon \mathbb{N} \to \mathbb{N}$ such that $\mathbf{M}(P \restrict n) \leq 2^{-r(n)}$ for all $n \in \mathbb{N}$. In particular, $\mathbf{M}(X \restrict n) \leq 2^{-r(n)}$ for every $X \in P$ and $n \in \mathbb{N}$, or equivalently $\apc(X \restrict n) \geq r(n)$ for every $X \in P$ and $n \in \mathbb{N}$, i.e., every $X \in P$ is strongly $r$-random. Thus, every $X \in P$ is $r$-random, so $P \subseteq \complex(r)$ and hence $\complex(r) \strongleq P$. \cref{complex reals compute ldnr functions specific} shows there exists a fast-growing order function $p\colon \mathbb{N} \to (1,\infty)$ such that $\ldnr(p) \strongleq \complex(r)$. Applying \cref{ldnr incomparable} yields a slow-growing order function $q\colon \mathbb{N} \to (1,\infty)$ such that $\ldnr(p)$ and $\ldnr(q)$ are weakly incomparable. But $\ldnr(q) \subseteq \ldnr_\slow$ and hence
\begin{equation*}
\ldnr(p) \strongleq \complex(r) \strongleq P \weakeq \ldnr_\slow \strongleq \ldnr(q)
\end{equation*}
giving a contradiction.
\end{proof}

Likewise, we can show that $\ldnr_\slow \nweakeq \ldnr(q)$ for every slow-growing order function $q$:

\begin{thm} \label{slow-growing hierarchy has no minimum}
There is no order function $q\colon \mathbb{N} \to (1,\infty)$ such that $\ldnr_\slow \weakeq \ldnr(q)$.
\end{thm}
\begin{proof}
Suppose for the sake of a contradiction that $\ldnr_\slow \weakeq \ldnr(q)$. By \cref{ldnr incomparable}, there exists a slow-growing order function $p \colon \mathbb{N} \to (1,\infty)$ such that $\ldnr(p)$ and $\ldnr(q)$ are weakly incomparable. But $p$ being slow-growing implies $\ldnr(p) \stronggeq \ldnr_\slow \weakeq \ldnr(q)$, a contradiction.
\end{proof}

Similarly, we can show that $\shiftcomplex \weaknleq \ldnr_\slow$.

\begin{thm} \label{SC not weakly below ldnr_slow}
$\shiftcomplex \weaknleq \ldnr_\slow$.
\end{thm}
\begin{proof}
For each rational $\delta \in (0,1)$, $\complex(\delta) \strongleq \shiftcomplex(\delta)$. Since $\sqrt{n} \leq \delta \cdot n$ for almost all $n$, it follows that $\complex(\lambda n.\sqrt{n}) \weakleq \shiftcomplex(\delta)$ for all rational $\delta \in (0,1)$, so $\complex(\lambda n.\sqrt{n}) \weakleq \shiftcomplex$. By \cref{complex reals compute ldnr functions recursive sum} there exists a fast-growing order function $p\colon \mathbb{N} \to (1,\infty)$ such that $\ldnr(p) \weakleq \complex(\lambda n.\sqrt{n})$. \cref{ldnr incomparable} implies there is a slow-growing order function $q\colon \mathbb{N} \to (1,\infty)$ such that $\ldnr(p)$ and $\ldnr(q)$ are weakly incomparable. Thus, if $\shiftcomplex \weakleq \ldnr_\slow$, then we would have
\begin{equation*}
\ldnr(p) \weakleq \complex(\lambda n. \sqrt{n}) \weakleq \shiftcomplex \weakleq \ldnr_\slow \weakleq \ldnr(q),
\end{equation*}
yielding a contradiction.
\end{proof}

\begin{cor} \label{SC not weakly below ldnr_slow corollary}
There exists a slow-growing order function $q\colon \mathbb{N} \to (1,\infty)$ such that $\shiftcomplex \weaknleq \ldnr(q)$.
\end{cor}

\section{Replacing Slow-Growing with Depth}
\label{replacing slow-growing with depth section}

The proof of \cref{weak degree of dnr_p is ill-defined} produced admissible enumerations $\varphi_\bullet$ and $\tilde{\varphi}_\bullet$ such that $\dnr^{(1)}_{\lambda n.2^n} \strongeq \dnr^{(2)}_{\lambda x.x}$, where $\dnr^{(1)}$ is $\dnr$ defined with respect to $\varphi_\bullet$ and $\dnr^{(2)}$ is $\dnr$ defined with respect to $\tilde{\varphi}_\bullet$. For this reason, the implications of $q$ being `slow-growing' on the weak degree of $\avoid^\psi(q)$ are dependent on the choice of $\psi$. In the case where $\psi$ is linearly universal partial recursive, $q$ is slow-growing if and only if $\avoid^\psi(q)$ is a deep r.b.\ $\Pi^0_1$ class. With that motivation in mind, we strengthen \cref{stronger main theorem}, with the depth of $\avoid^{\psi_2}(q \circ u)$ replacing slow-growing.

\begin{thm} \label{stronger main theorem deep}
Suppose $p_1\colon \mathbb{N} \to (1,\infty)$ and $p_2\colon \mathbb{N} \to (1,\infty)$ are order functions, $u\colon \mathbb{N} \to \mathbb{N}$ is a strictly increasing order function, $\psi_1$ and $\psi_2$ are universal partial recursive functions, and $\psi_3$ and $\psi_4$ are partial recursive functions. Then there exists a order function $q\colon \mathbb{N} \to (1,\infty)$ such that $\avoid^{\psi_2}(q \circ u)$ is of deep degree and
\begin{equation*}
\avoid^{\psi_1}(p_1) \nweakleq \avoid^{\psi_3}(q) \quad \text{and} \quad \avoid^{\psi_2}(q \circ u) \nweakleq \avoid^{\psi_4}(p_2).
\end{equation*} 
\end{thm}

Requiring that $\avoid^{\psi_2}(q \circ u)$ is deep is a stronger condition than simply requiring that $q \circ u$ be slow-growing.

\begin{prop} \label{deep implies slow-growing}
Suppose $\psi$ is a universal partial recursive function and $q\colon \mathbb{N} \to (1,\infty)$ is an order function. If $\avoid^\psi(q)$ is of deep degree, then $q$ is slow-growing.
\end{prop}

A benefit of working with linearly universal partial recursive functions rather than the diagonals of linear admissible enumerations is that if $\psi$ is linearly universal and $\tilde{\psi}$ is univeral, then we may take the translation $u$ from $\psi$ to $\tilde{\psi}$ to be strictly increasing. Moreover, \cref{slow-growing composed with linear is slow-growing} showed that composition with a linear map does not affect whether an order function is fast-growing or slow-growing.

\begin{proof}[Proof of \cref{deep implies slow-growing}.]
Suppose $\psi_0$ is a linearly universal partial recursive function and $\avoid^\psi(q)$ is of deep degree. There are $a,b \in \mathbb{N}$ such that $\psi_0(ax+b) \simeq \psi(x)$ for all $x \in \mathbb{N}$, so let $u\colon \mathbb{N} \to \mathbb{N}$ be defined by $u(x) \coloneq ax+b$ for $x \in \mathbb{N}$. Then $\psi_0 \circ u = \psi$. Because $\psi$ is universal, $a$ must be nonzero (otherwise, $\psi$ would either be constant or undefined everywhere depending on whether $\psi_0(b)$ converges or diverges, respectively). $u$ is hence a strictly increasing order function.

We extend $q$ and $u$ to real-valued functions $\overline{q}\colon \co{0,\infty} \to (1,\infty)$ and $\overline{u}\colon \co{0,\infty} \to \co{0,\infty}$, respectively, by letting $\overline{q}$ be linear between $\langle x,q(x)\rangle$ and $\langle x+1,q(x+1)\rangle$ for all $x \in \mathbb{N}$ and likewise with $\overline{u}$. To be exact, we define
\begin{align*}
\overline{q}(x) & \coloneq (q(\lfloor x \rfloor+1) - q(\lfloor x \rfloor))(x-\lfloor x \rfloor) + q(\lfloor x \rfloor), \\
\overline{u}(x) & \coloneq (u(\lfloor x \rfloor+1) - u(\lfloor x \rfloor))(x-\lfloor x \rfloor) + u(\lfloor x \rfloor).
\end{align*}
$\overline{u}$ is strictly increasing, so its inverse $\overline{u}^{-1} \colon [b,\infty) \to [0,\infty)$ is defined. Finally, define $\tilde{q}\colon \mathbb{N} \to (1,\infty)$ by 
\begin{equation*}
\tilde{q}(x) \coloneq \lfloor \overline{q}(\overline{u}^{-1}(x))\rfloor
\end{equation*}
for all $x \in \mathbb{N}$. $\tilde{q} \circ u$ is dominated by $q$, so $\avoid^\psi(q) \weakleq \avoid^{\psi_0}(\tilde{q})$. 

\cref{deep degrees form filter} implies $\avoid^{\psi_0}(\tilde{q})$ is of deep degree, so \cref{depth of dnr_p^psi} implies $\tilde{q}$ must be slow-growing. $\overline{u}^{-1}$ is a linear map, so $\tilde{q}$ being slow-growing implies $\tilde{q} \circ \overline{u}$ is also slow-growing by \cref{slow-growing composed with linear is slow-growing}. $q$ is dominated by $\tilde{q} \circ \overline{u} + 1$, so $q$ is slow-growing as well.
\end{proof}

\begin{proof}[Proof of \cref{stronger main theorem deep}.]
Let $\psi_0$ be a linearly universal partial recursive function, let $\varphi_\bullet$ be the admissible enumeration corresponding to $\psi_0$ as in \cref{properties of linearly univeral partial recursive functions section}, and let $\psi$ be the diagonal of $\varphi_\bullet$. Define $\dnr$ with respect to $\varphi_\bullet$ (i.e., so that $\dnr = \avoid^\psi$). Let $v\colon \mathbb{N} \to \mathbb{N}$ be a translation from $\psi_2$ to $\psi$ and let $\tilde{v}\colon \mathbb{N} \to \mathbb{N}$ be a strictly increasing order function dominating $v$ (e.g., $\tilde{v}(x) \coloneq \max-{i \leq x}{v(i)} + x$ for $x \in \mathbb{N}$). 

\cref{stronger main theorem} shows there is an order function $q\colon \mathbb{N} \to (1,\infty)$ such that $q \circ (u \circ \tilde{v})$ is slow-growing, $\avoid^{\psi_1}(p_1) \nweakleq \avoid^{\psi_3}(q)$, and $\avoid^\psi(q \circ (u \circ \tilde{v})) \nweakleq \avoid^{\psi_4}(p_2)$. Because $q \circ (u \circ \tilde{v})$ is slow-growing and $\psi$ is linearly universal, it follows that $\avoid^\psi(q \circ (u \circ \tilde{v}))$ is deep. Using \cref{deep pi01 classes form filter under strong reduction} and observing that
\begin{equation*}
\avoid^\psi(q \circ (u \circ \tilde{v})) \strongleq \avoid^\psi(q \circ (u \circ v)) = \avoid^{\psi_2 \circ v}((q \circ u) \circ v) \strongleq \avoid^{\psi_2}(q \circ u)
\end{equation*}
shows $\avoid^{\psi_2}(q \circ u)$ is deep. $\avoid^\psi(q \circ (u \circ \tilde{v})) \nweakleq \avoid^{\psi_4}(p_2)$ implies $\avoid^{\psi_2}(q \circ u) \nweakleq \avoid^{\psi_4}(p_2)$, so we have found the desired $q$.
\end{proof}

\section{Open Problems}

Although \cref{ldnr incomparable} significantly expands our understanding of the relationships between the fast and slow-growing $\ldnr$ hierarchies as well as the structure of the slow-growing $\ldnr$ hierarchy itself, there remain many open problems concerning these two subjects.

For example, \cref{ldnr incomparable} shows that to each order function $p$ there is a $q$ such that $\ldnr(p)$ and $\ldnr(q)$ are weakly incomparable, and hence we must have $p \domnleq q$ and $q \domnleq p$, but what more can be said?

\begin{question} \label{better understanding incomparable order functions}
Given an order function $p\colon \mathbb{N} \to (1,\infty)$, what can be said about how the growth rates of the slow-growing order functions $q\colon \mathbb{N} \to (1,\infty)$ for which $\ldnr(p)$ and $\ldnr(q)$ are weakly incomparable  compare to the growth rate of $p$? In particular, for the $q$ defined in the proof of \cref{ldnr incomparable} alternates between staying constant and making large jumps, but can we quantify the lengths of those constant periods or the size of those jumps? 
\end{question}

Although \cref{slow-growing hierarchy has no minimum} shows that there is no slow-growing order function $q$ for which $\ldnr_\slow \weakeq \ldnr(q)$, it does not eliminate the possibility that there exist slow-growing order functions $q$ for which $\weakdeg(\ldnr(q))$ is minimal among the weak degrees of the slow-growing $\ldnr$ hierarchy. 

\begin{question} \label{minimal slow-growing ldnr question}
Given a slow-growing order function $q\colon \mathbb{N} \to (1,\infty)$, does there exist a slow-growing order function $q^+\colon \mathbb{N} \to (1,\infty)$ such that $\ldnr(q^+) \weakle \ldnr(q)$? If so, can we take $q^+$ so that $q \domleq q^+$, and if that is true, can we quantify how much faster-growing $q^+$ must be than $q$ for $\ldnr(q^+) \weakle \ldnr(q)$ to hold?
\end{question}

A related question which would address \cref{minimal slow-growing ldnr question} if answered affirmative is the following:

\begin{question} \label{slow-growing ldnr downwards-directed question}
Given slow-growing order functions $p_1\colon \mathbb{N} \to (1,\infty)$ and $p_2\colon \mathbb{N} \to (1,\infty)$, is there a slow-growing order function $q\colon \mathbb{N} \to (1,\infty)$ such that $\ldnr(q) \weakleq \ldnr(p_1) \cup \ldnr(p_2)$?
\end{question}

An affirmative answer to \cref{slow-growing ldnr downwards-directed question} would provide an affirmative answer to the first half of \cref{minimal slow-growing ldnr question} thanks to \cref{ldnr incomparable}. In \cref{slow-growing ldnr downwards-directed question}, we cannot add the requirement that $q$ dominate both $p_1$ and $p_2$, as there exist slow-growing order functions $p_1$ and $p_2$ such that $\max\{p_1,p_2\}$ is fast-growing. 

A positive answer to \cref{slow-growing ldnr downwards-directed question} for $p_1 = \id_\mathbb{N}$:

\begin{prop} \label{max with identity is slow-growing}
Suppose $p$ is a slow-growing order function. Then There exists a slow-growing order function $q$ such that $\ldnr(q) \strongleq \ldnr(p) \cup \ldnr(\lambda n.n)$.
\end{prop}
\begin{proof}
It suffices to show that $q \coloneq \max\{p,\id_\mathbb{N}\}$ is slow-growing. Define $A \coloneq \{ n \in \mathbb{N} \mid p(n) \leq n\}$. If $A$ is finite, then $\id_\mathbb{N} \domleq p$, so $\max\{p(n),n\} = p(n)$ for almost all $n$, hence $q$ is slow-growing. So suppose $A$ is infinite. Given $n \in A$, let $m$ be maximal such that $2^m \leq n$, so that $p(2^m) \leq p(n) \leq n \leq 2^{m+1}$. Thus, $2^m \cdot \frac{1}{p(2^m)} \geq \frac{1}{2}$. It follows that $\sum_{m=0}^\infty{2^m \cdot \frac{1}{\max\{p(2^m),2^m\}}} = \infty$. By the Cauchy Condensation Test, this implies $\sum_{n=0}^\infty{q(n)^{-1}} = \infty$.
\end{proof}

\begin{cor}
There exists a slow-growing order function $q$ such that $\ldnr(q) \weakle \ldnr(\lambda n.n)$.
\end{cor}
\begin{proof}
By either \cref{ldnr incomparable} or combining \cite[Theorem 3.11]{khan2017forcing} and \cref{fast-growing dominates identity}, there exists a slow-growing order function $p$ such that $\ldnr(p)$ is weakly incomparable with $\ldnr(\lambda n.n)$. By \cref{max with identity is slow-growing}, there is a slow-growing $q$ such that $\ldnr(q) \weakleq \ldnr(p) \cup \ldnr(\lambda n.n)$, hence $\ldnr(q) \weakle \ldnr(\id_\mathbb{N})$.
\end{proof}

\cref{SC not weakly below ldnr_slow} suggests the following question:

\begin{question} \label{quantify which slow-growing order functions do not compute shift complexity}
Can we give a natural and specific slow-growing function $q\colon \mathbb{N} \to (1,\infty)$ such that $\shiftcomplex \weaknleq \ldnr(q)$?
\end{question}

\clearpage
\chapter{Structure of the Deep Region of \texorpdfstring{$\mathcal{E}_\weak$}{Ew}}
\label{structure of filter of deep degrees}

Two results stated in \cref{depth section} that gave some idea of the structure of the collection of deep degrees in $\mathcal{E}_\weak$ were \cref{deep degrees form filter}, which showed that the collection forms a filter in $\langle \mathcal{E}_\weak\rangle$, and \cref{difference randoms cannot compute member of sets of deep degree}, which shows that no difference random computes a member of any representative of a deep degree. The goal of this chapter is to examine the structure of the filter of deep degrees further. Our main goal is to prove the following main theorem.

\begin{thm} \label{proper nesting of deep-related filters}
Define
\begin{align*}
\mathscr{F}_\mathrm{deep} & \coloneq \{ \mathbf{p} \in \mathcal{E}_\weak \mid \text{$\mathbf{p}$ a deep degree}\}. \\
\mathscr{F}_\mathrm{pseudo} & \coloneq \{ \mathbf{p} \in \mathcal{E}_\weak \mid \text{$\mathbf{p} = \inf \mathcal{C}$ for some $\mathcal{C} \subseteq \mathscr{F}_\mathrm{deep}$}\}. \\
\mathscr{F}_\mathrm{diff} & \coloneq \{ \mathbf{p} \in \mathcal{E}_\weak \mid \forall P \in \mathbf{p} \forall X \in \mlr \qspace ( \exists Y \in P \qspace (Y \turingleq X) \to (0' \turingleq 0'))\}.
\end{align*}
Then $\mathscr{F}_\mathrm{pseudo}$ is a principal filter while $\mathscr{F}_\mathrm{deep}$ and $\mathscr{F}_\mathrm{diff}$ are nonprincipal filters. Consequently, $\mathscr{F}_\mathrm{deep} \subsetneq \mathscr{F}_\mathrm{pseudo} \subsetneq \mathscr{F}_\mathrm{diff}$.
\end{thm}

In \cref{infimum of deep degrees section}, we show that the infimum of the collection of deep degrees $\weakdeg(L)$ lies in $\mathcal{E}_\weak$ but is not a deep degree itself, showing the filter of deep degrees is nonprincipal. 

\begin{repprop}{union of deep pi01 classes is sigma03}
The union $L$ of all deep $\Pi^0_1$ classes is $\Sigma^0_3$. Consequently, $\weakdeg(L) \in \mathcal{E}_\weak$.
\end{repprop}

\begin{repthm}{L not of deep degree}
$\weakdeg(L)$ is not a deep degree in $\mathcal{E}_\weak$.
\end{repthm}

In \cref{filter of pseudo-deep degrees section}, we define the collection of \emph{pseudo}-deep degrees in $\mathcal{E}_\weak$ and characterize it as the principal filter generated by $\weakdeg(L)$.

\begin{repthm}{pseudo-deep degrees forms principal filter}
$\{ \mathbf{p} \in \mathcal{E}_\weak \mid \text{$\mathbf{p}$ pseudo-deep}\}$ is equal to the principal filter generated by $\weakdeg(L)$ in $\langle\mathcal{E}_\weak,\leq\rangle$.
\end{repthm}

In \cref{filter of deep degrees and difference randoms section}, we show that the filters of deep degrees and pseudo-deep degrees cannot be characterized by the property that no difference random computes a member of any representative of those degrees.

\begin{repthm}{f-pseudo not f-diff}
There exists a $\Pi^0_1$ class $P$ which is not of pseudo-deep degree but for which no difference random computes an element of $P$.
\end{repthm}

\section{The infimum of all deep degrees}
\label{infimum of deep degrees section}

An important observation about the collection of deep degrees is that its infimum is in $\mathcal{E}_\weak$. 

\begin{prop} \label{union of deep pi01 classes is sigma03}
The union $L$ of all deep $\Pi^0_1$ classes is $\Sigma^0_3$. Consequently, $\weakdeg(L) \in \mathcal{E}_\weak$.
\end{prop}
\begin{proof}
Let $\mathbf{M}$ be a fixed universal left r.e.\ continuous semimeasure on $\mathbb{N}^\ast$, and let the map $\langle s,\sigma\rangle \mapsto \mathbf{M}_s(\sigma)$ realize the left recursive enumerability of $\mathbf{M}$, i.e., it is a recursive function $\mathbb{N} \times \mathbb{N}^\ast \to \mathbb{Q}$ such that $\langle \mathbf{M}_s(\sigma)\rangle_{s \in \mathbb{N}}$ converges monotonically to $\mathbf{M}(\sigma)$ from below for each $\sigma \in \mathbb{N}$. Given $e,s \in \mathbb{N}$, let $P_{e,s} \coloneq \{ X \in \cantor \mid \varphi_{e,s}^{X \restrict s}(0) \diverge\}$; $\langle P_{e,s}\rangle_{s \in \mathbb{N}}$ is a sequence of uniformly recursive subsets whose intersection is $P_e$.

\begin{description}
\item[Claim 1.] The predicate $\{ \langle e,m,q \rangle \mid \mathbf{M}(P_e \restrict m) \leq q\}$ is a $\Pi^0_2$ subset of $\mathbb{N}^2 \times \mathbb{Q}_{\geq 0}$.

\begin{proof}
The predicate $\{ \langle e,m,q,s\rangle \mid \mathbf{M}_s(P_{e,s} \restrict m) \leq q\}$ is recursive, so it suffices to show that
\begin{equation*}
\mathbf{M}(P_e \restrict m) \leq q \iff \forall t \exists s \qspace (t < s \wedge \mathbf{M}_s(P_{e,s} \restrict m) \leq q).
\end{equation*}
This follows essentially from the observation that the sequence $\langle \mathbf{M}_s(P_{e,s} \restrict m) \rangle_{s \in \mathbb{N}}$ is eventually nondecreasing for all $e, m \in \mathbb{N}$.

In the forward direction, suppose $\mathbf{M}(P_e \restrict m) \leq q$. $P_e \restrict m$ is a finite set equal to $\bigcap_{s \in \mathbb{N}}{P_{e,s}}$, so there exists an $t \in \mathbb{N}$ such that $P_{e,s} \restrict m = P_e \restrict m$ for all $s > t$. Thus, there are arbitrarily large $s$ such that $\mathbf{M}_s(P_{e,s} \restrict m) \leq \mathbf{M}_s(P_e \restrict m) \leq q$. 

In the opposite direction, assume $\forall t \exists s \qspace (t < s \wedge \mathbf{M}_s(P_{e,s} \restrict m) \leq q)$ and suppose for the sake of a contradiction that $\mathbf{M}(P_e \restrict m) > q$. Let $t$ be large enough so that $P_{e,s} \restrict m = P_e \restrict m$ and $\mathbf{M}_s(P_e \restrict m) > q$  for all $s > t$. But then for all $s > t$ we have $\mathbf{M}_s(P_{e,s} \restrict m) > q$, contradicting our assumption. Thus, $\mathbf{M}(P_e \restrict m) \leq q$ is a $\Pi^0_2$ subset of $\mathbb{N}^2 \times \mathbb{Q}$.
\end{proof}

\item[Claim 2.] The predicate $\{ i \in \mathbb{N} \mid \text{$\varphi_i$ is total}\}$ is a $\Pi^0_2$ subset of $\mathbb{N}$.

\begin{proof}
$\varphi_i$ being total is equivalent to $\forall n \exists s \qspace \varphi_{i,s}(n) \converge$.
\end{proof}

\item[Claim 3.] The predicate $\{ e \in \mathbb{N} \mid \text{$P_e$ is deep}\}$ is a $\Sigma^0_3$ subset of $\mathbb{N}$.

\begin{proof}
$P_e$ is deep if and only if $\exists i ( (\text{$\varphi_i$ is total}) \wedge \forall n (\mathbf{M}(P_e \restrict \varphi_i(n)) \leq 2^{-n}))$. By Claim 1, this is $\Sigma^0_3$.
\end{proof}

\end{description}

Finally,
\begin{align*}
L & = \bigcup\{ P_e \mid \text{$P_e$ is deep}\} \\
& = \{ X \in \cantor \mid \exists e ((\text{$P_e$ is deep}) \wedge X \in P_e)\}
\end{align*}
shows $L$ is $\Sigma^0_3$. Since $L$ contains a nonempty $\Pi^0_1$ class, the \nameref{embedding lemma} implies $\weakdeg(L) \in \mathcal{E}_\weak$.
\end{proof}

That $\ldnr_\slow$ (\cref{ldnr-slow not deep}) is not of deep degree shows that $L$ is not of deep degree, further clarifying the structure of the filter of deep degrees in $\mathcal{E}_\weak$:

\begin{thm} \label{L not of deep degree}
$\weakdeg(L)$ is not a deep degree in $\mathcal{E}_\weak$.
\end{thm}
\begin{proof}
Suppose for the sake of a contradiction that $\weakdeg(L)$ is a deep degree in $\mathcal{E}_\weak$. Then
\begin{align*}
\weakdeg(L) & = \inf \{ \weakdeg(P) \mid \text{$P \subseteq \cantor$ is nonempty, deep}\} \\
& \leq \inf\{ \weakdeg(\ldnr(p)) \mid \text{$p$ slow-growing order function}\} \\
& = \weakdeg(\ldnr_\slow).
\end{align*}
Because $\weakdeg(L)$ is a deep degree in $\mathcal{E}_\weak$, \cref{deep degrees form filter} shows that $\weakdeg(\ldnr_\slow)$ is a deep degree in $\mathcal{E}_\weak$, contradicting \cref{ldnr-slow not deep}.
\end{proof}

\begin{cor} \label{filter of deep degrees is nonprincipal}
The filter of deep degrees in $\mathcal{E}_\weak$ is nonprincipal.
\end{cor}

\begin{cor}
For any deep degree $\mathbf{p} \in \mathcal{E}_\weak$, there exists a deep degree $\mathbf{q} \in \mathcal{E}_\weak$ such that $\mathbf{q} < \mathbf{p}$.
\end{cor}

\section{The Filter of Pseudo-Deep Degrees}
\label{filter of pseudo-deep degrees section}

Motivated by $L$, $\ldnr_\slow$, and $\shiftcomplex$, we define:

\begin{definition}[pseudo-deep degree in $\mathcal{E}_\weak$]
A weak degree $\mathbf{p} \in \mathcal{E}_\weak$ is a \textdef{pseudo-deep degree (in $\mathcal{E}_\weak$)} if $\mathbf{p}$ is an infimum of deep degrees in $\mathcal{E}_\weak$, or equivalently that there is a collection $\mathcal{C}$ of deep $\Pi^0_1$ classes such that $\mathbf{p} = \weakdeg(\bigcup{\mathcal{C}})$.

$P \subseteq \baire$ is \textdef{of pseudo-deep degree} if $\weakdeg(P)$ is a pseudo-deep degree in $\mathcal{E}_\weak$.
\end{definition}

\cref{difference randoms cannot compute member of sets of deep degree} continues to hold for $P \subseteq \baire$ of pseudo-deep degree.

\begin{prop} \label{difference randoms cannot compute member of sets of pseudo-deep degree}
Suppose $P \subseteq \baire$ is of pseudo-deep degree. If $X \in \cantor$ is difference random, then $X$ computes no member of $P$.
\end{prop}
\begin{proof}
Because $P$ is of pseudo-deep degree, there is a collection $\mathcal{C}$ of deep $\Pi^0_1$ classes such that $P \weakeq \bigcup{\mathcal{C}}$. If $Y \turingleq X$ for some $Y \in P$, then the fact that $P \weakeq \bigcup{\mathcal{C}}$ implies there is a $Q \in \mathcal{C}$ and a $Z \in Q$ such that $Z \turingleq Y \turingleq X$, contradicting \cref{difference randoms cannot compute members of deep pi01 classes}.
\end{proof}

\begin{notation}
Let \vspace{-1.5em}
\begin{align*}
\mathscr{F}_\mathrm{deep} & \coloneq \{ \mathbf{p} \in \mathcal{E}_\weak \mid \text{$\mathbf{p}$ is a deep degree}\}, \\
\mathscr{F}_\mathrm{pseudo} & \coloneq \{ \mathbf{p} \in \mathcal{E}_\weak \mid \text{$\mathbf{p}$ is a pseudo-deep degree}\}.
\end{align*}
\end{notation}

Just as $\mathscr{F}_\mathrm{deep}$ is a filter, $\mathscr{F}_\mathrm{pseudo}$ also forms a filter -- in fact, it is the principal filter generated by $\weakdeg(L)$.

\begin{thm} \label{pseudo-deep degrees forms principal filter}
$\mathscr{F}_\mathrm{pseudo}$ is equal to the principal filter generated by $\weakdeg(L)$ in $\langle\mathcal{E}_\weak,\leq\rangle$.
\end{thm}
\begin{proof}
Suppose $\mathbf{p}$ is a pseudo-deep degree in $\mathcal{E}_\weak$. Let $\mathcal{C}$ be a collection of deep $\Pi^0_1$ classes such that $\mathbf{p} = \weakdeg(\bigcup{\mathcal{C}})$. Then $L \supseteq \bigcup{\mathcal{C}}$, so $\weakdeg(L) \leq \mathbf{p}$. This shows that every pseudo-deep degree in $\mathcal{E}_\weak$ lies in the filter generated by $\weakdeg(L)$ in $\langle\mathcal{E}_\weak,\leq\rangle$.

Conversely, suppose $\weakdeg(L) \leq \mathbf{p} \in \mathcal{E}_\weak$. Let $P$ be a $\Pi^0_1$ class for which $\mathbf{p} = \weakdeg(P)$. $(\mathcal{D}_\weak,\leq)$ is completely distributive (\cref{properties of weak and strong reducibility}(f)), so
\begin{align*}
\weakdeg(P) & = \sup\{\weakdeg(P),\weakdeg(L)\} \\
& = \sup\{\weakdeg(P),\inf\{\weakdeg(P_e) \mid \text{$P_e$ is deep}\}\} \\
& = \inf\{\sup\{\weakdeg(P),\weakdeg(P_e)\} \mid \text{$P_e$ is deep}\} \\
& = \inf\{\weakdeg(P \times P_e) \mid \text{$P_e$ is deep}\}.
\end{align*}
Thus, $\mathbf{p}$ is pseudo-deep.
\end{proof}

\begin{cor} \label{no minimal pseudo-deep degree unequal to L}
There is no minimal element of $\mathscr{F}_\mathrm{pseudo} \setminus \{\weakdeg(L)\}$..
\end{cor}
\begin{proof}
By \cref{pseudo-deep degrees forms principal filter}, if $\mathbf{p}$ is a pseudo-deep degree distinct from $\weakdeg(L)$ then $\weakdeg(L) \weakle \mathbf{p}$. The Density Theorem for $\mathcal{E}_\weak$ \cite[Theorem 2]{binns2016mass} shows that there exists $\mathbf{q} \in \mathcal{E}_\weak$ such that $\weakdeg(L) \weakle \mathbf{q} \weakle \mathbf{p}$. A second application of \cref{pseudo-deep degrees forms principal filter} shows $\mathbf{q}$ is a pseudo-deep degree, and hence $\mathbf{p}$ is not a minimal element of $\mathscr{F}_\mathrm{pseudo} \setminus \{ \weakdeg(L)\}$.
\end{proof}

\section{The Filter of Deep Degrees in \texorpdfstring{$\mathcal{E}_\weak$}{Ew} and Difference Randoms}
\label{filter of deep degrees and difference randoms section}

\cref{difference randoms cannot compute member of sets of pseudo-deep degree} shows that no difference random computes a member of any $P \subseteq \baire$ of pseudo-deep degree. However, we can show that this does not characterize the pseudo-deep degrees.

\begin{thm} \label{f-pseudo not f-diff}
There exists a $\Pi^0_1$ class $P$ which is not of pseudo-deep degree but for which no difference random computes an element of $P$.
\end{thm}

\begin{notation}
Let $\mathscr{F}_\mathrm{diff}$ be the collection of all weak degrees $\weakdeg(P)$ in $\mathcal{E}_\weak$ such that no difference random computes a member of $P$.
\end{notation}

\begin{prop} \label{f-diff is a filter}
$\mathscr{F}_\mathrm{diff}$ is a filter.
\end{prop}
\begin{proof}
Given $\mathbf{p},\mathbf{q} \in \mathcal{E}_\weak$, let $P$ and $Q$ be $\Pi^0_1$ classes such that $\weakdeg(P) = \mathbf{p}$ and $\weakdeg(Q) = \mathbf{q}$.

Suppose $\mathbf{p} \in \mathscr{F}_\mathrm{diff}$ and $\mathbf{p} \leq \mathbf{q} \in \mathcal{E}_\weak$. If $Y \turingleq X$ for some $Y \in Q$, then $\mathbf{p} \leq \mathbf{q}$ implies $Z \turingleq Y$ for some $Z \in P$, from which we find $Z \turingleq X$, a contradiction. Thus, $\mathbf{q} \in \mathscr{F}_\mathrm{diff}$.

Now suppose $\mathbf{p},\mathbf{q} \in \mathscr{F}_\mathrm{diff}$. $\inf\{\mathbf{p},\mathbf{q}\} = \weakdeg(P \cup Q)$. As no difference random computes any member of $P$ or $Q$, no difference random computes any member of $P \cup Q$, i.e., $\inf\{\mathbf{p},\mathbf{q}\} \in \mathscr{F}_\mathrm{diff}$.
\end{proof}

We can show that $\mathscr{F}_\mathrm{pseudo} \subsetneq \mathscr{F}_\mathrm{diff}$ by showing that $\mathscr{F}_\mathrm{diff}$ is non-principal.

\begin{thm} \label{no difference random computation but not pseudo-deep}
$\mathscr{F}_\mathrm{diff}$ is non-principal.
\end{thm}

To prove \cref{no difference random computation but not pseudo-deep} we make use of the notion of $\pfc$-triviality.

\begin{definition}[$\pfc$-trivial]
$X \in \cantor$ is \textdef{$\pfc$-trivial} if there exists $c \in \mathbb{N}$ such that $\pfc(X \restrict n) \leq \pfc(n) + c$ for all $n \in \mathbb{N}$.
\end{definition}

\begin{proof}[Proof of \cref{no difference random computation but not pseudo-deep}.]
Suppose $\mathbf{q} \in \mathscr{F}_\mathrm{diff}$, and let $Q$ be a $\Pi^0_1$ class such that $\mathbf{q} = \weakdeg(Q)$. By \cite[Lemma 2]{binns2003splitting}, there exist r.e.\ sets $A, B \subseteq \mathbb{N}$ such that $0 \turingle A,B \turingle 0'$, $A \cap B = \emptyset$ and $A \cup B = 0'$, and for which neither $A$ nor $B$ compute any member of $Q$. Note that $A \oplus B \turingeq 0'$.

\cref{reducible to halting problem implies pi02 singleton} implies $\{A\}$ and $\{B\}$ are $\Pi^0_2$, so that $Q \cup \{A\}$ and $Q \cup \{B\}$ are each $\Pi^0_2$. Both sets contain a nonempty $\Pi^0_1$ class (namely, $Q$), so the \nameref{embedding lemma} implies there are $\Pi^0_1$ classes $P_A$ and $P_B$ such that $P_A \weakeq Q \cup \{A\}$ and $P_B \weakeq Q \cup \{B\}$. Because neither $A$ nor $B$ compute any member of $Q$, we have $P_A \weakeq Q \cup \{A\} \weakle Q$ and $P_B \weakeq Q \cup \{B\} \weakle Q$.

By \cite[Theorem 11.6.2]{downey2010algorithmic}, if $A$ and $B$ are both $\pfc$-trivial, then $A \oplus B \turingeq 0'$ is $\pfc$-trivial, which is a contradiction. Thus, at least one of $A$ and $B$ are not $\pfc$-trivial. Without loss of generality, say that $A$ is not $\pfc$-trivial, and let $P = P_A$.

Suppose $X$ is a difference random. \cref{difference randoms cannot compute member of sets of pseudo-deep degree} shows that $X$ computes no member of $Q$, and if $A \turingleq X$ then \cite[Corollary 3.6]{hirschfeldt2007using} implies $A$ is $\pfc$-trivial, contrary to hypothesis. Thus, $X$ computes no member of $Q \cup \{A\}$, and hence computes no member of $P$, showing $\weakdeg(P) \in \mathscr{F}_\mathrm{diff}$. As $\mathbf{q}$ was an arbitrary member of $\mathscr{F}_\mathrm{diff}$, it follows that $\mathscr{F}_\mathrm{diff}$ is non-principal.
\end{proof}

\begin{proof}[Proof of \cref{f-pseudo not f-diff}.]
\cref{f-diff is a filter} and \cref{no difference random computation but not pseudo-deep} show that $\mathscr{F}_\mathrm{diff}$ is a nonprincipal filter. By \cref{difference randoms cannot compute member of sets of pseudo-deep degree}, $\mathscr{F}_\mathrm{diff} \subseteq \mathscr{F}_\mathrm{pseudo}$, but $\mathscr{F}_\mathrm{pseudo}$ being principal means this inclusion must be proper.
\end{proof}

All these facts add up to imply \cref{proper nesting of deep-related filters}:

\begin{proof}[Proof of \cref{proper nesting of deep-related filters}.]
\cref{filter of deep degrees is nonprincipal} and \cref{no difference random computation but not pseudo-deep} shows that the outer two are nonprincipal, while \cref{pseudo-deep degrees forms principal filter} shows the middle is. This implies that the inclusions $\mathscr{F}_\mathrm{deep} \subseteq \mathscr{F}_\mathrm{pseudo} \subseteq \mathscr{F}_\mathrm{diff}$ must be proper. 
\end{proof}

\section{Open questions about the filter of pseudo-deep degrees}

\cref{pseudo-deep degrees forms principal filter} and \cref{no minimal pseudo-deep degree unequal to L} give important structural information about the filter of pseudo-deep degrees. However, there remain open questions about that structure, especially $\mathscr{F}_\mathrm{pseudo} \setminus \mathscr{F}_\mathrm{deep}$.

\begin{question} \label{size of gap between pseudo-deep and deep}
What is the cardinality of $\mathscr{F}_\mathrm{pseudo} \setminus \mathscr{F}_\mathrm{deep}$? I.e., how many pseudo-deep degrees are there which aren't deep degrees?
\end{question}

\cref{no minimal pseudo-deep degree unequal to L} puts constraints on $|\mathscr{F}_\mathrm{pseudo} \setminus \mathscr{F}_\mathrm{deep}|$.

\begin{prop}
$|\mathscr{F}_\mathrm{pseudo} \setminus \mathscr{F}_\mathrm{deep}| \in \{1,\aleph_0\}$.
\end{prop}
\begin{proof}
Because $\weakdeg(L) \in \mathscr{F}_\mathrm{pseudo} \setminus \mathscr{F}_\mathrm{deep}$, we know $1 \leq |\mathscr{F}_\mathrm{pseudo} \setminus \mathscr{F}_\mathrm{deep}|$. $\mathcal{E}_\weak$ is countable, so $|\mathscr{F}_\mathrm{pseudo} \setminus \mathscr{F}_\mathrm{deep}| \leq \aleph_0$. 

If $1 < |\mathscr{F}_\mathrm{pseudo} \setminus \mathscr{F}_\mathrm{deep}| < \aleph_0$, then $\mathscr{F}_\mathrm{pseudo} \setminus (\mathscr{F}_\mathrm{deep} \cup \{\weakdeg(L)\})$ has a minimal element, and such a minimal element is a minimal element of $\mathscr{F}_\mathrm{pseudo} \setminus \{\weakdeg(L)\}$ since $\mathscr{F}_\mathrm{deep}$ is upward-closed, contradicting \cref{no minimal pseudo-deep degree unequal to L}.
\end{proof}

Currently, the only two pseudo-deep degrees known to not be deep are $\weakdeg(L)$ and $\weakdeg(\ldnr_\slow)$, though they are not known to be distinct.

\begin{question} \label{L and ldnr_slow}
Are $L$ and $\ldnr_\slow$ weakly equivalent?
\end{question}

Something slightly stronger than asking whether $L \weakeq \ldnr_\slow$ is the following.

\begin{question} \label{L and ldnr_slow strong}
Given a deep $\Pi^0_1$ class $P$, does there exist a slow-growing order function $p\colon \mathbb{N} \to (1,\infty)$ such that $\ldnr(p) \weakleq P$?
\end{question}

\begin{prop}
\mbox{}
\begin{enumerate}[(a)]
\item An affirmative answer to \cref{L and ldnr_slow strong} gives an affirmative answer to \cref{L and ldnr_slow}.
\item An affirmative answer to \cref{L and ldnr_slow strong} gives an affirmative answer to \cref{slow-growing ldnr downwards-directed question}, i.e., for all slow-growing order functions $p\colon \mathbb{N} \to (1,\infty)$ and $q\colon \mathbb{N} \to (1,\infty)$ there exists a slow-growing order function $r\colon \mathbb{N} \to (1,\infty)$ such that $\ldnr(r) \weakleq \ldnr(p) \cup \ldnr(q)$.
\item An answer to \cref{size of gap between pseudo-deep and deep} of `$1$' gives an affirmative answer to \cref{L and ldnr_slow}.
\end{enumerate}
\end{prop}

An answer to \cref{size of gap between pseudo-deep and deep} of `$\aleph_0$' suggests further structural questions about antichains in $\mathscr{F}_\mathrm{pseudo} \setminus \mathscr{F}_\mathrm{deep}$ and related properties.

\begin{question} \mbox{}
\begin{enumerate}[(a)]
\item Do there exist weakly incomparable elements of $\mathscr{F}_\mathrm{pseudo} \setminus \mathscr{F}_\mathrm{deep}$? 
\item Do there exist infinitely many pairwise weakly incomparable elements of $\mathscr{F}_\mathrm{pseudo} \setminus \mathscr{F}_\mathrm{deep}$?
\item Is $\weakdeg(L)$ meet-irreducible? I.e., are there no pseudo-deep degrees $\mathbf{p},\mathbf{q}$ such that $\inf\{\mathbf{p},\mathbf{q}\} = \weakdeg(L)$?
\end{enumerate}
\end{question}

In contrast, if $|\mathscr{F}_\mathrm{pseudo} \setminus \mathscr{F}_\mathrm{deep}| = \aleph_0$, then $\mathscr{F}_\mathrm{pseudo} \setminus \mathscr{F}_\mathrm{deep}$ contains infinite chains.

One possible approach to answering \cref{size of gap between pseudo-deep and deep} would be by showing that $\shiftcomplex$ is not of deep degree, as it is of pseudo-deep degree and we know that $\shiftcomplex \weaknleq \ldnr_\slow$ by \cref{SC not weakly below ldnr_slow}. 

We observe that the known lattice theoretic properties available are not enough to determine the structure of $\mathscr{F}_\mathrm{pseudo} \setminus \mathscr{F}_\mathrm{deep}$. 

\begin{remark}
Some lattices $\langle P,\leq \rangle$ and filters $\mathscr{F}_1 \subsetneq \mathscr{F}_2 \subseteq P$ which give some idea of how the boundary between $\mathscr{F}_\mathrm{pseudo}$ and $\mathscr{F}_\mathrm{deep}$ might look include the following:
\begin{itemize}
\item Consider the lattice $\langle P,\leq \rangle \coloneq \langle \{ S \subseteq [0,1] \mid |S| \leq \aleph_0\} \cup \{[0,1]\},\supseteq \rangle$ and the filters $\mathscr{F}_1 \coloneq \{ S \subseteq [0,1] \mid |S| < \aleph_0\}$ and $\mathscr{F}_2 = P$. In this case, $|\mathscr{F}_2 \setminus \mathscr{F}_1| = \aleph_0$ and the minimum of $\mathscr{F}_2$, $[0,1]$, is meet-irreducible.

\item Consider the lattice $\langle P,\leq \rangle \coloneq \langle \{ S \subseteq \mathbb{N} \mid |S| < \aleph_0 \vee |\mathbb{N} \setminus S| < \aleph_0\}, \supseteq \rangle$ and the filters $\mathscr{F}_1 \coloneq \{ S \subseteq \mathbb{N} \mid |S| < \aleph_0\}$ and $\mathscr{F}_2 = P$. In this case, $|\mathscr{F}_2 \setminus \mathscr{F}_1| = \aleph_0$ and the minimum of $\mathscr{F}_2$, $\mathbb{N}$, is not meet-irreducible. 
\end{itemize}
\end{remark}

\clearpage
\bibliographystyle{amsplain}
\bibliography{biblio}

\end{document}